\documentclass[12pt,twoside,reqno]{amsart}

\usepackage{amsmath, amssymb, amsthm, enumitem, esint}
\usepackage[normalem]{ulem}
\usepackage{soul}
\usepackage{xcolor}
\usepackage[hyperindex]{hyperref}
\usepackage{mathrsfs}

\usepackage{hyperref}
\hypersetup{bookmarksdepth=3}

\newcommand{\lto}{\longrightarrow}
\newcommand{\lmapsto}{\longmapsto}

\newcommand{\cyl}{\textnormal{cyl}}

\newcommand{\sV}{\mathscr{V}}
\newcommand{\IR}{\mathbb{R}}
\newcommand{\lb}{\linebreak[1]}

\newcommand{\IB}{\mathbb{B}}
\newcommand{\IS}{\mathbb{S}}
\newcommand{\IZ}{\mathbb{Z}}

\newcommand{\YY}{\mathcal{Y}}
\newcommand{\Qu}{\mathsf{Q}}
\newcommand{\DD}{\mathcal{D}}

\newcommand{\PP}{\mathcal{P}}
\newcommand{\UU}{\mathcal{U}}

\newcommand{\MM}{\mathcal{M}}

\newcommand{\bY}{\mathbf{Y}}
\newcommand{\bO}{\mathbf{0}}
\newcommand{\bq}{\mathbf{q}}
\newcommand{\bp}{\mathbf{p}}

\newcommand{\bz}{\mathbf{z}}
\newcommand{\by}{\mathbf{y}}
\newcommand{\bx}{\mathbf{x}}
\newcommand{\be}{\mathbf{e}}
\newcommand{\fp}{\mathfrak{p}}

\newcommand{\eps}{\varepsilon}
\newcommand{\la}{\lambda}

\newcommand{\ov}[1]{\overline{#1}}
\newcommand{\td}[1]{\widetilde{#1}}

\DeclareMathOperator*{\osc}{osc}
\DeclareMathOperator{\sym}{sym}
\DeclareMathOperator{\spt}{spt}
\DeclareMathOperator{\Jac}{Jac}
\DeclareMathOperator{\rot}{rot}

\DeclareMathOperator{\Int}{Int}

\DeclareMathOperator{\tr}{tr}
\DeclareMathOperator{\id}{id}

\DeclareMathOperator{\diam}{diam}
\DeclareMathOperator{\bowl}{bowl}

\DeclareMathOperator{\rank}{rank}

\DeclareMathOperator{\proj}{proj}

\DeclareMathOperator{\reg}{reg}
\DeclareMathOperator{\sing}{sing}

\newcommand{\EMPTY}[1]{}

\newtheorem{Theorem}[equation]{Theorem}
\newtheorem{Lemma}[equation]{Lemma}
\newtheorem{Corollary}[equation]{Corollary}
\newtheorem{Proposition}[equation]{Proposition}
\newtheorem{Claim}[equation]{Claim}

\theoremstyle{definition}
\newtheorem{Definition}[equation]{Definition}
\theoremstyle{remark}
\newtheorem{Remark}[equation]{Remark}

\numberwithin{equation}{section}

\title[The PDE-ODI principle]{The PDE-ODI principle and \\ cylindrical mean curvature flows}
\author{Richard H Bamler and Yi Lai}
\address{Department of Mathematics, UC Berkeley, CA 94720, USA}
\email{rbamler@berkeley.edu}
\address{Department of Mathematics, UC Irvine, CA 92697, USA}
\email{ylai25@uci.edu}
\thanks{R.B. was supported by NSF grant DMS-2204364.
Y.L. was supported by NSF grant DMS-2506832.
This material is based upon work conducted at the Simons Laufer Mathematical Sciences Research Institute in Berkeley, California, during the Fall semester of 2024.}
\date{\today}

\begin{document}

\begin{abstract}
We introduce a new approach for analyzing {ancient solutions and singularities of mean curvature flow that are locally modeled on a cylinder}.
Its key ingredient is a general mechanism, called the \emph{PDE--ODI principle}, which converts a broad class of parabolic differential equations into systems of ordinary differential inequalities.
This principle bypasses many delicate analytic estimates used in previous work, and yields asymptotic expansions to arbitrarily high order.

As an application, we establish the uniqueness of the bowl soliton times a Euclidean factor among ancient, cylindrical flows with dominant linear mode.
This extends previous results on this problem to the most general setting and is made possible by the stronger asymptotic control provided by our analysis.
In the other case, when the quadratic mode dominates, we obtain a complete asymptotic expansion to arbitrary polynomial order, which will form the basis for a subsequent paper.
Our framework also recovers and unifies several classical results.
In particular, we give new proofs of the uniqueness of tangent flows (due to Colding-Minicozzi) and the rigidity of cylinders among shrinkers (due to Colding-Ilmanen-Minicozzi) by reducing both problems to a single ordinary differential inequality, without using the {\L}ojasiewicz-Simon inequality.

Our approach is independent of prior work and the paper is largely self-contained.
\end{abstract}

\maketitle

\tableofcontents

\section{Introduction}
\subsection{Overview}
In recent years, there has been growing interest mean curvature flows that are locally nearly cylindrical---such as flows forming cylindrical singularities or ancient solutions that arise as singularity models.
An early milestone in this direction is due to Colding-Minicozzi~\cite{colding_minicozzi_uniqueness_blowups}, who proved the uniqueness of cylindrical tangent flows.
Subsequent work of Angenent-Daskalopoulos-\v{S}esum \cite{ADS_2019} and Brendle-Choi~\cite{Brendle_Choi_2019} introduced powerful new techniques for analyzing ancient, asymptotically cylindrical flows.
These advances sparked a broader study of almost cylindrical flows and led to major breakthroughs, including the resolution of the Mean Convex Neighborhood Conjecture for 1-cylinders by Choi-Haslhofer-Hershkovits-White~\cite{Choi_Haslhofer_Hershkovits_2022, Choi_Haslhofer_Hershkovits_White_22}.
Despite this progress, many important questions have remained open, such as the complete classification of ancient, asymptotically cylindrical flows and, most importantly, the \emph{Mean Convex Neighborhood Conjecture} in the general case.
One motivation for this paper (and its sequel \cite{Bamler_Lai_MCF2}) is to resolve both of these problems.

From an analytic viewpoint, the main difficulty---one that recurs throughout this area and has been tackled in various ways in related works---is that the flow is only \emph{locally} close to a cylindrical model.
Consequently, one cannot work with a global graphical representation over a fixed model, as the geometry allows such a representation only on portions of the hypersurface.
Any effective approach must therefore begin with a careful localization step.
This step is delicate, because the linearized operator at the cylinder has kernel elements and unstable modes of unbounded growth.
Many existing approaches---impressive as they are---exploit fine structural features of the mean curvature flow equation and have so far been limited to the non-collapsed and convex or 1-cylindrical settings.
These limitations are precisely what have obstructed progress on the remaining conjectures.

\medskip

In this paper, we introduce a new method, called the \emph{PDE-ODI principle,} which overcomes these obstacles.
Rather than relying on delicate, problem-specific analytic estimates, it applies \emph{to a broad class of parabolic PDEs.}
The principle converts such PDEs into systems of ordinary differential inequalities, from which \emph{high-order asymptotics} for the original PDE can be extracted using standard ODE arguments.

To illustrate the strength of our method, we apply it to asymptotically cylindrical mean curvature flows and establish the uniqueness of the bowl soliton times a Euclidean factor in full generality.
Such flows follow a dichotomy~\cite{ADS_2019, Du_Haslhofer_2023, Du_Haslhofer_2024, Du_Zhu_2025}, which we also recover within  our framework: %
their asymptotics may either be governed by a dominant linear (unstable) mode or by a dominant quadratic (neutral) mode.

Our first main result shows that the only possible flow in the dominant linear mode case is a  bowl soliton times a Euclidean factor.
This completes a long line of earlier work by
Brendle-Choi~\cite{Brendle_Choi_2019} (in the 3-dimensional, convex case), Choi-Haslhofer-Hershkovits-White \cite{Choi_Haslhofer_Hershkovits_White_22} (in the 1-cylindrical case) and
Du-Zhu \cite{Du_Zhu_2025}, see also \cite{ChoiHaslhoferHershkovits2024} (in the non-collapsed and therefore convex case).

Our second main result concerns the second case, in which the quadratic mode is dominant.
In this case, we obtain a detailed asymptotic description of up to arbitrary polynomial order, which can be parameterized by an asymptotic invariant, which is a non-negative definite matrix. 
We study this invariant in more detail: we establish key continuity and convergence properties and we show that every possible value of this invariant is realized by an ancient oval. 
These results form the foundation for a complete classification of ancient asymptotically cylindrical flows in our sequel paper~\cite{Bamler_Lai_MCF2}, thereby completing and unifying a long line of work by many authors in this area~\cite{ADS_2019, Brendle_Choi_2019, ADS_2020,Brendle_Choi_higher_dim,zhu2022so,ChoiHaslhoferHershkovits2023,DuHaslhofer2021,Choi_Haslhofer_Hershkovits_2022,Choi_Haslhofer_Hershkovits_White_22,  CHH_2023,Du_Haslhofer_2023,ChoiHaslhofer2024,ChoiHaslhoferHershkovits2024,Du_Haslhofer_2024,ChoiDuZhu2025,Du_Zhu_2025,ChoiDaskalopoulosDuHaslhoferSesum2025,ChoiHaslhoferHershkovits2025, ADS_2025,  CHH_2025, Choi_Haslhofer_2025}.

Finally, our PDE–ODI principle framework streamlines the analysis of many asymptotic problems for ancient solutions and singularities of mean curvature flows that are locally cylindrical. 
As concrete examples, we rederive the uniqueness of tangent flows established by Colding–Minicozzi~\cite{Colding_Minicozzi_12} by reducing it to an ordinary differential inequality---without having to use a {\L}ojasiewicz-Simon inequality.
Similarly, we recover the rigidity of cylinders among shrinkers due to Colding–Ilmanen–Minicozzi~\cite{Colding_Ilmanen_Minicozzi}.

\medskip

We emphasize that our proofs are independent of prior work and the paper is largely self-contained.

\medskip
A crucial aspect of our work is that the PDE–ODI principle yields substantially sharper asymptotic control than previously available, capturing arbitrarily many modes to arbitrarily high order. 
It is helpful to recall why such precision is essential for establishing uniqueness in the dominant linear mode case. 
Such uniqueness arguments typically proceed in two steps: first, one analyzes the nearly cylindrical region using perturbative methods to obtain an asymptotic description; second, one must control the non-cylindrical part of the flow. 
The difficulty arises in the second step, because the geometry away from the cylindrical region rarely resembles an explicit model, leaving the comparison principle as the main tool. 
Applying the comparison principle, however, requires a closeness estimate to the model solution that decays sufficiently fast in both time and space---precisely the type of estimate that has been challenging to obtain in full generality. Our framework provides exactly this, via high-order asymptotic expansions that are uniform in time and have precise spatial control, which renders the comparison-principle step straightforward and leads to a clean uniqueness proof in the linear-mode case.

A final complication in deriving the flow's asymptotics comes from neutral and unstable modes associated with ambient Euclidean transformations. 
Since these modes are not geometrically significant, they must often be removed through a suitable gauging procedure---a recurring challenge in the field that has been addressed in various ways in prior works. 
In our framework, the PDE–ODI principle naturally accommodates these modes, allowing them to be eliminated via a relatively straightforward gauging procedure. 
Consequently, this step is considerably simpler than in most previous approaches.
\medskip

\subsection{The PDE-ODI principle} \label{subsec_intro_pdeodi}
The key estimate of this paper is the \emph{PDE-ODI principle}.
This principle makes precise---in the strongest form---an intuition that has guided much of the field over the past decade.
It transforms certain classes of parabolic PDEs into systems of \emph{ordinary differential inequalities (ODIs)}, which can often be analyzed more directly using standard ODE techniques.
Remarkably, the PDE-ODI principle applies to a broad class of equations and requires only a relatively coarse pseudolocality property, a feature shared by many geometric flows.

We now outline the underlying intuition.
Consider a solution to a nonlinear parabolic equation of the form
\begin{equation}\label{eq:general_PDE}
    \partial_\tau u = Lu + Q(u,\nabla u, \nabla^2 u). 
\end{equation}
Here $Q(u,\nabla u, \nabla^2 u)$ is a non-linear term of at least quadratic order, and the linear part is given by an operator of the following form for the potential $f = \frac14 |\bx|^2$: 
\[ L u = \triangle_f u + A u = \triangle u - \nabla f \cdot \nabla u + Au, \]
where $\triangle_f$ is the Ornstein-Uhlenbeck operator and  $A$ is an arbitrary linear map.
Equations of this type frequently arise in geometric analysis after a Type~I rescaling---for instance, in our setting $u$ represents a perturbation of the standard round cylinder evolving under the rescaled mean curvature flow.

We assume that $u$ is small in a sufficiently large neighborhood of the origin.
Crucially, however, we impose no bounds on $u$ outside this region, and we don't even require the solution to be defined on the entire domain.
To study the dynamics of $u$, one typically considers the eigenspace decomposition of the linearized operator within the $L^2_f$-space, with a Gaussian weight $e^{-f}$; these are often spanned by Hermite polynomials.
One may then wish to study the projection $U^+_\tau$ of $u_\tau$ onto the finite-dimensional direct sum $\sV_{> \la}$ of eigenspaces with eigenvalues greater than some fixed threshold $\la$.
Formally, one expects this projection to satisfy an approximate ODE of the form
\[ \partial_\tau U^+_\tau \approx L U^+_\tau + Q_J^+(U^+_\tau) + O(\Vert U^+_\tau \Vert^{J+1}). \]
Here $Q_J^+$ can roughly be viewed as the $J$-th Taylor approximation of the non-linear term restricted and projected to $\sV_{> \la}$, for some fixed $J \geq 2$.

Making this intuition rigorous has traditionally been challenging.
Past approaches relied on specific structural features of the governing PDE and often required restrictive assumptions.
One key difficulty is that $u_\tau$ may grow rapidly at spatial infinity or may not even be globally defined.
This usually requires a localization step, which is typically accomplished by multiplication with a cutoff function $\omega_{R(\tau)}$ supported in a large ball of radius $R(\tau) \gg 1$ before projection:
\[
U^+_\tau = \PP_{\sV_{> \la}}  (u_\tau  \omega_{R(\tau)}).
\]
However, a localization requires control over $u_\tau$ in the transition region $\{0 < \omega_{R(\tau)} < 1\}$ to estimate the flux between the interior region $\{\omega_{R(\tau)}=1\}$ and the exterior region $\{\omega_{R(\tau)}=0\}$, where $u_\tau$ may become large.
In the earlier work, \cite{ADS_2019,Brendle_Choi_2019}, the localization step was carried out via an ``inverse Poincar\'e inequality'' derived from barrier constructions and the shrinker foliation.
In \cite{colding_minicozzi_uniqueness_blowups}, it relied on an entropy bound and the {\L}ojasiewicz-Simon inequality.
While these methods are highly effective for the goals of these papers, they make essential use of the gradient–flow structure of mean curvature flow, and the asymptotic expansions they produce do not extend to arbitrarily high order.

In this paper, we show that these difficulties can be overcome \emph{in a very general setting}---and the heuristic picture described above can be made {fully} rigorous---once the problem is placed in the right framework.
Our method only requires a relatively coarse pseudolocality property of the original PDE.
Ordinarily, pseudolocality asserts that a local bound over a sufficiently large ball yields control at later times on a smaller ball.
For PDEs of the form \eqref{eq:general_PDE}, which typically arise via rescaling, it is natural for the future ball to be \emph{larger} than the initial one---a well known phenomenon that has been exploited in various settings \cite{colding_minicozzi_uniqueness_blowups, Colding_Minicozzi_RF_2025}. 
This phenomenon also plays a central role in our work, but we have to use it in a more subtle way: to achieve approximations of arbitrarily high order, we have to adjust the size of the approximating ball carefully according to the desired precision of the approximation.

To describe our result, let $\sV_{\leq \lambda}$ denote the orthogonal complement of $\sV_{> \lambda}$ and decompose
$$u_\tau \omega_{R(\tau)} = U^+_\tau + U^-_\tau \in \sV_{> \la} \oplus \sV_{\leq \la}.$$
We moreover introduce the following error term, for some small and fixed $\eps, { \eta} > 0$,
\[ \UU^-_\tau := \Vert U^-_\tau \Vert_{L^2_f} + {\eta}\, e^{-\frac{((1-\eps)R(\tau))^2}{8}}. \]
So $\UU^-_\tau$ is small if and only if the approximation $u_\tau \approx U^+_\tau$ holds up to a small integral error over a large ball. 
In the simplest case, our main result can be roughly summarized as follows (see Theorem~\ref{Thm_PDE_ODI_principle} for more details):

\begin{Theorem}[PDE-ODI principle, vague form]\label{thm:PDE-ODI_vauge}
There are suitable choices of radii $R(\tau)$, depending smoothly on time, such that we have a pointwise bound {$\Vert u \Vert_{C^m(\IB^n_{R(\tau)})} \leq \eta$ on the $R(\tau)$-balls and such that} the following ODIs hold:
\begin{align*}
 \Vert \partial_\tau U^+_\tau - L U^+_\tau - Q_J^+(U^+_\tau) \Vert &\leq C \Vert U^+_\tau \Vert^{J+1} + {C\sqrt\eta}\, \UU^-_\tau \\
 \partial_\tau \UU^-_\tau &\leq (\la+  {C\sqrt\eta})\, \UU^-_\tau + \Vert U^+_\tau \Vert^{J+1} + \Vert Q_J^- (U^+_\tau) \Vert 
\end{align*}
\end{Theorem}

If $\la$ is chosen sufficiently negative, as is often the case, then the second ODI expresses that the error term $\UU^-_\tau$ must roughly decay exponentially in time until it becomes comparable to the inhomogeneous term $\Vert U^+_\tau \Vert^{J+1} + \Vert Q_J^- (U^+_\tau) \Vert$.
Here $Q_J^-(U^+_\tau)$ is the projection of the Taylor expansion of non-linear part of \eqref{eq:general_PDE} onto the subspace $\sV_{\leq \la}$, and is typically\footnote{We are suppressing some details here. The term $Q_J^-(U^+_\tau)$ is generally only bounded by a sum of powers of different components of $U^+_\tau$, where the exponent depends on the component.
Components associated with lower exponents are, however, expected to be controlled by a higher power of $\Vert U^+_\tau \Vert$, which can be established a posteriori through an inductive argument.
Doing so requires a slight reordering of the logical steps in the ODI analysis, but the resulting argument is still standard.}
bounded by a high power of $\Vert U^+_\tau \Vert$.
So after some time, we must have control of the form $\UU^-_\tau \lesssim \Vert U^+_\tau \Vert^{J'}$ for some large $J' \leq J$, and plugging this back into the first ODI implies
\begin{equation} \label{eq_better_Up}
\Vert \partial_\tau U^+_\tau - L U^+_\tau - Q_J^+(U^+_\tau) \Vert \leq C \Vert U^+_\tau \Vert^{J'}. 
\end{equation}
Further analysis of this ODI often leads to a detailed understanding of the asymptotic behavior of $U^+_\tau$ and therefore $u_\tau$.
This is often achieved by analyzing the leading component of $U^+_\tau$ and showing that it---and therefore the norm $\Vert U^+_\tau \Vert$---decays exponentially or polynomially in time.
Consequently, one may often replace the right-hand side of \eqref{eq_better_Up} with a term of the form $Ce^{-J'' |\tau|}$ or $C|\tau|^{-J''}$ for large $J''$, which leads to a precise asymptotic expansion of $U^+_\tau$ to essentially the same order.

Finally, we remark that we also establish a more general version of the PDE-ODI principle, applicable to equations of the form
\[  \partial_\tau u = Lu + Q(u,\nabla u, \nabla^2 u, \bY_\tau),  \]
where $\bY_\tau$ is a finite-dimensional parameter serving as a global input that allows one to influence the flow.
Such parameters are often used to eliminate geometrically insignificant modes via a gauging procedure.
For instance, in the mean curvature flow, $\mathbf{Y}_\tau$ represents a family of ambient Killing fields that can be used to modify the flow and that can be switched on or off at will.
Since this input term may introduce unbounded contributions, it complicates our analysis and requires, among other things, a careful adaptation of our pseudolocality assumption.
For more details, we refer to Sections~\ref{sec_pseudolocality} and \ref{sec_PDE_ODI}.

\medskip
\subsection{Application to almost cylindrical mean curvature flow}
Let us now summarize the results of this paper pertaining to almost cylindrical mean curvature flows.
We will only present the most important results in a simplified form and refer to the detailed theorems in the body of the paper, in Sections~\ref{sec_mode_analysis}--\ref{sec_dom_quadratic}, where they are often accompanied by further theorems, corollaries and explanations.

The first result makes rigorous the dichotomy between flows with a dominant linear mode and those with a dominant quadratic mode.
While this dichotomy already appears in \cite{ADS_2019, Du_Haslhofer_2023, Du_Haslhofer_2024, Du_Zhu_2025}, our theorem goes substantially further: 
we extract a finite-dimensional invariant~$\Qu$, which quantifies the third order asymptotics more precisely in the dominant quadratic case.
Since the invariant vanishes if and only if the we have dominance of the linear mode, it can moreover be used as a clean criterion for distinguishing the two cases. 
The following theorem summarizes results, which are explained in more detail in the beginnings of Sections~\ref{sec_dom_lin} and \ref{sec_dom_quadratic}.

\begin{Theorem}[summary]\label{thm:dichotoy}
Let $\MM$ be an \emph{asymptotically $(n,k)$-cylindrical flow,} that is an ancient mean curvature flow in $\IR^{n+1} \times (-\infty, T)$ such that $(-t)^{-1/2} \MM_t$ converges to a cylinder of the form $\IR^k \times \IS^{n-k}$ as $t \to -\infty$.
Then the asymptotic behavior of the rescaled flow $\td\MM_\tau := e^{\tau/2} \MM_{-e^{-\tau}}$ gives rise to an invariant, which is a symmetric, non-negative definite matrix $\Qu(\MM) \in \IR^{k \times k}_{\geq 0}$ with the following properties:
\begin{enumerate}[label=(\alph*)]
\item If $\Qu(\MM) = 0$, then the rescaled flow converges to the cylinder at rate $\lesssim e^{\tau/2}$ as $\tau \to \-\infty$ and the dominant mode is linear. 
\item If $\Qu(\MM) \neq 0$, then the same convergence occurs at rate $|\tau|^{-1}$  and the dominant mode is a quadratic Hermite polynomial whose asymptotic behavior modulo $O(|\tau|^{-3})$ determines $\Qu(\MM)$.
\end{enumerate}
\end{Theorem}

In the dominant linear mode case we obtain a full classification:

\begin{Theorem}[Theorem~\ref{Thm_bowl_unique}]\label{thm:unique_of_bowl}
If $\Qu(\MM) = 0$, then $\MM$ is isometric to a parabolic rescaling of the round shrinking cylinder $\MM_{\cyl}^{n,k}$ or to $\IR^{k-1} \times \MM^{n-k+1}_{\bowl}$, where the second factor denotes the (n-k+1)-dimensional bowl soliton.
\end{Theorem}

In the dominant quadratic mode case, we obtain further characterizations of the flow, which will form the basis of a more complete analysis in this case in subsequent work \cite{Bamler_Lai_MCF2}.

The next result shows that $\Qu(\MM)$ uniquely determines the asymptotics of the convergence of the rescaled flow to the round cylinder up to \emph{arbitrary high polynomial order.}

\begin{Theorem}[Proposition~\ref{Prop_same_Q_close}] \label{thm:quadratic_polynomial}
If two asymptotically $(n,k)$-cylindrical flows $\MM^0$ and $\MM^1$ satisfy $\Qu(\MM^1) = \Qu(\MM^2)$, then for every $A \geq 1$ and $\tau \ll 0$, the rescaled flows $\td\MM^i_\tau := e^{\tau/2} \MM^i_{-e^\tau}$ are $C |\tau|^{-A}$-close in a ball around the origin of radius $A \sqrt{\log |\tau|}$.
\end{Theorem}

On the other hand we show that each admissible value of $\Qu$ can be realized by an ancient oval (see also \cite{White_03, DuHaslhofer2021, HaslhoferHershkovits2016} for different parameterizations).

\begin{Theorem}[Theorem~\ref{Thm_existence_oval}]\label{thm:oval_existence}
For any symmetric, non-negative definite matrix $\Qu' \in \IR^{k \times k}_{\geq 0}$, there exists an asymptotically $(n, k)$-cylindrical flow $\MM$ with $\Qu(\MM)=\Qu'$ satisfying the properties: $\MM$ is non-collapsed, convex and rotationally symmetric about the axis $\IR^{k} \times \bO^{n-k+1}$, and invariant under reflections perpendicular to all spectral directions of $\Qu(\MM)$. Moreover, if $\Qu(\MM)$ has non-trivial $l$-dimensional nullspace, then $\MM$ splits as  compact asymptotically $(n-l,k-l)$-cylindrical flow times an $\IR^l$-factor in the nullspace direction; if $l=k$, then $\MM$ is the round shrinking cylinder.
\end{Theorem}

We also establish several continuity and convergence properties involving $\Qu$.
First, we show that $\Qu$ is continuous under Brakke convergence:

\begin{Theorem}[Proposition~\ref{Prop_Q_continuous}\ref{Prop_Q_continuous_a}] \label{thm:Q_continuous}
The invariant $\Qu$ is continuous with respect to Brakke convergence: If $\MM^i \to \MM^\infty$ in the Brakke sense and if all flows and their limits are asymptotically $(n,k)$-cylindrical, then $\Qu(\MM^i) \to \Qu(\MM^\infty)$.
\end{Theorem}

As a corollary, we obtain the following properness-type result, which, allows us to conclude that non-trivial Brakke limits are asymptotically cylindrical if $\Qu$ is bounded.

\begin{Corollary}[Proposition~\ref{Prop_Q_continuous}\ref{Prop_Q_continuous_b}, \ref{Prop_Q_continuous_c}]\label{cor:limit_Q_0}
Consider a sequence of asymptotically $(n,k)$-cylindrical flows $\MM^i$ and suppose that the norm $\Vert \Qu(\MM^i)\Vert$ is uniformly bounded.
Then, after passing to a subsequence, we have $\MM^i \to \MM^\infty$ where the limit is one of the following:
\begin{itemize}
\item another $(n,k)$-cylindrical flow,
\item an affine plane, or
\item empty.
\end{itemize}
If, in addition $\Qu(\MM^i) \to 0$, then the limit is either isometric to a round shrinking cylinder, a bowl times a Euclidean factor, an affine plane, or empty.
\end{Corollary}

We remark that $\Qu$ is invariant under translation in space and time and its behavior under parabolic rescaling can be characterized by $\Qu(\la \MM) = \la \Qu(\MM)$.
So Corollary~\ref{cor:limit_Q_0} implies that at large scales, asymptotically cylindrical flows can be modeled on cylinders or $\IR^{k-1} \times \MM^{n-k+1}_{\bowl}$.

\begin{Corollary}\label{cor:blow-down}
If $\MM$ is asymptotically cylindrical and consider a sequence of points and times $(\bp_i, t_i) \in \spt \MM$ and scales $\la_i \to 0$.
Then, after passing to a subsequence, we have $\la_i ( \MM - (\bp_i,t_i)) \to \MM^\infty$, which must be isometric to a round shrinking cylinder, a bowl times a Euclidean factor, or an affine plane.
\end{Corollary}

In summary, Theorems~\ref{thm:quadratic_polynomial},  \ref{thm:oval_existence} and Corollary~\ref{cor:blow-down}  give the following partial asymptotic picture:
the flow it is close to an ancient oval with the same $\Qu(\MM)$, up to arbitrarily high polynomial order and in a large neighborhood of the origin, and close to a round cylinder or $\IR^{k-1} \times \MM^{n-k+1}_{\bowl}$ in regions where the curvature is small.
As we will see in \cite{Bamler_Lai_MCF2}, this covers all regions except for the ``cap-region'' in a flying wing soliton.

\medskip

In addition, we also reprove existing results on almost cylindrical flows.
The novelty here is that in our framework, these results reduce to a rather elementary analysis of a system of ODIs.
Recall that a point $\bp_0 \in \IR^{n+1+n'}$ is called \emph{center of a $(n,k,\delta)$-neck at scale $r > 0$ and time $t_0$} of a mean curvature flow $\MM$ if there is an orthogonal map $S \in O(n+1+n')$ such that $r^{-1}S(\MM_{t_0} - \bp_0)$ is $\delta$-close to the standard round cylinder $\IR^k \times \IS^{n-k}$.

The first result is the following rigidity result of cylinders within the space of shrinkers, originally due to Colding-Ilmanen-Minicozzi~\cite{Colding_Ilmanen_Minicozzi}:

\begin{Theorem}[Corollary~\ref{Cor_eternal}]
There is a dimensional constant $\delta > 0$ with the following property.
Let $M \subset \IR^{n+1+n'}$ be an $n$-dimensional shrinker and assume that $\bO$ is a center of an $(n,k,\delta)$-neck at scale~$1$.
Then $M$ is a round cylinder.
\end{Theorem}

The second result is the following stability theorem for $\delta$-necks.
This was originally proved by Colding-Minicozzi \cite{colding_minicozzi_uniqueness_blowups} (see also \cite{Gianniotis_Haslhofer_2020} for an adaptation to this specific form).

\begin{Theorem}[Theorem~\ref{Thm_stability_necks}]\label{thm:unique_tangent_flow}
For every $\eps > 0$ there is an $\delta(\eps) > 0$ with the following property.

Let $\MM$ be an $n$-dimensional mean curvature flow of codimension $n'+1 \geq 1$ in $\IR^{n+1+n'}$ defined over a time-interval of the form $[-T_1, -T_2]$ for $T_1, T_2 > 0$.
Suppose that for $t \in \{-T_1, -T_2 \}$ the point $\bO$ is center of an $(n,k,\delta)$-neck at scale $\sqrt{-t}$ and at time $t$.
Then there is a \emph{uniform} orthogonal map $S \in O(n+1+n')$ such that $(-t)^{-1/2} S\MM_t$ is $\eps$-close to the standard cylinder \emph{for all} times $t \in [-T_1, -T_2]$.
\end{Theorem}

This theorem directly implies the uniqueness of cylindrical tangent flows (at a singular point and at $-\infty$); see Corollaries~\ref{Cor_unique_tangent}, \ref{Cor_unique_tangent_infinity}.
{See also \cite{Zhu25_CylindricalSelfShrinkers} and \cite{Ghosh2025Cylindrical} for further alternative proofs. 
For uniqueness of tangent flow on non-cylindrical shrinkers see, e.g., \cite{schulze2014,CS2021_uniqueness, Choi_Mantoulidis_22,lotay2022neck, stolarski2023structure,LZ2024_uniqueness}.
}

\medskip
\subsection{Structure of the paper}
In Section~\ref{sec_Preliminaries}, we introduce notions that will be frequently used throughout the paper, and review some basic facts in mean curvature flow.

In Section~\ref{sec_pseudolocality}, we introduce and motivate the pseudolocality property, which plays a central role in the PDE-ODI principle, and we establish such a property for mean curvature flow.
More precisely, we prove a strong pseudolocality statement for rescaled mean curvature flows and a weaker variant for rescaled \emph{modified} mean curvature flows, which allows an additional smooth family of ambient isometries.
In the conventional (unrescaled) setting, pseudolocality typically asserts that a local bound on a sufficiently large ball at one time yields control at later times on a smaller ball.
In the rescaled setting, however, a key observation for our purposes is that the conclusion may hold on a \emph{larger} ball than the one appearing in the assumption.

In Section \ref{sec_PDE_ODI}, we prove the general PDE-ODI principle for a broad class of PDEs satisfying a pseudolocality property (Theorem~\ref{thm:PDE-ODI_vauge}). We explicitly describe how we track the evolution of certain non-linear PDEs by a finite dimensional ODI.

In Section \ref{sec_mode_analysis}, we apply the PDE–ODI principle to the rescaled mean curvature flow equation, analyze the resulting ODI, and derive strong asymptotic estimates for its solutions.
A crucial preparatory step is the gauging procedure of Subsection \ref{subsec_gauge}, which removes geometrically insignificant neutral and unstable modes by modifying the flow via a family of ambient Euclidean motions. This yields a rescaled \emph{modified} mean curvature flow to which our PDE–ODI principle applies directly.
From this, we obtain asymptotic estimates to arbitrarily high polynomial order in the leading term.
In particular, we show that the Euclidean motions used in the gauging procedure are themselves controlled to arbitrarily high order.
Immediate consequences include classical results such as the dichotomy between a dominant linear mode and a dominant quadratic mode; the rigidity of cylinders among shrinkers, due to Colding–Ilmanen–Minicozzi \cite{Colding_Ilmanen_Minicozzi}; the stability of cylinders; and the uniqueness of cylindrical tangent flows (Theorem~\ref{thm:unique_tangent_flow}), due to Colding-Minicozzi \cite{colding_minicozzi_uniqueness_blowups}.
The same a posteriori control on the ambient Euclidean motions allows us to undo the gauging procedure carried out at the beginning of our construction.
As a result, we obtain the same asymptotic characterization as in the modified setting for the original rescaled \emph{unmodified} flow.
This will be the key result of this section and will play a central role in the remainder of the paper.

In Section \ref{sec_dom_lin}, we discuss the dominant linear mode case and prove the uniqueness of the bowl soliton times a Euclidean factor.
In this case the leading term decays exponentially, which yields asymptotic control of the flow to arbitrary \emph{exponential} order in time.
This implies that any flow with dominant linear mode can be approximated by a bowl soliton times a Euclidean factor, up to an arbitrary high polynomial order \emph{in space.}
Finally, we use a comparison principle to show that the two flows must be equal (Theorem~\ref{thm:unique_of_bowl}).

In Section \ref{sec_dom_quadratic}, we treat the dominant quadratic mode case and find an asymptotic invariant $\Qu$ (a non-negative definite matrix), which determines the behavior of the flow to arbitrary polynomial order.
Our analysis proceeds by modeling the flow using a finite collection of modes $U^{+}(\tau)$, which satisfy an ODE up to an error term that decays to arbitrarily high polynomial order.
In particular, we extract the invariant $\Qu$ from the third-order expansion of this ODE system.
We also further characterize this $\Qu$ and establish results that will be crucial for our sequel \cite{Bamler_Lai_MCF2}.
For example, we show the continuity of the invariant $\Qu$ under the Brakke convergence (Theorem \ref{thm:Q_continuous}) and characterize limits of sequences of asymptotically cylindrical flows depending on bounds on $\Qu$.
Finally, we use our methods, combined with a new topological framework, to construct  ancient ovals that realize every possible value of $\Qu$.

\subsection{Acknowledgements}
The authors want to thank Or Hershkovits, Christos Mantoulidis, Jingze Zhu and Jonathan Zhu for helpful comments and correction of typos.
\bigskip

\medskip
\section{Preliminaries}\label{sec_Preliminaries}
In this paper, we will often use the short form ``$X \leq C(a,b, \ldots)$'' to express ``there is a constant $C$, which can be chosen only depending on constants $a, b, \ldots$, such that $X \leq C$''.
We will usually fix dimensions and omit these in the list of dependencies.
So a statement of the form ``$X \leq C$'' means that the constant $C$ is dimensional or even universal.
Likewise, a statement of the form ``if $\alpha \leq \ov\alpha(a,b,\ldots)$, then \ldots'' means ``there is a constant $\ov\alpha$, which can be chosen depending only on $a,b,\ldots$ and dimensions, such that if $0 < \alpha \leq \ov\alpha$, then \ldots''.
Lower bounds are expressed similarly as ``$A \geq \underline A(a,b,\ldots)$''.
We also use the notation ``$\Psi(a | b)$'' or ``$\Psi(A | b)$'' to denote a generic function that satisfies $\lim_{a \to 0} \Psi(a|b) = 0$ or $\lim_{A \to \infty} \Psi(A|b) = 0$ (depending on the context) for fixed $b$.
\medskip

We will denote the standard, open $r$-ball in $\IR^n$ by $\IB^n_r$, where we sometimes omit the superscript if the context is clear.
Throughout this paper we will also fix a smooth \textbf{cutoff function} $\omega : \IR \to [0,1]$ with $\omega \equiv 1$ on $(-\infty,-0.2]$ and $\omega \equiv 0$ on $[-0.1,\infty)$.
For every $R \geq 1$, we define $\omega_R \in C^\infty(\IR^n)$ by
\[ \omega_R (x) := \omega(|x|- R). \]
So $\omega_R$ is supported on $\IB^n_R$ and $\omega_R \equiv 1$ on $\IB^n_{R-0.2}$.
\medskip

We will frequently consider objects defined on spacetimes of the form $\IR^{n+1+n'} \times I$, for a time-interval $I \subset \IR$.
We denote by $\mathbf t : \IR^{n+1+n'} \times I \to I$ the projection onto the second factor and will use a time-index to denote the restrictions of such objects to the corresponding time-slices.
For example, if $U \subset \IR^{n+1+n'} \times I$ is a subset and $u : U \to \IR$ is a function, then for any $t \in I$ we write $U_t := U \cap \mathbf t^{-1}(t)$ and denote by $u_t = u(\cdot, t) : U_t \to \IR$ restriction of $u$ to this time-slice; here we identify $\mathbf t^{-1}(t) = \IR^{n+1+n'} \times \{ t \} \cong \IR^{n+1+n'}$.
\medskip

Our main objects of study will be \textbf{$n$-dimensional integral Brakke flows,} which are given by a family of measures $(\mu_t)_{t \in I}$, parameterized by a time-interval $I$, on Euclidean space $\IR^{n+1+n'}$; here $1+n'$ is the codimension of the flow.
We will often simply write $\MM := (\mu_t)_{t \in I}$ and introduce $\MM$ as an ``$n$-dimensional integral Brakke flow in $\IR^{n+1+n'} \times I$'' without referencing the measures $\mu_t$.
The \textbf{support} $\spt \MM \subset \IR^{n+1+n'} \times I$ is defined as the closure of $\bigcup_{t \in I} (\spt \mu_t) \times \{t \}$ within $\IR^{n+1+n'} \times I$.
A point $(\bp_0,t_0) \in \spt \MM$ is called \textbf{regular} (otherwise \textbf{singular}) if there exists a smooth family of smooth $n$-dimensional submanifolds $N_t \subset \IR^{n+1+n'}$ for $t \in I$ close to $t_0$ such that for a neighborhood $U \subset \IR^{n+n'+1}$ of $\bp_0$, the intersection $U \cap N_t$ is properly embedded in $U$ and $\mu_t \lfloor U_t = \mathcal{H}^n \lfloor N_t$.
We denote by $\MM^{\reg}$ and $\MM^{\sing} \subset \spt \MM$ the sets of regular and singular points and say that $\MM$ is \textbf{regular on} a subset $U \subset \IR^{n+1+n'} \times I$ if $U \cap \MM^{\sing} = \emptyset$.
It is clear that $\MM^{\reg}$ is a smooth $(n+1)$-dimensional submanifold of $\IR^{n+1+n'} \times I$ such that $\mathbf t$ restricted to $\MM^{\reg}$ has no critical points and whose time-slices $\MM^{\reg}_t$ move by the mean curvature flow equation
\begin{equation} \label{eq_mcf_prelim}
 \partial_{t} \mathbf x = \mathbf H, 
\end{equation}
where the left-hand side denotes the normal velocity and the right-hand side the mean curvature vector of $\MM^{\reg}_t$.

We say that $\MM$ is \textbf{convex} if there is a $T \leq \infty$ (its \textbf{extinction time}) such that $(\spt \MM)_t = \emptyset$ for $t > T$ and $\emptyset \neq (\spt \MM)_{T} = \MM^{\sing}_T$ if $T < \infty$ and for $t < T$ the set $(\spt \MM)_t = \MM^{\reg}_t$ is the boundary of a closed, convex subset $C_t = \IR^{n+1}$ with non-empty interior.
Note that the time-slice $\MM^{\reg}_t$ must have non-negative definite second fundamental form for a suitable choice of co-orientation.

\medskip

We will use the same language as introduced before to characterize solutions to \emph{modified} versions of the mean curvature flow equation.
These will always arise from an integral Brakke flow $(\td\mu_t)_{t \in I}$ by a transformation identity of the form $\td\mu_\tau = (F_\tau)_* \mu_{\phi(\tau)}$, where $\phi :  \td I \to I$ is a smooth diffeomorphism reparameterizing time and $(F_\tau : \IR^{n+1+n'} \to \IR^{n+1+n'})_{\tau \in \td I}$ is a smooth family of homotheties (i.e., isometries and rescalings).
The resulting flows $(\td\mu_\tau)_{\tau \in \td I}$ will be called ``rescaled (modified) mean curvature flows'' and denoted by $\td\MM, \MM'$ or $\td\MM'$.
Their regular and singular points are defined in the same way as discussed in the previous paragraph and it is clear that $\td\MM^{\reg}_{\tau} = F_\tau(\MM^{\reg}_{\phi(\tau)})$; the same is true for $\td\MM^{\sing}$ and the support of $\td\MM$.
Obviously, $\td\MM^{\reg}_\tau$ evolves by a modified version of the evolution equation \eqref{eq_mcf_prelim}, which we will discuss in due time.
\medskip

Consider now an $n$-dimensional integral Brakke flow $\MM$ in $\IR^{n+1+n'} \times I$, given by a family of measures $(\mu_t)_{t \in I}$.
For $(\bp_0, t_0) \in \IR^{n+1+n'} \times I$, we define the \textbf{spacetime translation} $\MM + (\bp_0 , t_0)$ to be the flow given by the meausures $\mu'_{t'} := (T_{\bp_0})_*\mu_{t'-t_0}$, $t' \in I+t_0$, where $T_{\bx_0}(\bx) := \bx + \bp_0$. 
Similarly, for $\la > 0$ we define the \textbf{parabolic scaling} $\la \MM$ to be the flow given measures $\mu'_{t'} := \la^{-n} \la_* \mu_{t/\la^2}$, $t' \in \la^2 I$, where ``$\la_*$'' denotes push-forward under the spatial dilation by $\la$.
Lastly, if $S \in O(n+1+n')$ is an orthogonal map or $S \in E(n+1+n')$ is a Euclidean motion, then we may define $S\MM$ to be the flow given by the measures $\mu'_t := S_* \mu_t$.
It is clear that $(\MM + (\bp_0, t_0))^{\reg} = \MM^{\reg} + (\bp_0, t_0)$ and $(\la\MM)^{\reg} = \la\MM^{\reg}$ and $(S\MM)^{\reg} = S \MM^{\reg}$, where in the second identity the time-coordinate is scaled $\la^2$.
The analogous identities hold for the singular parts and the supports.
We call two flows $\MM', \MM$ \textbf{isometric} if $\MM' = S(\MM+(\bp_0, t_0))$ for a suitable $(\bp_0, t_0) \in \IR^{n+1+n'} \times I$ and $S \in O(n+1+n')$ and \textbf{homothetic} if $\MM'$ and $\la\MM$ are isometric for a suitable $\la > 0$.

A submanifold $M \subset \IR^{n+1}$ or Brakke flow $\MM$ in $\IR^{n+1} \times I$ is called \textbf{$(n,k)$-rotationally symmetric,} if it is invariant under orthogonal maps in $O(n-k+1)$ applied to the second factor of $\IR^{n+1} = \IR^k \times \IR^{n-k+1}$.
The subspace $\IR^k \cong \IR^k \times \bO^{n-k+1}$ is called the \textbf{axis of rotation.}
We will often omit the prefix ``$(n,k)$'' if the dimensions are clear from the context.

We say that $\MM$ has \textbf{uniformly bounded area ratios at scales $< r_0$} if
\[ \sup_{r \in (0,r_0)}  \sup_{(\bp,t) \in \IR^{n+1+n'} \times I} \frac{\mu_t (B(\bp,r))}{r^n} < \infty. \]
This is clearly the case if $\MM$ has bounded support.
If $r_0=\infty$, then we say that $\MM$ has \textbf{uniformly bounded area ratios at all scales.}
We recall the definition of the \textbf{Gaussian area} at $(\bp_0, t_0) \in \IR^{n+1+n'} \times \IR$ for $t \in I$, $t < t_0$
\[ \Theta^{\MM}_{(\bp_0, t_0)} (t_0 - t) := \int_{\IR^{n+1+n'}} \frac1{(4\pi (t_0 - t))^{n/2}} e^{-\frac{|\bx - \bp_0|^2}{4(t_0 - t)}} d\mu_t. \]
If $\MM$ has uniformly bounded area ratios at scales $< r_0$, then the Gaussian area is finite and non-decreasing in its time-parameter (see for example \cite[Proposition~5.15]{Schulze_intro_Brakke}\footnote{If we allow the Gaussian area to attain the value $\infty$, then monotonicity is even true if we remove the uniformly bounded area ratio assumption. This follows from a modification of the proof in \cite{{Schulze_intro_Brakke}}.}), so the limits $\Theta^{\MM}_{(\bp_0, t_0)} (0)$ and $\Theta^{\MM}_{(\bp_0, t_0)} (\infty)$ are both defined, assuming $t_0 > \inf I$.
We recall that $\Theta^{\MM}_{(\bp_0, t_0)} (0) > 0$ if and only if $(\bp_0, t_0) \in \spt \MM$, in which case we even have $\Theta^{\MM}_{(\bp_0, t_0)} (0) \geq 1$.
Moreover, if $(\bp_0, t_0) \in \MM^{\reg}$, then $\Theta^{\MM}_{(\bp_0, t_0)} (0) =1$.
We call $\MM$ \textbf{unit-regular} if the converse is true. 
Recall that Brakke's Regularity Theorem \cite{Brakke_1978} implies that this property is property is closed under Brakke convergence.
The same result shows that if a sequence of unit-regular integral Brakke flows converge in the Brakke sense, $\MM^i \to \MM^\infty$, then the convergence is locally smooth near every regular point of the limit $\MM^\infty$.
Moreover, if $U \subset \IR^{n+1+n'} \times I$ is a bounded open subset whose closure is disjoint from $\MM^{\infty, \sing}$, then $\MM^i$ is regular on $U$ for sufficiently large $i$ and the convergence on $U$ is locally smooth.
So, for example, if $\MM^\infty$ is smooth, then $\MM^i$ is regular on a sequence of larger and larger subsets and the convergence is locally smooth everywhere.

The Gaussian density is constant $\Theta^{\MM}_{(\bp_0, t_0)} (t_0 - t)$ if and only if $\MM$ is a \textbf{shrinker,} that is, if it is invariant under parabolic rescalings based at $(\bp_0, t_0)$, so if $\MM - (\bp_0, t_0) = \la(\MM - (\bx_0, t_0))$ for all $\la > 0$.
The simplest non-trivial example of a shrinker is the \textbf{round shrinking sphere,} whose time-slices are given by $\sqrt{-t} \, \IS^n$, where here and for the remainder of this paper we define $\IS^n \subset \IR^{n+1}$ to be the sphere of radius $\sqrt{2n}$ centered at the origin.
Another class of examples is given by \textbf{the round shrinking $\mathbf{(n,k)}$-cylinder} $\MM_{\cyl}^{n,k}$ whose time-slices are given by $\sqrt{-t} \, M_{\cyl}^{n,k}$, where 
\[ M_{\cyl}^{n,k} := \IR^k \times \IS^{n-k}, \]
which is called the \textbf{round $\mathbf{(n,k)}$-cylinder.}
We will distinguish between ``\emph{the} round shrinking $(n,k)$-cylinder'' and ``\emph{a} round shrinking $(n,k)$-cylinder,'' where the latter is any flow that is homothetic to $\MM_{\cyl}^{n,k}$, i.e., of the form $\la S(\MM^{n,k} + (\bp_0, t_0))$.
We will omit the superscripts ``$n,k$'' or the prefix ``$(n,k)$'' if the dimensions are fixed.
The (constant) Gaussian area (a.k.a. \textbf{entropy}) of $\MM_{\cyl}^{n,k}$ is denoted by $\Theta_{M_{\cyl}^{n,k}}$; recall that $\Theta_{M_{\cyl}^{n,k}} = \Theta_{\IS^{n-k}}$ and
\[ 1 < \Theta_{\IS^n} < \Theta_{\IS^{n-1}} < \ldots < \Theta_{\IS^1} < 2. \]

Consider the cylinder $M_{\cyl}^{n,k} \subset \IR^{n+1} \cong \IR^{n+1} \times \bO^{n'} \subset \IR^{n+1+n'}$ and let $u = (u', u'') \in C^0(\DD; \IR \times \IR^{n'})$ be a vector-valued function on an open subset $\DD \subset M_{\cyl}^{n,k}$.
We define the \textbf{graph} of $u$ over $M^{n,k}_{\cyl}$ as
\begin{equation} \label{eq_Gamma_cyl}
\Gamma_{\cyl}(u) := \big\{ \big(\bx, (1+u'(\bx,\by))\by, u''(\bx, \by) \big) \;\; : \;\; (\bx, \by) \in \DD \big\} \subset \IR^{k} \times \IR^{n-k+1} \times \IR^{n'} = \IR^{n+1+n'}. 
\end{equation}
Recall that by our conventions $(\bx, \by) \in \DD$ implies that $|\by| = \sqrt{2(n-k)}$.
We say that an $n$-dimensional submanifold $M \subset \IR^{n+1+n'}$ is \textbf{$\delta$-close} to the round cylinder $M_{\cyl}^{n,k}$ \textbf{at scale $r > 0$} (which we leave out if $r=1$) if $r^{-1} M \cap \IB^{n+1+n'}_{\delta^{-1}}=\Gamma_{\cyl}(u)$ for some $u \in C^\infty(\DD; \IR \times \IR^{n'})$ with $\Vert u \Vert_{C^{[\delta^{-1}]}} < \delta$ and $\DD \supset M_{\cyl}^{n,k} \cap \IB^{n+1+n'}_{\delta^{-1}-1}$.
We say that a point $\bp_0 \in \IR^{n+1+n'}$ is a \textbf{center of an $(n,k,\delta)$-neck at scale $r> 0$} if $S(M - \bp_0)$ is $\delta$-close to $M_{\cyl}^{n,k}$ at scale $r$ for some $S \in O(n+1+n')$.
Likewise, we say that an integral Brakke flow $\MM$ in $\IR^{n+1+n'} \times I$ is \textbf{$\delta$-close to $M^{n,k}_{\cyl}$ at time $t \in I$ and scale $r > 0$} if $\MM^{\sing}_t \cap \ov\IB^n_{\delta^{-1}} = \emptyset$ and $\MM^{\reg}_t$ is $\delta$-close to $M^{n,k}_{\cyl}$ at scale $r$.
We say that a point $\bp_0 \in \IR^{n+1+n'}$ is a \textbf{center of an $(n,k,\delta)$-neck at time $t \in I$ and scale $r> 0$} if $S (\MM - (\bp_0,0))$ is $\delta$-close to $M^{n,k}_{\cyl}$ at time $t \in I$ and scale $r > 0$.
We adopt the same conventions for rescaled (modified) flows.

We note that much of the previous work in this area relied on the additional assumption that the Brakke flow is \textbf{cyclic}.
By contrast, our results hold without this assumption.
\bigskip

\section{Pseudolocality properties and rescaled (modified) mean curvature flows} \label{sec_pseudolocality}
\subsection{Introduction}
An centrol method of this paper is the \emph{PDE-ODI principle,} which converts a parabolic partial differential equation into a system of ordinary differential inequalities; it will be introduced in Section~\ref{sec_PDE_ODI}.
This principle applies to a broad class of parabolic equations and essentially relies only on a relatively coarse \emph{pseudolocality property,} which is available in many settings.
The purpose of this section is to motivate this pseudolocality property and establish it in the context of almost cylindrical mean curvature flows, the main focus of this paper.

As will become clear later, the presence of unstable modes associated with ambient Euclidean isometries requires us to perform a gauging procedure, resulting in a broader class of flows, called \emph{rescaled modified mean curvature flows.}
These flows evolve via an additional smooth family of ambient isometries.
While they do not satisfy the pseudolocality property in its original form, we show that a weaker version of this property still holds, which is sufficient for our purposes.

We begin by illustrating the pseudolocality property in the simplest, linear case in Subsection~\ref{subsec_lin_case}).
In Subsection~\ref{subsec_resc_mod_mcf}, we introduce the rescaled modified mean curvature flow and reduce it to an equation for graphs over cylinders.
In Subsection~\ref{subsec_pseudoloc_rmmcf}, we derive a pseudolocality property for rescaled and rescaled modified mean curvature flows, which will allow us to apply the PDE-ODI principle to these flows in Section~\ref{sec_mode_analysis}.

\subsection{The linear case} \label{subsec_lin_case}
In this subsection, we examine the most fundamental setup with a pseudolocality property: bounded solutions of the {(rescaled)} linear heat equation.
Let $v \in C^2(\IR^n \times (-T,0))$ be a solution to the linear heat equation%
\begin{equation} \label{eq_v_he}
 \partial_t v = \triangle v . 
\end{equation}
Suppose, in addition, that we have the global bound $|v| \leq 1$.

Such solutions satisfy an \emph{(unrescaled) pseudolocality property,} which can be obtained, for example, by expressing $ v$ as a convolution with the Gaussian heat kernel.
Roughly speaking, this property asserts that a local bound on $v$ over a sufficiently large ball at some time controls $v$ over a smaller ball at later times.
More precisely, we need the following special form: for every $\eta > 0$ there is an $\sigma(\eta) > 0$ such that
\[ \text{If} \qquad  \sup_{B(\bx,\sigma^{-1} \sqrt{-t})} |v|(\cdot, t) \leq \sigma\eta, \qquad \text{then} \qquad \sup_{B(\bx,\sqrt{-t})} |v| (\cdot, t') \leq \eta \quad \text{for all} \quad t' \in [t, e^{-1} t] \cap (-T,0).
\]
We emphasize that the bound $|v| \leq 1$ is crucial here;
general (bounded) solutions to \eqref{eq_v_he} do not obey a pseudolocality property for a uniform choice of $\sigma$.

While the pseudolocality property on $v$ requires the choice of a smaller ball in its conclusion, the situation is much improved if we consider the \emph{rescaled} setting.
In this setting we write $t = - e^{-\tau}$ and set $u(\bx,\tau) := v (e^{-\tau/2} \bx, - e^{-\tau})$, which satisfies the following \emph{rescaled} flow equation:
\[ \partial_\tau u = \triangle u - \tfrac12 \bx \cdot \nabla u =  \triangle_f u, \]
Here $f = \frac14 |\mathbf x|^2$ and $\triangle_f = \triangle - \nabla f \cdot \nabla$ is the Ornstein-Uhlenbeck operator.
The pseudolocality property of $v$ now implies the following rescaled pseudolocality property for $u$.
Using the same constants as before, we have for any radius $R \geq \sigma^{-1}$:
\[ \text{If} \qquad  \sup_{\IB^n_R} |u_\tau|\leq \sigma\eta, \qquad \text{then} \qquad \sup_{\IB^n_{e^{(\tau'-\tau)/2}(R - \sigma^{-1})}} |u_{\tau'}|  \leq \eta \quad \text{for all} \quad \tau' \in [\tau, \tau+1].
\]
Note that, in contrast to the pseudolocality property of $v$, both balls are centered at the origin and for large enough $R$ and $\tau' = \tau+1$, the ball in the conclusion is actually larger than the ball in the assumption, as it arises from  a ball of radius $R - \sigma^{-1}$ after rescaling the domain by a factor of $e$.
This observation will be relevant in the proof of the PDE-ODI principle.
\medskip

\subsection{Rescaled (modified) mean curvature flows} \label{subsec_resc_mod_mcf}
Our goal in this section will be to replicate the phenomenon detailed in the previous subsection in the setting of rescaled mean curvature flows, or more generally, rescaled \emph{modified} mean curvature flows.
These are rescaled mean curvature flows that are adjusted by an additional family of ambient isometries, which are often chosen during a gauging procedure (see Subsection~\ref{subsec_gauge}).

Fix some $n \geq 1$ and $n' \geq 0$ and let $\YY$ be the space of Killing fields on $\IR^{n+1+n'}$, which can be identified with the Lie algebra of the group $E(n+1+n')$ of Euclidean motions on the same space.
Consider an $n$-dimensional, integral Brakke flow $\MM \in \IR^{n+1+n'} \times I$ given by a family of measures $(\mu_t)_{t \in I}$.
For a smooth family $(S_t \in E(n+1+n'))_{t \in I}$ of Euclidean motions, we can consider the modified flow $\MM'$ that is given by the push-forward measures $\mu'_t := (S_t)_* \mu_t$.
Note that $\MM^{\prime, \reg}_t = S_t(\MM_t^{\reg})$, which evolves according to the following \textbf{modified mean curvature flow} equation:
\[ \partial_t \mathbf x = \mathbf H + \frac1{t_0-t} \bY_t^{\perp}, \]
where $t_0$ is some fixed time with $I \subset (-\infty,t_0)$ and $(\bY_t \in \YY)_{t \in I}$ is a smooth family of Killing fields generating $(S_t)_{t \in I}$ in the following sense
\[ (t_0-t) \partial_t S_t = \bY_t \circ S_t. \]

As in the example from the previous subsection, we will now pass to a rescaled version of this flow.
So consider the flow $\td\MM'$ that is defined over 
\[ \td I := \{ -\log(t_0-t) \;\; : \;\; t \in I \} \] 
and given by the push-forward measures $\mu'_\tau := (e^{\tau/2})_* \mu'_{t_0 - e^{-\tau}}$.
Note again that $\td\MM^{\prime, \reg}_\tau =  e^{\tau/2} \MM'_{t_0 - e^{-\tau}}$, which evolves according to the \textbf{rescaled modified mean curvature flow} equation:
\begin{equation} \label{eq_mrmcf_pseudo}
 \partial_\tau \mathbf x = \mathbf H + \frac{\mathbf x^\perp}2 + \td\bY_\tau^\perp, 
\end{equation}
where $\td\bY_\tau = (e^{\tau/2})_* \bY_{-e^{-\tau}}$.
We caution the reader that 
\[ \td\MM^{\prime, \reg}_\tau = e^{\tau/2} S_{t_0-e^{-\tau}}( \MM^{\reg}_{t_0-e^{-\tau}}) =  \td S_{\tau}(e^{\tau/2} \MM^{\reg}_{t_0-e^{-\tau}}) \]
for some family of Euclidean motions $(\td S_\tau)$, which usually satisfies $\partial_\tau \td S_\tau \neq \td\bY_\tau \circ \td S_\tau$.
For example, if $S_t$ is a non-trivial translation that is constant in time, then $\bY_t , \td\bY_\tau \equiv 0$, however $\td S_\tau$ is not constant.

If $S_t \equiv \id$, and hence $\bY_\tau = \td\bY_\tau \equiv 0$, then we call $\td\MM'$ the \textbf{rescaled mean curvature flow} associated with $\MM$ (based at time $t_0$) and write $\td\MM' = \td\MM$.

\medskip
Let us now fix $k \in \{ 1, \ldots, n-1 \}$ and a rescaled modified mean curvature flow $\td\MM'$ associated with a family of Killing fields $(\td\bY)_{\tau \in \td I}$.
Suppose that a subset of this flow can be described as the graph of a function $u = (u',u'') \in C^\infty(\DD ; \IR\times \IR^{n'})$ over the round $(n,k)$-cylinder $M^{n,k}_{\cyl}$; that is $\DD \subset M_{\cyl} \times \IR$ is an open subset whose time-slices are denoted by $\DD_\tau := \DD \cap M_{\cyl} \times \{ \tau \} \subset M_{\cyl} \times \{ \tau \} \cong M_{\cyl}$ and we assume that for all $\tau$
\[ \Gamma_{\cyl} (u_\tau) \subset \td\MM^{\reg,\prime}_\tau. \]
See \eqref{eq_Gamma_cyl} for the definition of $\Gamma_{\cyl}$ and recall our convention that the cross-sectional $\IS^{n-k}$-spheres of $M_{\cyl}^{n,k}$ have radius $\sqrt{2(n-k)}$. 
In this setting, the rescaled modified mean curvature flow equation \eqref{eq_mrmcf_pseudo} can naturally be expressesed as an evolution equation on $u_\tau$ of the form
\begin{equation} \label{eq_MCF_u_eq}
 \partial_\tau u_\tau = L u_\tau + Q[u_\tau, \td\bY_\tau]  
\end{equation}
The goal of the next lemma is to record important structural properties of this equation.

\begin{Lemma} \label{Lem_structure_MCF_graph_equation}
For any $\bY \in \YY$ we denote by $\bY (\bx, \by)$ the value of the vector field $\bY$ evaluated at the point and $(\bx,\by,0) \in M_{\cyl}$.
Decompose the vector $\bY(\bx,\by)$ into its component $\bY_\Vert (\bx,\by)$ that is tangent to $M_{\cyl}$, its component  $\bY'(\bx,\by) \in \IR$ in the normal direction of $M_{\cyl}$ within $\IR^{n+1}$ and its component $\bY''(\bx,\by)$ in the $\IR^{n'}$-direction. Set
\[ \bY_{\perp} (\bx,\by) := \big( \tfrac1{\sqrt{2(n-k)}} \bY'(\bx,\by), \bY''(\bx,\by) \big) \in \IR \times \IR^{n'}. \]
The factor $\tfrac1{\sqrt{2(n-k)}}$ has been chosen due to the fact that $|\by| = \sqrt{2(n-k)}$.
Then $u_\tau$ evolves according to an equation of the form \eqref{eq_MCF_u_eq}.
Using the decomposition $u = (u', u'')$, the linear part of this equation can be expressed as
\[ L (u', u'') = (\triangle_f u' + u', \triangle_f u'' + \tfrac12 u''), \]
where $f = \frac14 |\mathbf x|^2$ and $\triangle_f = \triangle - \nabla f \cdot \nabla$ is the Ornstein-Uhlenbeck operator.
The second, mostly non-linear, part is
\begin{align*}
 Q[u, \bY] &= Q(\bx,\by, u,\partial u, \partial^2 u, \bY) \\
&= Q_0(u, \partial u, \partial^2 u)
- \nabla_{\bY_{\Vert} (\bx,\by)} u + \bY_\perp (\bx,\by)   + Q_1( \by, u, \bY ) + Q_2(\by, u, \partial u ,   \bY ).
\end{align*}
Here $Q_0$, $Q_1$ and $Q_2$ denote smooth functions in the indicated parameters.
The function $Q_0$ and its first partial derivatives vanish at the origin.
The function $Q_1$ is bilinear in $u, \bY$ and $Q_2$ is trilinear in $u, \partial u, \bY$, where we view $\bY$ as an element in the vector space  $\YY$ of Killing fields on $\IR^{n+1+n'}$.
\end{Lemma}

\begin{proof}
Since the mean curvature vector $\mathbf H$ at any point can be calculated from the value and the first and the second derivatives of $u$ at that point, we find that \eqref{eq_mrmcf_pseudo} implies an evolution equation of the form
\[ \partial_\tau u = F_0(u, \partial u, \partial^2 u) + F_1 (\bx, \by, u, \partial u) + F_2 (\bx, \by,u, \partial u,  \bY), \]
A standard variational calculation implies that
\[ F_0 (u, \partial u , \partial^2 u) = \big( - \tfrac12 + \triangle u' + \tfrac12 u', \triangle u'' \big) + Q_0 (u, \partial u, \partial^2 u ), \]
where we write $u = (u', u'')$ as before and where $Q_0$ denotes a smooth function, which vanishes at to origin, together with its first partial derivatives.
The second term can be expressed explicitly as follows:
\[ F_1 ( \bx, \by, u, \partial u)
= \big( \tfrac12 - \tfrac12 \bx \cdot \nabla u' + \tfrac12 u',  - \tfrac12 \bx \cdot \nabla u'' + \tfrac12 u''\big). \]

To express the third term, consider cylindrical coordinates on $\IR^{n+1+n'}$, where for $(\bx,\by,0) \in M_{\cyl}$ the point $(\bx,\by,\rho, \bz)$ corresponds to $(\bx, (1+\rho) \by, \bz)$.
Using the same recipe as in the statement of the lemma, we can extend the  functions $\bY_\Vert (\bx,\by)$ and $\bY_\perp (\bx,\by)$ to functions $\bY_\Vert (\bx,\by,\rho,\bz)$ and $\bY_\perp (\bx,\by,\rho,\bz)$ in such a way that the decomposition of $\bY(\bx,\by, \rho,\bz)$ is taken with respect to the tangential and normal directions at $(\bx,\by,0) \in M_{\cyl}$. 
Note that we have $\bY_{\Vert / \perp} (\bx,\by) = \bY_{\Vert / \perp} (\bx,\by,0,0)$.
We now obtain the explicit representation
\begin{multline*}
 F_2(\bx, \by,u, \partial u,  \bY)
= - \nabla_{\bY_{\Vert} (\bx,\by, u', u'')} u +\bY_\perp (\bx,\by,u', u'') \\
= - \nabla_{\bY_{\Vert} (\bx,\by)} u +\bY_\perp (\bx,\by) + \nabla_{\bY_{\Vert} (\bx,\by, u', u'') - \bY_\Vert (\bx,\by, 0,0)} u 
- \big(\bY_\perp (\bx,\by,u', u'') - \bY_\perp (\bx,\by,0,0) \big). 
\end{multline*}
Since the Killing field $\bY$ is affine linear, we obtain that the last term is bilinear in $u, \bY$ and the second last term is trilinear in $u, \partial u, \bY$.
By the same reason both terms are independent of $\bx$.
\end{proof}
\medskip

\subsection{Pseudolocality for rescaled modified mean curvature flows} \label{subsec_pseudoloc_rmmcf}
We will now derive a pseudolocality property for solutions $u$ to \eqref{eq_MCF_u_eq} similar to the one provided in Subsection~\ref{subsec_lin_case}. 
This property will translate to a statement about closeness of $\td\MM'_\tau$ to $M_{\cyl}$ in a large ball assuming closeness at earlier times.
In contrast to the example from Subsection~\ref{subsec_lin_case}, we now allow the domains $\DD_\tau \subset M_{\cyl}$, over which we can represent $\td\MM'_\tau$ by a function $u_\tau$, to be incomplete.
Consequently, our pseudolocality property also needs to establish control over the size of $\DD_\tau$ at later times.

Let us first describe the construction of the domains $\DD_\tau$.
Suppose that $\td\MM'$ is rescaled modified flow in $\IR^{n+1+n'} \times \td I$ arising from an $n$-dimensional, unit-regular, integral Brakke flow, as explained in the previous subsection.
For each $\tau \in \td I$ choose $R'_\tau \in [0,\infty]$ maximal with the property that $\IB^{n+1+n'}_{R'_\tau} \times \{ \tau \}$ is disjoint from $\td\MM^{\prime,\sing}$ and such that
\[ \td\MM^{\prime,\reg}_\tau \cap \IB^{n+1+n'}_{R'_\tau} = \Gamma_{\cyl} (u_\tau) \]
for some $\DD_\tau \subset M_{\cyl}$ and $u_\tau \in C^\infty(\DD_\tau; \IR \times \IR^{n'})$.
It is not hard to see that $\DD := \bigcup_{\tau \in \td I} \DD_\tau \times \{ \tau \}$ is an open subset of $M_{\cyl} \times \td I$ and that $u$ is a smooth function on $\DD$, which satisfies the evolution equation from Lemma~\ref{Lem_structure_MCF_graph_equation}.

We will now derive a pseudolocality property for $u$ and $\DD$.
We first assume that $\td\bY_\tau \equiv 0$, so $\td\MM := \td\MM'$ is, in fact, a rescaled \emph{(unmodified)} mean curvature flow.
In this case, we obtain a stronger and simpler pseudolocality property:

\begin{Lemma}[Strong pseudolocality if $\td\bY_\tau \equiv 0$] \label{Lem_strong_pseudo}
For $\eta > 0$ and integers $1 \leq k < n $, $n' \geq 0$ and $m \geq 2$ there is a constant $\sigma(\eta,\lb m, \lb  n, \lb n') > 0$ with the following property.
Let $\td\MM$ be a rescaled mean curvature flow in $\IR^{n+1+n'} \times \td I$ corresponding to an $n$-dimensional, unit-regular, integral Brakke flow, as explained in Subsection~\ref{subsec_resc_mod_mcf}.
Consider the domain $\DD$ and the function $u$, defined in the beginning of this subsection.

Suppose that for some $\tau_0 \in \td I$ and $R > \sigma^{-1}$ we have $\IB^{k}_R \times \IS^{n-k} \subset \DD_{\tau_0}$ and
\begin{equation} \label{eq_time_tau_bound_strong}
 |u_{\tau_0}| + \ldots + |\nabla^{m-1} u_{\tau_0}| \leq \sigma \eta \qquad \text{on} \quad \IB^{k}_R \times \IS^{n-k}.   
\end{equation}
Then for all $\tau_2 \in [\tau_0 +\frac12, \tau_0+1] \cap \td I$ we have $\IB^{k}_{e^{(\tau_2-\tau_0)/2} (R - \sigma^{-1})} \times \IS^{n-k} \subset \DD_{\tau_2}$ and
\begin{equation} \label{eq_us_less_eta2}
 |u_{\tau_2}| + \ldots + |\nabla^{m} u_{\tau_2}| < \eta \qquad \text{on} \quad \IB^{k}_{e^{(\tau_2-\tau_0)/2} (R - \sigma^{-1})} \times \IS^{n-k}.   
\end{equation}
\end{Lemma}

\begin{proof}
Without loss of generality, we assume in the following that $\tau_0 = t_0 = 0$.
Let us first consider the corresponding \emph{unrescaled} flow $\MM$, so the flow with $\td\MM^{\reg}_\tau = e^{\tau/2} \MM^{\reg}_{-e^{-\tau}}$.
Note that the bound \eqref{eq_time_tau_bound_strong} (at time $\tau = 0$) implies a similar bound on $\MM$ (at time $t = -1$).

\begin{Claim} \label{Cl_pseudoloc}
If $\sigma \leq \ov\sigma(\eta, m, n,n')$, then 
for any $t \in [-e^{-1/2}, -e^{-1}] \cap I$ the ball $\IB^{n+1+n'}_{R  - \frac12 \sigma^{-1}} \times \{ t \}$ is disjoint from $\MM^{\sing}$ and we can express $\MM^{\reg}_{t} \cap \IB^{n+1+n'}_{R  - \frac12 \sigma^{-1} }$ as the graph of a smooth vector-valued function $v_{t}$ over an open subset $\DD'_t \subset \sqrt{-t} M_{\cyl}$ such that 
\begin{equation} \label{eq_v_dv_weaker}
 |v_{t}| + \ldots + |\nabla^m v_{t} | \leq \eta \qquad \text{on} \quad \IB^k_{R-\sigma^{-1}} \times (\sqrt{-t} \, \IS^{n-k} ) \subset \DD'_t.
\end{equation}
\end{Claim}

\begin{proof}
Fix $\eta,n,n' > 0$ and suppose by contradiction that the claim was false.
So we can find sequences $\sigma_i \to 0$, $\MM^{i}$, $R_i \geq \sigma_i^{-1}$ such that there are points $(\bp_i, t_i) \in \IB^{n+1+n'}_{R_i+1 -\sigma^{-1}_i} \times [-e^{-1/2}, -e^{-1}]$, for which one of the following is true:
\begin{itemize}
    \item $\bp_i \in \MM^{i,\sing}_{t_i}$.
    \item $\bp_i\in\MM^{i,\reg}_{t_i}$ has distance $> 0.01$ from $\sqrt{-t_i} M_{\cyl}$.
    \item $\bp_i\in\MM^{i,\reg}_{t_i}$ has distance $\le 0.01$ from $\sqrt{-t_i} M_{\cyl}$. There is another point within $(\spt \MM^i)_{t_i} \cap \IB^{n+1+n'}_{R_i   - \frac12 \sigma_i^{-1}}$ that has distance $\leq 0.01$ from $\sqrt{-t_i} M_{\cyl}$ and whose nearest-point projection to $\sqrt{-t_i} M_{\cyl}$ has the same image as $\bp_i$.
    \item $\bp_i \in \MM^{i,\reg}_{t_i}$ and the bound \eqref{eq_v_dv_weaker} fails at this point.
    \item $\bp_i \in \IB^k_{R_i-\sigma_i^{-1}} \times (\sqrt{-t_i} \, \IS^{n-k})$, but $\bp_i$ is not in the image of the nearest-point projection of $\MM^{i,\reg}_{t_i} \cap \IB^{n+1+n'}_{R_i  - \frac12 \sigma_i^{-1} }$ to $\sqrt{-t_i} M_{\cyl}$.
\end{itemize}
After passing to a subsequence, we may assume that we have convergence  $\MM^i - (\bp_i, 0) \to \MM^\infty$
(here we use, for example, \cite[Theorem~5.3]{Schulze_intro_Brakke} 
to ensure uniform local area bounds).
The limit $\MM^\infty$ is non-empty by the choice of $\bp_i$. Seeing that $\MM^i_{-1}$ smoothly converges to $M_{\cyl}$, the fact that $\MM^\infty_{-1}$ is non-empty implies the distance between $\bp_i$ and $M_{\cyl}$ must stay bounded.
By White's local regularity theorem \cite{White_local_regularity} (see also \cite{Ecker_regularity_book, Ilmanen_sing_MCF_surf,Haslhofer_lecture_MCF}) there is a $t^* > -1$ such that $\MM^\infty$ restricted to the time-interval $[-1,t^*]$ is smooth with uniformly bounded mean curvature, and the initial time-slice must be isometric to $M_{\cyl}$.
Let $t^{**}$ be the supremum over all $t^* > -1$ with this property.
By a basic uniqueness argument, we obtain that $\MM^\infty$ restricted to the time-interval $[0,t^{**})$ must be isometric to the round cylindrical flow $\MM_{\cyl}$ restricted to the same time-interval.
So using White's regularity theorem again, implies that $t^{**} = 0$ due to its supremal choice and hence $\MM^\infty$ must be isometric to $\MM_{\cyl}$.
Hence we have local smooth convergence $\MM^i \to \MM^\infty$, which yields a contradiction to the assumptions on $(\bp_i,t_i)$ for large $i$.
\end{proof}

The lemma now follows by reexpressing Claim~\ref{Cl_pseudoloc} in terms of the \emph{rescaled} flow $\td\MM$ and adjusting the constants $\eta$ and $\sigma$.
\end{proof}
\medskip

A slightly weaker pseudolocality property holds for rescaled \emph{modified} flows $\td\MM'$, so if $\td\bY_\tau \not\equiv 0$.

\begin{Lemma}[Weak pseudolocality, general case] \label{Lem_mcf_weak_pseudo}
For $\eta , A > 0$ and integers $1 \leq k < n $, $n' \geq 0$ and $m \geq 2$ 
 there is a constant $\sigma(\eta,A,m,n,n') > 0$ with the following property.
Let $\td\MM'$ be a rescaled modified mean curvature flow in $\IR^{n+1+n'} \times \td I$ for a family of Killing fields $(\td\bY_\tau)_{\tau \in \td I}$ corresponding to an $n$-dimensional, unit-regular, integral Brakke flow, as explained in Subsection~\ref{subsec_resc_mod_mcf}.
Consider the domain $\DD$ and the function $u$, defined in the beginning of this subsection.
Fix a norm on the (finite-dimensional) space $\YY$ of Killing fields on $\IR^{n+1+n'}$ and suppose that $\Vert \td\bY_\tau \Vert \leq A$ for all $\tau \in \td I$.

Suppose that for some $\tau_0, \tau_1, \tau_2 \in \td I$ with $\tau_0 \leq \tau_1 \leq \tau_2$ and $\tau_2 - \tau_0 \in [\frac12, 1]$ and $R > \sigma^{-1}$ the following is true for $i = 0,1$:
\begin{equation} \label{eq_u_less_sigma_eta_pseudo}
 |u_{\tau_i}| + \ldots + |\nabla^{m-1} u_{\tau_i}| \leq \sigma\eta  \qquad \text{on} \quad 
 \IB^k_R \times \IS^{n-k} \subset \DD_{\tau_i}
 \end{equation}
Suppose moreover that
\begin{equation} \label{eq_t2t1sR}
 \tau_2 - \tau_1 \leq \sigma R^{-1}. 
\end{equation}
Then 
\[ |u_{\tau_2}| + \ldots + |\nabla^{m} u_{\tau_2}| < \eta  \qquad \text{on} \quad \IB^k_{e^{\sigma (\tau_2 - \tau_0)} (R-\sigma^{-1})} \times \IS^{n-k} \subset \DD_{\tau_2}. \]
\end{Lemma}

We remark that the pseudolocality property of this lemma is indeed weaker than that of Lemma~\ref{Lem_strong_pseudo}, as it only holds under an additional assumption at time $\tau_1$.
This weaker version will be sufficient for our purposes.

\begin{proof}
We may assume without loss of generality that $\tau_0  = t_0 = 0$.
Consider the \emph{unrescaled} and \emph{unmodified} flow $\MM$, so the flow with $\td\MM^{\prime,\reg}_\tau = e^{\tau/2} S_{-e^{-\tau}} (\MM^{\reg}_{-e^{-\tau}})$.
Let moreover $\td\MM$ be the corresponding rescaled flow, so the flow with $\td\MM^{\reg}_\tau = e^{\tau/2} \MM^{\reg}_{-e^{-\tau}}$.
Without loss of generality, we may assume that $S_{-1} = \id$, so $\td\MM_{0}^{\prime, \reg} = \MM^{\reg}_{-1} = \td\MM^{\reg}_{0}$.
It is not hard to see that there is a family $(\td S_\tau)_{\tau \in  \td I}$ of Euclidean motions such that $\td\MM_{\tau}^{\prime,\reg} = \td S_\tau ( \td\MM_{\tau}^{\reg} )$.
The bound $\Vert \td\bY_\tau \Vert \leq A$ moreover implies bounds of the following form, as long as $\sigma \leq \ov\sigma(A, n,n')$,
\begin{equation} \label{eq_St1St2_At2t1}
 d_{E(n+1+n')} \big( \td S_{\tau_2} \circ \td S_{\tau_1}^{-1} , \id \big) \leq C(n,n', A)  (\tau_2 - \tau_1), \qquad d_{E(n+1+n')} \big( \td S_{\tau_i} , \id \big) \leq C(n,n', A) ,
\end{equation}
where $d_{E(n+1+n')}$ is the length-metric of a fixed left-invariant metric on the Lie group $E(n+1+n')$.

We can now apply Lemma~\ref{Lem_strong_pseudo} to the rescaled and \emph{unmodified} flow $\td\MM$, implying a $\Psi (\sigma | n, n')$-closeness of $\td\MM^{\reg}_{\tau_i}$ to $M_{\cyl}$ on $\IB^{n+1+n'}_{e^{\sigma \tau_i}(R - \sigma^{-1}-C(n,n'))}$, $i=1,2$, where here and in the following $\Psi(\sigma |n,n')$ is a generic function with $\Psi(\sigma |n,n')$ as $\sigma \to 0$.
From this we obtain that $\td\MM^{\prime,\reg}_{\tau_1} \cap \IB^{n+1+n'}_{0.9 R} = \td S_{\tau_1} (\td\MM^{\reg}_{\tau_1}) \cap \IB^{n+1+n'}_{0.9 R}$ is $\Psi (\sigma | n, n')$-Hausdorff-close to a subset of $\td S_{\tau_1} (M_{\cyl})$.
On the other hand, the bound from \eqref{eq_u_less_sigma_eta_pseudo} at time $\tau_1$ implies that $\td\MM^{\prime,\reg}_{\tau_1} \cap \IB^{n+1+n'}_{0.9 R}$ is $\Psi(\sigma |n,n')$-Hausdorff-close to an open subset of $M_{\cyl}$.
Therefore, $\td S_{\tau_1} (M_{\cyl}) \cap \IB^{n+1+n'}_{0.8 R}$ must be $\Psi(\sigma |n,n')$-Hausdorff-close to a subset of $M_{\cyl}$.
Combined with the first bound in \eqref{eq_St1St2_At2t1} and \eqref{eq_t2t1sR}, we obtain that $\td S_{\tau_2} (M_{\cyl}) \cap \IB^{n+1+n'}_{10 R}$ has Hausdorff distance from a subset of $M_{\cyl}$ that is at most
\[ \Psi(\sigma |n,n') + C(n,n', A)  (\tau_2 - \tau_1)R \leq \Psi(\sigma |n,n') + C(n,n',A) \sigma. \]
Since $\td\MM^{\prime, \reg}_{\tau_2} = \td S_{\tau_2} (\td\MM^{\reg}_{\tau_2})$, the desired bound now follows from the closeness of $\td\MM^{\reg}_{\tau_2}$ to $M_{\cyl}$ discussed in the beginning of this paragraph, for sufficiently small $\sigma$.
\end{proof}

\bigskip
\section{The PDE-ODI principle} \label{sec_PDE_ODI}
\subsection{Introduction and motivation}
In this section, we establish a general principle for tracking the evolution of certain non-linear PDEs using a \emph{finite-dimensional} ordinary differential inequality (ODI) that governs the behavior of the low-frequency modes.
Specifically, we will focus on non-linear parabolic PDEs whose linearization takes the form $\partial_\tau u = \triangle_f u + Au$, where $\triangle_f u = \triangle u - \nabla f \cdot \nabla u$ is the  Ornstein-Uhlenbeck operator with $f = \frac14 |\bx|^2$, and $Au$ is a zeroth-order term.
Such PDEs naturally arise in the study of rescaled flows, such as the rescaled mean curvature flow from Subsection~\ref{subsec_resc_mod_mcf}.
Our principle will allow us to control \emph{arbitrarily many modes} (i.e., not just leading modes) of solutions to such PDEs, to \emph{arbitrarily high vanishing order} (i.e., with arbitrarily fine asymptotic precision). 
The principle will be entirely local in nature and rely only on a pseudolocality property similar to the one discussed in Section~\ref{sec_pseudolocality}.
\medskip

To motivate the discussion, consider first the linear  case
\begin{equation} \label{eq_PDE_lin_intro}
 \partial_\tau u = L u = \triangle_f u, 
\end{equation}
which is, of course, much simpler than the situations of interest and serves only as a warm-up.
The operator $L$ is self-adjoint on the weighted $L^2_f$-space
\[ \Vert u \Vert_{L^2_f}^2 = \int_{\IR^n} u^2 \, e^{-f} d\bx, \]
with eigenfunctions given by products of Hermite polynomials
\[ \mathfrak p^{(\vec i)}(\mathbf x) = \mathfrak p^{(i_1)} (\bx_1) \cdots \mathfrak p^{(i_n)} (\bx_n), \qquad \la_{\vec i} = - \tfrac12 (i_1 +\ldots + i_n). \]
If $u_\tau \in L^2_f (\IR^n)$, then one may expand
\begin{equation} \label{eq_spectral_expansion}
 u_\tau = \sum_{\vec i} a_{\vec i}(\tau) \mathfrak p^{(\vec i)}, \qquad \text{where} \quad \partial_\tau a_{\vec i} = \la_{\vec i} \, a_{\vec i} . 
\end{equation}
So the PDE \eqref{eq_PDE_lin_intro} reduces to decoupled ODEs for the coefficients $a_{\vec i}$.
This expansion allows us to draw conclusions on the asymptotic behavior of $u$.
For example, if $u$ is defined over a time-interval of the form $I = (-\infty, T)$ and $\Vert u_\tau \Vert_{L^2_f}$ is uniformly bounded, then all $a_{\vec i}$ for $\vec i \neq 0$ must vanish, implying that $u \equiv const$.
\medskip

In contrast to the linear model, the general, non-linear PDE setting that is of interest to our discussion presents substantially greater complexity.
First, the evolution equation for $u$ is usually non-linear, so the coefficients $a_{\vec i}$ must satisfy \emph{non-linear} ODEs, which typically interact and usually don't reduce to finite-dimensional systems.
Second---and far more importantly---one usually lacks global control on $u$, which makes a global spectral expansion like \eqref{eq_spectral_expansion} impossible.
In particular, one typically does \emph{not} have $u_\tau \in L^2_f(\IR^n)$ and $u_\tau$ is usually not even defined on all of~$\IR^n$---it is often only defined on a large ball around the origin, and its spatial growth is uncontrolled.
This presents a serious obstacle; for example, it may occur that $u_{\tau_1}\equiv 0$ wherever defined, yet $u_{\tau_2}\not\equiv 0$ for some later time $\tau_2>\tau_1$.
Lastly, our framework must accommodate additional input terms, such as the family of Killing fields $\td\bY_\tau$  in the rescaled modified mean curvature flow equation (see Subsection~\ref{subsec_resc_mod_mcf}), which often introduce spatially unbounded terms and which further complicate the analysis.

For these reasons, one must instead work with a \emph{localized} spectral expansion, truncated to \emph{finite degree}, and derive evolution equations for the truncated coefficients that include manageable error terms.
The PDE-ODI principle developed in this section achieves exactly this: it performs such a localization in a robust way and ensures that the associated error terms remain sufficiently small.
Consequently, the resulting finite-dimensional ODI system still accurately reflects the asymptotic behavior of the original PDE.
This will be crucial for the applications in this paper and may also be helpful in other geometric settings.

An additional brief overview of the PDE-ODI principle was given in the introduction, in Subsection~\ref{subsec_intro_pdeodi}.

\subsection{Statement of the result} \label{subsec_PDE_ODI_statement}
We will now state the PDE-ODI principle in its most general form.
Let us first explain our setup, which is a generalization of the rescaled (modified) mean curvature flow equation for graphs over the round cylinder, which was discussed in the previous section.

Fix a finite dimensional normed space $(\YY, \Vert \cdot \Vert)$, a compact Riemannian manifold $N$ and integers $n, n' \geq 1$.
Fix the potential function $f(\bx,\by) := \frac14 |\bx|^2$ on $\IR^n \times N$.
Let $I \subset \IR$ be a time-interval and consider an open subset $\DD \subset \IR^n \times N \times I$.
In the following we will consider vector-valued solution  $u \in C^\infty(\DD; \IR^{n'})$ of the following PDE, corresponding to a smooth family $(\mathbf Y_\tau \in \YY)_{\tau \in I}$:
\begin{equation} \label{eq_general_PDE}
 \partial_\tau u_\tau = L u_\tau + Q[u_\tau, \mathbf Y_\tau]. 
\end{equation}
Here we assume that:
\begin{itemize}
\item We have $L = \triangle_f + A$, where $\triangle_f u = \triangle u - \nabla f \cdot \nabla u$ is the  Ornstein-Uhlenbeck operator and $A (\bx,\by) = A (\by) \in \IR^{n' \times n'}_{\sym}$ is a self-adjoint linear operator on $\IR^{n'}$, which is constant in the $\IR^n$- and time-directions, but may smoothly depend on the $N$ factor.
\item The non-linear term $Q[u,\mathbf Y] = Q(\bx,\by,u, \nabla u, \nabla^2 u, \mathbf Y)$ is a smooth, time-independent function, defined for $|u| + |\nabla u | + |\nabla^2 u | + \Vert \bY \Vert < c_0$, for some uniform $c_0 > 0$.
We assume that, on its domain of definition, we have uniform bounds of the form
\begin{equation} \label{eq_derQ_bounds_asspt}
|\partial_{u}^{i_0} \partial_{\nabla u}^{i_1} \partial_{\nabla^2 u}^{i_2} \partial_{\bY}^{i_3}  Q| \leq C_{i_0,i_1,i_2,i_3} (r+1)^{C_{i_0,i_1,i_2,i_3}}  
\end{equation}
and for any $0 \leq i , j \leq 2$
\begin{equation} \label{eq_condition_on_Q}
 |\partial_{\nabla^i u} \nabla^j Q| \leq C_j (r+1)^{2-i} (|u| + |\nabla u| + |\nabla^2 u| + \Vert \bY \Vert ) . 
\end{equation}
Here $\partial_u Q, \partial_{\nabla u} Q, \partial_{\nabla^2 u} Q, \partial_{\bY} Q$ denote the partial derivatives in the respective parameters, $\nabla^j$ denotes arbitrary $j$-th \emph{partial} derivatives in the spatial parameters $\bx, \by$ (not \emph{total} derivatives of $Q[u,\bY]$ as a function in space!) and $r = |\bx|$ is the radial distance on the $\IR^n$-factor.
\end{itemize}

\begin{Remark}
These conditions hold for the rescaled modified mean curvature flow equation \eqref{eq_MCF_u_eq} for graphs over cylinders, which will be of main interest to this paper.
This can seen directly from Lemma~\ref{Lem_structure_MCF_graph_equation}.
\end{Remark}

\begin{Remark}
It is straightforward to generalize our discussion by assuming that $u$ takes values in a Euclidean vector bundle over $\IR^n \times N$, which is equipped with a connection.
In this case, we need to assume that there is an isometric action on this bundle, which leaves $A$ invariant and acts transitively on the $\IR^n$-factor.
\end{Remark}

\medskip

Consider the weighted Hilbert space $L^2_f(\IR^n \times N;\IR^{n'})$ defined by the inner product and norm
\begin{equation} \label{eq_L2f_integration}
 \langle u_1, u_2 \rangle_{L^2_f} := \int_{\IR^n \times N} (u_1 \cdot u_2) e^{-f} dg, \qquad \Vert u \Vert_{L^2_f}^2 = \langle u, u \rangle_{L^2_f} = \int_{\IR^n \times N}|u|^2 e^{-f} dg , 
\end{equation}
where $g$ denotes the Cartesian product metric on $\IR^n \times N$. 
Recall that $L$ is self-adjoint with respect to this inner product and has discrete spectrum, which is bounded from above \cite{Reed_Simon_IV}.
Let us now fix some $\lambda \in \IR$ and consider the  decomposition
\begin{equation} \label{eq_splittingVgeqless}
 L^2_{f} := L^2_f(\IR^n \times N;\IR^{n'}) = \sV_{> \la} \oplus \sV_{\leq \la}, 
\end{equation}
where $\sV_{> \la}$ is the (finite-dimensional) direct sum of eigenspaces corresponding to eigenvalues $> \la$ and $\sV_{\leq \la}$ is its orthogonal complement.
Our goal will be to approximate the solution $u_\tau$ by a smooth family of finite-dimensional elements $U^+_\tau \in \sV_{> \la}$.
For most applications, it will be advantageous to choose $\lambda < 0$ or $\la \ll 0$, where a smaller $\la$ results in a better approximation.

Let us now fix a $J \geq 1$, which will govern the \emph{order} of this approximation, and for any fixed $(\bx,\by) \in \IR^n \times N$ consider the $J$\emph{th} Taylor polynomial $\td Q_J$ of $Q[u,\bY] =Q(\bx,\by, u, \nabla u, \nabla^2 u, \mathbf Y)$  in the remaining parameters:
\[  Q(\bx,\by, u, \nabla u, \nabla^2 u, \mathbf Y) = \td Q_J (\bx,\by, u, \nabla u, \nabla^2 u, \mathbf Y) + O \big( |u|^{J+1} + |\nabla u|^{J+1} + |\nabla^2 u|^{J+1}). \]
Note that $\td Q_J (\bx,\by, u, \nabla u, \nabla^2 u, \mathbf Y)$ is defined for all choices of $(\bx,\by, u, \nabla u, \nabla^2 u, \mathbf Y)$.
It is a polynomial in the last four entries, whose coefficients are smooth functions in $(\bx,\by)$ that are bounded by $C (r+1)^{C}$ for some $C = C(J) > 0$ due to \eqref{eq_derQ_bounds_asspt}.
We can therefore project these coefficients onto $\sV_{> \la}$ and $\sV_{\leq \la}$ and define the following polynomial map
\[ Q_J = Q_J^+ + Q_J^- : \; \; \sV_{> \la} \times \YY \lto L^2_{f} = \sV_{> \la} \oplus \sV_{\leq \la}, \qquad
(U, \mathbf Y) \longmapsto \td Q_J (U, \nabla U, \nabla^2 U, \mathbf Y). \]
This map will can be viewed as the \textbf{$J$\emph{th} Taylor approximation of $Q[u, \bY]$} ``restricted'' to the space $\sV_{> \la} \times \YY$ and it will play a crucial role in controlling the evolution of an approximation to $u$ within $\sV_{> \la}$.

We remark that $Q[u,\bY]$ is often not defined if $u\in \sV_{>\lambda}$, as such functions typically exhibit polynomial growth.
So the true non-linear expression $Q[u,\bY]$ may not make sense globally.
For this reason, $Q_J$ can generally not be obtained by simply restricting $Q[u,\bY]\) to \(\sV_{>\lambda}\times\YY$ and taking its Taylor polynomial in the usual way.
\bigskip

In order to state our main result, we need the following definition. 

\begin{Definition}[Pseudolocality property] \label{Def_pseudolocality}
We say that a function $u \in C^\infty(\DD;\IR^{n'})$ , for some open $\DD \subset \IR^n \times N \times I$, satisfies the \textbf{$(\sigma,\eta, m)$-pseudolocality property,} for some $m \geq 1$, if the following holds.
Suppose that for some $\tau_0, \tau_1, \tau_2 \in I$ with $\tau_0 \leq \tau_1 \leq \tau_2$ and $\tau_2 - \tau_0 \in [\frac12,1]$ and $R > \sigma^{-1}$ the following is true for {$i=0,1$}:
\begin{equation} \label{eq_u_less_sigma_eta}
 |u_{\tau_i}| + \ldots + |\nabla^{m-1} u_{\tau_i}| \leq \sigma\eta  \qquad \text{on} \quad 
 \IB^n_R \times N \subset \DD_{\tau_i} .
\end{equation}
Suppose moreover that
\begin{equation} \label{eq_tau2tau1pseudo}
\tau_2 - \tau_1 \leq \sigma R^{-1/\sigma}. 
\end{equation}
Then 
\[ |u_{\tau_2}| + \ldots + |\nabla^{m} u_{\tau_2}| < \eta  \qquad \text{on} \quad \IB^n_{e^{\sigma (\tau_2 - \tau_0)} (R-\sigma^{-1})} \times N \subset \DD_{\tau_2}. \]
\end{Definition}
\medskip

\begin{Remark}
If \eqref{eq_general_PDE} is the rescaled modified mean curvature flow equation for graphs over the round cylinder and $\DD$ is defined as in Subsection~\ref{subsec_resc_mod_mcf}, then Lemma~\ref{Lem_mcf_weak_pseudo} implies that $u$ satisfies the $(\sigma,\eta, m)$-pseudolocality property whenever $\sigma \leq \ov\sigma(\eta,m,n,n')$.
In the case of \emph{unmodified} mean curvature flow, so when the family $(\mathbf  Y_\tau)$ vanishes, then a stronger version of the pseudolocality property holds; see Lemma~\ref{Lem_strong_pseudo}.
In that stronger version, the condition at time $\tau_1$  in \eqref{eq_u_less_sigma_eta} is omitted.
\end{Remark}

\begin{Remark}
The extra assumption at time $\tau_1$, so the case $i=1$ in \eqref{eq_u_less_sigma_eta}, expresses that a bound on $u$ at time $\tau_2$ is only required to hold if a similar bound holds at a slightly earlier time, in accordance with the polynomial relation \eqref{eq_tau2tau1pseudo}.
This additional requirement makes the pseudolocality property \emph{weaker,} as it applies in fewer situations.
Conversely, omitting the condition at time $\tau_1$ yields a \emph{stronger} version, which may be more intuitive.
Since pseudolocality appears solely as an assumption in Theorem~\ref{Thm_PDE_ODI_principle} below, using the weaker version makes the theorem correspondingly more \emph{general.}
\end{Remark}

\begin{Remark}
Note that the validity of the pseudolocality property remains unchange if we replace $\DD$ with the smaller domain $\{ |u| + \ldots + |\nabla^m u| < \eta' \}$ for $\eta' > \eta$.
\end{Remark}

\begin{Remark}
A pseudolocality condition of the above form is false for solutions of general non-linear PDEs.
For example, even solutions of the linear heat equation fail to satisfy such a property if they don't obey a uniform bound as in Subsection~\ref{subsec_lin_case}.
Moreover, even in settings where a pseudolocality condition does hold, the constant $\sigma$ typically must depend on $\eta$.
Otherwise, by considering a sequence of solutions $u^{(i)} $and rescaling $u^{(i)}/\eta^{(i)}$ with $\eta^{(i)} \to 0$, one would obtain a pseudolocality property for the linearized equation.
\end{Remark}

\medskip

We can now state our key result. 

\begin{Theorem}[PDE-ODI principle] \label{Thm_PDE_ODI_principle}
Consider a PDE of the form described in the begining of this subsection.
There is a constant $c > 0$ such that the following holds if
\[ 
m \geq 4, \quad  
J \geq 1, \quad
\la \in \IR, \quad 
\eta \leq \ov\eta(\la), \quad 
\sigma > 0, \quad 
\eps \leq \ov\eps (\sigma),  
\quad R^* \geq \underline R^*( m, J, \la,   \eta, \sigma,\eps). \]

Let $u \in C^\infty(\DD;\IR^{n'})$ be a solution to \eqref{eq_general_PDE} over the time-interval $I$ for a smooth family $(\mathbf Y_\tau \in \YY)_{\tau \in I}$.
Suppose that:
\begin{enumerate}[label=(\roman*)]
\item \label{Thm_PDE_ODI_principle_0} $\Vert \bY_\tau \Vert \leq c$ for all $\tau \in I$.
\item \label{Thm_PDE_ODI_principle_i} $u$ satisfies the $(\sigma,\eta,m)$-pseudolocality property.
\item\label{Thm_PDE_ODI_principle_ii} For all $\tau \in I$ we have 
\[ \IB_{ R^*}^n \times N  \subset \DD_{\tau} \qquad \text{and} \qquad \| u_{\tau}  \|_{L^2_f(\IB^n_{ R^*} \times N)} \leq \eta e^{-(R^*)^2/8}. \]
Here the weighted $L^2_f$ integral \eqref{eq_L2f_integration} is taken only over the domain $\IB^n_{ R^*} \times N$.
\item \label{Thm_PDE_ODI_principle_iv} If $T_0 = \min I > -\infty$ exists, then for some $R_0 \in [ R^*, \infty]$ the following holds for all $\tau \in [T_0 , T_0 +1]$:
\begin{equation} \label{eq_atT0bound}
 \IB^n_{R_0+1} \times N  \subset \DD_{\tau} \qquad \text{and} \qquad  \| u_{\tau}  \|_{C^{m}(\IB^n_{R_0+1} \times N)} < \eta. 
\end{equation}
If $\inf I = -\infty$, then \eqref{eq_atT0bound} holds for $R_0 = R^*$ and  sufficiently small $\tau$.
\end{enumerate}
Then there is a continuous function of radii $R : I \to [ R^*, \infty]$, which is smooth whenever it is finite, such that the following is true.
We have $R(T_0) = R_0$ if $T_0$ exists and for all $\tau \in I$ we have $\IB^n_{R(\tau)} \times N  \subset \DD_\tau$, so $u_\tau \omega_{R(\tau)}$ can be extended to a smooth function on $\IR^n \times N$ by setting it zero outside its domain.
So we can define
\begin{equation} \label{eq_UpUUm_def}
 U^+_\tau :=  \PP_{\sV_{> \lambda}} (u_{\tau} \omega_{R(\tau)})  , \qquad
\mathcal U^{-}_\tau := \big\|   \PP_{\sV_{\leq \lambda}}  (u_{\tau} \omega_{R(\tau)} ) \big\|_{L^2_{f}} + \eta e^{-\frac{((1-\eps)R(\tau))^2}{8}}. 
\end{equation}
If $R(\tau) = \infty$, then we set $\omega_{R(\tau)} \equiv 1$ and $e^{-\frac{((1-\eps)R(\tau))^2}{8}} = 0$.
The following is true:
\begin{enumerate}[label=(\alph*)]
\item \label{Thm_PDE_ODI_principle_a} We have the bound
$  \| u_{\tau} \|_{C^{m}(\IB^n_{R(\tau)} \times N )} \leq \eta$ for all $\tau \in I$.
\item \label{Thm_PDE_ODI_principle_b} The quantities $U^+_\tau$ and $\mathcal{U}^-_\tau$ satisfy the evolution inequalities
\[ \Big\| \partial_\tau U^+_\tau - L U^+_\tau  - Q_J^+ (U^+_\tau ,\mathbf Y_\tau) \Big\|_{L^2_{f}}
\leq  C(\la, J) \Vert U^+_\tau \Vert^{J+1}_{L^2_{f}} + \big( C(\la) \sqrt\eta + C \Vert \bY_\tau \Vert  \big)  \mathcal{U}^-_\tau  \]
\[  \partial_\tau \mathcal U^-_\tau 
\leq  \big(\la + C(\la) \sqrt\eta + C(\la) \Vert \bY_\tau \Vert  \big) \mathcal U^-_\tau  
+  \Vert Q^-_J(U^+_\tau, \mathbf Y_\tau) \Vert_{L^2_{f}}
+ C(\la,J) \Vert U^+_\tau \Vert^{J+1}_{L^2_{f}}
  \]
\item \label{Thm_PDE_ODI_principle_d} There is a $T' \geq -\infty$ such that $R(\tau) = \infty$ if $\tau \leq T'$ and $R(\tau) < \infty$ if $\tau > T'$. 
\end{enumerate}
\end{Theorem}
\medskip

Let us digest the content of Theorem~\ref{Thm_PDE_ODI_principle}.
The key outcome is the system of evolution inequalities from Assertion~\ref{Thm_PDE_ODI_principle_b}, which control the dynamics of $U^+_\tau$ and $\UU^-_\tau$.
Here $U^+_\tau$ represents a localized approximation of $u_\tau$, obtained by restricting to a ball of radius $\approx R(\tau)$ and projecting to the $> \la$ modes.
The quantity $\UU^-$ measures both the quality of this approximation and the size of the region on which it is valid.
In most applications one may choose $\la < 0$ small enough to ensure that the coefficient in the second evolution inequality satisfies
\[ \la + C(\la) \sqrt\eta + C(\la) \Vert \bY_\tau \Vert   < - c' < 0. \]
Under this condition, the second evolution inequality typically implies that $\UU^-_\tau$ roughly decays exponentially to a number comparable to $\Vert Q^-_J(U^+_\tau, \mathbf Y_\tau) \Vert
+ C(\la,J) \Vert U^+_\tau \Vert^{J+1}$, which is usually bounded by higher power of $\Vert U^+_\tau \Vert$.
So asymptotically, one often has
\[ \UU^-_\tau \lesssim \Vert U^+_\tau \Vert^{J'} \]
for some $J' \leq J$.
Plugging this into the first evolution inequality yields (see also Subsection~\ref{subsec_intro_pdeodi})
\[  \partial_\tau U^+_\tau = L U^+_\tau  + Q_J^+ (U^+_\tau ,\mathbf Y_\tau) + O(\Vert U^+_\tau \Vert^{J'}). \]
The important feature here is that the error term is proportional to $\Vert U^+_\tau \Vert$ (or even to a higher power of it).
This is in sharp contrast with the outcome of a conventional localization procedure, where the error term is typically only comparable to the fixed threshold $\eta$ from Assertion~\ref{Thm_PDE_ODI_principle_a}.
The absence of such an error term is essential for our applications, as it allows to conclude smallness of $\UU^-_\tau$ solely based on a smallness assumption on $\Vert U^{+}_\tau \Vert$. %
This dependence makes the proof of Theorem~\ref{Thm_PDE_ODI_principle} quite subtle, as the radius function $R(\tau)$ must be chosen carefully  \emph{depending on the evolving sizes of $\Vert U^+_\tau \Vert$ and $\UU^-_\tau$.}

The second important outcome is the \emph{unweighted $C^m$-estimate}  in Assertion~\ref{Thm_PDE_ODI_principle_a}, on $u_\tau$ over $\IB^n_{R(\tau)} \times N$---despite the rapidly decaying Gaussian weight used in the definition of $U^+$ and $\UU^-$.
Note that as $\UU^-$ decreases, the radius $R(\tau)$ \emph{increases,} enlarging the region on which this uniform bound holds.
However, this improvement is not guaranteed: if some component of $U^+_\tau$ grows, then $\UU^-_\tau$ typically grows as well, which normally forces $R(\tau)$ to \emph{decrease.}
Thus the domain of validity of the $C^m-$bound is tightly coupled to the evolution of $U^+_\tau$ and $\UU^-_\tau$, which is in turn encoded in Assertion~\ref{Thm_PDE_ODI_principle_b}.
This interdependence shows again why $R(\tau)$ must be chosen depending on these quantities in a careful manner.
\medskip

Among the assumptions, two are particularly crucial.
The first is the pseudolocality property from Assumption~\ref{Thm_PDE_ODI_principle_i}, which is indispensable: without it, no result of this form can hold.
The second is the smallness condition in Assumption~\ref{Thm_PDE_ODI_principle_ii} on the weighted $L^2_f$-bound over a ball of fixed radius $R^*$.
The precise term on the right-hand side is less relevant; the important point here is that the $L^2_f$-norm is bounded by small number that depends on the same parameters as $R^*$.

It is helpful to interpret Assumption~\ref{Thm_PDE_ODI_principle_ii} in light of the dynamic description of $ U^+_\tau $ and $\UU^-$ from Assertion~\ref{Thm_PDE_ODI_principle_b}.
In many settings, the evolution inequalities from this assertion themselves yield bounds on $\Vert U^+_\tau \Vert+\UU^-$ that ensure Assumption~\ref{Thm_PDE_ODI_principle_ii} \emph{a posteriori}---so in such cases Assumption~\ref{Thm_PDE_ODI_principle_ii} can often be omitted.

More concretely, suppose a solution is defined on a time interval $I$ and satisfies Assumptions~\ref{Thm_PDE_ODI_principle_0}, \ref{Thm_PDE_ODI_principle_i} and \ref{Thm_PDE_ODI_principle_iv}, while Assumption~\ref{Thm_PDE_ODI_principle_ii} is known only on $I\cap(-\infty,T)$, up to some time $T$.
Applying the theorem on this subinterval yields evolution inequalities for $ U^+_\tau $ and $\UU^-_\tau$, which sometimes provide control of the form $\Vert U^+_\tau \Vert+\UU^- \leq \delta$.
If $\delta=\delta(\eta, R^*) > 0$ is sufficiently small, then such a bound implies $R(\tau) > R^* + 1$ and 
\[ \| u_{\tau}  \|_{L^2_f(\IB^n_{ R^*} \times N)} \leq \Vert U^+_\tau \Vert+\UU^- <  \eta e^{-(R^*)^2/8}. \]
This is stronger than Assumption~\ref{Thm_PDE_ODI_principle_ii} near $\tau=T$ and therefore shows that the assumption continues to hold beyond time $T$.
Consequently, if $T$ was chosen maximal with the property that Assumption~\ref{Thm_PDE_ODI_principle_ii} holds on $I\cap(-\infty,T)$, then we obtain a contradiction unless $I\cap(-\infty,T) = I$.

In summary, whenever the evolution inequalities ensure sufficient control on $U^+_\tau$ and $\UU^-_\tau$, Assumption~\ref{Thm_PDE_ODI_principle_ii} is automatically propagated forward in time.
Conversely, a failure of this assumption, at some time $\tau$, must be accompanied by an increase in $\|U^+_\tau\|+\UU^-_\tau$, which must typically be due to the growth of an unstable component of $U^+_\tau$.

\medskip

The remainder of this section is devoted to proving Theorem~\ref{Thm_PDE_ODI_principle}.
We will first establish some key preparatory lemmas in Subsection~\ref{subsec_PDEODE_prep} and then finish the proof in Subsection~\ref{subsec_PDEODE_proof}.
\bigskip

\subsection{Preparation of the proof} \label{subsec_PDEODE_prep}
In the following we will fix the dimension $n$, the manifold $N$, and the non-linear PDE \eqref{eq_general_PDE}.
We will omit the dependence of constants on these data in the lemmas leading up to the proof of Theorem~\ref{Thm_PDE_ODI_principle}.
Define $\Vert u \Vert_{H^1_f} := \Vert u \Vert_{L^2_f} + \Vert \nabla u \Vert_{L^2_f}$.

We will frequently use the following two basic lemmas:

\begin{Lemma}[Poincar\'e inequality] \label{Lem_Poincare}
For any (possibly vector-valued) $u \in C^1(\IR^n \times M)$ we have
\[ \Vert (r+1) u \Vert_{L^2_f} \leq C \Vert u \Vert_{H^1_f}. \]
Here $r$ denotes the radial function on the $\IR^n$-factor and $C$ is a dimensional constant.
\end{Lemma}

\begin{proof}
Suppose first that $u$ is compactly supported.
Using integration by parts, we obtain since $\triangle f = \frac{n}2$,
\begin{multline*}
 \big\| |\nabla f| \, u \big\|_{L^2_f}^2
= \int_{\IR^n \times N} |\nabla f|^2 |u|^2 e^{-f} dg
= - \int_{\IR^n \times N} |u|^2 \nabla f  \cdot \nabla e^{-f} dg \\
= 2\int_{\IR^n \times N} u\cdot \nabla u \cdot \nabla f  \, e^{-f} dg
+ \int_{\IR^n \times N} \triangle f \, |u|^2 \, e^{-f} dg
\leq \tfrac12 \big\| |\nabla f| \, u \big\|^2_{L^2_f} + 2\big\| \nabla u \big\|^2_{L^2_f} + \tfrac{n}2 \big\| u \|_{L^2_f}^2. 
\end{multline*}
Subtracting the first term on the right-hand side on both sides of this inequality, yields the desired bound; note that $|\nabla f| = \tfrac12 r$.
If $u$ is not compactly supported, then we can consider the functions $u_i := u\omega_{R_i}$ for $R_i \to \infty$; note that $\Vert u_i \Vert_{L^2_f} \to \Vert u \Vert_{L^2_f}$ and $\Vert u_i \Vert_{H^1_f} \to \Vert u \Vert_{H^1_f}$.
\end{proof}
\medskip

\begin{Lemma} \label{Lem_polynomial_bounds}
For any $U^+ \in \sV_{> \la}$ we have polynomial bounds of the form
\[ |\nabla^k U^+| \leq C(\la, k) (r+1)^{C(\la)} \Vert U^+ \Vert_{L^2_f}. \]
\end{Lemma}

\begin{proof}
Since $\sV_{> \la}$ is finite dimensional, it suffices to prove this bound for eigenfunctions $L U^+ = \la' U^+$, $\la' > \la$.
Consider an eigenfunction $L_N v = \la_N v$ of the operator $L_N = \triangle_{N} + A$ over $N$.
Then $w := \int_N (U^+ \cdot v) dg_N$ satisfies the eigenvalue equation $\triangle_f w = (\la'-\la_N) w$ on $\IR^n$, so it is a product $\mathfrak p^{(j_1)} (\bx_1) \cdots \mathfrak p^{(j_n)}(\bx_n)$ of Hermite polynomials with $\la' - \la_N = - \frac12( j_1 + \ldots + j_n)$.
Since $\la_N$ is bounded from above, we obtain upper bounds on the degrees $j_1, \ldots, j_n$ and a lower bound on $\la_N$ whenever $w$ is non-zero.
So $U^+$ must be a linear combination of terms of the form $\mathfrak p^{(j_1)} (\bx_1) \cdots \mathfrak p^{(j_n)}(\bx_n) v$, where $v$ is an eigenfunction of $L_N$ whose eigenvalue is bounded by a constant depending on $\la$.
These terms satisfy the asserted polynomial bounds.
\end{proof}
\medskip

The following lemma allows us to relate $Q_J$ with the non-linear term $Q$.

\begin{Lemma}\label{lem:Q_initial}
Suppose that $\la \in \IR$ and $0 < \eta \leq \ov\eta(\la)$ and $J \geq 1$.
Consider a smooth (possibly vector-valued) function $u \in C^\infty (\IR^n \times N)$ with
\[ |u|, |\nabla u|, \ldots, |\nabla^4 u| \leq \eta. \]
Consider the decomposition $u = U^+ + U^-$ according to the splitting \eqref{eq_splittingVgeqless}. 
Then we have the following bounds:
\begin{align}
 \Big\| \PP_{\sV_{> \la}} \Big( Q[U^+ + U^-, \mathbf Y] - Q_J (U^+, \mathbf Y) \Big) \Big\|_{L^2_{f}} 
 &\leq \big( C(\la) \sqrt\eta + C \Vert \mathbf Y \Vert \big) \| U^- \|_{L^2_{f}}  \notag \\
 & \qquad\qquad\qquad + C(\la, J) \Vert U^+ \Vert^{J+1}_{L^2_f}, \label{eq_Q_approx_1} \\
 \Big\langle U^-,  Q[U^+ + U^-, \mathbf Y] - Q_J (U^+, \mathbf Y) \Big\rangle_{L^2_{f}} 
 &\leq \big( C(\la) \sqrt\eta + C \Vert \mathbf Y \Vert \big)  \| U^- \|_{H^1_{f}}^2 \notag \\
 & \qquad\qquad\qquad + C(\la, J) \| U^- \|_{L^2_{f}}  \Vert U^+ \Vert^{J+1}_{L^2_f}. \label{eq_Q_approx_2}
\end{align}
\end{Lemma}

\begin{proof}
We will denote by $C$ a generic constant that depends on the parameters listed in parentheses.
For convenience we will frequently drop the subscript $L^2_{f}$ next  inner products and norms.

We will prove both inequalities \eqref{eq_Q_approx_1}, \eqref{eq_Q_approx_2} in tandem.
To do this we fix some arbitrary $U' \in \sV_{> \la}$ with $\Vert U' \Vert_{L^2_{f}} = 1$ (for inequality \eqref{eq_Q_approx_1}) and $U' = U^- \in \sV_{\leq \la}$ (for inequality \eqref{eq_Q_approx_2}).
It is then enough to derive an upper bound on
\[ \Big\langle U',  Q[U^+ + U^-, \mathbf Y] - Q_J (U^+, \mathbf Y) \Big\rangle  \]

Choose $c_0 > 0$ such that $Q(u,\nabla u, \nabla^2 u, \mathbf Y)$ is defined whenever $|u|, |\nabla u|, |\nabla^2 u| \leq 2c_0$.
By Lemma~\ref{Lem_polynomial_bounds} there is a constant $C_0(\lambda)$ such that
\begin{equation}\label{eq:higher_derivatives}
    |U^+|, \; |\nabla U^+|, \; \ldots, \; |\nabla^4 U^+| \leq C_0 (r+1)^{C_0} \| U^+ \|,
\end{equation}
where $r$ denotes the radial distance on the $\IR^{n}$-factor.
So if we choose
\begin{equation} \label{eq_defRUp}
 R :=  \| U^+ \|^{-\frac1{2C_0}} , \qquad 
 \omega := \omega_R , \qquad
 \td U^\pm := U^\pm \omega,
\end{equation}
and $\td U^{\pm} := U^{\pm}$ if $U^+ = 0$, then 
\begin{equation} \label{eq_tdUpm_bound}
 \Vert \td U^\pm \Vert_{L^2_f} \leq \Vert U^\pm \Vert_{L^2_f}, \qquad
 \Vert \td U^\pm \Vert_{H^1_f} \leq C \Vert U^\pm \Vert_{H^1_f}
\end{equation}
and we have the following bounds if $\eta \leq \ov\eta(\la)$
\begin{align}
 \big\Vert \td U^+ \big\Vert_{C^4} 
 &\leq C(\la) (R^{C_0}+1) \Vert U^+ \Vert 
 \leq C(\la) \Vert U^+ \Vert^{1/2}
 \leq C(\la)\sqrt{\eta}
 \leq c_0, \label{eq_tdUp_C2} \\ 
\big\Vert \td U^- \big\Vert_{C^4} 
&= \big\Vert u\omega - \td U^+ \big\Vert_{C^4}
\leq C \eta + C(\la) \sqrt{\eta}
\leq C(\la)\sqrt{\eta}
 \leq c_0. \label{eq_tdUm_C2} 
\end{align}
Thus $Q[\td U^+, \mathbf Y]$ is defined on all of $\IR^n \times N$.

Since $Q[ U^+ +  U^-, \mathbf Y] = Q[\td U^+ + \td  U^-, \mathbf Y]$ on $\IB^n_{R-1} \times N$ and since 
$$ \big|Q[ U^+ +  U^-, \mathbf Y] \big|, \; \big|Q[\td U^+ + \td U^-, \mathbf Y] \big| \leq {C(r+1)^C} ,$$
we have 
\begin{multline} 
 \big\|  Q[U^+ + U^-, \mathbf Y]  - Q[\td U^+ + \td U^-, \mathbf Y]  \big\|_{L^2} 
\leq C (R+1)^C \bigg( \int_{(\IR^n \setminus \IB^n_{R-1}) \times N}  e^{-f} dg \bigg)^{1/2}  \\
\leq C(R+1)^C   e^{-(R-2)^2/8} 
\leq C(J) R^{-2C_0(J+1)}
\le C(J) \Vert U^+ \Vert^{J+1}.\label{eq_Quputd} 
\end{multline}

Next, since $Q[\td U^+, \mathbf Y]$ is defined everywhere, we can derive the following bound in a similar manner
\begin{align}
 \big\|  Q[\td U^+, \mathbf Y] &- Q_J(U^+, \mathbf Y) \big\|^2_{L^2} \notag \\
&= \int_{\IB^n_{R-1} \times N} \big| Q[ U^+, \mathbf Y] - Q_J(U^+, \mathbf Y) \big|^2 e^{-f} dg \notag \\
&\qquad+ \int_{(\IR^n \setminus \IB^n_{R-1}) \times N} \big| Q[ \td U^+, \bY] - Q_J(U^+, \bY) \big|^2  e^{-f} dg \notag \\
&\leq C \int_{\IB^n_{R-1} \times N}  \big( |U^+|^{J+1} + | \nabla U^+ |^{J+1} + | \nabla^2 U^+ |^{J+1} \big)^2 (r+1)^C  e^{-f} dg \notag \\
&\qquad+ C(\la, J) \int_{(\IR^n \setminus \IB^n_{R-1}) \times N} \big(  1 + (r+1)^{JC_0}  (r+1)^C \big)^2 e^{-f} dg \notag \\
&\leq C(\la) \| U^+ \|^{2J+2} \int_{\IB^n_{R-1} \times N}   (r+1)^{C_0(2J+2) +C}  e^{-f} dg + C(\la,J)  e^{-(R-2)^2/4} \notag \\
&\leq C(\la,J)  \| U^+ \|^{2J+2} +  C(\la,J)  e^{-(R-2)^2/4}. \label{eq_Qtdupup}
\end{align}
Combining \eqref{eq_Quputd}, \eqref{eq_Qtdupup} and using the definition of $R$ in \eqref{eq_defRUp} implies that
\begin{align}
 \big\langle U', Q[U^+ + U^-,& \bY] - Q[\td U^++ \td U^-, \bY] + Q[\td U^+, \bY] - Q_J (U^+, \bY) \big\rangle \notag \\
 &\leq \Vert U' \Vert_{L^2_f} \cdot \big\Vert Q[U^+ + U^-, \bY] - Q[\td U^+ + \td U^-, \bY] + Q[\td U^+, \bY] - Q_J (U^+,\bY) \big\Vert_{L^2_f} \notag \displaybreak[1] \\
&\leq C(\la,J) \Vert U' \Vert_{L^2_f} 
\big(  \Vert U^+ \Vert^{J+1} + e^{-(R-2)^2/8}  \big), \notag \displaybreak[1]  \\
&\leq C(\la,J) \Vert U' \Vert_{L^2_f} 
\big( \Vert U^+ \Vert^{J+1} + R^{-2C_0(J+1)} \big), \notag \\
&\leq C(\la,J) \Vert U' \Vert_{L^2_f} 
\Vert U^+ \Vert^{J+1} . \label{eq_QQQQ_bound}
\end{align}

It remains to bound $\langle  U', Q[\td U^+ + \td U^-, \bY] - Q[\td U^+, \bY] \rangle$ from above.
Due to the bounds \eqref{eq_tdUp_C2} \eqref{eq_tdUm_C2} we can write
\[ Q[\td U^+ + \td U^-,\bY] - Q[\td U^+,\bY] 
= \sum_{i=0}^2 Q^*_i * \nabla^i \td U^-, \]
where
\[ Q^*_i(\bx,\by) := \int_0^1 (\partial_{\nabla^i u} Q)(\bx, \by, \td U^+ + s \td U^-, \nabla \td U^+ + s \nabla \td U^-,  \nabla^2 \td U^+ + s \nabla^2 \td U^-,  \bY ) ds. \]
are tensor-valued smooth functions and $*$ denotes an unspecified tensor multiplication.
The assumed bounds \eqref{eq_condition_on_Q} on $Q$ combined with \eqref{eq_tdUp_C2} and \eqref{eq_tdUm_C2} imply that, for $i = 0,1,2$,
\[ | Q^*_i | + \ldots + | \nabla^2 Q^*_i | \leq  (r+1)^{2-i} \big( C(\la) \sqrt\eta + C\Vert\bY\Vert \big). \]
So we can write
\begin{equation} \label{eq_sumQstari}
 \big\langle  U', Q[\td U^+ + \td U^-, \bY] - Q[\td U^+, \bY] \big\rangle = \sum_{i=0}^2  \big\langle  U', Q^*_i * \nabla^i \td U^- \big\rangle. 
\end{equation}
In order to bound this term, we need to consider two cases.
\medskip

\textit{Case 1: $U' \in \sV_{> \la}$ \quad}
In this case, we obtain using Lemma~\ref{Lem_polynomial_bounds}, since $\sV_{> \la}$ is finite dimensional and since $\Vert U' \Vert = 1$:
\[  \big\langle  U', Q^*_0 * \td U^- \big\rangle 
\leq (C (\la) \sqrt\eta + C\Vert\bY\Vert) \big\Vert (r+1)^2 U' \big\Vert_{L^2_f} \big\Vert \td U^- \big\Vert_{L^2_f}
\leq (C (\la) \sqrt\eta + C\Vert\bY\Vert) \big\Vert \td U^- \big\Vert_{L^2_f} . \]
The term involving $Q^*_1$ can be bounded via integration by parts as follows:
\begin{align}
  \big\langle  U', Q^*_1 * \nabla \td U^- \big\rangle 
&= \int_{\IR^n \times N} U' * Q^*_1 * \nabla \td U^- \, e^{-f} dg \notag \\
&= \int_{\IR^n \times N} \nabla \big(U' * Q^*_1 * e^{-f} \big)  * \td U^- \,  dg \notag \\
&\leq \big\Vert \nabla U' * Q^*_1 + U' * \nabla Q^*_1 - U' * Q^*_1 * \nabla f \big\Vert_{L^2_f} \big\Vert \td U^- \big\Vert_{L^2_f} \notag \\
 &\leq (C (\la) \sqrt\eta + C\Vert\bY\Vert)  \big\Vert \td U^- \big\Vert_{L^2_f}. \label{eq_Qs1_bound_Case1}
\end{align}
The term involving $Q^*_2$ can be bounded similarly, by applying integration by parts twice:
\[ \big\langle  U', Q^*_2 * \nabla^2 \td U^- \big\rangle 
\leq (C (\la) \sqrt\eta + C\Vert\bY\Vert) \big\Vert \td U^- \big\Vert_{L^2_f}. \]
It follows that
\[ \big\langle  U', Q[\td U^+ + \td U^-, \bY] - Q[\td U^+, \bY] \big\rangle \leq (C (\la) \sqrt\eta + C\Vert\bY\Vert)  \big\Vert \td U^- \big\Vert_{L^2_f}. \]
Combining this bound with \eqref{eq_tdUpm_bound}, \eqref{eq_QQQQ_bound} and taking the supremum over all $U' \in \sV_{> \la}$ with $\Vert U' \Vert = 1$ implies  \eqref{eq_Q_approx_1}.
\medskip

\textit{Case 2: $U' = U^-$ \quad}
In this case, we use the Poincar\'e inequality, Lemma~\ref{Lem_Poincare}, to obtain
\begin{multline*}
  \big\langle  U', Q^*_0 * \td U^- \big\rangle 
\leq (C (\la) \sqrt\eta + C\Vert\bY\Vert) \big\Vert (r+1) U' \big\Vert_{L^2_f} \big\Vert (r+1) \td U^- \big\Vert_{L^2_f} \\
\leq (C (\la) \sqrt\eta + C\Vert\bY\Vert) \big\Vert  U' \big\Vert_{H^1_f} \big\Vert \td U^- \big\Vert_{H^1_f} . 
\end{multline*}
The term involving $Q^*_1$ can be bounded similarly:
\begin{multline*}
 \big\langle  U', Q^*_1 * \nabla \td U^- \big\rangle 
\leq (C (\la) \sqrt\eta + C\Vert\bY\Vert) \big\Vert (r+1)  U' \big\Vert_{L^2_f} \big\Vert  \td U^- \big\Vert_{H^1_f} \\
\leq (C (\la) \sqrt\eta + C\Vert\bY\Vert) \big\Vert  U' \big\Vert_{H^1_f} \big\Vert \td U^- \big\Vert_{H^1_f} .
\end{multline*}
To bound the term involving $Q^*_2$ we only perform a single integration by parts, as in \eqref{eq_Qs1_bound_Case1}:
\begin{align*}
  \big\langle  U', Q^*_2 * \nabla^2 \td U^- \big\rangle 
&\leq \big\Vert \nabla U' * Q^*_2 + U' * \nabla Q^*_2 - U' * Q^*_2 * \nabla f \big\Vert_{L^2_f} \big\Vert  \nabla \td U^- \big\Vert_{L^2_f} \\
&\leq (C (\la) \sqrt\eta + C\Vert\bY\Vert) \big( \Vert  U' \Vert_{H^1_f} + \Vert (r+1) U' \Vert_{L^2_f} \big) \big\Vert \td U^- \big\Vert_{H^1_f} \\
&\leq (C (\la) \sqrt\eta + C\Vert\bY\Vert)  \big\Vert  U' \big\Vert_{H^1_f}  \big\Vert \td U^- \big\Vert_{H^1_f} .
\end{align*}
It follows that
\[ \big\langle  U', Q[\td U^+ + \td U^-, \bY] - Q[\td U^+, \bY] \big\rangle \leq  (C (\la) \sqrt\eta + C\Vert\bY\Vert)  \big\Vert  U' \big\Vert_{H^1_f}  \big\Vert \td U^- \big\Vert_{H^1_f}. \]
Combining this bound again with \eqref{eq_tdUpm_bound}, \eqref{eq_QQQQ_bound} implies  \eqref{eq_Q_approx_2}, finishing the proof of the lemma.
\end{proof}
\medskip

\bigskip
The next lemma establishes the evolution inequalities from Assertion~\ref{Thm_PDE_ODI_principle_b} of Theorem~\ref{Thm_PDE_ODI_principle} under the assumption that the solution is small in the $C^4$-sense.

\begin{Lemma}\label{l:evolution_ODE}
There is a constant $c > 0$ such that the following is true.
Suppose that $\lambda \in \IR$ and $0 < \eta \leq \ov\eta(\la)$ and $J \geq 1$.
Consider a smooth solution $u \in C^\infty(\DD;\IR^m)$ to \eqref{eq_general_PDE} over some time-interval $I$, for a smooth family $(\mathbf Y_\tau \in \YY)_{\tau \in I}$.
Let $R : I \to \IR_+$ be a smooth function of radii and suppose that for all $\tau \in I$ we have
\begin{equation} \label{eq_h_bound}
\Vert \bY_\tau \Vert \leq c, \qquad
 \IB^n_{R(\tau)} \times N  \subset \DD_\tau, \qquad  \Vert u_\tau \Vert_{C^4(\IB^n_{R(\tau)} \times N )} \leq \eta, \qquad |\partial_\tau R(\tau)| \leq 10 e^{R(\tau)/10}.
\end{equation}
Then the quantities (note that $(\td\UU^-_\tau)$ differs from $(\UU^-_\tau)$ in \eqref{eq_UpUUm_def} in the lack of the last term)
\begin{equation} \label{eq_UpUUm_def_without}
 U^+_\tau :=  \PP_{\sV_{> \lambda}} (u_{\tau} \omega_{R(\tau)})  , \qquad
  U^-_\tau :=  \PP_{\sV_{\leq \lambda}} (u_{\tau} \omega_{R(\tau)})  , \qquad
\td\UU^-_\tau := \|   U^- \|_{L^2_{f}} . 
\end{equation}
satisfy the following evolution inequalities: 
\begin{equation} \label{eq_evol_ineq_1_lem}
 \Big\| \partial_\tau U^+_\tau - L U^+_\tau  - Q_J^+ (U^+_\tau ,\mathbf Y_\tau) \Big\|_{L^2_{f}}
\leq C(\la, J) \Vert U^+_\tau \Vert^{J+1}_{L^2_{f}} + (C(\la)\sqrt{\eta} + C \Vert \mathbf Y_\tau \Vert) \, \td\UU^-_\tau  + C \eta e^{-\frac{(R(\tau)-2)^2}8} 
\end{equation}
\begin{equation} \label{eq_evol_ineq_2_lem}
  \partial_\tau \td\UU^-_\tau 
\leq  \big(\la +C (\la) \sqrt\eta + C(\la) \Vert \bY_\tau \Vert \big) \td\UU^-_\tau  
+  \Vert Q^-_J(U^+_\tau, \mathbf Y_\tau) \Vert_{L^2_{f}}
+ C(\la,J) \Vert U^+_\tau \Vert^{J+1}_{L^2_{f}}
+ C \eta e^{-\frac{(R(\tau)-2)^2}8}  
\end{equation}
\end{Lemma}

\begin{proof}
As always, we will denote by $C$ a generic constant and we will indicate dependencies in parentheses.
For convenience we will write $\omega_\tau = \omega_{R(\tau)}$ and we will frequently drop the ``$\tau$''-subscript in time-dependent quantities.
When the context is clear, we will also frequently omit the subscript ``$L^2_{f}$'' on norms and inner products.

By direct computation, we obtain the following evolution equation for $u\omega$: 
\begin{equation} \label{eq_evol_eq_withE}
\partial_\tau (u\omega)=L (u\omega) + Q[u\omega, \mathbf Y] + E, 
\end{equation}
where
\begin{equation}
   E= \partial_\tau  \omega \cdot  u - 2\nabla \omega \cdot \nabla u -\Delta_{f} \omega \cdot u   + \big( Q[u, \mathbf Y] \omega - Q[u\omega, \mathbf Y] \big).
\end{equation}
Note that $E_\tau$ is supported on $(\IB^n_{R(\tau) } \setminus \IB^n_{ R(\tau) -1}) \times N$, where we can use the bounds $e^{-f}\leq e^{-(R(\tau)-1)^2/4}$ and $|\partial_\tau \omega| \leq Ce^{R(\tau)/10}$.
Combining this with \eqref{eq_h_bound}, we obtain
\begin{equation} 
\Vert E \Vert 
  \leq \Big(  C \eta^2  R^{C} e^{\frac{R}{10}} e^{-\frac{( R-1)^2}{4}} \Big)^{1/2} 
   \leq  C\eta e^{-\frac{( R-2)^2}8}. \label{eq_I4}
\end{equation}

Projecting \eqref{eq_evol_eq_withE} to $\sV_{>\lambda},\sV_{\leq\lambda}$ and noting that $L$ commutes with $\PP_{\sV_{>\lambda}},\PP_{\sV_{\leq\lambda}}$ yields
\begin{alignat*}{2}
 \partial_\tau U^+ &= L U^+ + Q^+_J(U^+, \mathbf Y) &&+ \PP_{\sV_{> \la}} \big(   Q[U^+ + U^-, \mathbf Y] - Q_J(U^+, \mathbf Y) \big)  
+ \PP_{\sV_{> \la}}  E , \\
 \partial_\tau U^- &= L U^-  + Q^-_J(U^+, \mathbf Y) &&+ \PP_{\sV_{\leq \la}} \big(   Q[U^+ + U^-, \mathbf Y] - Q_J(U^+, \mathbf Y) \big) 
 + \PP_{\sV_{\leq \la}}  E.
\end{alignat*}
We therefore obtain using \eqref{eq_Q_approx_1} from Lemma \ref{lem:Q_initial} and assuming $0 <  \eta \leq \ov\eta(\la)$ that
\begin{multline*}
 \Big\| \partial_\tau U^+ - L U^+ - Q_J^+(U^+, \mathbf Y) \Big\| 
\leq \big\| \PP_{\sV_{> \la}} \big( Q[U^+ + U^-, \mathbf Y] - Q_J(U^+, \mathbf Y) \big) \big\|   
 +\| \PP_{\sV_{>\lambda}} E  \| \\
\leq (C(\la) \sqrt\eta + C\Vert\bY\Vert )\| U^- \|_{L^2_{f}} 
+ C(\la, J) \Vert U^+ \Vert^{J+1} 
+ \Vert E \Vert_{L^2_f} ,
\end{multline*}
which shows \eqref{eq_evol_ineq_1_lem} in combination with \eqref{eq_I4}.
Also by \eqref{eq_Q_approx_2}
\begin{align}
\| U^- \|^2_{L^2_f} \cdot \partial_\tau \| U^- \|_{L^2_f} &=
 \tfrac12 \partial_\tau \| U^- \|^2_{L^2_f}  \notag \\
&= \big\langle U^-, L U^- \big\rangle 
+ \big\langle U^-, Q_J^- (U^+, \mathbf Y) \big\rangle \notag \\
&\qquad + \big\langle U^- , Q[U^+ + U^-, \mathbf Y] - Q_J(U^+, \mathbf Y) \big\rangle
+ \big\langle U^-, E \big\rangle \notag \\
&\leq \big\langle U^-, L U^- \big\rangle 
+ \Vert U^- \Vert_{L^2_f} \cdot \Vert Q_J^-(U^+, \mathbf Y) \Vert_{L^2_f}  + (C(\la) \sqrt\eta + C\Vert \bY \Vert) \| U^- \|_{H^1_f}^2 \notag \\
&\qquad + C(\la, J) \| U^- \|_{L^2_f} \cdot \Vert U^+ \Vert^{J+1}+ \Vert U^- \Vert_{L^2_f} \cdot \Vert E \Vert_{L^2_f}. \label{eq_Um_lem_many}
\end{align}
It remains to bound the terms $\langle U^-, L U^- \rangle$ and $E$.
Let $\eps \in [0,1)$ be a constant whose value we will determine later.
Since that the maximal eigenvalue of $ L - \eps \triangle_{f} = (1-\eps) L  + \eps A $ restricted to $\sV_{\leq \la}$ is at most $(1-\eps)\lambda  + C \eps \leq \la + C(\la)\eps$, for some generic constant $C$, we have
\begin{equation}\label{eq_UmLumeps}
 \big\langle U^-, L U^- \big \rangle  
    =  \big\langle U^-, ( L - \eps \triangle_{f} )  U^-   \big\rangle  - \eps \|\nabla U^- \|^2_{L^2_f}  \le (\la  + C(\la) \eps) \|U^- \|^2_{L^2_f}-\eps \big\| U^- \big\|_{H^1_f}^2.
\end{equation}
Let us now assume that $\eta \leq \ov\eta(\la)$ and $\Vert \bY \Vert \leq c$ for some suitable constant $c > 0$ such that the term $\eps := C(\la) \sqrt\eta + C \Vert \mathbf Y\Vert$ in \eqref{eq_Um_lem_many} is bounded by $0.1$.
Combining \eqref{eq_I4}, \eqref{eq_Um_lem_many} and \eqref{eq_UmLumeps} for this choice of $\eps$ yields \eqref{eq_evol_ineq_2_lem}.
\end{proof}

\bigskip

\subsection{Proof of Theorem~\ref{Thm_PDE_ODI_principle}}
\label{subsec_PDEODE_proof}

\begin{proof}[Proof of Theorem~\ref{Thm_PDE_ODI_principle}]
For convenience we will omit all ``$L^2_f$''-subscripts.
Choose $T' \geq -\infty$ maximal such that $u_\tau \equiv 0$ and $\DD_\tau = \IR^n \times N$ for all $\tau \leq T'$ and write $\IB_R = \IB^n_R$.
Since we can set $R(\tau) = \infty$ for $\tau \leq T'$, {we may shrink $I$ and assume $T'=\min I$.}
Hence it suffices to construct $R(\tau)$ with $R(\tau) < \infty$ for $\tau > \min I$.

Let us now assume for a moment that we have already chosen $R^*$ and constructed the function $R : I \to [R^*,\infty]$ such that the bounds \eqref{eq_h_bound} from Lemma~\ref{l:evolution_ODE} hold.
Define
\begin{equation} \label{eq_UpUmE}
  U^+_\tau :=  \PP_{\sV_{> \lambda}} (u_{\tau} \omega_{R(\tau)})  , \qquad
\td\UU^-_\tau := \big\|   \PP_{\sV_{\leq \lambda}}  (u_{\tau} \omega_{R(\tau)} ) \big\|, \qquad E_\tau := e^{-\frac{((1-\eps)R(\tau))^2}{8}}.  
\end{equation}
So $\UU^-_\tau = \td\UU^- + \eta E_\tau$.
The evolution inequalities from Assertion~\ref{Thm_PDE_ODI_principle_b} of our theorem differ from \eqref{eq_evol_ineq_1_lem} and \eqref{eq_evol_ineq_2_lem} as they involve the quantity $\UU^-_\tau = \td\UU^-_\tau + \eta E_\tau$ instead of $\td\UU^-_\tau$.
If we adjust the constants $C(\la)$ and $C(\la, J)$ in front of the last term of the second inequality, then the bounds \eqref{eq_evol_ineq_1_lem} and \eqref{eq_evol_ineq_2_lem} imply Assertion~\ref{Thm_PDE_ODI_principle_b} if the following additional bounds can be guaranteed for all $\tau \in I$ (with $\tau \neq T_0$ if $R_0 = \infty$):
\begin{align} \label{eq_eR8_leq_E}
 e^{-\frac{(R(\tau)-2)^2}{8}} &\leq \eta e^{-\frac{((1-\eps)R(\tau))^2}{8}} = \eta E_\tau , \\ 
 \label{eq_need}
 \eta \partial_\tau  E_\tau
&\leq   \la'  \eta E_\tau  + \eta \td\UU^-_\tau  
+ \| U^+_\tau \|^{J+1}.
\end{align}
Here we set $\la' := \min \{ \la , -1 \}$ for reasons that will become clear later.
To complete the proof of the theorem, it therefore suffices to construct a function $R$ such that these bounds hold and such that for all $\tau \in I$
\begin{equation} \label{eq_dtR_between}
-e^{R(\tau)/10} \leq \partial_\tau R(\tau) \leq 1, 
\end{equation}
\begin{equation} \label{eq_BND_ueta}
\IB_{R(\tau)} \times N \subset \DD_\tau \qquad \text{and} \qquad \Vert u_\tau \Vert_{C^{m}(\IB_{R(\tau)} \times N )} \leq  \eta.
\end{equation}

In the remainder of the proof we will construct a function $R : I \to [R^*, \infty]$, for a suitable choice of $R^* (m,J,\la,\eta, \sigma,\eps)$, that satisfies Properties~\eqref{eq_eR8_leq_E}--\eqref{eq_BND_ueta} assuming $\eps \leq \ov\eps(\sigma)$.
Once such a function has been produced, the theorem follows.

{
Next, we reduce the construction of such a function to the case where $I = [T_0, T_1]$ is compact and $R_0 < \infty$, and we argue that it suffices to construct $R(\tau)$ to be locally Lipschitz whenever it takes finite values.
To see this, observe that the required conditions are equivalent to a collection of inequalities involving $R(\tau)$ and $\partial_\tau R(\tau)$; for example, \eqref{eq_BND_ueta} is equivalent to a pointwise upper bound on $R(\tau)$.
If $R(\tau)$ is only locally Lipschitz where finite, then it can be smoothened, and the resulting function will still satisfy slightly weakened versions of Properties~\eqref{eq_eR8_leq_E}--\eqref{eq_dtR_between}, which are sufficient for our purposes. Moreover, the smoothing can be chosen so as not to increase $R(\tau)$ pointwise, which ensures that Property~\eqref{eq_BND_ueta} remains preserved. 
Thus it suffices to construct $R(\tau)$ locally Lipschitz wherever it is finite.
Next, we argue that we may assume without loss of generality that $R_0 < \infty$.
Indeed, suppose that $R_0 = \infty$ and assume that for every finite initial value we can construct $R(\tau)$ satisfying Properties~\eqref{eq_eR8_leq_E}--\eqref{eq_BND_ueta}.
Choose a sequence of such functions $R_i(\tau)$ with $R_i(T_0) \to \infty$.
The bound \eqref{eq_dtR_between} implies that the functions $(R_i(\tau) + 1)^{-1}$ (with $(\infty+1)^{-1} = 0$) are equicontinuous and uniformly bounded. 
So by Arzela-Ascoli, a subsequence converges locally uniformly to a function of the form $(R_\infty(\tau)+1)^{-1}$, where $R_\infty(\tau)$ has the desired properties and $R_\infty(T_0)= \infty$.
It is also clear that Assertion~\ref{Thm_PDE_ODI_principle_d} passes to the limit due to the bound \eqref{eq_dtR_between}.
Finally, if $I$ is not compact, we can approximate it by a sequence of intervals $[T_{0,i},T_{1,i}] \subset I$, find $R_i(\tau)$ for each of these intervals and pass to a a limit as before.}

The following claim accomplishes such a construction up to a certain time $T$.

\begin{Claim} \label{Cl_construct_R}
There is a constant $c > 0$ such that if we assume $\Vert \bY_\tau \Vert \leq c$ for all $\tau \in I$, then the following is true for $R^* \geq \underline{R}^*(\la,\eta,\eps)$.
Suppose that the constant $R_0$ from Assumption~\ref{Thm_PDE_ODI_principle_iv} satisfies $R_0 \geq R^*$.
Then there is a $T \leq T_1$ and a piecewise smooth function $R : [T_0,T] \to [R^*,\infty)$ with $R(T_0) = R_0$ such that 
for all $\tau \in [T_0, T]$ the Properties~\eqref{eq_eR8_leq_E}--\eqref{eq_BND_ueta} hold and such that:
\begin{enumerate}[label=(\alph*)]
\item \label{Cl_construct_R_b} If $T < T_1$, then one of the following is true: 
\begin{equation}\label{eq:contra_u_large}
 \ov{\IB}_{R(T)} \times N \not\subset \DD_{T} \qquad \text{or} \qquad   \Vert u_{T} \Vert_{C^{m}(\IB_{R(T)} \times N )} = \eta.
\end{equation}
\item \label{Cl_construct_R_e} If $E_\tau <  ( \frac{\td\UU^-_{\tau}}{\eta} + (\frac{\|U^+_{\tau}\|}{\eta})^{J+1} )^{1+\eps}$ for some $\tau \in I$, then $\partial_\tau R(\tau) = - e^{R(\tau)/10}$.
\end{enumerate}
\end{Claim}

This claim reduces our proof to showing that $T = T_1$.

\begin{proof}
The constant $c$ is the same constant as in Lemma~\ref{l:evolution_ODE}.
The bound \eqref{eq_eR8_leq_E} can be guaranteed easily by assuming a bound of the form $R(\tau) \geq R^* \geq \underline{R}^*(\eta)$.
We will now describe an iterative process to construct the function $R$.

Set $\tau_0 := T_0$ and choose $R(\tau_0) := R_0$.
We will define $R$ successively on time-intervals of the form $(\tau_i, \tau_{i+1}]$ based on the value $R(\tau_i)$ and the solution $u$ at time $\tau_i$.
So suppose that $R$ has already been chosen on $[\tau_0,\tau_i] = \{ \tau_0 \} \cup (\tau_0, \tau_1] \cup \ldots \cup (\tau_{i-1}, \tau_i]$.
Use this function to define $(U^+_\tau)_{\tau \in [T_0,\tau_i]}$, $(\td\UU^-_\tau)_{\tau \in [T_0,\tau_i]}$ and $(E_\tau)_{\tau \in [T_0, \tau_i]}$ via \eqref{eq_UpUmE}.
If $\tau_i = T_1$ or if one of the bounds of Assertion~\ref{Cl_construct_R_b} holds for $T = \tau_i$, then we are done.
Otherwise, we consider two cases.
\medskip

\textit{Case 1:  
$E_{\tau_i} >  2( \frac{\td\UU^-_{\tau_i}}{\eta} + (\frac{\|U^+_{\tau_i}\|}{\eta})^{J+1} )^{1+\eps}$ \quad}
In this case we first extend $R(\tau)$ onto $[T_0,T_1]$ such that  for all $\tau \in (\tau_i, T_1]$
\begin{equation} \label{eq_E_choice}
\partial_\tau E_{\tau} = \la' E_{\tau}.
\end{equation}
This guarantees \eqref{eq_need} and ensures that $\partial_\tau R > 0$ (since $\la' < 0$), so the bound   $R(\tau) \geq R^*$ remains preserved.
Moreover, Property~\eqref{eq_dtR_between} follows if we assume a bound of the form $R^* \geq \underline R^*(\la)$.

Choose $\tau'_{i+1} \geq \tau_i$ maximal such that Property~\eqref{eq_BND_ueta} holds for all $\tau \in [\tau_i, \tau'_{i+1})$.
Since we have assumed that $\tau_i < T_1$ and that none of the bounds in Assertion~\ref{Cl_construct_R_b} hold at time $\tau_i$, we must have $\tau'_{i+1} > \tau_i$.
Let us now extend $(U^+_\tau)$, $(\UU^-_\tau)$ and $(E_\tau)$ to $[T_0,\tau'_{i+1}]$ via \eqref{eq_UpUmE}.
Choose $\tau_{i+1} \in (\tau_i, \tau'_{i+1}]$ maximal such that 
\begin{equation} \label{eq_stop_case1}
E_\tau \geq  \bigg(  \frac{\td\UU^-_\tau}{\eta}  + \Big( \frac{\|U^+_\tau\|}{\eta} \Big)^{J+1} \bigg)^{1+\eps}\qquad \text{for all} \quad \tau \in [\tau_i, \tau_{i+1}].
\end{equation}
So Assertion~\ref{Cl_construct_R_e} is vacuously true for $\tau \in [\tau_i, \tau_{i+1}]$.
\medskip

\textit{Case 2:  
$E_{\tau_i} \le  2( \frac{\td\UU^-_{\tau_i}}{\eta} + (\frac{\|U^+_{\tau_i}\|}{\eta})^{J+1} )^{1+\eps}$ \quad} 
The construction in this case is similar, except that  instead of \eqref{eq_E_choice}, we solve
\[ \partial_\tau R(\tau) = - e^{R(\tau)/10}. \]
Moreover, instead of \eqref{eq_stop_case1} we require that
\begin{equation} \label{eq_stop_case2}
E_\tau \leq  4\bigg(  \frac{\td\UU^-_\tau}{\eta}  +  \Big( \frac{\|U^+_\tau \|}{\eta} \Big)^{J+1}  \bigg)^{1+\eps} \qquad \text{for all} \quad \tau \in [\tau_i, \tau_{i+1}].
\end{equation}
Then Properties~\eqref{eq_dtR_between}, \eqref{eq_BND_ueta} and Assertion~\ref{Cl_construct_R_e} hold by definition.
It remains to establish the bound $R(\tau) \geq R^*$ and Property~\eqref{eq_need} for $\tau \in [\tau_i, \tau_{i+1}]$.
To see this suppose that $R(\tau) = R^*$ for such a time $\tau$.
By Assumption~\ref{Thm_PDE_ODI_principle_ii} and \eqref{eq_BND_ueta} $$\Vert U^+_\tau \Vert = \Vert \PP_{\sV_{> \la}} (u_\tau \omega_{R(\tau)}) \Vert_{L^2_f} \leq \Vert u_\tau \omega_{R^*} \Vert_{L^2_f} \leq  \eta e^{-(R^*)^2/8}$$
and similarly $\td\UU^-_\tau \leq \eta e^{-(R^*)^2/8}$.
Therefore, \eqref{eq_stop_case2} implies a bound of the form $E_\tau \leq C e^{-(1+\eps)(R^*)^2/8}$, which implies $R(\tau) > R^*$ for $R^* \geq \underline{R}^*( \eps)$, in contradiction to our assumption.
To see Property~\eqref{eq_need}, we find that, assuming  $R^* \geq \underline{R}^*(\la, \eta, \eps)$,
\begin{multline*}
 \eta \partial_\tau E_\tau - \la' \eta E_\tau \leq \big( R(\tau) e^{R(\tau)/10} +|\la'| \big) \eta E_\tau 
\leq  \eta^{J+1} E_\tau^{\frac1{1+\eps}} \\
\leq  \eta^{J+1} \bigg(   \frac{\td\UU^-_\tau}{\eta}  +  \Big( \frac{\|U^+_\tau \|}{\eta} \Big)^{J+1} \bigg)
\leq \eta\td\UU^-_\tau +  \|U^+_\tau\|^{J+1}.  
\end{multline*}
\medskip

Let us now argue that our iterative process must terminate after a finite number of steps.
Assume by contradiction that the process would go on indefinitely, producing an increasing sequence $\tau_i \to \tau_\infty < \infty$.
Then by \eqref{eq_dtR_between}, the resulting function $R : [0,\tau_\infty) \to \IR_+$ can be extended continuously onto the closed interval $[0, \tau_\infty]$, so $(U^+_\tau)$, $(\UU^-_\tau)$ and $(E_\tau)$ also allow continuous extensions.
In each step of our iteration, we must have $\tau_{i+1} < \tau'_{i+1}$, because otherwise the process would terminate in the next step.
So if Case~1 applies over $[\tau_i, \tau_{i+1}]$, then we must have equality in \eqref{eq_stop_case1} for $\tau =\tau_{i+1}$.
However, since $E_{\tau_i} > 2  (\frac{\td\UU^-_{\tau_i}}{\eta} + (\frac{\|U^+_{\tau_i}\|}{\eta})^{J+1})^{1+\eps}$ and since the quantities on both sides of \eqref{eq_stop_case1} are continuous, we find that Case~1 cannot apply for infinitely many $i$.
The same reasoning shows that Case~2 can only apply for finitely many $i$, which yields the desired contradiction.
\end{proof}
\medskip

It remains to show that $T=T_1$ if $R^* \geq \underline R^* (m,J,\la, \eta, \sigma,\eps)$ and $\eps \leq \ov\eps(\sigma)$, so assume by contradiction that $T < T_1$, which implies \eqref{eq:contra_u_large}.

\begin{Claim}
$T \geq T_0 + 1$.
\end{Claim}

\begin{proof}
Suppose $T < T_0 +1$.
Then \eqref{eq_dtR_between} implies $R(T) < R_0 + 1$ and \eqref{eq:contra_u_large} contradicts Assumption~\ref{Thm_PDE_ODI_principle_iv}.
\end{proof}

In the following, we will use the pseudolocality property from Assumption~\ref{Thm_PDE_ODI_principle_i} applied  times $\tau_0, \tau_1$, which are chosen according to the following claim.

\begin{Claim} \label{Cl_E_geq_U}
Assuming $R^* \geq \underline R^*(\sigma)$, there are times $\tau_0, \tau_1 \in [T-1, T]$ with $\tau_0 \leq \tau_1$ such that the following is true for each $i = 0,1$:
\begin{alignat}{1} \label{eq_E_geq_UUU}
 E_{\tau_i} &\geq \bigg(\frac{\td\UU^-_{\tau_i}}{\eta} +  \Big( \frac{ \|U^+_{\tau_i}\|}{\eta} \Big)^{J+1}  \bigg)^{1+\eps} , \\
 T - \tau_0 &\geq \tfrac12,  \label{eq_Ts2mtau1} \\
 T - \tau_1 &\leq \sigma  R^{-1/\sigma}(T)  .  \label{eq_RTs2leq}
\end{alignat}
\end{Claim}

\begin{proof}
Suppose first that there is no $\tau_0 \in [T -  1, T - \tfrac12 ]$ such that \eqref{eq_E_geq_UUU} holds for $i = 0$. 
So the reverse inequality must hold over the entire time-interval $[T-1, T-\frac12]$.
Therefore, by Claim~\ref{Cl_construct_R}\ref{Cl_construct_R_e}, and assuming that $R^* \geq \underline R^*(\sigma)$, we have the following bound on this interval:
\begin{equation} \label{eq_partial_tau_R_2sigma}
 \partial_\tau R(\tau) = -e ^{R(\tau)/10} \leq -2 R^{1+1/\sigma}(\tau) \qquad \Rightarrow \qquad \partial_\tau  R^{-1/\sigma}(\tau) \geq  \tfrac{2}{\sigma}. 
\end{equation}
Assuming $R^* > 1$, we can ensure that  $R^{-1/\sigma}(\tau) < 1$, so the differential inequality in \eqref{eq_partial_tau_R_2sigma} cannot hold over time-intervals of length $\frac12 \geq \frac{\sigma}2$ and hence not on all of $[T-1, T-\frac12]$. 
So $\tau_0$ can be chosen as required.

Next choose $\tau_1 \in [T - \tfrac12,T]$ to be maximal with the property that \eqref{eq_E_geq_UUU} holds; if no such time exists, set $\tau_1 :=T-\frac12$.
As before, we obtain the differential inequality \eqref{eq_partial_tau_R_2sigma} on $(\tau_1, T]$ and we conclude that $T - \tau_1 < \frac{\sigma}2$, which rules out $\tau_1 = T-\frac12$ and implies \eqref{eq_E_geq_UUU} for $i=1$.
Lastly, we integrate \eqref{eq_partial_tau_R_2sigma} and obtain
\[ R^{-1/\sigma}(T) \geq   \tfrac{1}{\sigma} ( T - \tau_1) + R^{-1/\sigma}(\tau_1) \geq \tfrac{1}{\sigma} (T - \tau_1) , \]
which proves \eqref{eq_RTs2leq}.
\end{proof}
\medskip

In the next claim we derive bounds on $u$ at the times $\tau_0$ and $\tau_1$ over a balls whose radii are slightly smaller than $R(\tau_0)$ and $R(\tau_1)$.

\begin{Claim} \label{Cl_mm1_bounds}
Assuming $R^* \geq \underline{R}^*(m,J,\la,\sigma,\eps)$, we have for $i=0,1$
\[ |u_{\tau_i}|+ \ldots + |\nabla^{m-1} u_{\tau_i}| \leq \sigma\eta  \qquad \text{on} \qquad 
\IB_{\frac{(1-\eps)^2}{1+\eps} R(\tau_i)} \times N .
\]
\end{Claim}

\begin{proof}
Writing $u_{\tau_i} \omega_{R(\tau_i)} = U^+_{\tau_i} + U^-_{\tau_i} \in \sV_{> \la} \oplus \sV_{\leq \la}$ and setting $B := \IB_{\frac{(1-\eps)^2}{\sqrt{1+\eps}} R(\tau_i)} \times N$, we find, using \eqref{eq_E_geq_UUU},
\begin{multline*}
 \int_{B} |U^-_{\tau_i}|^2 dg
\leq   \exp\bigg( \frac{(1-\eps)^4}{1+\eps}  \, \frac{R^2(\tau_i)}{4} \bigg)\int_{B} |U^-_{\tau_i}|^2 e^{-f} dg \\
\leq (E_{\tau_i})^{-\frac{2(1-\eps)^2}{1+\eps}} (\td\UU^-_{\tau_i})^2
\leq (E_{\tau_i})^{\frac{2}{1+\eps} - \frac{2(1-\eps)^2}{1+\eps}} \eta^2
= E_{\tau_i}^{2\eps'(\eps)} \eta^2,
\end{multline*}
where $\eps' (\eps) > 0$.
On $B$ we also have the following bounds due to Lemma~\ref{Lem_polynomial_bounds} and \eqref{eq_E_geq_UUU}
\[ |U^+_{\tau_i}| + \ldots + |\nabla^{m} U^+_{\tau_i}| 
\leq C(\la,\sigma,m)  R^{C(\la,\sigma,m)}(\tau_i) \| U^+_{\tau_i} \|
\leq C(\la,\sigma,m) R^{C(\la,\sigma,m)}(\tau_i)  E_{\tau_i}^{\frac{1}{(1+\eps)J}} \eta . \]
So if we denote by $\xi \in (0,1)$ a constant, which we will determine later, then assuming a bound of the form $R^* \geq \underline R^*(\la, \sigma, \eps, m, J, \xi)$ we have
\begin{equation}\label{eq:U-_integral}
    \int_{B} |U^-_{\tau_i}|^2 dg
\leq \xi \eta^2, \qquad \sup_{B} \big( |U^+_{\tau_i}| + \ldots |\nabla^{m} U^+_{\tau_i}| \big)  \leq  \xi \eta.
\end{equation}
On the other hand, we have the following bound due to Property~\eqref{eq_BND_ueta} 
\[ \sup_B \big( |u_{\tau_i} | + \ldots + |\nabla^{m} u_{\tau_i}| \big) \leq \eta. \]
Combining this with the second bound in \eqref{eq:U-_integral} implies, since $u_{\tau_i} = U^+_{\tau_i} + U^-_{\tau_i}$ on $B$ for $R^* \geq \underline R^*(\eps)$,
\[ \sup_B  \big( |U^-_{\tau_i}| + \ldots + |\nabla^{m} U^-_{\tau_i}| \big)  \leq 2\eta. \]
Assuming $\xi \leq \ov\xi (\sigma, \eps)$, we can use the integral bound in \eqref{eq:U-_integral} to improve this bound at the cost of losing one derivative and slightly shrinking the domain:
\[ \sup_{\IB_{\frac{(1-\eps)^2}{1+\eps}R(\tau_i)} \times N} \big( |U^-_{\tau_i}| + \ldots + |\nabla^{m-1} U^-_{\tau_i}| \big) \leq \tfrac12 \sigma\eta.\]
This step follows from a basic limit argument or the Galiardo-Nirenberg inequality; see also Lemma~\ref{Lem_GN}.
Lastly, combining this bound again with the second bound in \eqref{eq:U-_integral}, and assuming $\xi \leq \frac{1}{2} \sigma$, implies the claim.
\end{proof}
\medskip

Since $\partial_\tau R \leq 1$ we have for $R^* \geq \underline{R}^*(\eps)$ and $i = 0,1$
\[ R(\tau_i) \geq R(T) - 1 \geq (1+\eps)^{-1} R(T). \]
So Claim~\ref{Cl_mm1_bounds} implies that
\begin{equation} \label{eq_u_bound_RTs2}
 |u_{\tau_i}|+ \ldots + |\nabla^{m-1} u_{\tau_i}| \leq \sigma\eta  \qquad \text{on} \qquad 
\IB_{\frac{(1-\eps)^2}{(1+\eps)^2} R(T)} \times N 
\end{equation}
We are now in a position to apply the $(\sigma, \eta, m)$-pseudolocality property, Definition~\ref{Def_pseudolocality} at times $\tau_0, \tau_1$ and $\tau_2 = T$.
The assumptions of Definition~\ref{Def_pseudolocality} hold for $R = \frac{(1-\eps)^2}{(1+\eps)^2} R(T) < R(T)$ due to \eqref{eq_RTs2leq} and \eqref{eq_u_bound_RTs2}.
As a result, we obtain control over the $C^m$-norm of $u_{T}$ over a ball of radius, see \eqref{eq_Ts2mtau1},
\[ e^{\sigma(T-\tau_0)} \Big(\tfrac{(1-\eps)^2}{(1+\eps)^2} R(T) - \tfrac1\sigma \Big)
\geq  e^{\sigma/2} \Big(\tfrac{(1-\eps)^2}{(1+\eps)^2} R(T) - \tfrac1\sigma \Big). \]
Assuming bounds of the form $\eps \leq \ov\eps(\sigma)$ and $R^* \geq \underline{R}^*(\sigma, \eps)$ this is larger than $R(T)$, in contradiction to \eqref{eq:contra_u_large}.
This concludes the proof.
\end{proof}

\bigskip

\section{Mode analysis of mean curvature flow}\label{sec_mode_analysis}
\subsection{Overview}
In this section we apply the PDE-ODI principle (Theorem~\ref{Thm_PDE_ODI_principle}) to nearly cylindrical rescaled modified mean curvature flows.
This reduction allows us to describe the evolution of the flow via a finite-dimensional vector $U^+_\tau \in \sV_{>\la}$, for some $\la < 0$, which encodes the modes associated to eigenvalues greater than $\la$.
By analyzing the resulting ODI system, we will extract geometric information and establish several of our main theorems.

A central theme is that the asymptotics of the ODI---and hence of the flow---are governed by a small number of dominant modes.
These modes decay either \emph{exponentially} (with rates $e^\tau$ or $e^{\tau/2}$) or \emph{polynomially} (with rate $|\tau|^{-1}$) as $\tau \to -\infty$.
All remaining modes are controlled by powers of these dominant ones and therefore decay strictly faster.
In particular, we will see that all non-rotationally symmetric modes must decay faster than any power of the dominant mode.

A key step in our argument is the construction of an appropriate gauge obtained by modifying of the flow by a carefully chosen family of Killing fields $(\bY_t)_{t \in I}$.
This gauge allows us to control the geometrically insignificant, neutral and unstable Jacobi modes corresponding to infinitesimal ambient Euclidean motions.
The gauging procedure is somewhat technical, but it is confined to this section and does not appear explicitly in the main results that will be needed in subsequent sections.
\medskip

We now summarize the main outcomes of this section.
Because several of the results require further definitions and motivation, they will be presented progressively.
The central result is Proposition~\ref{Prop_PO_ancient}, together with its addendum Proposition~\ref{Prop_Add_tau2}.
These propositions compile the key outcomes of this section in the case of ancient, asymptotically cylindrical mean curvature flows and are stated in their most general, yet elementary form without reference to gauging or the rescaled \emph{modified} flows.
They also justify the notions of flows with ``dominant linear or quadratic mode''; see Definition~\ref{Def_dominant_modes}.

In addition, in Subsections~\ref{subsec_eternal} and \ref{subsec_stability_neck} we show how our framework yields straight-forward proofs of classical results of Colding-Ilmanen-Minicozzi \cite{Colding_Ilmanen_Minicozzi} on the rigidity of  cylinders among shrinkers (Corollary~\ref{Cor_eternal}), and Colding-Minicozzi \cite{colding_minicozzi_uniqueness_blowups} on the stability of cylinders, and uniqueness of cylindrical tangent flows (Theorem~\ref{Thm_stability_necks} and Corollaries~\ref{Cor_unique_tangent}, \ref{Cor_unique_tangent_infinity}) without the use of a {\L}ojasiewicz-Simon inequality.
\medskip

This section is organized as follows.
In Subsection~\ref{subsec_linearization}, we discuss the setup used throughout this section and establish basic results on the linearization of the rescaled modified mean curvature flow equation and the associated spectral decomposition.
In Subsection~\ref{subsec_gauge}, we define our gauge and show that every almost cylindrical mean curvature flow can be gauged by a suitable family of ambient Euclidean motions, resulting in a gauged, rescaled \emph{modified} flow.
In Subsection~\ref{subsec_PDEODE_to_MCF}, we apply the PDE-ODI principle to the gauged rescaled modified flow and show that the Jacobi modes arising from ambient Euclidean motions must in fact be small.
Subsections~\ref{subsec_higher_modes} and \ref{subsec_evol_leading} contain the analysis of the resulting ODI system:
we first bound the higher and non-rotationally symmetric modes by powers of the leading modes
and then characterize the finer asymptotics of the leading neutral or unstable, rotationally symmetric modes.
In Subsections~\ref{subsec_eternal} and \ref{subsec_stability_neck}, we apply these results to reprove the rigidity and stability of the cylinder.
Finally, in Subsection~\ref{subsec_asymp_cyl} we apply theory developed so far to ancient, asymptotically cylindrical flows and summarize the resulting picture in a more elementary and useful language.

\subsection{Setup, linearization and mode decomposition} \label{subsec_linearization}
Let us review fundamental facts and introduce definitions that will be used throughout this section and the remainder of this paper.
Throughout this section we fix dimensions $1 \leq k < n$ and $n' \geq 0$ and we will omit dimensional dependencies henceforth.
Recall the round cylinder 
\[ M_{\cyl}  = M^{n,k}_{\cyl} = \IR^k \times \IS^{n-k} \subset \IR^{n+1} \cong \IR^{n+1} \times \bO^{n'} \subset \IR^{n+1+n'}, \] 
where $\IS^{n-k}$ is chosen to have radius $\sqrt{2(n-k)}$.
Let $f(\bx, \by) = \frac14 |\bx|^2$ be the standard quadratic potential on $M_{\cyl}$ and define the weighted $L^2_f$-space as in \eqref{eq_L2f_integration}.
We will always consider $n$-dimensional mean curvature flows  of codimension $1+n'$ in $\IR^{n+1+n'}$.
Denote by $\YY$ the space of Killing fields on the ambient space $\IR^{n+1+n'}$, which we equip with an arbitrary fixed norm $\Vert \cdot \Vert$.
\medskip

The discussion in this section will be based on the notion of rescaled (modified) mean curvature flows $\td\MM'$ associated with a family of Killing fields $(\td\bY_\tau \in \YY)$,  as introduced in Subsection~\ref{subsec_pseudoloc_rmmcf}, and the corresponding PDE \eqref{eq_MCF_u_eq}, which expresses this flow in terms of the graph of a function $u \in C^\infty(\DD; \IR \times \IR^{n'})$ over the cylinder $M_{\cyl}$.
Recall that this PDE takes the form
\begin{equation} \label{eq_main_ev_eq}
  \partial_\tau u_\tau = L u_\tau + Q[u_\tau, \td\bY_\tau]   
\end{equation}
We will fix the linear term $L$ and non-linear term $Q$ henceforth.
Recall from Lemma~\ref{Lem_structure_MCF_graph_equation} that the linearization takes the form
\begin{equation} \label{eq_L_recall}
 L(u',u'') = (\triangle_f u' + u', \triangle_f u'' + \tfrac12 u''), 
\end{equation}
For any $\la \in \IR$ and $J \geq 1$, the procedure outlined in the beginning of Subsection~\ref{subsec_PDE_ODI_statement} produces the $J$\emph{th} Taylor expansion of the non-linear term, which we will study in more detail throughout this section:
\[ Q_J = Q_J^+ + Q_J^- : \sV_{> \la} \times \YY \lto \sV_{>\la} \oplus \sV_{\leq \la}. \]
\medskip

Next, consider the action of $O(n-k+1) \times O(n')$ on $\IR^{n+1+n'} = \IR^{k} \times \IR^{n-k+1} \times \IR^{n'}$, which preserves the cylinder $M_{\cyl}$.
This action naturally induces an action on $C^\infty (M_{\cyl}; \IR \times \IR^{n'})$ via the formula 
\[ \big((A,B). (u',u'') \big)(\bx, \by) = \big( u'(\bx, A^{-1} \by), B u''(\bx, A^{-1} \by) \big), \qquad (A,B) \in O(n-k+1) \times O(n'). \]
Likewise, it induces an action on the weighted $L^2$-space $L^2_f(M_{\cyl}; \IR \times \IR^{n'})$.
We write
\[ L^2_f(M_{\cyl}; \IR \times \IR^{n'})= \sV_{
\rot} \oplus \sV_{\osc}, \]
where $\sV_{\rot}$ is the set of fixed points under this action and $\sV_{\osc}$ is its $L^2_f$-orthogonal complement.
It is not hard to see that $u = (u', u'') \in \sV_{\rot}$ if and only if $u'' \equiv 0$ and $u'(\bx,\by)= u'(\bx)$ is independent of $\by$.

For any $\la \in \IR$, we denote by $\sV_\la \subset L^2_f(M_{\cyl}; \IR \times \IR^{n'})$ the eigenspace of $L$ corresponding to the eigenvalue $\la$.
As in the last section, we will frequently consider the decomposition
\[ L^2_f(M_{\cyl}; \IR \times \IR^{n'}) = \sV_{> \la} \oplus \sV_{\leq \la}, \]
into the finite-dimensional sum of eigenspaces $\sV_{\la'}$ of $L$ corresponding to eigenvalues $\la' > \la$ and its orthogonal complement.
Since $L$ is equivariant under the $O(n-k+1) \times O(n')$-action, it preserves the spaces {$\sV_\la, \sV_{> \la}$ and $\sV_{\leq \la}$}.
So we obtain invariant orthogonal splittings of the form
\[ \sV_{ \la} =  \sV_{\rot,  \la} \oplus \sV_{\osc,  \la}, \qquad
\sV_{> \la} = \sV_{\rot, > \la} \oplus \sV_{\osc, > \la}, \qquad
\sV_{\leq \la} =  \sV_{\rot, \leq \la} \oplus \sV_{\osc, \leq \la}, \]
where $\sV_{\rot, \la} = \sV_{\rot} \cap \sV_\la$, $\sV_{\osc, \la} = \sV_{\osc} \cap \sV_\la$, and so on.

\medskip

Consider now the case $u_\tau \equiv 0$ in equation \eqref{eq_main_ev_eq}.
Lemma~\ref{Lem_structure_MCF_graph_equation} shows that the map 
\begin{equation} \label{eq_bYtoQbY}
\Jac : \bY \longmapsto Q[0, \bY] 
\end{equation}
is linear.
In the picture of the rescaled modified flow, this map corresponds to the map $\bY \mapsto \frac1{\sqrt{2(n-k)}} \bY^\perp |_{M_{\cyl}}$, which assigns to every ambient Killing field the corresponding Jacobi field along $M_{\cyl}$, expressed pair consisting of its radial part in the $\bO^k \times \IR^{n-k+1} \times \bO^{n'}$-direction and its part in the $\bO^{n+1}  \times \IR^{n'}$-direction.
More specifically, consider a smooth family of Euclidean motions $(S_\tau \in E(n+1+n'))_{\tau \approx 0}$ with $S_0 = \id$ and $\partial_\tau |_{\tau=0} S_\tau = \bY$ and write $S_\tau(M_{\cyl})$ locally as the graph $\Gamma_{\cyl}(u_\tau)$ over $M_{\cyl}$; so $u_0 \equiv 0$.
Then $\partial_\tau |_{\tau = 0} u_\tau = \Jac \bY$.

It is not hard to see that $\Jac$ is equivariant under the natural $O(n-k+1) \times O(n')$-actions on $\YY$ and $M_{\cyl}$.
The kernel of \eqref{eq_bYtoQbY} is the subspace $\YY_{\cyl} \subset \YY$ of Killing fields whose flow preserves $M_{\cyl}$, so linear combinations of infinitesimal translations and rotations in the $\IR^{k}$-factor, infinitesimal rotations in the $\IR^{n-k+1}$-factor and infinitesimal rotations in the $\IR^{n'}$-factor.
We can write $\YY = \YY_{\cyl} \oplus \YY_\perp$, where the latter space denotes the span of infinitesimal translations orthogonal to $\IR^{k} \times \bO^{n-k+1+n'}$ and rotations of the form $e_i \wedge e_j$ where $e_i$ and $e_j$ are contained in different factors of $\IR^{k} \times \IR^{n-k+1} \times \IR^{n'}$.
Denote the image of \eqref{eq_bYtoQbY} by $\sV_{\Jac} \subset C^\infty(M_{\cyl}; \IR \times \IR^{n'})$.
Then the restriction of \eqref{eq_bYtoQbY} to $\YY_\perp$
\[ \Jac |_{\YY_\perp} : \YY_\perp \lto \sV_{\Jac}\]
is an isomorphism of vector spaces and $\sV_{\Jac}$ is invariant under the $O(n-k+1) \times O(n')$-action.

\medskip

The next lemma characterizes the space $\sV_{\Jac}$ and the eigenspaces $\sV_\la$ for $\la \geq 0$.

\begin{Lemma} \label{Lem_mode_dec}
The first four eigenvalues of $L$, viewed as an operator on $L^2_f(M_{\cyl}; \IR \times \IR^{n'})$, are $1, \frac12, 0, - \frac1{n-k}$.
We have
\[ \sV_{\osc, 1} = 0, \qquad \sV_{\Jac} = \sV_{\osc,\frac12} \oplus V_{\osc, 0} \]
and $\sV_{\rot} = \sV_{\rot,1} \oplus \sV_{\rot,\frac12} \oplus \ldots$, where $\sV_{\rot, -\frac{i}2}$, for $-\frac{i}2 \leq 1$, is spanned by elements of the form $(u',0)$, where $u'(\bx,\by) = u'(\bx)$ is a polynomial of degree $i+2$ on $\IR^k$, which is a product of Hermite polynomials in each coordinate.
Specifically, $\sV_{\rot, 1}$, $\sV_{\rot, \frac12}$ and $\sV_{\rot, 0}$ are spanned by elements of the form $(\mathfrak p^{(0)}, 0)$, $(\mathfrak p^{(1)}_i, 0)$ and $(\mathfrak p^{(2)}_{ij}, 0)$, where
\begin{equation} \label{eq_Hermite}
 \mathfrak p^{(0)}(\bx) = 1, \qquad \mathfrak p^{(1)}_i(\bx) = \frac1{\sqrt{2}} \bx_i, \qquad \mathfrak p^{(2)}_{ij} (\bx) =  \frac1{2\sqrt{2}} (\bx_i \bx_j -2\delta_{ij}) . 
\end{equation}
The space $\sV_{\osc, \frac12}$ is spanned by elements of the form $(u', 0)$, where $u'(\bx,\by) = u'(\by)$ is a first spherical harmonic on $\IS^{n-k}$ and elements $(0, u'')$, where $u''$ is constant.
The space $\sV_{\osc, 0}$ is spanned by elements of the form $(u',0)$, where $u'(\bx,\by) = a(\bx)b(\by)$ is product of a linear function and a first spherical harmonic, and elements $(0, u'')$, where $u''(\bx,\by) = u''(\bx)$ is a linear function.
\end{Lemma}

\begin{proof}
Due to the diagonal form of $L$ (see \eqref{eq_L_recall}), it is enough to study eigenfunctions $v \in L^2_f(\IR^{k} \times \IS^{n-k})$ with $\triangle_f v = \la v$.
By separation of variables, any such eigenfunction is a linear combination of products $v_i (\bx,\by) = a_i(\bx) b_i(\by)$, where $a_i \in L^2_f(\IR^{k})$ and $b_i \in L^2(\IS^{n-k})$ with
\[ \triangle_f a_i = \alpha_i a_i, \qquad \triangle b_i = \beta_i b_i, \qquad \alpha_i + \beta_i = \la. \]
The eigenvalues of $\triangle_f$ on $\IR^{k}$ are $0, -\frac12, -1, \ldots$ and the eigenfunctions are given by products of Hermite polynomials in each coordinate.
As $\IS^{n-k}$ is the sphere of radius $\sqrt{2(n-k)}$, the first three eigenvalues of the Laplacian on $\IS^{n-k}$ are $0, -\frac12, - \frac{n-k+1}{n-k} = -1 -\frac1{n-k}$.
So the first four eigenvalues for $\triangle_f$ are $\la = 0, -\frac12, -1, -1 -\frac1{n-k}$; if $\la \in \{ 0, -\frac12, -1 \}$, then $\alpha_i \in \{ 0, -\frac12, -1 \}$ and $\beta_i \in \{ 0, -\frac12 \}$.

Due to \eqref{eq_L_recall}, the first four eigenvalues of $L$ must therefore be $1, \frac12, 0, -\frac1{n-k}$.
The eigenvalue $1$ can only be obtained for eigenfunctions of the form $u = (u', 0)$, where $\triangle_f u' = 0$, so $u'$ must be constant.
This proves the characterization of $\sV_{\rot, 1}$ and $\sV_{\osc, 1}$.
The spaces $\sV_{\rot,\frac12}$ and $\sV_{\rot,0}$ consists of functions of the form $u = (v, 0)$, where $v(\bx, \by) = a(\bx)$ for $\triangle_f a = -\frac12 v$ and $\triangle_f a = - v$, respectively.
This implies that these spaces are spanned by the corresponding Hermite polynomials.

The space $\YY$ is spanned by infinitesimal translations and rotations.
If $\td\bY \in \YY$ is an infinitesimal translation, then one can see easily that $\Jac \bY = (u' , u'')$, where $u' (\bx, \by) = a'(\bx) b'(\by)$ for constant $a'$ and $\triangle b' = -\frac12 b'$ and $u'' \equiv const$.
So $\Jac \bY \in \sV_{\osc,\frac12}$.
If $\td\bY \in \YY$ is an infinitesimal rotation, then $\Jac \bY = (u', u'')$, where $u' (\bx, \by) = a'(\bx) b'(\by)$ for linear $a'$ and $\triangle b' = -\frac12 b'$, as well as $u'' (\bx, \by) = a''(\bx) b''(\by)$ for linear $a''$ and constant $b''$ or constant $a''$ and $\triangle b'' = - \frac12 b''$.
So $\Jac \bY \in \sV_{\osc, 0}$.
In summary, $\sV_{\Jac} = \Jac(\YY) \subset \sV_{\osc, \frac12} \oplus \sV_{\osc, 0}$.
In order to see that we have equality, note that a dimension count gives
\[ \dim \sV_{\Jac} = \dim \YY_{\perp} = (n-k+1+n')+k(n-k+1) + kn' + (n-k+1)n' = \dim (\sV_{\osc, \frac12} \oplus \sV_{\osc, 0}). \]
This finishes the proof of the lemma.
\end{proof}
\medskip

Lemma~\ref{Lem_mode_dec} implies that the neutral and unstable directions of the linearized rescaled mean curvature flow equation $\partial_\tau u = Lu$ consist only of rotationally symmetric modes, which are polynomials on $\IR^k$ of degree $\leq 2$, together with the modes in the image $\sV_{\Jac}$ of the Jacobi-operator.
With a suitable choice of gauge, which will be chosen in the next subsection, the modes in $\sV_{\Jac}$ can be controlled, so that ultimately only the rotationally symmetric modes govern the dynamics.
\medskip

\subsection{Gauging the flow} \label{subsec_gauge}
We will now describe a class of rescaled modified mean curvature flows for which the Killing fields $(\td\bY_\tau)$ are chosen in such a way that the mode of $u$ within $V_{\Jac}$ almost vanishes.
Recall the notion of a rescaled modified mean curvature flow from Subsection~\ref{subsec_resc_mod_mcf}.
We will use the following notion:

\begin{Definition}[$R^\#$-gauge] \label{Def_gauged}
Let $M \subset \IR^{n+1+n'}$ be an $n$-dimensional, not necessarily properly, embedded submanifold.
We say that $M$ is \textbf{$R^{\#}$-gauged,} for some $R^\# > 1$, if there is an $R' > 0$ such that $M \cap {\IB^{n+1+n'}_{R'}}$ is properly embedded in ${\IB^{n+1+n'}_{R'}}$ and can be expressed as the graph $\Gamma_{\cyl}(u)$ of some smooth, vector-valued function $u \in C^\infty(\DD; \IR \times \IR^{n'})$ with the property that $\IB^{k}_{R^\#} \times \IS^{n-k} \subset \DD \subset M_{\cyl}$ and
\[ |u|+ |\nabla u| < (R^\#)^{-1} \qquad \text{and} \qquad \PP_{\sV_{\Jac}} (u \omega_{R^{\#}})  = 0. \]
We say that a rescaled modified mean curvature flow $\td\MM'$ corresponding to a family of Killing fields $(\td\bY_\tau)_{\tau \in \td I}$ is \textbf{$R^{\#}$-gauged} if $\td\MM^{\prime,\reg}_\tau$ is $R^{\#}$-gauged for all $\tau \in \td I$ and if in addition $\td\bY_\tau \in \YY_\perp$ for all $\tau \in \td I$.
\end{Definition}

The following proposition shows that we can always modify a rescaled mean curvature flow by a slowly varying family of Euclidean motions so that the resulting flow is $R^\#$-gauged, as long as the flow is sufficiently close to a cylinder.
The proposition summarizes all the necessary tools that are needed throughout this paper.

\begin{Proposition} \label{Prop_gauging}
For any $\eps\leq \ov\eps$ and  $R^\# \geq \underline{R}^\#$ there is a constant $\delta(\eps, R^\#) > 0$ such that the following holds.

Consider a unit-regular, integral, $n$-dimensional Brakke flow $\MM$ in $\IR^{n+1+n'} \times I$ and a point $(\bq_0,t_0) \in \IR^{n+1+n'} \times \IR$ with $I \subset (-\infty, t_0)$.
Suppose that some fixed $\bq_0 \in \IR^{n+1+n'}$ is a center of a $(n,k,\delta)$-neck at scale $\sqrt{t_0-t}$ of $\MM$ at all times $t \in I$.
Then there is a smooth family of Euclidean motions $(S_t \in E(n+1+n'))_{t \in I}$ such that the following is true for the rescaled modified flow $(\td\MM'_\tau := e^{\tau/2}  S_{t_0-e^{-\tau}}( \MM_{t_0-e^{-\tau}}))_{\tau \in \td I}$, where $\td I := \{ -\log(t_0-t) \;\; : \;\; t \in I\}$:
\begin{enumerate}[label=(\alph*)]
\item \label{Prop_gauging_a} $\td\MM^{\prime}$ is $R^\#$-gauged.
\item \label{Prop_gauging_b} $(t_0-t) \partial_t S_t = \bY_t \circ S_t$ for some smooth  family $(\bY_t \in \YY)_{t \in  I}$ and $\td\MM^{\prime,\reg}$ evolves by the following rescaled modified mean curvature flow equation for $\td \bY_\tau =  (e^{\tau/2})_* \bY_{t_0-e^{-\tau}}$:
\[ \partial_\tau \mathbf x=\mathbf H+\frac{\mathbf x^\perp}{2} + \td\bY^\perp. \]
\item \label{Prop_gauging_c} We have the bounds $\| \td\bY_\tau \| \leq \eps$ and $|S_t (\bq_0) | < \eps \sqrt{t_0-t}$.
\item \label{Prop_gauging_d} If $T_1 := \sup I \in I$ and $S' \in E(n+1+n')$ is a Euclidean motion with $|S' (\bq_0) | < \eps \sqrt{t_0-T_1}$ such that $(t_0-T_1)^{-1/2} S' ( \MM^{\reg}_{T_1})$ is $R^\#$-gauged, then we may choose $S_{T_1} = S'$, which determines $(S_t)_{t \in I}$ uniquely.
\item \label{Prop_gauging_f} For any $\eps' > 0$ there is a constant $\delta'(\eps', R^\#) > 0$ with the following property. Suppose that $\bq_0$ is a center of a $\delta'$-neck at scale $\sqrt{t_0 - t}$ of $\MM$ at some time $t \in I$.
Then $\td\MM^{\prime}$ is $\eps'$-close to $M^{n,k}_{\cyl}$ at time $- \log(t_0 - t)$.
\end{enumerate}
\end{Proposition}
\bigskip

The proof of Proposition~\ref{Prop_gauging} relies on the following lemma, which shows that the gauging condition can be achieved via an application of the implicit function theorem.

\begin{Lemma} \label{Lem_gauge_prep}
Suppose that $R^\# \geq \underline R^\#$, $\eps >0$ and $\delta \leq \ov\delta(\eps, R^\#)$.
There is an open neighborhood $W \subset E(n+1+n')$ of $\id$, whose choice only depends on $R^\#, n,n'$, such that the following holds.

Let $M \subset \IR^{n+1+n'}$ be a, not necessarily properly, embedded $n$-dimensional submanifold.

\begin{enumerate}[label=(\alph*)]
\item \label{Lem_gauge_prep_a} If $M$ is $R^\#$-gauged and some point $\bq_0 \in \IR^{n+1+n'}$ with $| \bq_0| \leq \eps^{-1}$ is a center of an $(n,k,\delta)$-neck at scale $1$ of $M$, then $M$ is $\eps$-close to $M^{n,k}_{\cyl}$.
\end{enumerate}
Next, suppose that $M$ is $\delta$-close to $M^{n,k}_{\cyl}$.
Then there is an $S \in W$ such that $S(M)$ is $R^\#$-gauged.
Moreover, there is an $R' > 0$ such that for any $S \in W$ we can express $ S(M) \cap \IB^{n+1+n'}_{R'} = \Gamma_{\cyl} (u_{S})$ as the graph of a function $u_{S} \in C^\infty(\DD_{S}; \IR \times \IR^{n'})$ with $\IB^{k}_{R^\#} \times \IS^{n-k} \subset \DD_{S} \subset M^{n,k}_{\cyl}$ and $|u_S| + |\nabla u_S | < (R^\#)^{-1}$.
Consider the following smooth map:
\[ F^M : W \lto \sV_{\Jac}, \qquad S \lmapsto \PP_{\sV_{\Jac}}(u_{S} \omega_{R^\#}). \] 
Observe that the condition that $S (M)$ is $R^\#$-gauged  is equivalent to $F(S) = 0$.
The map $F^M$ has the following properties:
\begin{enumerate}[label=(\alph*), start=2]
\item \label{Lem_gauge_prep_c}  The differential $(dF^M)_{\id}$ is surjective.
\item \label{Lem_gauge_prep_d} If we identify the tangent space of $E(n+1+n')$ at $\id$ with $\YY$ in the natural way, then $(dF^M)_{\id} |_{\YY_\perp}$ is invertible with inverse bounded by some  universal constant.
\item \label{Lem_gauge_prep_e} $\Vert F^M \Vert_{C^{m}(W)} \leq C(R^\#, m)$ for $m \leq \frac12 \delta^{-1}$.
\end{enumerate}
\end{Lemma}

\begin{proof}
Assertion~\ref{Lem_gauge_prep_c} is a direct consequence of Assertion~\ref{Lem_gauge_prep_d}.
The fact that $S(M) \cap \IB^{n+1+n'}_{R'}$ is a graph over a subset of $M_{\cyl}$ as claimed is clear and the bound in Assertion~\ref{Lem_gauge_prep_e} is also clear.
So it remains to prove Assertions~\ref{Lem_gauge_prep_a} and \ref{Lem_gauge_prep_d} for appropriate choices of $\ov\delta(\eps,R^\#)$ and $W$.

Consider first the case in which $M = M_{\cyl}$.
Then Assertion~\ref{Lem_gauge_prep_a} holds trivially.
To see Assertion~\ref{Lem_gauge_prep_d}, consider a Killing field $\bY \in \YY_{\perp}$ and note that for $R^\# \geq \underline R^\#$
\[ \big\Vert (dF^M)_{\id} (\bY) - \Jac (\bY) \big\Vert =  \big\Vert \PP_{\sV_{\Jac}} \big( \Jac(\bY) \omega_{R^\#} \big) - \Jac (\bY) \big\Vert
\leq \Vert \Jac(\bY) (\omega_{R^\#} - 1) \Vert \leq \tfrac12 \Vert \Jac (\bY) \Vert. \]

\begin{Claim}
If $R^\# \geq \underline R^\#$, then whenever $M = S (M_{\cyl})$ is $R^\#$-gauged for some $S \in E(n+1+n')$, then we must have $S(M_{\cyl}) = M_\cyl$.
\end{Claim}
\begin{proof}
Suppose not and choose a sequence of counterexamples $S_i \in E(n+1+n')$ for some $R^\#_i \to \infty$.
After precomposing each $S_i$ with a suitable element of the subgroup of $E(n+1+n')$ that fixes $M_{\cyl}$, we may assume that $S_i \to \id$ and moreover that $S_i = \exp (\bY_i)$ for some non-zero sequence $\bY_i \in \YY_\perp$, which we view as elements of the Lie algebra of $E(n+1+n')$.
Then over $\IB^k_{R^\#_i} \times \IS^k$ we can express $S_i (M_{\cyl})$ as $\Gamma_{\cyl}(u_i)$ with $| u_i - \Jac \bY_i | \leq C (r \| \bY_i \|)^2$ for some uniform $C$.
So
\[ 0 = \PP_{\sV_{\Jac}} (u_i \omega_{R^\#_i}) = \bY_i + \PP_{\sV_{\Jac}} ((\Jac \bY_i)(\omega_{R_i^\#} -1)) + \PP_{\sV_{\Jac}} ((u_i - \Jac \bY_i )\omega_{R_i^\#}) \]
Since
\begin{align*} 
\| \PP_{\sV_{\Jac}} ((\Jac \bY_i)(\omega_{R_i^\#} -1)) \|_{L^2_f} 
&\leq \| (\Jac \bY_i)(\omega_{R_i^\#} -1) \|_{L^2_f} \leq \| \bY_i \| \cdot \| \omega_{R_i^\#} -1 \|_{L^2_f}, \\
\| \PP_{\sV_{\Jac}} ((u_i - \Jac \bY_i )\omega_{R_i^\#})  \|_{L^2_f} 
&\leq \| (u_i - \Jac \bY_i )\omega_{R_i^\#}  \|_{L^2_f}
\leq C \| \bY_i \|^2,
\end{align*}
and since  $ \| \omega_{R_i^\#} -1 \|_{L^2_f} \to 0$, we obtain a contradiction for large $i$.
\end{proof}

The claim shows that Assertion~\ref{Lem_gauge_prep_a} and \ref{Lem_gauge_prep_c} hold in the case $M = S(M_{\cyl})$ for an appropriate choice of $\underline R^\#$.

To see the general case, we argue by contradiction.
Fix $R^\# \geq \underline R^\#$ and $\eps > 0$ and consider a sequence of counterexamples $M_i$ violating Assertion~\ref{Lem_gauge_prep_a} or \ref{Lem_gauge_prep_d} for $\delta_i \to 0$.
In Assertion~\ref{Lem_gauge_prep_a} we have subsequential smooth convergence $M_i \to S(M_{\cyl})$ for some $S \in E(n+1+n')$ and in Assertion~\ref{Lem_gauge_prep_d} we even have $M = M_{\cyl}$.
So the conclusions hold by the previous two paragraphs for large enough $i$.
\end{proof}
\bigskip

We can now prove Proposition~\ref{Prop_gauging}.

\begin{proof}[Proof of Proposition~\ref{Prop_gauging}.]
After application of a translation and time-shift, we may assume that $(\bq_0, t_0) = (\bO, 0)$.
We first choose $S_{T_1}$ at the final time, if it exists.

\begin{Claim} \label{Cl_final_T0}
If $\delta \leq \ov\delta(\eps, R^\#)$ and $T_1:=\sup I\in I$, then there is a Euclidean motion $S_{T_1} \in E(n+1+n')$ such that $(-T_1)^{-1/2} \MM^{\reg}_{T_1}$ is $R^\#$-gauged and $|S_{T_1}(\bO)| < \eps \sqrt{-T_1}$.
If $S'$ is given in Assertion~\ref{Prop_gauging_d}, then we can choose $S_{T_1} = S'$.
\end{Claim}

\begin{proof}
After application of a rotation, we may assume without loss of generality that $ \MM$ is $\delta$-close to $M_{\cyl}$ at scale $\sqrt{-T_1}$ at time $T_1$.
Then the function $F^{(T_1)^{-1/2} \MM^{\reg}_{T_1}}$ from Lemma~\ref{Lem_gauge_prep} is defined and satisfies
$\Vert F^{(T_1)^{-1/2} \MM^{\reg}_{T_1}} (\id) \Vert \leq \Psi(\delta | R^\#)$.
So the claim follows from the implicit function theorem and Lemma~\ref{Lem_gauge_prep}.
\end{proof}

Next, we show that the gauge can be chosen locally.

\begin{Claim} \label{Cl_S_small_I}
If $R^\# \geq \underline R^\#$ and $\delta \leq \ov\delta(\eps, R^\#)$, then the following holds.
Fix some $t^* = - e^{-\tau^*} \in I$ and a Euclidean motion $S_{t^*} \in E(n+1+n')$ such that $(-t^*)^{-1/2} S_{t^*} \MM^{\reg}_{t^*}$ is $R^\#$-gauged and $|S_{t^*}(\bO)| < \eps \sqrt{-t^*}$.

Then there is a small neighborhood $t^* \in I^* \subset I$ on which $S_{t^*}$ can be extended uniquely to a smooth family $(S_{t^*})_{t^* \in I^*}$ such that $(-t)^{-1/2} S_{t}( \MM^{\reg}_{t})$ is $R^\#$-gauged and such that if $\partial_t S_t = \bY_t \circ S_t$ and $\td\bY_\tau =  (e^{\tau/2})_* \bY_{-e^{-\tau}}$, then $\td\bY_\tau \in \YY_\perp$.
Moreover, we have a bound of the form $\| \td\bY_{\tau^*} \| \leq \Psi(\delta |R^\#)$.
\end{Claim}

\begin{proof}
Consider the smooth family of submanifolds $\td M_\tau := e^{\tau/2} S_{t^*}( \MM^{\reg}_{-e^{-\tau}})$ for $\tau \approx \tau^*$ evolving by rescaled mean curvature flow.
Lemma~\ref{Lem_gauge_prep}\ref{Lem_gauge_prep_a} implies that $\td M_{\tau^*}$ is $\Psi(\delta |  R^\#)$-close to $M_{\cyl}$, so the same is true for $\tau \approx \tau^*$.
So for $\delta \leq \ov\delta(R^\#)$ and for a sufficiently small neighborhood $\tau^* \in \td I^* \subset \td I$ the map
\[ F : W \times \td I^* \lto \sV_{\Jac}, \qquad (S'', \tau) \mapsto F^{\td M_{\tau}} (S'') \]
is defined and smooth with $F(\id, \tau^*) = 0$.
Lemma~\ref{Lem_gauge_prep}\ref{Lem_gauge_prep_d} combined with the implicit function theorem implies that $F^{-1}(0)$ is a smooth submanifold whose tangent space at $(\id, \tau^*)$ is complementary to $\YY_\perp \times 0$, after possibly shrinking $W$ and $\td I^*$.

In order to construct $S_t$ for $t \approx t^*$, let us make the Ansatz $S_{-e^{-\tau}} = S'_\tau \circ S_{t^*}$ for some $S'_\tau(\mathbf x) = A'_\tau \mathbf x + \mathbf b'_\tau$ with $A'_\tau \in O(n+1+n')$ and $\mathbf b'_\tau \in \IR^{n+1+n'}$, which we need to determine and which needs to satisfy $S'_{\tau^*} = \id$.
If we set $S''_\tau (\mathbf x) := A'_\tau \mathbf x + e^{\tau/2} \mathbf b'_\tau$, then 
\[ e^{\tau/2} S_{-e^{-\tau}} (\MM^{\reg}_{-e^{-\tau}}) = S''_\tau (\td M_\tau), \]
so $\td\bY_\tau =(\partial_\tau S''_\tau)\circ (S''_\tau)^{-1}$ and the gauging condition at each $\tau \approx \tau^*$ is equivalent to $F(S''_\tau, \tau) = 0$.

We also need to satisfy the condition $\td\bY_\tau \in \YY_\perp$.
At $\tau = \tau^*$ we have $\mathbf b'_{\tau^*} = 0$, so $S''_{\tau^*}=\id$ and this condition at time $\tau = \tau^*$ is equivalent to 
\[ \td\bY_{\tau^*} = \partial_\tau |_{\tau = \tau^*} S''_\tau \in \YY_\perp. \]
For $\tau \approx \tau^*$ the condition $\td\bY_\tau \in \YY_\perp$ can be expressed in the form  $\partial_\tau S''_\tau \in \YY_{\perp, S''_\tau,\tau}$ for some smooth and time-dependent subbundle $S'' \mapsto \YY_{\perp, S'',\tau}$ of the tangent bundle of $W \subset E(n+1+n')$, where $\YY_{\perp, \id, \tau^*} = \YY_\perp$.
Since $\YY_\perp \times 0$ is complementary to the tangent space of $F^{-1}(0)$ at $(\id, \tau^*)$, we obtain
by solving the equation
\[0=\tfrac{d}{d\tau} F(S''_\tau,\tau)=(dF)_{(S''_\tau,\tau)}(\partial_\tau S''_\tau,0)+\partial_\tau  F(\id,\tau),\]
the existence of a unique family $( S''_\tau)_{\tau \in \td I^*}$ with $(S''_\tau, \tau) \in F^{-1}(0)$ such that $\partial_\tau S''_\tau \in \YY_{\perp, S''_\tau,\tau}$, possibly after shrinking $\td I^*$.
Since $\td M_{\tau}$, $\tau\in\td I^*$, is $\Psi(\delta|R^\#)$-close to $M_{\cyl}$ by Lemma \ref{Lem_gauge_prep}\ref{Lem_gauge_prep_a}, and $\tau \mapsto \td M_{\tau}$ is a rescaled mean curvature flow, we obtain moreover that $\|\partial_\tau |_{\tau = \tau^*} F(\id,\tau) \| \leq \Psi(\delta|R^\#)$, 
which implies that $\| \td\bY_{\tau^*} \| \leq \Psi(\delta |R^\#)$.
\end{proof}

Let us now focus on the setting of the proposition and suppose first that $T_1 = -e^{-\tau_1} := \sup I \in I$.
By Claims~\ref{Cl_final_T0}, \ref{Cl_S_small_I}, we may find a maximal solution $(S_t)_{t \in (T', T_1]}$, $T' \geq -\infty$, satisfying Assertions~\ref{Prop_gauging_a}--\ref{Prop_gauging_d}.
Let $(\bY_t)$ and $(\td\bY_\tau)$ be the corresponding families of Killing fields.

\begin{Claim}
If $\delta \leq \ov\delta(\eps, R^\#)$, then $T' = - \infty$ and $(S_t)$ extends to a smooth family on $I \cap (-\infty,T_1]$.
\end{Claim}

\begin{proof}
Suppose by contradiction that $-e^{-\tau'} := T' > -\infty$.
Due to the uniform bound on $\td\bY_\tau$, we obtain that $S_t \to S_{T'}$ for $t \searrow T'$.
If $|S_{T'}(\bO)| < \eps \sqrt{-T'}$, then we can apply Claim~\ref{Cl_S_small_I} for $t^* = T'$ and extend $(S_t)$ smoothly and uniquely to a slightly larger time-interval.
So by minimality of $T'$, we must have $|S_{T'}(\bO)| =\eps \sqrt{-T'}$.
To obtain a contradiction to this note that the bound on $\td\bY_\tau$ from Claim~\ref{Cl_S_small_I} implies that for some dimensional constant $C$
\begin{multline*}
 \bigg|\frac{d}{d\tau} S_{-e^{-\tau}} (\bO) \bigg| 
= e^{-\tau} |\partial_t S_{-e^{-\tau}} (\bO)| 
=  |\bY_{-e^{-\tau}} (S_{-e^{-\tau}} (\bO))| 
\leq e^{-\tau/2} \big|\td\bY_{\tau} (e^{\tau/2} S_{-e^{-\tau}} (\bO)) \big| \\
\leq C e^{-\tau/2} \Vert \td\bY_{\tau} \Vert 
\leq \Psi(\delta | R^\#) e^{-\tau/2} .
\end{multline*}
Note that in the second last inequality, we have used that $|S_{-e^{-\tau}}(\bO)| \leq \eps e^{-\tau/2}$ for $\tau > \tau'$.
Therefore, if $\delta \leq \ov\delta(\eps, R^\#)$, then for $\tau \approx \tau'$ we have
\[ \frac{d}{d\tau} \Big( e^{\tau/2} | S_{-e^{-\tau}} (\bO)|   \Big) 
\geq \frac12 e^{\tau/2} | S_{-e^{-\tau}} (\bO)| - \frac\eps4   > 0.  \]
Since $e^{\tau/2} | S_{-e^{-\tau}} (\bO)| \leq \eps$ for $\tau \geq \tau'$, this implies that $e^{\tau'/2} | S_{-e^{-\tau'}} (\bO)| < \eps$, giving us the desired contradiction.
\end{proof}

This concludes the proof in the case $T_1 = -e^{-\tau_1} := \sup I \in I$; Assertion~\ref{Prop_gauging_f} is a direct consequence of Lemma~\ref{Lem_gauge_prep}\ref{Lem_gauge_prep_a}.

Let us now consider the case in which $T_1 := \sup I \not\in I$.
In this case we can apply our conclusions to the flow restricted to time-intervals $I_i = I \cap (-\infty, T_{1,i}]$ for $T_{1,i} \nearrow T_1$.
This leads to families of solutions $(S_{i,t})_{t \in I_i}$, which satisfy the uniform bounds from Assertion~\ref{Prop_gauging_c}.
We can now construct $(S_t)_{t \in I}$ by passing to a subsequential limit.
\end{proof}
\bigskip

\subsection{The PDE-ODI principle for rescaled modified mean curvature flows} \label{subsec_PDEODE_to_MCF}
In this subsection we apply the PDE-ODI principle (Theorem~\ref{Thm_PDE_ODI_principle}) to the rescaled modified mean curvature flow equation \eqref{eq_main_ev_eq}.
This will allow us to describe the flow via the evolution of a finite-dimensional vector $U^+_\tau \in \sV_{> \la}$.
In the case in which the flow is $R^\#$-gauged---which will be our primary concern---we will show that the Jacobi projections $\PP_{\sV_{\Jac}} U^+_\tau$ and the Killing fields $\td\bY_\tau$ are small.

The following is a direct consequence of Theorem~\ref{Thm_PDE_ODI_principle} combined with Lemmas~\ref{Lem_structure_MCF_graph_equation} and \ref{Lem_mcf_weak_pseudo}.

\begin{Proposition}[PDE-ODI principle for rescaled modified mean curvature flow] \label{Prop_PDE_ODI_MCF}
There is a constant $c > 0$ such that the following holds if
\[ 
m \geq 4,\quad J \geq 1, \quad 
\la \in \IR,  \quad 
\eta \leq \ov\eta(\la), \quad 
\eps \leq \ov\eps(\eta,m),  \quad 
R^* \geq \underline R^*(m,J,\la, \eta, \eps). \]

Let $\td\MM^{\prime}$ be a unit-regular, integral rescaled modified mean curvature flow in $\IR^{n+1+n'} \times \td I$ for a smooth family of Killing fields $(\td\bY_\tau \in \YY)_{\tau \in \td I}$ (compare with Subsection~\ref{subsec_pseudoloc_rmmcf}).
Consider the domain $\DD \subset M_{\cyl} \times \td I$ and the function $u \in C^\infty(\DD;  \IR^{1+n'})$ whose time slices $\DD_\tau$ and $u_\tau \in C^\infty(\DD_\tau; \IR^{1+n'})$, for each $\tau \in \td I$, are chosen such that there is a maximal $R'_\tau \in [0,\infty]$ with the property that $\IB^{n+1+n'}_{R'_\tau} $ is disjoint from $\td\MM^{\prime,\sing}_\tau$ and such that
\[ \td\MM^{\prime,\reg}_\tau \cap \IB^{n+1+n'}_{R'_\tau} = \Gamma_{\cyl} (u_\tau). \]
So $u$ satisfies the equation \eqref{eq_main_ev_eq}. Suppose that:
\begin{enumerate}[label=(\roman*)]
\item \label{Prop_PDE_ODI_MCF_i} $\Vert \td\bY_\tau \Vert \leq c$ for all $\tau \in \td I$.
\item \label{Prop_PDE_ODI_MCF_ii}
For all $\tau \in \td I$ we have 
\[ \IB_{ R^*}^k \times \IS^{n-k}  \subset \DD_{\tau} \qquad \text{and} \qquad \| u_{\tau}  \|_{L^2_f(\IB^k_{ R^*} \times \IS^{n-k})} \leq \eta e^{-(R^*)^2/8}. \]
\item \label{Prop_PDE_ODI_MCF_iii} If $T_0 = \min \td I > -\infty$ exists, then for some $R_0 \in [ R^*, \infty]$ the following holds for all $\tau \in [T_0 , T_0 +1]$:
\begin{equation} \label{eq_atT0bound2}
 \IB^k_{R_0+1} \times \IS^{n-k}  \subset \DD_{\tau} \qquad \text{and} \qquad  \| u_{\tau}  \|_{C^{m}(\IB^k_{R_0+1} \times \IS^{n-k})} < \eta. 
\end{equation}
If $\inf \td I = -\infty$, then \eqref{eq_atT0bound} holds for $R_0 = R^*$ and  sufficiently small $\tau$.
\end{enumerate}
Then there is a continuous function $R : \td I  \to [R^*, \infty]$ that is smooth when finite and satisfies $R(T_0) = R_0$ if $T_0$ exists, with the following properties.
For all $\tau \in \td I$ we have $\IB^{k}_{ R(\tau)} \times \IS^{n-k}  \subset \DD_\tau$, so $u_\tau \omega_{R(\tau)}$ can be extended to a smooth function on $M_{\cyl}$ by setting it zero outside its domain.
So we can set
\begin{equation} \label{eq_UpUUm_def_apply}
 U^+_\tau :=  \PP_{\sV_{> \lambda}} (u_{\tau} \omega_{R(\tau)}) , \qquad
\mathcal U^-_\tau := \big\|   \PP_{\sV_{\leq \lambda}}  (u_{\tau} \omega_{R(\tau)} ) \big\|_{L^2_{f}} + \eta e^{-\frac{((1-\eps)R(\tau))^2}{8}}. 
\end{equation}
If $R(\tau) = \infty$, then we set $\omega_{R(\tau)} \equiv 1$ and $e^{-\frac{((1-\eps)R(\tau))^2}{8}} = 0$.
The following is true:
\begin{enumerate}[label=(\alph*)]
\item \label{Prop_PDE_ODI_MCF_a} We have the bound
$  \| u_{\tau} \|_{C^{m}(\IB^k_{R(\tau)} \times \IS^{n-k} )} \leq \eta$. 
\item \label{Prop_PDE_ODI_MCF_b}  The quantities $U^+_\tau$ and $\mathcal{U}^-_\tau$ satisfy the evolution inequalities
\[ \Big\| \partial_\tau U^+_\tau - L U^+_\tau  - Q_J^+ (U^+_\tau ,\td\bY_\tau) \Big\|_{L^2_{f}}
\leq  C(\la, J) \Vert U^+_\tau \Vert^{J+1}_{L^2_{f}} + \big( C(\la) \sqrt\eta + C \Vert \td\bY_\tau \Vert  \big)  \mathcal{U}^-_\tau  \]
\[  \partial_\tau \mathcal U^-_\tau 
\leq  \big(\la + C(\la) \sqrt\eta + C(\la) \Vert \td\bY_\tau \Vert  \big) \mathcal U^-_\tau  
+  \Vert Q^-_J(U^+_\tau,  \td\bY_\tau) \Vert_{L^2_{f}}
+ C(\la,J) \Vert U^+_\tau \Vert^{J+1}_{L^2_{f}}
  \]
\item \label{Prop_PDE_ODI_MCF_c} There is a $T' \geq -\infty$ such that $R(\tau) = \infty$ and $u_\tau \equiv 0$ if $\tau \leq T'$ and $R(\tau) < \infty$ if $\tau > T'$.
\end{enumerate}
\end{Proposition}

\begin{proof}
Lemma~\ref{Lem_mcf_weak_pseudo} (applied for $A =1$) shows that we can choose $\sigma = \sigma(\eta, m)$ to fulfill the $(\sigma, \eta, m)$-pseudolocality property.
Thus we need to assume that $\eps \leq \ov\eps(\sigma) \leq \ov\eps(\eta, m)$.
\end{proof}
\medskip

The following addendum to Proposition~\ref{Prop_PDE_ODI_MCF_gauged} further characterizes the gauged case.

\begin{Proposition}[Addendum to Proposition~\ref{Prop_PDE_ODI_MCF_gauged}] \label{Prop_PDE_ODI_MCF_gauged}
If in Proposition~\ref{Prop_PDE_ODI_MCF}, we additionally assume that $R^\# \geq \underline R^\#$, $\la < 0$, $\eta \leq \ov\eta(\la, R^\#)$ and $R^* \geq R^\# + 10$, then the following bound holds if $\td\MM'$ is $R^\#$-gauged
\begin{equation*}
 \Vert \PP_{\sV_{\Jac}} U^+_\tau \Vert_{L^2_f}  + \Vert \td\bY_{\tau} \Vert_{L^2_f}
 \leq \Psi(R^\#) \big( \| \PP_{\sV_{\osc, < 0}} U^+_\tau \|_{L^2_f} + \UU^-_\tau \big) . %
\end{equation*}
Here $\Psi(R^\#)$ is a universal function with $\Psi(R^\#) \to 0$ as $R^\# \to \infty$.
\end{Proposition}

\begin{proof}
Since $R(\tau) \geq R^* \geq R^\#$, we have $\omega_{R(\tau)} \omega_{R^\#} = \omega_{R^\#}$, so the $R^{\#}$-gauge property is equivalent to
\begin{equation} \label{eq_PPuomega0}
 \PP_{\sV_{\Jac}}( u_\tau \omega_{R(\tau)} \omega_{R^{\#}}) =\PP_{\sV_{\Jac}}( u_\tau \omega_{R^{\#}})=0.
\end{equation}
Choose a basis $\{ \ov U_i \}_{i=1}^N \subset \sV_{\Jac}$ consisting of functions of bounded linear growth.
Then \eqref{eq_PPuomega0} is equivalent to
\begin{equation} \label{eq_gauge_ovU}
 \big\langle u_\tau, \ov U_i \omega_{R^\#} \big\rangle_{L^2_f} = 0 \qquad \text{for all} \quad i = 1, \ldots, N. 
\end{equation}
In the next claim we consider this pairing with every term of the evolution equation for $u_\tau$.
For convenience, we will drop the ``$L^2_f$''-subscript.

\begin{Claim}
We have the following bounds:
\begin{align}
\big| \langle u_\tau \omega_{R(\tau)} , \ov U_i \rangle\big| &\leq \Psi(R^\#) \big( \Vert \PP_{\sV_{\osc}} U^+_\tau \Vert + \UU^-_\tau \big) \label{eq_withovU_1} \\
\big| \langle \Jac \td\bY_\tau , \ov U_i \omega_{R^\#}\rangle - \langle \Jac \td\bY_\tau,  \ov U_i \rangle\big| &\leq \Psi(R^\#) \Vert \td\bY_\tau \Vert \label{eq_withovU_2} \\
 \langle \partial_\tau u_\tau, \ov U_i \omega_{R^\#} \rangle  &= 0 \label{eq_withovU_3} \\
 \big| \langle L u_\tau, \ov U_i \omega_{R^\#} \rangle \big| &\leq \Psi(R^\#) \big( \Vert \PP_{\sV_{\osc}} U^+_\tau \Vert + \UU^-_\tau \big) \label{eq_withovU_4} \\
  \big| \langle Q[u_\tau, \td\bY_\tau] - \Jac \td\bY_\tau, \ov U_i \omega_{R^\#} \rangle \big| &\leq C(\la, R^\#) \eta\big( \Vert \PP_{\sV_{\osc}} U^+_\tau \Vert + \UU^-_\tau \big) + C(\la, R^\#) \eta \Vert \td\bY_\tau \Vert  \label{eq_withovU_5} 
\end{align}
\end{Claim}

\begin{proof}
Fix $\tau$ and write $R = R(\tau)$, $\td\bY = \td\bY_\tau$, $U^\pm = U^\pm_\tau$, $\UU^- = \UU^-_\tau$ and $u = u_\tau = u_{\rot} + u_{\osc}$, where $u_{\rot}$ is the average of $u$ under the action of the group $O(n-k+1) \times O(n')$, wherever defined.
Note that $u_{\rot} \omega_R$ (respectively $u_{\rot} \omega_{R^\#}$, $(L u_{\rot}) \omega_{R^\#}$ or $L(u_{\rot} \omega_{R^\#})$) is the $L^2_f$-projection of $u \omega_R$ (respectively $u \omega_{R^\#}$, $(L u) \omega_{R^\#}$ or $L(u \omega_{R^\#})$) onto $\sV_{\rot}$.
So since $\sV_{\Jac} \subset \sV_{\osc} \perp \sV_{\rot}$ and due to \eqref{eq_gauge_ovU}, the following pairings vanish 
\[ \langle u_{\rot}\omega_{R}, \ov U_i  \rangle = \langle u_{\rot} \omega_{R^\#}, \ov U_i \rangle = \langle u_{\osc}\omega_{R^\#}, \ov U_i  \rangle = \langle (L u_{\rot}) \omega_{R^\#}, \ov U_i  \rangle = \langle L (u_{\rot} \omega_{R^\#}), \ov U_i  \rangle = 0. \]
We also note that
\[ \Vert u_{\osc} \omega_{R} \Vert \leq \Vert \PP_{\sV_{\osc}} ( U^+ + U^- ) \Vert 
\leq \Vert \PP_{\sV_{\osc}}  U^+ \Vert + \Vert \PP_{\sV_{\osc}}  U^- \Vert
\leq \Vert \PP_{\sV_{\osc}}  U^+ \Vert + \UU^-. \] 

The bound \eqref{eq_withovU_1} now follows using Cauchy-Schwarz:
\begin{multline*}
 \big| \langle u \omega_{R}, \ov U_i  \rangle\big|
=  \big| \langle u_{\osc} \omega_{R}, \ov U_i  \rangle - \langle u_{\osc} \omega_{R^\#}, \ov U_i  \rangle\big|
= \big| \langle  u_{\osc} \omega_R, \ov U_i (1- \omega_{R^\#}) \rangle\big| \\
\leq \Vert u_{\osc} \omega_R \Vert \, \Vert \ov U_i (\omega_R- \omega_{R^\#}) \Vert
\leq \Psi(R^\#) \big( \Vert \PP_{V_{\osc}} U^+ \Vert + \UU^- \big) .
\end{multline*}
The bound \eqref{eq_withovU_2} follows similarly:
\begin{equation*}
  \big| \langle \Jac \td\bY, \ov U_i \omega_{R^\#} -\ov U_i \rangle\big| 
\leq \Vert \Jac \td\bY \Vert \, \big\Vert \ov U_i (\omega_{R^\#} -1) \big\Vert
\leq \Psi(R^\#) \Vert \td\bY \Vert.  
\end{equation*}
Equation~\eqref{eq_withovU_3} follows directly from \eqref{eq_gauge_ovU}.
For the bound \eqref{eq_withovU_4} we need to use that $L \sV_{\Jac} \subset \sV_{\Jac}$, so $\PP_{\sV_{\Jac}}(L(u \omega_{R^\#})) = 0$, which implies $\langle L (u \omega_{R^\#}), \ov U_i \rangle = 0$.
Therefore
\begin{align*}
 \big| \langle L u, \ov U_i \omega_{R^\#} \rangle \big|
 &= \big| \langle L u_{\rot} + L u_{\osc}, \ov U_i \omega_{R^\#} \rangle -  \langle L(u_{\osc}\omega_{R^\#} + u_{\rot}\omega_{R^\#}), \ov U_i \rangle \big| \\
 &= \big| \langle L(u_{\osc} \omega_R), \ov U_i \omega_{R^\#} \rangle - \langle u_{\osc} \omega_{R^\#}, (L\ov U_i) \rangle \big|\\
 &= \big| \langle   u_{\osc} \omega_R, L( \ov U_i \omega_{R^\#}) \rangle -  \langle u_{\osc} \omega_R, (L\ov U_i) \omega_{R^\#} \rangle \big| \\
  &\leq \| u_{\osc} \omega_R \| \, \big\|  L( \ov U_i \omega_{R^\#}) - (L\ov U_i) \omega_{R^\#}  \big\|  \\
  & 
\leq \Psi(R^\#)   \big( \Vert \PP_{V_{\osc}} U^+ \Vert + \UU^- \big).
\end{align*}
To see \eqref{eq_withovU_5} we write $u \omega_R = \PP_{\sV_{\rot}} U^+ + (\PP_{\sV_{\osc}} U^+ + U^-) = U^+_{\rot} + u'$; note that $\Vert u' \Vert \leq \Vert \PP_{\sV_{\osc}} U^+ \Vert + \UU^-$. 
As in the proof of Lemma~\ref{lem:Q_initial}, we can write over $\IB^k_{R^\#} \times \IS^{n-k} \subset \IB^k_{R-1} \times \IS^{n-k}$ and as long as $\eta \leq \ov\eta(\la, R^\#)$
\[   Q[u,\td\bY]  - Q[U^+_{\rot},\td\bY] 
 = \sum_{j=0}^2 Q^*_j * \nabla^j u', \]
where $Q^*_j = Q^*_j(\mathbf x, \mathbf y, u, \nabla u, \nabla^2 u, U_{\rot}, \nabla U_{\rot}, \nabla^2 U_{\rot}, \td\bY)$ depend smoothly on the indicated parameters.
Since the linearization of of $Q[u,\td\bY]$ at $u \equiv 0$ and $\td\bY = 0$ is given by $\Jac \td\bY$ (see Lemma~\ref{Lem_mode_dec}), we must have $Q^*_j(\mathbf x, \mathbf y, 0, 0,0, 0,0,0,0) = 0$.
Therefore,
\[ \Vert Q^*_j \Vert_{C^2(\IB^{k}_{R^\#} \times \IS^{n-k})} \leq C(\la, R^\#) \big( \Vert u \Vert_{C^4(\IB^{k}_{R^\#} \times \IS^{n-k})} + \Vert U^+_{\rot} \Vert + \Vert \td\bY \Vert \big)
 \leq C(\la, R^\#) \big( \eta + \Vert \td\bY \Vert \big). \]
It follows using integration by parts that
\begin{multline} \label{eq_QUrotu}
 \big| \langle Q[U^+_{\rot},\td\bY] - Q[u,\td\bY], \ov U_i \omega_{R^\#} \rangle \big| 
\leq \sum_{j=0}^2 \bigg| \int_{M_{\cyl}} \nabla^j (Q^*_j * \ov U_i \omega_{R^\#} e^{-f} ) * u' \, dg_{\cyl} \bigg| \\
\leq C(\la,R^\#) \big( \eta + \Vert \td\bY \Vert \big) \Vert u' \Vert 
\leq C(\la, R^\#) \big( \eta + \Vert \td\bY \Vert \big) \big( \Vert \PP_{\sV_{\osc}} U^+ \Vert + \UU^- \big).
\end{multline}
For symmetry reasons, the term $Q[U^+_{\rot}, 0]$ must be invariant under the action of $O(n-k+1) \times O(n')$, so as before
\[ \langle Q[U^+_{\rot}, 0], \ov U_i \omega_{R^\#} \rangle = 0. \]
So since $Q[0, \td\bY] = \Jac \td\bY$ and since $U^+_{\rot} \in \sV_{> \la}$ is an element of a finite-dimensional space, it follows that
\begin{align}
 \big| \langle Q[U^+_{\rot}, \td\bY] - \Jac \td\bY, \ov U_i \omega_{R^\#} \rangle \big|
&= \big| \langle Q[U^+_{\rot}, \td\bY] - Q[U^+_{\rot}, 0] - Q[ 0,\td\bY], \ov U_i \omega_{R^\#} \rangle \big| \notag \\
&\leq \int_0^1 \Big|  \big\langle \partial_{\td\bY} Q[ U^+_{\rot}, s\td\bY ] -  \partial_{\td\bY} Q[0,s\td\bY], \ov U_i \omega_{R^\#} \big\rangle\Big| ds \cdot \Vert \td\bY \Vert \notag \\
&\leq C(\la, R^\#) \Vert U^+_{\rot} \Vert \cdot \Vert \td\bY \Vert 
\leq C(\la,R^\#) \eta \Vert \td\bY \Vert. \label{eq_QUrotY}
\end{align}
Combining \eqref{eq_QUrotu} and \eqref{eq_QUrotY} implies \eqref{eq_withovU_5}.
\end{proof}

Combining the bounds from the claim with the evolution equation for $u$ gives us for $\eta \leq\ov\eta(\la, R^\#)$
\begin{align*}
\Vert \PP_{\sV_{\Jac}} U^+_\tau \Vert + \Vert \Jac \td\bY_\tau \Vert &\leq 
 \sum_{i=1}^N |\langle U^+_\tau, \ov U_i \rangle| + \sum_{i=1}^N |\langle  \Jac \td\bY_\tau , \ov U_i \rangle | \\
 &\leq (\Psi(R^\#)+ C(\la, R^\#) \eta) \big( \Vert \PP_{\sV_{\osc}} U^+_\tau \Vert + \UU^-_\tau  + \Vert \td\bY_\tau \Vert \big) \\
&\leq \Psi(R^\#) \big( \Vert \PP_{\sV_{\osc, < 0}} U^+_\tau \Vert + \Vert \PP_{\sV_{\Jac}} U^+_\tau \Vert  + \UU^-_\tau  + \Vert \td\bY_\tau \Vert \big).  
\end{align*}
Since $\td\bY_\tau \in \YY_{\perp}$, we have $\Vert \td\bY_\tau \Vert \leq C \Vert \Jac \td\bY_\tau \Vert$, so the desired bound follows for $R^\# \geq \underline R^\#$.
\end{proof}
\bigskip

\subsection{Bounding higher modes} \label{subsec_higher_modes}
While the functions $U^+_\tau$ (and $\UU^-_\tau$) from Proposition~\ref{Prop_PDE_ODI_MCF} evolve by a \emph{finite-dimensional} ODI system, which in principle could be analyzed using standard ODE techniques, the dimension of this system---depending on $\la$---may need to be chosen quite large in order to obtain strong asymptotic bounds.
To simplify matters, we will now refine the analysis and condense $U^+_\tau$ further.
We will see that stable modes of $U^+_\tau$, such as non-rotationally symmetric modes and rotationally symmetric modes corresponding to Hermite polynomials of degree $\geq 3$, decay exponentially fast provided the flow is appropriately gauged.
As a result the component of $U^+_\tau$ within $\sV_{\rot, 1} \oplus \sV_{\rot,\frac12} \oplus \sV_{\rot, 0}$ typically dominates, which reduces our analysis to the study of the projection of $U^+_\tau$ onto this subspace.

For the remainder of this section, we will frequently consider $R^\#$-gauged mean curvature flows as in Propositions~\ref{Prop_PDE_ODI_MCF} and \ref{Prop_PDE_ODI_MCF_gauged}.
Since we will often analyze projections of $U^+_\tau$ onto various subspaces of $\sV_{>\la}$, we adopt the following notation for convenience: the relevant subspace will appear as a subscript of $U^+$, and the dependence on $\tau$ will be shown in parentheses, or omitted when clear.
For example, we will write $U^+_{\rot, \frac12} (\tau)$ instead of $\PP_{\sV_{\rot,\frac12}} U^+_\tau$ and define $U^+_{\osc, \leq \la} (\tau)$, $U^+_{\Jac}(\tau)$, $\UU^-(\tau)$, etc., similarly.
Then we have $U^+ = U^+_{\rot} + U^+_{\osc}$ and for $\la \in \frac12\IZ$ and $-\frac12 i > \la$
\[ U^+_{\rot, \leq -\frac12 i} = U^+_{\rot, - \frac12 i} + \ldots + U^+_{\rot, \la + \frac12}. \]
Note that since  $U^+_{\rot, \leq - \frac12i}$ is a projection of $U^+ \in \sV_{> \la}$, it even lies in $\sV_{\leq -\frac12 i} \cap \sV_{> \la}$.
We also recall that (see Lemma~\ref{Lem_mode_dec})
\[ U^+_{\osc} = U^+_{\Jac} + U^+_{\osc,<0}. \]
Next, we define $\sV^{++} := \sV_{\rot, 1} \oplus \sV_{\rot,\frac12} \oplus \sV_{\rot, 0}$ and write
\[ U^{++} := \PP_{\sV^{++}} U^+ = U^+_{\rot, 1} + U^+_{\rot,\frac12} + U^+_{\rot,0}, \] 
so
\[ U^+_{\rot} = U^{++} + U^+_{\rot, \leq -\frac12}. \]
For convenience we will drop the subscript ``$L^2_f$'' in norms and inner products when there is no chance of confusion.
\medskip

The following proposition, which is the main result of the subsection, states that the terms $\| U^+_{\osc, < 0} \|$, $\| U^+_{\rot, \leq -\frac12 i} \|$ and $\UU^-$ decay exponentially until they are eventually bounded by a sufficiently high power of the leading term $\| U^{++} \|$.

\begin{Proposition} \label{Prop_Tisosc}
Consider the setting of Proposition~\ref{Prop_PDE_ODI_MCF} and its addendum, Proposition~\ref{Prop_PDE_ODI_MCF_gauged}, where we assume in addition that $\la < 0$ and $\la \in \frac12 \IZ$ and bounds $R^\# \geq \underline R^\#(J)$ and $\eta \leq \ov\eta(\la, J,  R^\#)$.
Set $J_0 := \min \{ 2J , -  2\la  \}$ and consider the time-dependent quantities
\begin{equation} \label{eq_tdUUi_def}
 \td\UU_i := \| U^+_{\osc, < 0} \| + \| U^+_{\rot, \leq -\frac{i}2} \| + \UU^-, \qquad i = 1, \ldots, J_0. 
\end{equation}
Then there are times $ T_0 := \inf \td I = T'_0 \leq T'_1 \leq \ldots \leq T'_{J_0} \leq T_1 := \sup \td I$
such that for any $i = 1, \ldots, J_0$ the following is true for some constant $C_i = C_i (\la, J)$:
\begin{enumerate}[label=(\alph*)]
\item \label{Prop_Tisosc_a} For all $\tau \in (T'_{i-1}, T'_i)$ we have $\td\UU_i(\tau) \geq C_i \|U^{++}(\tau)\|^{\lceil i/2 \rceil +1}$ and the following exponential decay:
\[ \partial_\tau \td\UU_i \leq -\tfrac1{10(n-k)} \td\UU_i \]
\item \label{Prop_Tisosc_b} If $T'_i < T_1= \sup \td I$, then 
\begin{equation} \label{eq_tdUUppfortauTi}
\td\UU_i(\tau) \leq C_i \|U^{++}(\tau)\|^{\lceil i/2 \rceil +1} \qquad \text{for} \quad \tau \geq T'_i, \quad \tau \in \td I. 
\end{equation}
\item \label{Prop_Tisosc_c} If $T_0 = \inf \td I = -\infty$, then $T'_0 = \ldots = T'_{J_0} = -\infty$ and for all $\tau \in \td I$ we have
\begin{align*}
 \| U^+_{\osc} \| + \UU^- + \| \td\bY \| 
 &\leq C(\la, J) \| U^{++} \|^{\lceil J_0/2 \rceil + 1} \\
 \| U^+_{\rot, \leq - \frac12 i} \| 
 &\leq C(\la, J) \| U^{++} \|^{\lceil i/2 \rceil + 1}, \qquad i = 1, \ldots, J_0 . 
\end{align*}
\end{enumerate}
\end{Proposition}
\medskip

The proof of Proposition~\ref{Prop_Tisosc} relies on the following lemma, which characterizes the evolution of the quantities $U^{++}(\tau)$, \lb $U^+_{\osc,< 0}(\tau)$, \lb $U^+_{\rot, \leq -\frac12 i}(\tau)$ and $\UU^-$.

\begin{Lemma}\label{lem:all_decay_modes}
Consider the setting of Proposition~\ref{Prop_PDE_ODI_MCF} and its addendum, Proposition~\ref{Prop_PDE_ODI_MCF_gauged}.
In addition we assume that $\la < 0$ and $\la \in \frac12 \IZ$ and bounds of the form $\xi > 0$, $R^\# \geq \underline{R}^\#(\xi)$ and $\eta \leq \ov\eta(\la,J, \xi)$.
Then we have the following evolution bounds for $-2 \leq i \leq -2\la$, for some constant $C = C(\lambda, J)$
\begin{alignat}{2}
	\partial_\tau \| U^{++} \| 
&\geq -\xi \|U^{++}\| - C \big( \|U^+_{\osc,<0}\|^2 + \| U^+_{\rot,< 0}\|^2 \big) &&-C \|U^+ \|^{J+1}-\xi \UU^-  \label{eq_U^++}\\
	\partial_\tau \|U^+_{\osc,<0}\| 
&\leq (-\tfrac1{n-k}+\xi) \| U^+_{\osc,<0} \| && + C \| U^+ \|^{J+1}+\xi \UU^-  \label{eq_U^+_osc}\\
	\partial_\tau \big\| U^+_{\rot,\leq -\frac12 i} \big\| 
&\leq -\tfrac12 i  \big\| U^+_{\rot, \leq - \frac12 i} \big\| + \xi \| U^+_{\osc, < 0} \| && \notag \\
&\qquad  + C \sum_{\substack{-2 \leq i_1, \ldots, i_l < -2\la, \; l \geq 2 \\  (i_1+2) + \ldots + (i_l+2) \geq i+2 }}  \| U^+_{\rot, - \frac12 i_1} \| \cdots \| U^+_{\rot, - \frac12 i_l} \|  &&+C \| U^+ \|^{J+1} + \xi \UU^- \label{eq_U^+_rot} \\
    \partial_\tau \UU^-
&\leq (\lambda + \xi) \UU^- 
+ \xi \| U^+_{\osc,<0} \|
\notag \\
&\qquad  + C \sum_{\substack{-2 \leq i_1, \ldots, i_l < -2\la, \; l \geq 2 \\  (i_1+2) + \ldots + (i_l+2) \geq - 2\la + 2 }}  \| U^+_{\rot, -\frac12 i_1} \| \cdots \| U^+_{\rot, - \frac12 i_l} \|  &&+C \| U^+ \|^{J+1} \label{eq_U^-},
\end{alignat}
Moreover, 
\begin{multline} \label{eq_evolution_sum}
	\partial_\tau \Big( \|U^+_{\osc,<0}\| + \big\| U^+_{\rot,\leq - \frac12 i} \big\| + \UU^- \Big)
\leq  -\tfrac1{3(n-k)} \Big( \|U^+_{\osc,<0}\| + \big\| U^+_{\rot,\leq - \frac12 i} \big\| + \UU^- \Big)
 \\
+ C \sum_{\substack{-2 \leq i_1, \ldots, i_l < -2\la , \; l \geq 2 \\  (i_1+2) + \ldots + (i_l+2) \geq i+2 }}  \| U^+_{\rot, - \frac12 i_1} \| \cdots \| U^+_{\rot, - \frac12 i_l}\| +C \| U^+ \|^{J+1}.
\end{multline}
\end{Lemma}

\begin{proof}
We will frequently use the following bound, which follows from Proposition~\ref{Prop_PDE_ODI_MCF_gauged} if $R^\# \geq \underline R^\#(\xi)$
\begin{align}
 \Vert U^+ \Vert 
&\leq \Vert U^{++} \Vert + \Vert U^+_{\Jac} \Vert + \Vert U^+_{\rot, < 0} \Vert + \Vert U^+_{\osc, < 0} \Vert \notag \\
&\leq \Vert U^{++} \Vert + \Psi (R^\#) \big( \Vert U^+_{\osc,<0} \Vert + \UU^- \big) + \Vert U^+_{\rot, < 0} \Vert + \Vert U^+_{\osc, < 0} \Vert \notag \\
&\leq \Vert U^{++} \Vert + \Vert U^+_{\rot, < 0} \Vert + 2\Vert U^+_{\osc, < 0} \Vert + \xi \UU^-. \label{eq_rough_Up}
\end{align}
Moreover, we will use the bounds
\[ \Vert U^+ \Vert, \; \Vert U^{++} \Vert, \; \Vert U^+_{\rot, <0} \Vert, \; \Vert U^+_{\osc, < 0} \Vert, \; \UU^- \leq C \Vert u \Vert_{C^0(\IB^k_{R} \times \IS^{n-k})} \leq C \eta. \]

Let us now consider the inequalities \eqref{eq_U^++}--\eqref{eq_U^-}.
Since we can adjust $\xi$ by a fixed factor, it suffices to show the desired bounds with $\xi$ replaced by $10\xi$.
Denote by $\sV$ one of the subspaces $\sV^{++}$, $\sV_{\osc, < 0}$ or $\sV_{\rot, \leq - \frac{i}2}$ and let $\la_{\min/\max, \sV}$ be the minimal and maximal eigenvalue of $L$ restricted to $\sV$.
Using Assertion~\ref{Prop_PDE_ODI_MCF_b} of Proposition~\ref{Prop_PDE_ODI_MCF} we find that
\begin{multline*}
 \partial_\tau \Vert \PP_{\sV} U^+ \Vert 
 = \frac{\langle \partial_\tau \PP_{\sV} U^+, \PP_{\sV} U^+ \rangle}{\Vert \PP_{\sV} U^+ \Vert } 
=  \frac{\langle \PP_{\sV}(  \partial_\tau  U^+ - LU^+ - Q^+_J(U^+, \td\bY) ), \PP_{\sV} U^+ \rangle}{\Vert \PP_{\sV} U^+ \Vert } \\+ \frac{\langle L \PP_{\sV} U^+, \PP_{\sV} U^+ \rangle}{\Vert \PP_{\sV} U^+ \Vert }  + \frac{\langle  \PP_{\sV} Q^+_J(U^+, \td\bY), \PP_{\sV} U^+ \rangle}{\Vert \PP_{\sV} U^+ \Vert } \\
\geq - C(\la,J) \Vert U^+ \Vert^{J+1} - \big( C(\la) \sqrt{\eta} + C \Vert \td\bY \Vert  \big) \UU^- + \la_{\min,{\sV}} \Vert \PP_{\sV} U^+ \Vert  - \big\Vert \PP_{\sV} Q^+_J(U^+, \td\bY) \big\Vert.
\end{multline*}
Assuming $\eta \leq \ov\eta(\la, \xi)$ and using Proposition~\ref{Prop_PDE_ODI_MCF_gauged} if $R^\# \geq \underline{R}^\# (\xi)$, we have
\[ C(\la) \sqrt{\eta} + C \| \td\bY \|  
\leq \xi +  C(\Vert U^+ \Vert + \UU^-) 
\leq 2 \xi. \]
Therefore,
\begin{equation} \label{eq_dtPVgeq}
 \partial_\tau \Vert \PP_{\sV} U^+ \Vert \geq - C(\la, J) \Vert U^+ \Vert^{J+1} - 2\xi \UU^- + \la_{\min, \sV} \Vert \PP_{\sV} U^+ \Vert  - \big\Vert \PP_{\sV} Q^+_J(U^+, \td\bY) \big\Vert. 
\end{equation}
Similarly, we obtain
\begin{align} 
\partial_\tau \Vert \PP_{\sV} U^+ \Vert &\leq  C(\la, J) \Vert U^+ \Vert^{J+1} + 2\xi \UU^- + \la_{\max, \sV} \Vert \PP_{\sV} U^+ \Vert  + \big\Vert \PP_{\sV} Q^+_J(U^+, \td\bY) \big\Vert,  \label{eq_dtPVleq} \\
\partial_\tau \UU^- &\leq  C(\la, J) \Vert U^+ \Vert^{J+1} + (\la + 2\xi) \UU^-  + \big\Vert Q^-_J(U^+, \td\bY) \big\Vert. \label{eq_dtUUm}
\end{align}
Next, observe that (see Lemma~\ref{Lem_mode_dec})
\begin{equation} \label{eq_lambdas}
 \la_{\min, \sV^{++}} = 0, \qquad
\la_{\max, \sV_{\osc, < 0}} = - \tfrac1{n-k}, \qquad
\la_{\max, \sV_{\rot, \leq-\frac12 i}} = - \tfrac12 i. 
\end{equation}

So it remains to bound the norms of the projections of $Q_J(U^+, \td\bY)$.
The next claim allows us to reduce this task to the case in which $U^+_{\Jac} = \td\bY = 0$.
Note that $U^+ - U^+_{\Jac}$ is the projection of $U^+$ onto the orthogonal complement of $\sV_{\Jac}$ in $\sV_{>\la}$.

\begin{Claim} \label{Cl_Q_diff}
If $\eta \leq \ov\eta(\la,J, \xi)$ and $R^\# \geq \underline R^\#(\xi)$, then the term 
\[ \big\Vert \PP_{\sV} \big( Q_J (U^+, \td\bY ) - Q_J(U^+ - U^+_{\Jac}, 0) \big) \big\Vert \]
is bounded by each of the following three terms, where $C= C(\la,J)$:
\begin{itemize}
\item $C \Psi(R^\#) \big(\Vert U^{++} \Vert + \Vert U^+_{\rot, < 0} \Vert + 2\Vert U^+_{\osc, < 0} \Vert \big) \Vert U^{+}_{\osc,<0} \Vert  + 2\xi \UU^-$,
\item $\xi \Vert U^{++} \Vert + C \big( \Vert U^+_{\osc, < 0} \Vert^2 + \Vert U^+_{\rot, <0} \Vert^2 \big) + 2 \xi \UU^- $,
\item  $\xi \Vert U^+_{\osc,<0} \Vert + 2 \xi \UU^-$.
\end{itemize}
\end{Claim}

\begin{proof}
Let us view $Q_J$ as a polynomial function on the finite dimensional vector space $\sV_{>\la} \times \YY_\perp$.
Due to the structure of $Q[u, \td\bY]$, as detailed in Lemma~\ref{Lem_structure_MCF_graph_equation}, and Proposition~\ref{Prop_PDE_ODI_MCF_gauged}, we have $Q_J(0,\td\bY) = \Jac \td\bY \in \sV_{\Jac}$ and $U^+ \mapsto Q_J(U^+,0)$ has no constant or linear term.
So since $\sV \perp \sV_{\Jac}$, we have for $\eta \leq \ov\eta( \xi)$ and $R^\# \geq \underline R^\#(\xi)$ and for a generic constant $C= C(\la,J)$
\begin{align*}
 \big\Vert \PP_{\sV} & \big( Q_J (U^+, \td\bY ) - Q_J(U^+ - U^+_{\Jac}, 0) \big) \big\Vert \\
&\leq \big\Vert \PP_{\sV} \big( Q_J (U^+, \td\bY ) - Q_J(U^+ , 0) \big)  \big\Vert + \big\Vert \PP_{\sV} \big( Q_J (U^+, 0 ) - Q_J(U^+ - U^+_{\Jac}, 0) \big) \big\Vert \\
&\leq C \Vert U^+ \Vert \, \big( \Vert \td\bY \Vert  + \Vert U^+_{\Jac} \Vert \big) \\
&\leq C  \Vert U^+ \Vert \cdot \Psi (R^\#) \big( \Vert U^+_{\osc,<0} \Vert + \UU^- \big) \\
&\leq C \Psi(R^\#)   \Vert U^+ \Vert \,  \Vert U^+_{\osc,<0} \Vert +  C \eta \Psi(R^\#) \UU^- \\
&\leq C \Psi(R^\#) \big(\Vert U^{++} \Vert + \Vert U^+_{\rot, < 0} \Vert + 2\Vert U^+_{\osc, < 0} \Vert + \xi \UU^- \big) \Vert U^{+}_{\osc,<0} \Vert + \xi \UU^- \\
&\leq C \Psi(R^\#) \big(\Vert U^{++} \Vert + \Vert U^+_{\rot, < 0} \Vert + 2\Vert U^+_{\osc, < 0} \Vert \big) \Vert U^{+}_{\osc,<0} \Vert  + 2\xi \UU^-.
\end{align*}
This establishes the first desired bound.
The second and third bounds follow from this bound.
\end{proof}

By combining \eqref{eq_dtPVgeq}--\eqref{eq_dtUUm} with \eqref{eq_lambdas} and Claim~\ref{Cl_Q_diff}, we can reduce the lemma to the case in which $\td\bY = U^+_{\Jac}=0$.
That is if we set $Q_J(U^+) := Q_J(U^+,0)$, then we need to show the following bounds, for any $U^+ \in \sV_{> \la}$ with $U^+_{\Jac} = 0$, where $C=C(\la,J)$,
\begin{align}
	\big\Vert \PP_{\sV^{++}}  Q_J(U^+) \big\Vert
&\leq \xi \Vert U^{++} \Vert + C \big( \Vert U^+_{\osc, < 0} \Vert^2 + \Vert U^+_{\rot, < 0} \Vert^2 \big)   \label{eq_PUppQ} \\
	\big\Vert \PP_{\sV_{\osc, <0}}  Q_J(U^+) \big\Vert 
&\leq \xi \Vert U^+_{\osc, <0} \Vert  \label{eq_PoscQ}  \\
	\big\Vert \PP_{\sV_{\rot, \leq - \frac12 i}}  Q_J(U^+) \big\Vert 
&\leq  \xi \| U^+_{\osc, < 0} \| + C \sum_{\substack{-2 \leq i_1, \ldots, i_l < -2\la , \;  l \geq 2 \\  (i_1+2) + \ldots + (i_l+2) \geq i+2 }}  \| U^+_{\rot, - \frac12 i_1} \| \cdots \| U^+_{\rot, - \frac12 i_l} \|  \label{eq_ProtQ}  \\
	\big\Vert \PP_{\sV_{< \la}}  Q_J(U^+) \big\Vert 
&\leq 
\xi \| U^+_{\osc,<0}\|
    + C \sum_{\substack{-2 \leq i_1, \ldots, i_l < -2\la, \; l \geq 2 \\  (i_1+2) + \ldots + (i_l+2) \geq - 2\la + 2 }}  \| U^+_{\rot, -\frac12 i_1} \| \cdots \| U^+_{\rot, - \frac12 i_l} \|   \label{eq_PllaQ}
\end{align}
Note that these bounds purely concern the polynomial function $U^+ \mapsto Q_J(U^+)$.
The following claim allows us to reduce these bounds to the case in which $U^+_{\osc} = 0$.

\begin{Claim} \label{Cl_rotoscquadratic}
We have the following bounds
\begin{align}
 \big\Vert \PP_{\sV_{\rot}} \big( Q_J(U^+) - Q_J(U^+_{\rot}) \big) \big\Vert &\leq C(\la, J) \Vert U^+_{\osc, < 0} \Vert^2     \label{eq_PProt_quadratic} \\
 \big\Vert \PP_{\sV_{\osc}} Q_J(U^+)  \big\Vert = \big\Vert \PP_{\sV_{\osc}} \big( Q_J(U^+) - Q_J(U^+_{\rot}) \big) \big\Vert &\leq C(\la, J) \Vert U^+ \Vert \, \Vert U^+_{\osc, < 0} \Vert   \label{eq_PPosc_quadratic} 
\end{align}
\end{Claim}

\begin{proof}
As $Q_J$ is a polynomial, we can write
\[ Q_J(U^+_{\rot} + U^+_{\osc}) - Q_J(U^+_{\rot}) = \sum_{j=0}^{J} A_j (U^+_{\rot}) (U^+_{\osc}, \ldots, U^+_{\osc}), \]
where the $A_j (U^+_{\rot})$ are $j$-linear multilinear forms in $U^+_{\osc}$, whose coefficients are polynomials in $U^+_{\rot}$.
Setting $U^+_{\osc} = 0$ implies that $A_0(U^+_{\rot}) = 0$. 
Moreover, $A_1(0) (U^+_{\osc}) = 0$ since $U^+ \mapsto Q_J(U^+)$ has no linear term.
This shows \eqref{eq_PPosc_quadratic}.
Note that $U^+_{\osc} = U^+_{\osc, < 0}$, since $U^+_{\Jac} = 0$.

To see \eqref{eq_PProt_quadratic} we claim that 
\begin{equation} \label{eq_PPVrot_A1}
\PP_{\sV_{\rot}} \big (A_1(U^+_{\rot})(U^+_{\osc}) \big) = 0.
\end{equation}
To see this, consider the action of $O(n-k+1) \times O(n')$ on $\sV_{>\la} \subset L^2(M_{\cyl} ; \IR \times \IR^{n'})$, as defined in Subsection~\ref{subsec_linearization}, and note that $Q_J$ is equivariant with respect to this action.
Hence the maps $U^+ \mapsto A_j (U^+_{\rot}) (U^+_{\osc}, \ldots, U^+_{\osc})$ are equivariant as well.
Since $U^+_{\rot}$ is invariant under this action, it follows that for fixed $U^+_{\rot}$ and arbitrary fixed $\ov U \in \sV_{\rot}$ the linear map $\sV_{\osc, > \la} \to \IR$, $U^+_{\osc} \mapsto \langle A_1 (U^+_{\rot}) (U^+_{\osc}), \ov U \rangle$ must also be equivariant.
But there is an $\langle \ov U' \in \sV_{\osc, > \la}$ such that $A_1 (U^+_{\rot}) (U^+_{\osc}), \ov U \rangle = \langle U^+_{\osc}, \ov U' \rangle$ for all $U^+_{\osc} \in \sV_{\osc, > \la}$.
So the map $U^+_{\osc} \mapsto \langle U^+_{\osc}, \ov U' \rangle$ must be equivariant under the same group action and hence $\ov U'$ must be invariant under it, implying that $\ov U' = 0$.
This proves \eqref{eq_PPVrot_A1} and hence the bound \eqref{eq_PProt_quadratic}.
\end{proof}

Claim~\ref{Cl_rotoscquadratic} implies \eqref{eq_PoscQ} if $\eta \leq \ov\eta(\la, J, \xi)$, and it shows that in order to show \eqref{eq_PUppQ}, \eqref{eq_ProtQ}, \eqref{eq_PllaQ}, we may assume that $U^+_{\osc} = 0$.
The bound \eqref{eq_PUppQ} now follows if $\eta \leq \ov\eta( \la,J, \xi)$
\begin{multline*}
 \big\Vert \PP_{\sV^{++}}  Q_J(U^+) \big\Vert 
\leq  C(\la,J) \Vert U^+ \Vert^2 
\leq C(\la, J) \big( \Vert U^{++} \Vert^2 + \Vert U^+_{\rot, <0} \Vert^2 \big) \\
\leq \xi \Vert U^{++} \Vert + C(\la,J)  \Vert U^+_{\rot, <0} \Vert^2.  
\end{multline*}
The bounds \eqref{eq_PUppQ}, \eqref{eq_ProtQ} are a consequence of the following claim.

\begin{Claim}
For any $-2 \leq i \leq -2\la$ and any $U^+ \in \sV_{> \la}$ with $U^+_{\osc} = 0$ we have the bound
\[ \big\Vert \PP_{ \sV_{\leq - \frac{i}2}} Q_J(U^+) \big\Vert \leq C(\la,J) \sum_{\substack{-2 \leq i_1, \ldots, i_l < -2\la, \; l\geq 2 \\  (i_1+2) + \ldots + (i_l+2) \geq i+2 }}  \| U^+_{\rot, - \frac12 i_1} \| \cdots \| U^+_{\rot, - \frac12 i_l} \|. \]
\end{Claim}

\begin{proof}
We recall that
\[ Q_J (U^+) = \td Q_J(U^+, \nabla U^+, \nabla^2 U^+, 0), \]
where the latter is a polynomial of degree $J$.
Since
\[ U^+ = U^+_{\rot, 1} + U^+_{\rot, \frac12} + \ldots + U^+_{\rot,  \la + \frac12}, \]
we can express $Q_J(U^+)$ as a sum of terms of the form
\[ Q_{m_1, \ldots, m_l; i_1, \ldots, i_l} := \alpha_{m_1, \ldots, m_l; i_1, \ldots, i_l} ( \nabla^{m_1} U^+_{\rot, -\frac12 i_1}, \ldots, \nabla^{m_l} U^+_{\rot, -\frac12 i_l} ), \]
where $\alpha_{m_1, \ldots, m_l; i_1, \ldots, i_l}$ is a spatially constant multilinear form, $m_1, \ldots, m_l \in \{0,1,2\}$ and $-2 \leq i_1, \ldots, i_l < -2\la$ and $l \geq 2$.
By Lemma~\ref{Lem_mode_dec}, we know that $U^+_{\rot, -\frac12 i'} (\bx,\by) = U^+_{\rot, -\frac12 i'} (\bx)$ is a polynomial of degree $i'+2$.
So $\nabla^{m'}  U^+_{\rot, -\frac12 i'}$ has degree $\leq i'+2-m'$ and hence the degree of $Q_{m_1, \ldots, m_l; i_1, \ldots, i_l}$ is at most $(i_1+2) + \ldots + (i_l+2)$.
It follows that $Q_{m_1, \ldots, m_l; i_1, \ldots, i_l} \in \sV_{> -\frac{i}2}$ and thus $\PP_{\sV_{\leq -\frac{i}2}} Q_{m_1, \ldots, m_l; i_1, \ldots, i_l} = 0$ unless $(i_1+2) + \ldots + (i_l+2) \geq i+2$, in which case we have
\[ \Vert Q_{m_1, \ldots, m_l; i_1, \ldots, i_l} \Vert \leq C(\la,J) \| U^+_{\rot, - \frac12 i_1} \| \cdots \| U^+_{\rot, - \frac12 i_l} \|. \]
Summing over all possible terms implies the claim.
\end{proof}

This concludes the proof of \eqref{eq_U^++}--\eqref{eq_U^-}.
Adding these bounds implies \eqref{eq_evolution_sum} for $\xi \leq \ov\xi$.
\end{proof}
\medskip

\begin{proof}[Proof of Proposition~\ref{Prop_Tisosc}.]
Assume that the constants $C_i$ are already given; we will determine their values depending only on $\la, J$ in the course of this proof.
Set $T'_0 := T_0$ and iteratively choose $T'_1, T'_2, \ldots, T'_{J_0}$ minimal such that Assertion~\ref{Prop_Tisosc_b} holds and such that $T'_0 \leq T'_1  \leq \ldots \leq T_1$.
Note that we may need to choose $T'_{i} = T'_{i+1} = \ldots = T = T_1$ for some $i$ if the bound from Assertion~\ref{Prop_Tisosc_b} is violated near $T_1$.
Then
\begin{equation} \label{eq_tdUUUpp_equality}
 \td\UU_i(T'_i) = C_i \| U^{++} (T'_i) \|^{\lceil i/2 \rceil +1} \qquad \text{or} \qquad T'_i \in \{ T'_{i-1}, T_1 \}. 
\end{equation}

We will now prove Assertion~\ref{Prop_Tisosc_a}; so fix some $i \in \{ 1, \ldots, J_0 \}$ and assume that $T'_{i-1} < T'_i$ and assume that the constants $C_1, \ldots, C_{i-1}$ have already been chosen.
The bound \eqref{eq_evolution_sum} from Lemma~\ref{lem:all_decay_modes} implies that for some generic constant $C = C(\la, J)$ we have the following bound on $(T'_{i-1}, T'_i)$
\begin{equation*}
 \partial_\tau \td\UU_i \leq -\tfrac1{3(n-k)} \td\UU_i 
+ C \sum_{\substack{-2 \leq i_1, \ldots, i_l < -2\la , \; l \geq 2 \\  (i_1+2) + \ldots + (i_l+2) \geq i+2 }}  \| U^+_{\rot, - \frac12 i_1} \| \cdots \| U^+_{\rot, - \frac12 i_l}\| +C \| U^+ \|^{J+1}.  
\end{equation*}
Let us analyze the sum on the right-hand side.
Any summand containing an index $i_{j} \geq i$ can be bounded by $C(\la, J) \eta \td\UU_i$.
So if $\eta \leq \ov\eta(\la, J, \eps)$, then all such terms can be absorbed into the first term on the right-hand side at the cost of increasing the factor $-\frac1{3(n-k)} $ to $-\frac1{4(n-k)}$.
So it remains to consider only those summands with $-2 \leq i_1, \ldots, i_l < i$.
Fix such a summand.
Any factor of the form  $\| U^+_{\rot, - \frac12 i_j} \|$ can be bounded by $\td\UU_{i_j} \leq C_{i_j} \| U^{++} \|^{\lceil i_j/2 \rceil + 1}$ if $i_j \in \{1, \ldots, i-1 \}$ and by $\| U^{++} \|$ if $i_j \in \{ -2,-1,0 \}$.
After rearranging these factors, we may assume that $i_1, \ldots, i_j \in \{ -2,-1,0 \}$ while $i_{j+1}, \ldots, i_{l} > 0$.
Then
\[ 2j + (i_{j+1}+2) + \ldots + (i_l+2) \geq (i_1+2) + \ldots + (i_l+2) \geq i+2, \]
which implies
\begin{multline*}
 j + \bigg( \bigg\lceil \frac{i_{j+1}}2 \bigg\rceil + 1 \bigg) + \ldots + \bigg( \bigg\lceil \frac{i_{l}}2 \bigg\rceil + 1 \bigg) 
 = \bigg\lceil \frac{2j}{2} \bigg\rceil + \bigg\lceil \frac{i_{j+1}+2}2 \bigg\rceil + \ldots + \bigg\lceil \frac{i_{l}+2}2 \bigg\rceil \\
 \geq \bigg\lceil \frac{2j + (i_{j+1}+2) + \ldots + (i_l+2)}{2} \bigg\rceil = \bigg\lceil \frac{i}2 \bigg\rceil + 1,
\end{multline*}
so
\begin{multline*}
  \| U^+_{\rot, - \frac12 i_1} \| \cdots \| U^+_{\rot, - \frac12 i_l}\|
\leq C( C_1, \ldots, C_{i-1}) \| U^{++} \|^{j + ( \lceil i_{j+1}/2 \rceil +1) + \ldots + ( \lceil i_{l}/2 \rceil +1) } \\
\leq C(\la, C_1, \ldots, C_{i-1}) \| U^{++} \|^{\lceil i/2 \rceil + 1 }. 
\end{multline*}
Next, consider the term $\| U^+ \|^{J+1}$.
As derived in \eqref{eq_rough_Up} we have for $R^\# \geq \underline R^\#$
\begin{equation} \label{eq_UpbyUpptdU1}
\Vert U^+ \Vert \leq \Vert U^{++} \Vert + \Vert U^+_{\rot, <0} \Vert + 2 \Vert U^+_{\osc, <0}\Vert + \eps \UU^- \leq \Vert U^{++} \Vert + 2 \td\UU_1. 
\end{equation}
Therefore, and since $\td\UU_1 \leq C_1 \Vert U^{++} \Vert^2$ for $\tau \geq T'_1$, we obtain that on $(T'_{i-1}, T'_i)$
\begin{equation} \label{eq_UppJp1bound}
 \| U^+ \|^{J+1} \leq C(\la, J) \|U^{++} \|^{J+1} + C(\la, J) \td\UU^{J+1}_1 
\leq \begin{cases} C(\la,J) \|U^{++} \|^{J+1} + C(\la,J) \eta \, \td\UU_i &\text{if $i=1$} \\
C(\la,J,C_1) \|U^{++} \|^{J+1} &\text{if $i>1$} \end{cases} 
\end{equation}
It follows that if $\eta \leq \ov\eta(\la,J)$, then for some $C'_i = C'_i(\la, J, C_1, \ldots, C_{i-1})$ we have on $(T'_{i-1}, T'_i)$
\begin{equation*}
 \partial_\tau \td\UU_i \leq -\tfrac1{5(n-k)} \td\UU_i 
+ C'_i \| U^{++} \|^{\lceil i/2 \rceil +1}.  
\end{equation*}
So if $C_i \geq 10(n-k) C'_i$, then for all $\tau \in (T'_{i-1}, T'_i)$ we have
\begin{equation} \label{eq_dttdUor}
 \partial_\tau \td\UU_i(\tau) \leq -\tfrac1{10(n-k)} \td\UU_i(\tau) \qquad \text{if} \qquad \td\UU_i (\tau) \geq  \tfrac12 C_i\| U^{++}(\tau) \|^{\lceil i/2 \rceil +1}. 
\end{equation}
So Assertion~\ref{Prop_Tisosc_a} holds once we can show that for $\tau \in (T'_{i-1}, T'_i)$ we have
\begin{equation}  \label{eq_tdUUC_want}
\td\UU_i (\tau) \geq  \tfrac12 C_i\| U^{++}(\tau) \|^{\lceil i/2 \rceil +1} 
\end{equation}

If $T'_i = T_1$, then this bound holds for $\tau$ close to $T'_i$ due to the minimal choice of $T'_i$ and if $T'_i < T_1$, then the same is true due to \eqref{eq_tdUUUpp_equality}.
So we can choose a minimal $\tau^*\in [T'_{i-1}, T'_i)$ such that \eqref{eq_tdUUC_want} is true for all $\tau \in [\tau^*, T'_i)$.
Suppose by contradiction that $\tau^* > T'_{i-1}$, so 
\begin{equation} \label{eq_attimetaustar} 
\td\UU_i (\tau^*) =  \tfrac12 C_i\| U^{++}(\tau^*) \|^{\lceil i/2 \rceil +1}, \qquad \partial_\tau |_{\tau = \tau^*} \td\UU_i (\tau) \geq  \tfrac12 C_i \,  \partial_\tau |_{\tau = \tau^*}  \| U^{++}(\tau) \|^{\lceil i/2 \rceil +1}. 
\end{equation}
We will now control the derivative of $\| U^{++} \|$ using \eqref{eq_U^++} from Lemma~\ref{lem:all_decay_modes}.
Fix some $\xi > 0$, which we will determine later, and assume that $R^\# \geq \underline R^\#(\xi)$ and $\eta \leq \ov\eta (\la, J, \xi)$ according to that lemma.
Again, with the help of \eqref{eq_UpbyUpptdU1} or \eqref{eq_UppJp1bound} we obtain the following bound for $\eta \leq \ov\eta(\la,J,\xi)$
\begin{align*}
 \partial_\tau |_{\tau = \tau^*} \| U^{++}(\tau) \| &\geq - \xi \| U^{++}(\tau^*) \| - C(\la, J) \td\UU_1^2(\tau^*)  - C(\la,J) \| U^{+} (\tau^*)\|^{J+1} - \xi \td\UU_1(\tau^*) \\
&\geq - \xi \| U^{++}(\tau^*) \| - \big( C(\la,J) \eta + \xi \big) \td\UU_1(\tau^*)  - C(\la, J) \eta \Vert U^{++} (\tau^*)\Vert \\
&\geq - 2\xi \| U^{++}(\tau^*) \| - 2\xi  \td\UU_1(\tau^*).
\end{align*}
If $i > 1$, then we can apply \eqref{eq_tdUUppfortauTi} at time $\tau^*$ and obtain that $\td\UU_1(\tau^*) \leq C_1 \Vert U^{++}(\tau^*) \Vert^2 \leq CC_1\eta  \Vert U^{++}(\tau^*) \Vert$.
If $i = 1$, then the first identity in \eqref{eq_attimetaustar} implies $\td\UU_1(\tau^*) = \frac12 C_1 \Vert U^{++}(\tau^*) \Vert^2 \leq CC_1\eta  \Vert U^{++}(\tau^*) \Vert$.
So in both cases if $\eta \leq \ov\eta(\la,J)$, then
\[  \partial_\tau |_{\tau = \tau^*} \| U^{++}(\tau) \| \geq - 3\xi \| U^{++}(\tau^*) \|. \]
Combining this with \eqref{eq_dttdUor} and the second identity in \eqref{eq_attimetaustar} implies
\begin{multline*} -\tfrac1{10(n-k)} \td\UU_i(\tau^*) 
\geq \partial_\tau |_{\tau = \tau^*} \td\UU_i(\tau)
\geq \tfrac12 C_i \, \partial_\tau |_{\tau = \tau^*} \Vert U^{++}(\tau) \Vert^{\lceil i/2 \rceil +1} \\
\geq - \tfrac32 C_i \xi J_0  \Vert U^{++}(\tau)\Vert^{\lceil i/2 \rceil +1} = -3\xi J_0  \td\UU_i(\tau^*). 
\end{multline*}
So if $\xi \leq \ov\xi(J_0)$ (which implies the requirements $R^\# \geq \underline R^\# (J)$ and $\eta \leq \ov\eta(\la, J)$), then we must have $\td\UU_i(\tau^*) = 0$.
But by \eqref{eq_dttdUor} this would imply that $\td\UU_i$ vanished on $[\tau^*, T'_i)$, in contradiction to the minimal choice of $T'_i$.
Therefore $\tau^* = T'_{i-1}$ and \eqref{eq_tdUUC_want} must hold on all of $(T'_{i-1}, T'_i)$, which implies Assertion~\ref{Prop_Tisosc_a} via \eqref{eq_dttdUor}.

Note that the condition $C_i \geq 10(n-k) C'_i (\la, J, C_1, \ldots, C_{i-1})$ can be used to determine $C_i$ successively.
Since we require $\eta \leq \ov\eta(C_i)$, the constant $\eta$ must be chosen in the end according to a bound of the form $\eta \leq\ov\eta(\la, J)$.

For Assertion~\ref{Prop_Tisosc_c} consider the case $T_0 = \inf \td I = - \infty$.
Suppose by contradiction that $T'_i > -\infty$ and choose $i \geq 1$ minimal with this property. 
So the exponential decay from Assertion~\ref{Prop_Tisosc_a} holds on  $(T'_{i-1}, T'_i) = (-\infty,T'_i)$.
Since $\td\UU_i$ is uniformly bounded, this implies $\td\UU_i \equiv 0$ on $(-\infty,T'_i]$ in contradiction to the minimal choice of $T'_i$.
The remaining bounds follow directly from those in  Assertion~\ref{Prop_Tisosc_b} combined with Proposition~\ref{Prop_PDE_ODI_MCF_gauged}; recall again that $U^+_{\osc}  =  U^+_{\osc, <0} + U^+_{\Jac}$.
\end{proof}
\bigskip

\subsection{Evolution of the leading modes} \label{subsec_evol_leading}
As established in Proposition~\ref{Prop_Tisosc}, the leading mode
\[ U^{++} \in \sV^{++} =  \sV_{\rot, 0} \oplus  \sV_{\rot, \frac12} \oplus \sV_{\rot, 1}   \]
eventually dominates the evolution of $U^+$.
In this subsection we analyze the dynamics of this leading mode.
We will see that the time interval $[T'_3, \infty) \cap \td I$ from Proposition~\ref{Prop_Tisosc} can be divided into three successive intervals, on which $U^+_{\rot,0}$, $U^+_{\rot,\frac12}$, and $U^+_{\rot,1}$ dominate, respectively (in this order). 
The first evolves according to a polynomial law of the form $\frac1\tau$, while the latter two grow approximately exponentially at rates $\frac12$ and $1$.

To obtain these estimates, we explicitly calculate the quadratic Taylor approximation $Q_2(U^{++})$ of the nonlinear term applied to the leading modes.
In Sections~\ref{sec_dom_lin} and \ref{sec_dom_quadratic} we will refine the asymptotic description of $U^+_0$ and reuse these calculations.

The next proposition summarizes the main result of this subsection.

\begin{Proposition} \label{Prop_ODE_subintervals}
Consider the setting of Proposition~\ref{Prop_PDE_ODI_MCF}, its addendum, Proposition~\ref{Prop_PDE_ODI_MCF_gauged} and Proposition~\ref{Prop_Tisosc} with the corresponding assumptions on the parameters.
Suppose that $J \geq 2$, $\xi > 0$ and $\eta \leq \ov\eta (\la,J,\xi)$ and let $T'_0=T_0, T'_1, T'_2, \ldots$ be the times obtain from Proposition~\ref{Prop_Tisosc}.
Recall that $U^{++} = U_{\rot, 0}^+ + U_{\rot,\frac12}^+ + U_{\rot, 1}^+$ and write $U_0 := U^+_{\rot, 0}$, $U_{\frac12} = U^+_{\rot, \frac12}$, etc., for convenience.
Write $U_0(\tau) = \sum_{i,j} c_{ij}(\tau) \mathfrak p^{(2)}_{ij}$ for a time-dependent symmetric matrix $(c_{ij}(\tau))$ and the Hermite polynomials from \eqref{eq_Hermite} and denote by $U_{0,\min/\max}(\tau)$ the minimal and maximal spectral values of the matrix $(c_{ij}(\tau))$.

If $\td I = (-\infty, \infty)$, then $U^+, \UU^-, u \equiv 0$, so $(\spt \td\MM')_\tau = \td\MM^{\prime, \reg}_\tau = M_{\cyl}$ and $\td\bY_\tau = 0$ for all $\tau$.
Suppose now that $u \not\equiv 0$.
Then there is a constant $c_0(\la, J, \xi) >0$ and unique times $T'_3 \leq \tau_0 \leq \tau_{\frac12} \leq  \tau_1 \leq  T_1 = \max \td I$ such that the following is true:
\begin{enumerate}[label=(\alph*)]
\item \label{Prop_ODE_subintervals_a} If $\td I$ is not finite, then the following is true:
\begin{itemize}
\item \emph{(Ancient case)} \quad If $T_0 = \inf \td I = -\infty$, then $T'_3 = \tau_0 = -\infty$.
\item \emph{(Immortal case)} \quad If $T_1 = \sup \td I = \infty$, then $\tau_0 = \tau_{\frac12} = \tau_1 = \infty$.
\end{itemize}
\item \label{Prop_ODE_subintervals_b} \emph{ 
(Quadratic mode dominates)}  \quad For all $\tau \in (T'_3, \tau_{\frac12})$ we have $\|U_{\frac12} (\tau) \|, \| U_{1} (\tau)\| \leq c_0 \| U_{0}(\tau) \| $ and
\begin{equation} \label{eq_U_evol}
  \big\| \partial_\tau U_{0}(\tau) \big\| \leq \xi \|U_{0}(\tau) \|. 
\end{equation}
Moreover, we have the following behaviors depending on the dominance of $U_{0,\min}$ and $U_{0,\max}$:
\begin{itemize}
\item If $\tau \in (T'_3, \tau_0)$, then $|U_{0,\min}|(\tau) \leq U_{0,\max}(\tau)$ and
\begin{equation} \label{eq_cmax_evol}
 \big| \partial_\tau U_{0,\max}(\tau) + \sqrt{2} U_{0,\max}^2(\tau) \big| \leq \xi U_{0,\max}^2(\tau). 
\end{equation}
\item If $\tau \in (\tau_0, \tau_{\frac12})$, then $U_{0,\min}(\tau) \leq -|U_{0,\max}|(\tau)$ and
\begin{equation} \label{eq_cmin_evol}
 \big| \partial_\tau U_{0,\min}(\tau) + \sqrt{2} U_{0,\min}^2(\tau) \big| \leq \xi U_{0,\min}^2(\tau). 
\end{equation}
\end{itemize}
\item \label{Prop_ODE_subintervals_c} \emph{(Linear mode dominates)} \quad If $\tau \in (\tau_{\frac12}, \tau_1)$, then then $c_0 \|U_{0} (\tau)\|, \| U_{1} (\tau)\| \leq \| U_{\frac12} (\tau) \| $ and
\begin{equation} \label{eq_U12_evol}
 \big\| \partial_\tau U_{\frac12}(\tau) - \tfrac12 U_{\frac12}(\tau) \big\| \leq \xi \|U_{\frac12} (\tau)\|. 
\end{equation}
\item \label{Prop_ODE_subintervals_d} \emph{(Constant mode dominates)} \quad If $\tau \in (\tau_1, T_1)$, then $ c_0 \|U_{0} (\tau) \|, \| U_{\frac12} (\tau) \| \leq \| U_{1} (\tau) \|$ and
\begin{equation} \label{eq_U1_evol}
 \big\| \partial_\tau U_{1}(\tau) -  U_{1}(\tau) \big\| \leq \xi \|U_{1}(\tau) \|. 
\end{equation}
\end{enumerate}
\end{Proposition}
\medskip

The proof of Proposition~\ref{Prop_ODE_subintervals} relies on an explicit calculation of the second Taylor approximation $Q_2(U^{++})$ of the non-linear term in the rescaled mean curvature flow equation \eqref{eq_main_ev_eq} (we will also calculate $Q_3(U^{++})$ for future use).
To do so, we first derive this evolution equation in the rotationally symmetric codimension $1$ case and for $\td\bY \equiv 0$.

\begin{Lemma} \label{Lem_MCF_u_explicit}
If $n'=0$ and $u(\bx,\by) = u(\bx)$ is invariant under rotations on the $\IS^{n-k}$-factor, then the evolution equation \eqref{eq_evolution_rMCF_rot} of $u$ under rescaled mean curvature flow becomes
\begin{equation} \label{eq_evolution_rMCF_rot}
 \partial_\tau u = \triangle_f u + u- \frac{ \frac12 u^2}{1 + u} - \frac{\nabla^2 u(\nabla u, \nabla u)}{(2(n-k))^{-1}+|\nabla u|^2} . 
\end{equation}
If $u(\bx,\by)$ is not rotationally invariant, then we have for small perturbations
\begin{equation} \label{eq_evolution_rMCF_non_rot}
\partial_\tau u = \triangle_f u + \triangle_{\IS^{n-k}} u + u- \tfrac12 u^2 -  u \cdot \triangle_{\IS^{n-k}} u - |\nabla_{\by} u|^2   + O\big( (|u| +|\nabla u|+|\nabla^2 u|)^3 \big).
\end{equation}
Here $\triangle_f$ denotes the $f$-Laplacian on $\IR^k$, $\triangle_{\IS^{n-k}}$ the spherical Laplacian and $\nabla_{\by} u$ the projection of $\nabla u$ onto the $\bO^k \times \IS^{n-k}$-factor.
\end{Lemma}

The identity \eqref{eq_evolution_rMCF_non_rot} will not be used in this paper, but in \cite{Bamler_Lai_MCF2}.

\begin{proof}
Set $r := \sqrt{2(n-k)}$ and recall that 
$$\Gamma_{\cyl}(u) = \big\{ \big(\bx, (1+u(\bx))y \big) \;\; : \;\; \bx \in \IR^{k}, \by \in \IR^{n-k+1}, |\by| = r \big\}.$$
We will now express each term in the rescaled mean curvature flow equation $\partial_\tau \mathbf z = \mathbf H + \frac12 \mathbf z^\perp$ in terms of $u$; here $\bz = (\bx, \by) \in \IR^{n+1}$ denotes a generic point.
Fix a point $(\bx,\by) \in M_{\cyl}$ and assume without loss of generality $\by = r \mathbf e_{n+1}$ and $\partial_2 u (\bx,\by) = \ldots = \partial_{k} u(\bx,\by) = \partial_{k+1} u(\bx,\by) = \ldots = \partial_{n-1} u(\bx,\by) = 0$; set $u := 1+u(\bx,\by)$ and $u_x := \partial_1 u(\bx,\by)$, $u_y := \partial_{n} u(\bx,\by)$, so $\nabla u = u_x \mathbf e_1 + u_y \mathbf e_{n}$.
Then the tangent space of $\Gamma(u)$ at $(\bx,(1+u)\by)$ is spanned by $\mathbf e_1 + u_x r \, \mathbf e_{n+1}, \mathbf e_2, \ldots, \mathbf e_{n-1}, \mathbf e_{n}+ u_y r \, \mathbf e_{n+1}$ and the unit normal vector is 
\[ \mathbf n := \frac1{\sqrt{1 + u_x^2 r^2 + u_y^2 r^2}} \big( - u_x r \, \mathbf e_1 - u_y r \, \mathbf e_{n} + \mathbf e_{n+1} \big). \]
It follows that
\begin{equation} \label{eq_dtx_n}
 \partial_\tau \mathbf z \cdot \mathbf n =\frac{\partial_\tau u( \by \cdot \mathbf e_{n+1})}{\sqrt{1+u_x^2 r^2 + u_y^2 r^2}}  =  \frac{\partial_\tau u}{\sqrt{r^{-2}+|\nabla u|^2}}   
\end{equation}
and
\begin{equation} \label{eq_xperp_n}
 \mathbf z^\perp \cdot \mathbf n
 = (\bx, (1+u) \by) \cdot \mathbf n 
= \frac{ \bx \cdot (-u_x r\be_1)  +  (1+u)( \by \cdot \be_{n+1})}{\sqrt{1 + u_x^2 r^2 + u_y^2 r^2}} 
= \frac{- \bx \cdot \nabla u  + 1 + u}{\sqrt{r^{-2} + |\nabla u|^2}}.  
\end{equation}
Lastly, we need to compute $\mathbf H \cdot \mathbf n$.
To do this, let $(v_1, \ldots, v_n) \in \IR^n$ and consider a  curve
\[ s \mapsto \bigg(\bx+ \sum_{i=1}^k s v_i \be_i, \Big(1+u \Big(\bx+ \sum_{i=1}^k s v_i \be_i, \gamma_{v_{k+1}, \ldots, v_{n}} (s) \Big) \Big) \gamma_{v_{k+1}, \ldots, v_{n}} (s) \bigg), \]
where $\gamma_i : \IR \to \IS^{n-k}$ is a unit-speed geodesic with $\gamma_{v_{k+1}, \ldots, v_n} (0)=\by = r e_{n+1}$ and velocity $\gamma'_{v_{k+1}, \ldots, v_n}(0) =  \sum_{i=k+1}^n v_i \be_i$.
At $s = 0$ its velocity is contained in the tangent space at $(\bx, (1+u)\by)$ and the square of its norm is
\begin{equation} \label{eq_quad_form1}
 v_1^2 + \ldots + v_k^2 + (1+u)^2  v_{k+1}^2 + \ldots + (1+u)^2  v_n^2 + u_x^2 r^2 v_1^2  +  u_y^2 r^2 v_{n}^2 + 2 u_x u_y  r^2  v_1 v_{n} 
\end{equation}
and its second derivative is
\[ \bigg( \bO, \sum_{i,j=1}^n \partial^2_{ij} u \, v_i v_j  \, \by + 2 \sum_{i=1}^n \sum_{j=k+1}^n \partial_i u \, v_i v_j   \, \be_j - (1+u) r^{-2} \sum_{i=k+1}^n v_i^2 \, \by \bigg). \]
The inner product of this with $\mathbf n$ is equal to 
\begin{equation} \label{eq_quad_form2}
 \frac1{\sqrt{1+ |\nabla u|^2r^2}} \bigg( \sum_{i,j=1}^n \partial^2_{ij} u \, r \,  v_i v_j -  2 u_x u_y r \, v_1 v_{n}  -  2    u_y^2 r \, v_{n}^2 - (1+u) r^{-1} \sum_{i=k+1}^n v_i^2 \bigg).  
\end{equation}
Now $\mathbf H \cdot \mathbf n$ is equal to the trace of the quotient of the matrices representing the quadratic forms \eqref{eq_quad_form2} and \eqref{eq_quad_form1}.
If $u_y = 0$, then the second matrix is diagonal and 
\begin{multline} \label{eq_H_n}
 \mathbf H \cdot \mathbf n = \frac{r}{\sqrt{1 + |\nabla u|^2 r^2}} \bigg( \frac{\partial^2_{11} u}{1 + u_x^2 r^2 } + \sum_{i=2}^{k} \partial^2_{ii} u + \sum_{i=k+1}^n \frac{\partial^2_{ii} u}{(1+u)^2} - \frac{(n-k)r^{-2}}{1+u} \bigg) \\
  = \frac1{\sqrt{r^{-2} + |\nabla u|^2}}  \bigg( \triangle_{\IR^k} u - \frac{\nabla^2 u (\nabla u, \nabla u)}{r^{-2} + |\nabla u|^2} - \frac{(n-k)r^{-2}}{1+u}  + \sum_{i=k+1}^n \frac{\partial^2_{ii} u}{(1+u)^2} \bigg). 
\end{multline}
The identity \eqref{eq_evolution_rMCF_rot}, in the rotationally symmetric case, follows by combining \eqref{eq_dtx_n}, \eqref{eq_xperp_n}, \eqref{eq_H_n}.

To see identity \eqref{eq_evolution_rMCF_non_rot} note that \eqref{eq_dtx_n}, \eqref{eq_xperp_n} still hold in the non-rotationally symmetric case.
In order to generalize \eqref{eq_H_n} denote the matrices representing the quadratic forms  \eqref{eq_quad_form1} and \eqref{eq_quad_form2} by $D_1 + E_1$ and $D_2+ E_2$, respectively, where $D_1, D_2$ denote the matrices from before (obtained by setting $u_y = 0$) and $E_1, E_2$ denote the remaining terms involving $u_y$.
Note $D_1$ is diagonal, $D_2$ is block-diagonal and the entries of $E_1, E_2$ vanish to order $O(u_x^2+u_y^2)$.
Therefore
\begin{multline*}
 (D_1+E_1)^{-1} = D_1^{-1} (\mathsf I + E_1 D_1^{-1} )^{-1}
= D_1^{-1} \big(\mathsf I - E_1 D_1^{-1} + O(\| E_1 D_1^{-1} \|) \big) \\
= D_1^{-1} - D_1^{-1} E_1 D_1^{-1} + O(\| E_1 \|^2 )
= D_1^{-1} - D_1^{-1} E_1 D_1^{-1} + O(u_x^4+u_y^4), 
\end{multline*}
so
\begin{multline*}
 \mathbf H \cdot \mathbf n
= \tr  (D_2 + E_2)(D_1^{-1} +E_1)^{-1} 
= \tr (D_2 + E_2) (D_1^{-1} - D_1^{-1} E_1 D_1^{-1}) + O(u_x^4+u_y^4) \\
= \tr D_2 D_1^{-1} + \tr E_2 D_1^{-1} - \tr D_2 D_1^{-1} E_1 D_1^{-1} +  O(u_x^4+u_y^4).
\end{multline*}
The term $\tr D_2 D_1^{-1}$ is the same as the right-hand side of \eqref{eq_H_n} and it is not hard to verify that
\begin{align*}
 \tr E_2 D_1^{-1} &= \frac{1}{\sqrt{1+|\nabla u|^2 r^2}} \, \frac{-2 u_y^2r}{(1+u)^2} = \frac{1}{\sqrt{r^{-2}+|\nabla u|^2}} \big( {- 2u_y^2 \big)} \big) , \\
\tr D_2 D_1^{-1} E_1 D_1^{-1} &= \frac{1}{\sqrt{1+|\nabla u|^2 r^2}} \, \frac{- (1+u) r^{-1}u_y^2 r^2}{(1+u)^4}  = \frac{1}{\sqrt{r^{-2}+|\nabla u|^2}} \, \big( {-u_y^2  + O\big( (|u| + |u_y|)^3\big)} \big).
\end{align*}
Putting everything together gives us
\begin{multline*}
 \mathbf H \cdot \mathbf n =  \frac1{\sqrt{r^{-2} + |\nabla u|^2}} \Big( \triangle_{\IR^k} u - \tfrac12 + \tfrac12 u - \tfrac12 u^2 + \triangle_{\IS^{n-k}} u - 2 u \,  \triangle_{\IS^{n-k}} u - u_y^2  \\
 + O\big( (|u| + |\nabla u| + |\nabla^2 u|)^3\big) \Big),
 \end{multline*}
which combined with \eqref{eq_dtx_n}, \eqref{eq_xperp_n} proves \eqref{eq_evolution_rMCF_non_rot}.
\end{proof}
\medskip

We can now calculate the second Taylor approximation $Q_2(U^{++})$ of the non-linear term in \eqref{eq_evolution_rMCF_rot}.
We will use Hermite polynomials $\fp^{(0)}, \fp^{(1)}_i$ and $\fp^{(2)}_{ij}$ from \eqref{eq_Hermite}.

\begin{Lemma} \label{Lem_Q2}
If $\la < 0$, then we have for any $U \in \sV_{\rot,> \la}$
\begin{equation} \label{eq_Q2}
 Q_2(U) = Q_2(U,0) = - \tfrac12 U^2.  
\end{equation}
Moreover, if $U = a \mathfrak p^{(0)} + \sum_i b_i \mathfrak p^{(1)}_i + \sum_{i,j} c_{ij} \mathfrak p^{(2)}_{ij}$, then
\[ \PP_{\sV_{\geq 0}} Q_2(U) = \ov a \fp^{(0)} + \sum_i \ov b_i \fp^{(1)}_i + \sum_{i,j} \ov c_{ij} \fp^{(2)}_{ij}, \]
 where
\begin{align}
 \ov a &= - \tfrac12 a^2 - \tfrac12 \sum_i  b_i^2 - \tfrac12 \sum_{i,j} c_{ij}^2  \label{eq_Q2p0} \\
 \ov b_i &= - ab_i- \sqrt{2} \sum_l c_{il} b_l  \\
 \ov c_{ij} &= - \sqrt{2} \sum_l c_{il} c_{lj} - \tfrac{1}{\sqrt{2}} b_i b_j - a c_{ij} \label{eq_Q2p2}
\end{align}
\end{Lemma}

\begin{proof}
In the following, we will use the \emph{standard} inner product $$\langle u_1, u_2 \rangle = \int_{\IR^k} u_1 u_2 (4\pi)^{-k/2} e^{-f} dx^1\cdots dx^k$$ on $\IR^k$ equipped with the standard Gaussian weight.
This differs from the inner product $\langle \cdot, \cdot \rangle_{L^2_f}$ by a dimensional factor, especially because the latter involves integration over $\IS^{n-k}$.
The Hermite polynomials $\fp^{(0)}, \fp^{(1)}_i$ and $\fp^{(2)}_{ij}$ are orthogonal with respect to $\langle \cdot, \cdot \rangle$ and $\fp^{(0)}, \fp^{(1)}_i$ and $\fp^{(2)}_{ii}$ have norm 1; the norm of $\fp^{(2)}_{ij}$, for $i \neq j$ equals $\frac12$.

The identity \eqref{eq_Q2} follows from \eqref{eq_evolution_rMCF_rot}.
To see the remaining identities, we may assume without loss of generality that the matrix $(c_{ij})$ is diagonal and compute, using the Einstein summation convention,
\[ U^2 = a^2 \mathfrak p^{(0)} + b_i b_j  \mathfrak p_i^{(1)} \mathfrak p_j^{(1)} + c_{ii} c_{jj} \mathfrak p^{(2)}_{ii}  \mathfrak p^{(2)}_{jj} + 2 ab_i \mathfrak p_i^{(1)} + 2 ac_{ii}  \mathfrak p_{ii}^{(2)} + 2 b_l c_{ii} \mathfrak p_l^{(1)} \mathfrak  p_{ii}^{(2)} \]
The lemma now follows from the identities
\begin{align*}
 \mathfrak p_i^{(1)} \mathfrak p_j^{(1)} &= \sqrt{2} \mathfrak  p^{(2)}_{ij} + \delta_{ij} \mathfrak p^{(0)}, \displaybreak[1] \\ 
\langle \mathfrak p_{ii}^{(2)} \mathfrak p^{(2)}_{jj}, \mathfrak p^{(0)} \rangle 
&=  \langle \mathfrak p_{ii}^{(2)}, \mathfrak p^{(2)}_{jj} \rangle = \delta_{ij} , \displaybreak[1] \\
\langle \mathfrak p_{ii}^{(2)} \mathfrak p^{(2)}_{jj}, \mathfrak p^{(1)}_l \rangle  
&= 0, \displaybreak[1] \\
\langle \mathfrak p_{ii}^{(2)} \mathfrak p^{(2)}_{jj}, \mathfrak p^{(2)}_{lm} \rangle &= 0  \qquad \text{unless $i=j=l=m$,} \displaybreak[1] \\
\langle \mathfrak p_{ii}^{(2)} \mathfrak p^{(2)}_{ii}, \mathfrak p^{(2)}_{ii} \rangle &=  {\textstyle \frac1{\sqrt{4\pi}}\int_{-\infty}^\infty  \frac1{16  \sqrt{2}} (x^2-2)^3 e^{-x^2/4}dx = 2 \sqrt{2}},  \displaybreak[1] \\
\langle \mathfrak p_l^{(1)} \mathfrak p_{ii}^{(2)} , \mathfrak p^{(0)} \rangle
= \langle \mathfrak p_l^{(1)} \mathfrak p_{ii}^{(2)} , \mathfrak p^{(2)}_{jj} \rangle
 &=  0, \displaybreak[1] \\
 \langle \mathfrak p_l^{(1)} \mathfrak p_{ii}^{(2)} , \mathfrak p_m^{(1)} \rangle
 &=  \langle  \mathfrak p_{ii}^{(2)} , \mathfrak p_l^{(1)} \mathfrak p_m^{(1)} \rangle
 = \langle  \mathfrak p_{ii}^{(2)} , \sqrt{2} \mathfrak p_{lm}^{(2)} + \delta_{ij} \mathfrak p^{(0)} \rangle  = \delta_{il} \delta_{im} \sqrt{2}. 
\end{align*}
This concludes the proof.
\end{proof}
\medskip

We also record the following refinement of Lemma~\ref{Lem_Q2}, which will be required to further study the evolution of the quadratic mode in Section~\ref{sec_dom_quadratic}.
It will not be used in this section.

\begin{Lemma}
If $U = \sum_{i} c_{ii} \mathfrak p^{(2)}_{ii}$ is in diagonal form, then
\begin{align}
 \PP_{\sV_1} \big( Q_2(U) \big) &= -\tfrac12 \sum_j c_{jj}^2  \mathfrak p^{(0)} \label{eq_Q2_V1} \\
 \PP_{\sV_{\frac12}} \big( Q_2(U) \big) &= \PP_{\sV_{-\frac12}} \big( Q_2(U) \big) = 0 \label{eq_Q2_V12} \\
 \PP_{\sV_{-1}} \big( Q_2(U) \big) &=  -\tfrac12 \sum_{i \neq j} c_{ii} c_{jj} \mathfrak p_{ii}^{(2)} \mathfrak p_{jj}^{(2)} - \tfrac{\sqrt{6}}2 \sum_i c_{ii}^2 \mathfrak p^{(4)}_{iiii},   \label{eq_Q2_Vm1}
\end{align}
where $\mathfrak p^{(4)}_{iiii} (\bx) = \frac1{8 \sqrt{6}} (\bx_i^4-12\bx_i^2+12)$ denotes the 4th pure Hermite polynomial.

Moreover, 
If $\la < -1$, then we have for any $U \in \sV_{\rot, >\la}$
\begin{equation} \label{eq_3_taylor}
  Q_3(U) = Q_3(U, 0) = - \tfrac12 U^2 + \tfrac12 U^3 - 2(n-k) \, \nabla^2 u (\nabla u, \nabla u). 
\end{equation}
If for some coefficients $a, v_{ij}, w_i \in \IR$ with $v_{ii} = 0$ we have
\[ U = a \mathfrak p^{(0)} + \sum_{i} c_{ii} \mathfrak p^{(2)}_{ii} + \sum_{i,j} v_{ij} \mathfrak p_{ii}^{(2)} \mathfrak p_{jj}^{(2)} + \sum_i w_i \mathfrak p_{iiii}^{(4)}, \]
then, assuming that $|a|, |c_{ii}|, |v_{ij}|, |w_i| < 1$, we have for some  dimensional and universal constants $C_1(n-k), C_2 \in \IR$
\begin{align}
 \PP_{\sV_0} \big( Q_3(U) \big) &= \sum_i \Big( - \sqrt{2} c_{ii}^2 - ac_{ii} + \tfrac32 \sum_j c_{jj}^2 c_{ii} + C_1(n-k) c_{ii}^3    - 2 \sum_j c_{jj} v_{ji}  + C_2 c_{ii} w_i \Big) \mathfrak p_{ii}^{(2)} \notag \\
 &\qquad  + O\big(a^2+ |\vec v|^2 + |\vec w|^2 \big) + O\big(  (|a|+ |\vec v| + |\vec w|) |\vec c|^2 \big).  \label{eq_Q3_V0}
\end{align}
\end{Lemma}

\begin{proof}
Identity \eqref{eq_3_taylor} follows again from \eqref{eq_evolution_rMCF_rot}.

For simplicity, we will write $c_{i} = c_{ii}$.
Note that in each scenario $U$ is invariant under the coordinate transformation $\bx_j \mapsto -\bx_j$ for some fixed $j$.
Since $Q_2$ and $Q_3$ are equivariant under this transformation, we obtain that the images $Q_2(U)$ and $Q_3(U)$ are invariant as well, so they must be polynomials on $\IR^k$ that are even in every coordinate.

The identity \eqref{eq_Q2_V1} follows directly from \eqref{eq_Q2p0} and the identities \eqref{eq_Q2_V12} follow from the fact that $Q_2(U)$ is even in each coordinate, while $\sV_{\frac12}$ and $\sV_{-\frac12}$ are spanned by Hermite polynomials that are odd in at least one coordinate.

To see \eqref{eq_Q2_Vm1}, note that the subspace of $\sV_{-1}$ consisting of polynomials that are even in each coordinate is spanned by the degree 4 Hermite polynomials $\mathfrak p_{ii}^{(2)} \mathfrak p_{jj}^{(2)}$, for $i \neq j$ and $\mathfrak p^{(4)}_{iiii}$.
Since the image $Q_2(U)$ is a polynomial of degree 4, it suffices to consider its leading terms, which are
\[ -\frac1{16} \sum_{i, j} c_i c_j x_i^2 x_j^2 
= -\frac1{16} \sum_{i \neq j} c_i c_j x_i^2 x_j^2  - \frac1{16} \sum_{i} c_i^2 x_i^4. \]
Since the leading terms of $\mathfrak p_{ii}^{(2)} \mathfrak p_{jj}^{(2)}$ and $\mathfrak p^{(4)}_{iiii}$ are $\frac18 x_i^2 x_j^2$ and $\frac1{8 \sqrt{6}} x_i^4$, respectively, the identity \eqref{eq_Q2_Vm1} follows.

Let us now verify identity \eqref{eq_Q3_V0}.
Since $\PP_{\sV_0} (Q_3(U))$ is an even polynomial in each coordinate function, it must be a linear combination of the Hermite polynomials $\mathfrak p_{ii}^{(2)}$.
Let us first consider the case in which $v_{ij} = w_i = 0$.
By direct computation, and using the same inner product as in the proof of Lemma~\ref{Lem_Q2},
\[ \langle \tfrac12 U^3, \mathfrak p_{ii}^{(2)} \rangle
= \tfrac12 \sum_{j, l, m} c_j c_l c_m \langle  \mathfrak p_{jj}^{(2)} \mathfrak p_{ll}^{(2)} \mathfrak p_{mm}^{(2)}, \mathfrak p_{ii}^{(2)} \rangle
+ O(|a| \, |\vec c|^2 + |a|^2 \, |\vec c|) \]
The inner product is only non-zero if $i$ is equal to one of the numbers $j, l, m$ and the other two numbers are equal.
Therefore we have for some universal constant $C \in \IR$
\[ \langle \tfrac12 U^3, \mathfrak p_{ii}^{(2)} \rangle
= \tfrac32 \sum_{j} c_j^2 c_i \langle  \mathfrak p_{ii}^{(2)} \mathfrak p_{jj}^{(2)} \mathfrak p_{jj}^{(2)}, \mathfrak p_{ii}^{(2)} \rangle -  c_i^3 \langle  \mathfrak p_{ii}^{(2)} \mathfrak p_{ii}^{(2)} \mathfrak p_{ii}^{(2)}, \mathfrak p_{ii}^{(2)} \rangle 
=  \tfrac32 \sum_{j} c_j^2 c_i^2 - C c_i^4. \]
We can also verify easily that
\[ \langle \nabla^2 U (\nabla U, \nabla U) , \mathfrak p^{(2)}_{ii} \rangle = c_i^3 \]
Combining these calculations with \eqref{eq_3_taylor} and Lemma~\ref{Lem_Q2} implies \eqref{eq_Q3_V0} if $v_{ij} = w_i = 0$.

Consider now the general case in which $v_{ij}, w_i$ don't vanish altogether.
Recall that $\PP_{\sV_0} (Q_3(U))$ is a polynomial of degree 3 in $a, c_i, v_{ij}, w_i$.
It suffices to study only those monomials that contain at least one $v_{ij}$ or $w_i$ factor.
Due to the first $O(\cdots)$-term on the right-hand side of \eqref{eq_Q3_V0}, we can  ignore all monomials containing at least two such factors, or monomials containing one such factor and the factor $a$.
This only leaves monomials containing exactly one factor $v_{ij}$ or $w_i$ and one or two $c_i$-factors.
Due to the second $O(\cdot)$, we can further ignore those monomials containing two $c_i$-factors reducing our analysis to only studying monomials of the form $c_l v_{ij}$ and $c_i w_j$.
These monomials already appear in the projection $\PP_{\sV_0} (Q_2(U))$ of the \emph{second} Taylor approximation.
We compute that (using again Lemma~\ref{Lem_Q2})
\begin{align*}
 \langle Q_2(U), \mathfrak p_{ii}^{(2)} \rangle
&= - \tfrac12 \langle U^2, \mathfrak p_{ii}^{(2)} \rangle \\
&= - \sqrt{2} c_i^2 - \Big\langle \Big( \sum_m c_m \mathfrak p_{mm}^{(2)} \Big) \Big( \sum_{j,l} v_{jl} \mathfrak p_{jj}^{(2)} \mathfrak p_{ll}^{(2)} + \sum_j w_j \mathfrak p_{jjjj}^{(4)} \Big) , \mathfrak p_{ii}^{(2)} \Big\rangle \\
&\qquad + O\Big( |\vec v|^2 + |\vec w|^2 + |a| (|a| + |\vec v| + |\vec w|) \Big).
\end{align*}
As before, we find that the inner product on the right-hand side is equal to
\begin{multline*}
 \sum_{j,l,m} c_m v_{jl}  \langle \mathfrak p_{mm}^{(2)} \mathfrak p_{jj}^{(2)} \mathfrak p_{ll}^{(2)}, \mathfrak p^{(2)}_{ii} \rangle 
+ \sum_{m,j} c_m w_j \langle \mathfrak p_{mm}^{(2)} \mathfrak p_{jjjj}^{(4)}, \mathfrak p_{ii}^{(2)} \rangle \\
= - 2 \sum_j c_j v_{ji}  - \sum_j c_i v_{jj} + 2 c_i v_{ii} + c_i w_i \langle \mathfrak p_{ii}^{(2)} \mathfrak p_{iiii}^{(4)}, \mathfrak p_{ii}^{(2)} \rangle .
\end{multline*}
The terms involving $v_{jj}$ and $v_{ii}$ vanish.
The remaining terms are the additional terms on the right-hand side of \eqref{eq_Q3_V0}.
\end{proof}
\medskip

\begin{proof}[Proof of Proposition~\ref{Prop_ODE_subintervals}.]
Without loss of generality, we may assume that $\xi < \frac1{10}$.
By Proposition~\ref{Prop_PDE_ODI_MCF_gauged} and Proposition~\ref{Prop_Tisosc}\ref{Prop_Tisosc_b} we have for $\tau \geq T'_3$
\begin{align}
 \Vert U^+ - U^{++} \Vert &\leq  \Vert U^+_{\Jac} \Vert  + \Vert U^+_{\osc, < 0} \Vert + \Vert U^+_{\rot, \leq -\frac12}  \Vert \leq C(\la, J) \Vert U^{++} \Vert^2, \label{eq_UpUpp} \\
  \UU^- + \Vert \td\bY \Vert &\leq  C(\la, J) \Vert U^{++} \Vert^3. \label{eq_UmYpp}
\end{align}
So by Proposition~\ref{Prop_PDE_ODI_MCF}\ref{Prop_PDE_ODI_MCF_b}
\[ \big\| \partial_\tau U^+ - L U^+ - Q^+_J(U^+,\bY) \big\| \leq C(\la, J) \Vert U^+ \Vert^{J+1} +  C(\la) \UU^- \leq  C(\la, J) \Vert U^{++} \Vert^{3}.   \]
For the same reasons we have
\begin{multline*}
\big\| Q^+_J (U^+, \bY ) - Q^+_2(U^{++},0) \big\| 
\leq \big\| Q^+_J (U^+, \bY ) - Q^+_J(U^{++},0) \big\| + \big\| Q^+_J (U^{++}, 0 ) - Q^+_2(U^{++},0) \big\| \\
\leq C(\la, J ) \Vert U^{++} \Vert^3
\end{multline*}
It follows that for $\tau \geq T'_3$
\begin{equation} \label{eq_evol_simplified}
 \big\| \partial_\tau U^+ - L U^+ - Q^+_2(U^{++},0) \big\|  \leq  C(\la, J) \Vert U^{++} \Vert^{3}.   
\end{equation}
In the next claim we use this bound to derive the desired evolution inequalities for $U_{0}, U_{\frac12}$ and $U_1$.

\begin{Claim} \label{Cl_U0U12U1}
There is a constant $c_0(\la, J,\xi) > 0$  such that for $\eta \leq \ov\eta(\la, J, \xi, c_0)$ the following is true:
\begin{enumerate}[label=(\alph*)]
\item \label{Cl_U0U12U1_a} Whenever $\|U_{\frac12} \|, \| U_{1} \|  \leq c_0 \| U_{0} \| $,  the bound \eqref{eq_U_evol} holds. Moreover, we have \eqref{eq_cmax_evol} or \eqref{eq_cmin_evol}, depending on whether $|U_{0,\max}|$ or $|U_{0,\min}|$ is larger.
\item \label{Cl_U0U12U1_b} Whenever $c_0 \| U_{0} \|,  \| U_{1} \|  \leq  \|U_{\frac12} \|$,  the bound \eqref{eq_U12_evol} holds.
\item \label{Cl_U0U12U1_c} Whenever $c_0 \| U_{0} \|,  \| U_{\frac12} \|  \leq  \|U_{1} \|$,  the bound \eqref{eq_U1_evol} holds.
\end{enumerate}
\end{Claim}

\begin{proof}
Consider first the setting of Assertion~\ref{Cl_U0U12U1_a}.
If we view $U_0$ as a time-dependent symmetric matrix as in the statement of the proposition, then by \eqref{eq_evol_simplified} and Lemma~\ref{Lem_Q2}
\begin{align*}
| \partial_\tau U_0  + \sqrt 2 U_0^2 |
&\leq  \big\| - \sqrt 2 U_0^2 + \PP_{\sV_0} Q_2^+(U^{++},0)  \big\| + C(\la,J) \| U^{++} \|^3  \\
&\leq 
C \| U_{\frac12} \|^2 + C \| U_{1} \| \cdot \| U_0 \| + C(\la, J) \| U_0 \|^3 \\
&\leq 
C c_0^2 \| U_{0} \|^2 + C c_0  \| U_{0} \|^2 + C(\la, J) \| U_0 \|^3 \\
&\leq \big(Cc_0 + C(\la, J) \eta \big) \|U_0\|^2.
\end{align*}
If we choose $c_0 \leq \ov c_0(\xi)$ and then $\eta \leq \ov\eta(  \la, J,  \xi, c_0)$, then the right-hand side is $\leq \xi \max\{ U_{0,\min}^2, U_{0,\max}^2 \}$.
This shows \eqref{eq_cmax_evol} or \eqref{eq_cmin_evol}.
The bound \eqref{eq_U_evol} follows directly if $\eta \leq \ov\eta$.

Next, consider the setting of Assertion~\ref{Cl_U0U12U1_b} or \ref{Cl_U0U12U1_c} and let $U_{i}$, $i \in \{ \frac12, 1 \}$, be the dominant mode and let $i' \in  \{ \frac12 ,1 \}$ be the other index.
Again, by \eqref{eq_evol_simplified}
\begin{multline*}
 \Vert \partial_\tau U_i - i U \Vert
\leq \Vert \PP_{\sV_{i}} Q^+_2(U^{++},0) \Vert + C(\la, J) \Vert U^{++} \Vert^3
\leq C(\la, J) \Vert U^{++} \Vert^2 \\
\leq C(\la, J) \eta \big( \Vert U_0 \Vert + \Vert U_{i'} \Vert + \Vert U_i \Vert \big)
\leq C(\la, J) c_0^{-1} \eta \Vert U_i \Vert. 
\end{multline*}
This shows \eqref{eq_U12_evol} or \eqref{eq_U1_evol} again if $\eta \leq \ov\eta(  \la, J,  \xi, c_0)$.
\end{proof}

Combining the bounds from Claim~\ref{Cl_U0U12U1} implies that in the barrier sense
\[ \Big| \partial_\tau \max \big\{ c_0 \Vert U_0 \Vert, \Vert U_{\frac12} \Vert, \Vert U_1 \Vert \big\} \Big| \leq (1+\xi)  \max \big\{ c_0 \Vert U_0 \Vert, \Vert U_{\frac12} \Vert, \Vert U_1 \Vert \big\} \]
So either 
\begin{equation} \label{eq_max_not_zero}
 \max \big\{ c_0 \Vert U_0 (\tau) \Vert, \Vert U_{\frac12} (\tau) \Vert, \Vert U_1 (\tau) \Vert \big\}  \neq 0 \qquad \text{for all} \quad \tau \in \td I \end{equation}
or $\Vert U_0 \Vert = \Vert U_{\frac12} \Vert = \Vert U_1 \Vert\equiv 0$ at all times, which implies that $u \equiv 0$ and $R \equiv \infty$ due to \eqref{eq_UpUpp} and \eqref{eq_UmYpp}.
Assume for the rest of the proof that \eqref{eq_max_not_zero} is true.
It remains to choose the times $\tau_0, \tau_{\frac12}$ and $\tau_1$.

\begin{Claim}
Consider a time $\tau \in \td I$, $\tau \geq T'_3$ at which at least \emph{two} of the inequalities  from Claim~\ref{Cl_U0U12U1}, involving the quantities $\|U_{\frac12} \|$, $\| U_{1} \|$  and $c_0 \| U_{0} \|$, hold.
Then, for slightly smaller (resp.~larger) times, only the inequality appearing earlier (resp.~later) in this claim remains valid.
\end{Claim}

\begin{proof}
We consider three cases.
\medskip

\textit{Case 1: $\|U_{\frac12} (\tau) \| , \| U_{1} (\tau)\| \leq c_0 \| U_{0}(\tau) \| $ and $ c_0 \|U_{0} (\tau)\|, \| U_{1} (\tau)\| \leq \| U_{\frac12} (\tau) \| $} \\ 
So $\| U_{1} (\tau)\| \leq c_0 \| U_{0}(\tau) \| = \|U_{\frac12} (\tau) \|$.
Using the inequalities \eqref{eq_U_evol} and \eqref{eq_U12_evol}, we obtain that 
\[ \partial_\tau \big(c_0 \Vert U_{0}(\tau) \Vert \big) \leq  \xi \big( c_0 \Vert U_{0}(\tau) \Vert \big), 
\qquad 
\partial_\tau \Vert U_{\frac12}(\tau) \Vert \geq (\tfrac12 - \xi) \Vert U_{\frac12}(\tau) \Vert  = (\tfrac12 - \xi) \big( c_0 \Vert U_{0}(\tau) \Vert \big). \]
Thus $\partial_\tau (c_0 \Vert U_{0}(\tau) \Vert ) < \partial_\tau \Vert U_{\frac12}(\tau) \Vert$, which implies that $c_0 \Vert U_{0}(\tau') \Vert < \Vert U_{\frac12}(\tau') \Vert$ for $\tau' > \tau$ sufficiently close to $\tau$ and the reverse inequality holds for $\tau' < \tau$ sufficiently close to $\tau$.

\medskip

\textit{Case 2: $ c_0 \|U_{0} (\tau)\|, \| U_{1} (\tau)\| \leq \| U_{\frac12} (\tau) \| $ and $c_0 \|U_{0} (\tau) \|, \| U_{\frac12} (\tau) \| \leq \| U_{1} (\tau) \|$ \quad}\\ 
So $c_0 \Vert U_0(\tau) \Vert \leq \Vert U_{\frac12} (\tau) \Vert = \Vert U_1 (\tau) \Vert$, and we obtain as before that
\[ \partial_\tau \Vert U_{\frac12}(\tau) \Vert \leq (\tfrac12 + \xi) \Vert U_{\frac12}(\tau) \Vert, 
\qquad 
\partial_\tau \Vert U_{1}(\tau) \Vert \geq (1 - \xi) \Vert U_{1}(\tau) \Vert  = (1 - \xi) \Vert U_{\frac12}(\tau) \Vert  \]
Thus $\partial_\tau \Vert U_{\frac12}(\tau) \Vert < \partial_\tau \Vert U_{1}(\tau) \Vert$, which implies that $\Vert U_{\frac12}(\tau') \Vert < \Vert U_{1}(\tau') \Vert$ for $\tau' > \tau$ sufficiently close to $\tau$ and the reverse inequality holds for $\tau' < \tau$ sufficiently close to $\tau$.

\medskip

\textit{Case 3: $ \|U_{\frac12} (\tau) \|, \| U_{1} (\tau)\| \leq c_0 \| U_{0}(\tau) \| $ and $c_0 \|U_{0} (\tau)\|, \| U_{\frac12} (\tau)\| \leq \| U_{1} (\tau) \| $} \\ 
The discussion in this case is similar to that in Case~1.
\end{proof}

It follows that the times at which the inequalities from Claim~\ref{Cl_U0U12U1} hold partition $\td I \cap [T'_3, \infty)$ into subintervals, arranged in the order in which the corresponding inequalities appear in the claim.
We now choose the partition points to be $\tau_{\frac12}$ and $\tau_1$.

Next, let us discuss the choice of $\tau_0 \in [T'_3, \tau_{\frac12}]$.
At every time $\tau \in  [T'_3, \tau_{\frac12}] \cap \td I$, at least one of the inequalities $|U_{0,\min} | (\tau) \leq U_{0,\max}(\tau)$ and $U_{0,\min}(\tau) \leq - |U_{0,\max}|(\tau)$ must hold.
Consider a time at which both hold, so $U_{0,\max}(\tau) = - U_{0,\min}(\tau) > 0$.
Then \eqref{eq_cmax_evol} and \eqref{eq_cmin_evol} imply that
\[ \partial_\tau U_{0,\max}(\tau) \leq (-\sqrt{2}+\xi) U_{0,\max}^2 (\tau), \quad
\partial_\tau (- U_{0,\min} (\tau) ) \geq (\sqrt{2} -\xi) U_{0,\min}^2(\tau) = (\sqrt{2} -\xi) U_{0,\max}^2(\tau). \]
It follows that $\partial_\tau U_{0,\max}(\tau) < \partial_\tau (- U_{0,\min}(\tau))$, so as before we have $|U_{0,\min}|(\tau') \leq U_{0,\max} (\tau')$ for $\tau' < \tau$ close enough to $\tau$ and $U_{0,\min} \leq - |U_{0,\max}|(\tau)$ for $\tau' > \tau$ close enough to $\tau$.
So, as before, the times where each of these bounds holds partition $[T'_3, \tau_{\frac12}] \cap \td I$ into  consecutive subintervals and $\tau_0$ is the partition point.

It remains to prove Assertion~\ref{Prop_ODE_subintervals_a} and the characterization for $\td I = (-\infty, \infty)$.

Suppose first that $T_0 = -\infty$, so also $T'_3 = -\infty$ by Proposition~\ref{Prop_Tisosc}\ref{Prop_Tisosc_c}.
If $\tau_0 > -\infty$, then on $(-\infty, \tau_0)$ we have
\[ \partial_\tau U_{0,\max} \leq (-\sqrt{2}+\xi) U_{0,\max}^2
\qquad \Longrightarrow \qquad
\partial_\tau U_{0,\max}^{-1} \geq \sqrt{2} - \xi. \]
This is impossible, since $U_{0,\max}$ must be positive on $(-\infty, \tau_0)$.
So we must have $\tau_0 = -\infty$.

Next, suppose that $T_1 = \infty$.
If $\tau_1 < \infty$, then on $(\tau_1, \infty)$ we have $\partial_\tau \Vert U_1 \Vert \geq (1-\xi) \Vert U_1 \Vert$, which contradicts the fact that $\Vert U_1 \Vert$ must be uniformly bounded.
So $\tau_1 = \infty$ and similarly, we obtain $\tau_{\frac12} =\infty$.
If $\tau_0 < \infty$, then on $(\tau_0, \infty)$ we have
\[ \partial_\tau U_{0,\min} \leq (-\sqrt{2}+\xi) U_{0,\min}^2
\qquad \Longrightarrow \qquad
\partial_\tau U_{0,\min}^{-1} \geq \sqrt{2} + \xi, \]
which contradicts the fact that $U_{0,\min} < 0$ on $(\tau_0, \infty)$.
So $\tau_0 = \infty$.

Lastly, if $T_0 = -\infty$ and $T_1 = \infty$, we obtain a contradiction for the choice of $\tau_0$, which contradicts our assumption \eqref{eq_max_not_zero}.
\end{proof}
\medskip

\subsection{Eternal flows and rigidity of cylinders}  \label{subsec_eternal}
As a direct corollary to Proposition~\ref{Prop_PDE_ODI_MCF_gauged}, we can give an alternative proof of a theorem due to Colding-Ilmanen-Minicozzi~\cite{Colding_Ilmanen_Minicozzi}.
\begin{Corollary} \label{Cor_eternal}
There is a dimensional constant $\delta > 0$ with the following property.
Let $M \subset \IR^{n+1+n'}$ be an $n$-dimensional shrinker and assume that $\bO$ is a center of an $(n,k,\delta)$-neck at scale~$1$.
Then $M$ is a round cylinder.

More generally, the following is true:
Let $\MM$ be an $n$-dimensional, unit-regular Brakke flow in $\IR^{n+1+n'} \times \IR_-$.
Suppose that for all $t < 0$ the point $\bO$ is a center of an $(n,k,\delta)$-neck at scale $\sqrt{-t}$ at time $t$.
Then $\MM = S \MM_{\cyl}^{n,k}$ for some $S \in O(n+1+n')$.
\end{Corollary}

\begin{proof}
Let $\eps', R^\# > 0$ be constants which we will determine later and apply Proposition~\ref{Prop_gauging} for $(\bq_0, t_0) = (\bO, 0)$ and with $\eps$ replaced by $\eps'$, assuming a bound of the form 
\[ \delta \leq \ov\delta(\eps', R^\#). \]
We will now study the resulting eternal, $R^\#$-gauged, rescaled modified flow $(\td\MM'_\tau = e^{\tau/2} S_{-e^{-\tau}} (\MM_{-e^{-\tau}}))_{\tau \in \IR}$ using Propositions~\ref{Prop_PDE_ODI_MCF}, \ref{Prop_PDE_ODI_MCF_gauged}, \ref{Prop_Tisosc} and \ref{Prop_ODE_subintervals}.
To do so, we fix arbitary constants $m, J \geq 4$ and $\la < 0$, $\la \in \frac12 \IZ$ and we let $\eps, \eta, R^*, \xi > 0$ be constants whose values we will determine later.
Assume $$\eps' \leq \ov\eps'(m, \eta, R^*, R^\#)$$ to ensure the bounds on $\td\bY_\tau$ and $u_\tau$ from Propositions~\ref{Prop_PDE_ODI_MCF}\ref{Prop_PDE_ODI_MCF_i}--\ref{Prop_PDE_ODI_MCF_iii}.
We also assume 
\[ \eta \leq \ov\eta(\la), \quad \eps \leq \ov\eps(\eta,m),  \quad R^* \geq \underline R^*(m,J,\la, \eta,\eps) \]
according to this proposition.
In order to be compliant with Proposition~\ref{Prop_PDE_ODI_MCF_gauged}, we moreover assume bounds of the form 
\[ R^\# \geq \underline R^\#, \quad \eta \leq \ov\eta (\la, R^\#), \quad  R^* \geq R^\# + 10, \] 
in order to apply Proposition~\ref{Prop_Tisosc}, we assume that 
\[ R^\# \geq \underline R^\# (J), \quad \eta \leq \ov\eta(\la, J, R^\#) \]
and for Proposition~\ref{Prop_ODE_subintervals} we assume
\[ \eta \leq \ov\eta(\la, J, \xi). \]
Note that these bounds can be fulfilled by choosing parameters in this order:
\[ \xi, \; m, \; J, \; \la,\; R^\#, \;  \eta, \; \eps, \; R^*, \; \eps', \; \delta. \]
Now Proposition~\ref{Prop_ODE_subintervals} is applicable and implies $(\spt \td\MM')_\tau = \td\MM^{\prime, \reg}_\tau = M_{\cyl}$ and $\td\bY_\tau \equiv 0$, which implies the corollary.
\end{proof}
\medskip

\subsection{Stability of necks} \label{subsec_stability_neck}
Another direct application of Proposition~\ref{Prop_PDE_ODI_MCF_gauged} is the following alternative proof of a result by Colding-Minicozzi~\cite{colding_minicozzi_uniqueness_blowups}.
The following statement is slightly more general, as it gives an angle bound between two finite scales.
It was first derived in \cite[Proposition~4.1]{Gianniotis_Haslhofer_2020} based on methods from \cite{colding_minicozzi_uniqueness_blowups}. 

\begin{Theorem} \label{Thm_stability_necks}
For every $\eps > 0$ there is an $\delta(\eps) > 0$ with the following property.

Let $\MM$ be an $n$-dimensional, unit-regular Brakke flow in $\IR^{n+1+n'}$ defined over a time-interval of the form $[-T_1, -T_2] \subset \IR_-$.
Suppose that for $t \in \{-T_1, -T_2 \}$ the point $(\bO,t)$ is center of a $\delta$-neck at scale $\sqrt{-t}$.
Then there is an $S \in O(n+1+n')$ such that $\MM$ is $\eps$-close to $S M_{\cyl}$ at time $t$ and scale $\sqrt{-t}$ for all $t \in [-T_1, -T_2]$.
\end{Theorem}

We remark that, crucially, $\eps$ does not depend on the times $T_1$ and $T_2$.

\begin{proof}
Let $\delta' > 0$ be a constant whose value we will determine later.

\begin{Claim}
If $\delta' \leq \ov\delta'$ and $\delta \leq \ov\delta(\delta')$, then the point $\bO$ is a center of a $(n,k,\delta')$-neck at scale $\sqrt{-t}$ at all times $t \in [-T_1, -T_2]$.
\end{Claim}

\begin{proof}
Choose $\ov\delta'$ smaller than $\frac1{10}$ times the constant from Corollary~\ref{Cor_eternal}.
Suppose that the claim was false for $\delta' \leq \ov\delta'$ and some $\delta < \delta'$.
Let $t^*$ be the maximum over all $t \in [-T_1,-T_2]$ such that the assertion of the claim is true for all $t \in [-T_1, t^*)$.
After parabolic rescaling we may assume that $t^* = -1$.
Fix $\delta' > 0$ and a sequence $\delta_i \to 0$ and pick a sequence of counterexamples $\MM_i$ defined over $[-T_{1,i}, -T_{2,i}]$ with $T_{1,i} > 1$ and $T_{2,i} < 1$.
If $T_{1,i}$ remains bounded for a subsequence, then by basic pseudolocality (see also the proof of Claim~\ref{Cl_pseudoloc}) we have smooth subsequential convergence $\MM_i \to \MM_\infty$ to a flow whose initial condition is a round cylinder.
By uniqueness $\MM_\infty$ must be a round shrinking cylinder in contradiction to our assumption at time $t^* = -1$.
So we must have $T_{1,i} \to \infty$. 
For any $t \leq -T_{2,i}$, $r \geq \max\{\sqrt{-t},1\}$ and large $j$, we can use \cite[Theorem~5.3]{Schulze_intro_Brakke} applied to the parabolic domain $\IB^{n+1}_{C'r} \times [t-(C'r)^2, t]$ (for some uniform $C'$) to derive an upper bound of the form $C r^n$ on the area of $\MM_i$ at time $t$ within $\IB^{n+1}_r$; note that the flow at the initial time of the parabolic domain is close to a cylinder at scale $\sim r$.
So we can pass to a subsequence such that we have convergence $\MM_i \to \MM_\infty$ in the Brakke sense.
Taking this bound to the limit implies a polynomial area bound for $\MM_\infty$ and hence a uniform upper bound on the Gaussian area $\Theta^{\MM_\infty}_{(\bO,0} (-t)$.
Therefore its blow-down $\MM'_\infty$ (so any subsequential Brakke limit $\la_j \MM_\infty \to \MM'_\infty$ for $\la_j \to 0$) must be a shrinker with an $(n,k,2 \delta')$-neck at the origin.
So Corollary~\ref{Cor_eternal} implies that $\MM'_\infty$ must be a round cylinder.
It follows that $\Theta^{\MM_{\infty}}(\infty) \leq \Theta_{\IR^k \times \IS^{n-k}}$.
On the other hand, we obtain that for all $t \leq -T_{2,i}$  we have $\Theta^{\MM_{\infty}}_{(\bO,0)}(-t) \geq \liminf_{i \to \infty}\Theta^{\MM_i}_{(\bO,0)} (-t) \geq \liminf_{i \to \infty} \Theta^{\MM_i}_{(\bO,0)} (-T_{2,i}) \geq \Theta_{\IR^k \times \IS^{n-k}}$; here we have used again that $\MM_\infty$ has uniformly bounded area ratios.
So $\Theta^{\MM_{\infty}}_{(\bO,0)}(-t)$ must be constant and thus $\MM_\infty$ can be extended to a shrinker onto the time-interval $\IR_-$.
By construction, $\bO$ is center of a $2\delta'$-neck at scale $\sqrt{-t}$ at all times $t < 0$, so Corollary~\ref{Cor_eternal} implies that $\MM_\infty$ is a round shrinking cylinder, which is again a contradiction.
\end{proof}

As explained in the proof of Corollary~\ref{Cor_eternal}, we can now apply Proposition~\ref{Prop_gauging}, pass to a gauged flow $(\td\MM'_\tau = e^{\tau/2} S_{-e^{-\tau}} (\MM_{-e^{-\tau}}))_{\tau \in [-\log (T_1), -\log(T_2)]}$ and analyze this flow via Propositions~\ref{Prop_PDE_ODI_MCF}, \ref{Prop_PDE_ODI_MCF_gauged}, \ref{Prop_Tisosc} and \ref{Prop_ODE_subintervals}.
The choice of constants is the same as explained in this proof, except that instead of $\delta$-closeness to $M_{\cyl}$ we require a $(n,k,\delta')$-neck condition; we also require $\xi \leq \frac1{10}$.
The order of these constants also allows us to adjust the size of $\eta$, as long as $\eps, R^*, \eps'$ and $\delta'$ are adjusted accordingly.
The following claim summarizes the relevant conlusions on the functions $U^+ : [-T_1, -T_2] \to  \sV_{> \la}$ and $\UU^- : [-T_1, -T_2] \to [0,\infty)$.

\begin{Claim}
There is a dimensional constant $C > 0$ such that for $\eta \leq \ov\eta$ and $\delta' \leq \ov\delta'(\eta)$ we have
\begin{equation} \label{eq_St_0_eta}
 |S_t (\bO) | \leq \eta \sqrt{-t}, \qquad \text{for} \quad t \in [-T_2, -T_1]. 
\end{equation}
and 
\begin{equation} \label{eq_UpUmCetaY}
 \Vert U^+ \Vert + \UU^- \leq C \eta, \qquad \Vert \td\bY \Vert \leq C \big( \Vert U^+_{\osc,<0}  \Vert + \UU^- \big). 
\end{equation}
Moreover, there are times $-\log(T_1) = T'_0 \leq T'_1 \leq T'_2 \leq T'_3 \leq \tau_0 \leq \tau_{\frac12} \leq \tau_1 \leq -\log(T_2)$ such that the quantities $\td\UU_1, \td\UU_2, \td\UU_3$ from \eqref{eq_tdUUi_def} satisfy the following bounds:
\begin{alignat}{2}
 \partial_\tau \td\UU_i &\leq - \tfrac1{10(n-k)} \td\UU_i \qquad &&\text{on} \quad (T'_{i-1}, T'_i) \label{eq_tdUUdecay_2nk} \\
 \td\UU_1 &\leq C \Vert U^{++} \Vert^2 \qquad &&\text{on} \quad (T'_{3}, -\log(T_2)) \notag
\end{alignat}
On the time-intervals $(T'_3, \tau_0)$, $(\tau_0, \tau_{\frac12})$, $(\tau_{\frac12}, \tau_1)$ and $(\tau_1, -\log(T_2))$ the norm $\Vert U^{++} \Vert$ is bounded by $CU_{0,\max}$, $- CU_{0,\min}$, $C\Vert U_{\frac12} \Vert$, $C\Vert U_1 \Vert$, respectively.
Moreover, the bounds \eqref{eq_U_evol}, \eqref{eq_cmax_evol}, \eqref{eq_cmin_evol}, \eqref{eq_U12_evol} and \eqref{eq_U1_evol}, respectively, hold for $\xi \leq \frac1{10}$.
\end{Claim}

We obtain the following consequence.

\begin{Claim}
There is a dimensional constant $C > 0$ such that
\begin{equation} \label{eq_inttdYY_bound}
 \int_{-\log(T_1)}^{-\log(T_2)} \Vert \td\bY (\tau) \Vert d\tau \leq C \eta. 
\end{equation}
\end{Claim}

\begin{proof}
Let $C$ be a generic dimensional constant.
By \eqref{eq_UpUmCetaY} and the definition of $\td\UU_i$ we have $\Vert \td\bY \Vert \leq C \td\UU_i$, so it suffices to show that
\begin{multline*}
 \sum_{i=1}^3 \int_{T'_{i-1}}^{T'_i} \td\UU_i(\tau) d\tau + \int_{T'_{3}}^{\tau_0} U^2_{0,\max}(\tau) d\tau + \int_{\tau_0}^{\tau_{\frac12}} U^2_{0,\min}(\tau) d\tau \\ + \int_{\tau_{\frac12}}^{\tau_1} \Vert U_{\frac12}(\tau) \Vert^2 d\tau + \int_{\tau_1}^{-\log(T_2)} \Vert U_{1}(\tau) \Vert^2 d\tau \leq C \eta.
\end{multline*}
The bounds on the first three and last two integrals follows easily, because its integrands are bounded by $C \eta$ due to \eqref{eq_UpUmCetaY}, and because \eqref{eq_tdUUdecay_2nk}, \eqref{eq_U12_evol} and \eqref{eq_U1_evol} dictate exponential decay and growth of these integrands.
To obtain the bound on the fourth integral, we rewrite \eqref{eq_cmax_evol} to obtain
\[ |\partial_\tau U_{0,\max}^{-1} - \sqrt{2} | \leq \tfrac1{10}. \]
It follows that for $\tau \in (T'_3, \tau_0)$ we have
\[ U_{0,\max}(\tau) \leq ( U_{0,\max}^{-1}(T'_3) + \tau - T'_3 )^{-1} \leq (C^{-1} \eta^{-1}   + \tau - T'_3 )^{-1}. \]
This readily implies the desired bound on the fourth integral.
The fifth integral can be bounded similarly.
\end{proof}

The rotational part $dS_t \in (n+1+n')$ satisfies a bound of the form $\Vert \partial_\tau S_{-e^{-\tau}} \Vert \leq \Vert \td\bY_\tau \Vert$.
So the bound \eqref{eq_inttdYY_bound} implies that $\Vert dS_{-T_2} - dS_t \Vert \leq C \eta$ for all $t \in [-T_1, -T_2]$.
Combined with \eqref{eq_St_0_eta}, we obtain that $\Vert S_{t} - S \Vert \leq C \eta$ for all $t \in [-T_1, -T_2]$ if we set $S := dS_{-T_2}$.
The theorem now easily follows by choosing $\delta' \leq \ov\delta'(\eps)$ and $\eta \leq \ov\eta(\eps)$.
\end{proof}
\medskip

We record the following direct consequences, which show that cylindrical tangent flows at a singular point and at infinity are unique; see \cite{colding_minicozzi_uniqueness_blowups} for the original results.

\begin{Corollary}[Uniqueness of cylindrical blow-ups] \label{Cor_unique_tangent}
There is a dimensional constant $\delta > 0$ with the following property.
Suppose that $\MM$ is an $n$-dimensional, unit-regular Brakke flow in $\IR^{n+1+n'} \times I$ and let $t_0 > \inf I$.
Suppose that for some sequence $t_i \nearrow t_0$ some point $\bp_0$ is a center of an $(n,k,\delta)$-neck at scale $\sqrt{t_0- t_i}$ and at time $t_i$.
Then there is an $S \in O(n+1+n')$ such that $\la S(\MM - (\bp_0, t_0)) \to \MM_{\cyl}^{n,k}$ smoothly as $\la \to \infty$.
\end{Corollary}

\begin{Corollary}[Uniqueness of cylindrical blow-downs] \label{Cor_unique_tangent_infinity}
There is a dimensional constant $\delta > 0$ with the following property.
Suppose that $\MM$ is an ancient, $n$-dimensional, unit-regular Brakke flow in $\IR^{n+1+n'} \times (-\infty, T)$.
Suppose that there is a sequence $t_i \to -\infty$ such that some point $\bp$ is a center of an $(n,k,\delta)$-neck at scale $\sqrt{- t_i}$ and at time $t_i$.
Then there is an $S \in O(n+1+n')$ such that $\la S\MM  \to \MM_{\cyl}^{n,k}$ smoothly as $\la \to 0$.
\end{Corollary}

\begin{proof}[Proofs of Corollaries~\ref{Cor_unique_tangent}, \ref{Cor_unique_tangent_infinity}.]
By Theorem~\ref{Thm_stability_necks} and after adjusting $\delta$, we can assume without loss of generality that the $(n,k,\delta)$-neck condition holds for $t < t_0$ sufficiently close to $t_0$ (resp. $t$ sufficiently small)---as opposed to for a sequence $t_i \nearrow t_0$ (resp. $t_i \to -\infty$).
So for every sequence $\la_i \to \infty$ (resp. $\la_i \to 0$) we can find a subsequence such that we have convergence $\la_i (\MM - (\bp_0, t_0)) \to \MM_{\infty}$ (resp. $\la_i \MM  \to \MM_{\infty}$).
This limit is a shrinker and if $\delta$ is chosen sufficiently small, then $\MM_\infty$ satisfies the assumptions from Corollary~\ref{Cor_eternal}, so $S\MM_\infty =  \MM^{n,k}_{\cyl}$ for some $S \in O(n+1+n')$.
It follows that there is a function $\delta(t) \to 0$ as $t \to 0$  (resp. $t \to -\infty$) such that $\bp_0$ (resp. $\bO$) is a center of an $(n,k,\delta(t))$-neck at scale $\sqrt{t_0- t}$  (resp.  $\sqrt{- t}$).
The desired convergence now follows from Theorem~\ref{Thm_stability_necks}.
\end{proof}
\bigskip

\subsection{Asymptotically cylindrical flows} \label{subsec_asymp_cyl}
In this subsection, we summarize the results obtained so far for \emph{ancient,} asymptotically cylindrical flows.
The main outcome is Proposition~\ref{Prop_PO_ancient} and its addendum, Proposition~\ref{Prop_Add_tau2}, which motivates Definition~\ref{Def_dominant_modes}.
These are the only statements from this section that are needed in the remainder of this paper and in the sequel~\cite{Bamler_Lai_MCF2}.
Compared with the preceding discussion, the proposition is more accessible: although its proof requires gauging the flow and introducing a family of Killing fields, none of these technical features appear in the final statement.

We begin with the basic definition:

\begin{Definition}[Asymptotically cylindrical flows]
An \textbf{asymptotically $(n,k)$-cylindrical mean curvature flow} $\MM$ in $\IR^{n+1+n'} \times (-\infty,T)$ (or $\IR^{n+1+n'} \times (-\infty,T]$) is an $n$-dimensional, integral, unit-regular Brakke flow with the property that as $\la \to 0$ the parabolic rescalings satisfy $\la \MM \to \MM_{\cyl}^{n,k}$ locally smoothly.
We will frequently leave out the prefix ``$(n,k)$'' if the context is clear.
\end{Definition}

Note that our convention differs slightly from other conventions in that we require the blow-down to be \emph{equal} to $\MM^{n,k}_{\cyl}$ and not a rotation of it; this is purely for convenience.
We recall the following useful properties:

\begin{Lemma} \label{Lem_cyl_properties}
If $\MM$ is asymptotically $(n,k)$-cylindrical, then the following is true:
\begin{enumerate}[label=(\alph*), start=1]
\item \label{Lem_cyl_properties_b} $\MM$ has uniformly bounded area ratios at all scales and $\Theta^\MM(\infty) = \Theta_{\IR^k \times  \IS^{n-k}} = \Theta_{\IS^{n-k}}$.

\item \label{Lem_cyl_properties_c} For any $\eps > 0$ there is a $\delta (n,k,\eps) > 0$ with the following property.
Suppose $\MM$ is $\delta$-close to $M_{\cyl}^{n,k}$ at scale $\sqrt{-t_0}$ for some time $t_0 < 0$.
Then $\MM$ is $\eps$-close to $M_{\cyl}^{n,k}$ at scale $\sqrt{-t}$ for all times $t \leq t_0$.
\end{enumerate}
\end{Lemma}

\begin{proof}
To see Assertion~\ref{Lem_cyl_properties_b} notice that for any $t \ll 0$ and  $r \geq \sqrt{-t}$, we can use \cite[Theorem~5.3]{Schulze_intro_Brakke} applied to the parabolic domain $\IB^{n+1}_{C'r} \times [t-(C'r)^2, t]$ (for some uniform $C'$) to derive an upper bound of the form $C r^{n}$ on the area of $\MM$ at time $t$ within $\IB^{n+1}_r$; note that the flow at the initial time of the parabolic domain is close to a cylinder at scale $\sim r$.
This establishes a uniform bound on the Gaussian area and the bound on $\Theta^\MM(\infty)$ follows from a blow-down limit argument.
The area ratio bound follows from this.

Assertion~\ref{Lem_cyl_properties_c} follows by applying Theorem~\ref{Thm_stability_necks} to time-intervals of the form $[t', t_0]$ for $t' \ll t_0$.
\end{proof}

The following is the main result of this subsection:

\begin{Proposition}[PDE-ODI principle for ancient, asymptotically cylindrical flows]\label{Prop_PO_ancient}
Fix $1 \leq k < n \leq n$ and $n' \geq 0$; we will omit dependencies on these constants in the following.
Let $J, m \geq 0$ be integers and choose $0 < \eta, \xi \leq \frac1{10}$. 
Then there is a constant $\delta (J,  m,\eta,\xi)> 0$ such that the following is true.

Let $\MM$ be an asymptotically $(n,k)$-cylindrical mean curvature flow in $\IR^{n+1+n'} \times (-\infty,T)$.
Let $(\bq_0, t_0) \in \IR^{n+1+n'} \times \IR$ be a point and $r_0 > 0$ a scale such that $\MM - (\bq_0, t_0)$ is $\delta$-close to $M_{\cyl}$  at time $- r_0^2$ and at  scale $r_0$.

If $\MM$ is rotationally symmetric and if $\bq_0 \in \IR^k \times \bO^{n-k+1}$ lies on the axis of rotation, then set $\bq := \bq_0$.
Otherwise we can find a point $\bq \in \IR^{n+1+n'}$ with
\begin{equation} \label{eq_bqmbq0}
 |\bq - \bq_0 | \leq \eta r_0 
\end{equation}
and with the following property.

Consider the rescaled (but unmodified!) flow $\td\MM$ arising from $\MM - (\bq, t_0 )$ over the time-interval $(-\infty, \td\tau]$ for $\td\tau := -2\log r_0$; so $\td\MM^{\reg}_\tau = e^{\tau/2}( \MM^{\reg}_{t_0  - e^{-\tau}} - \bq)$.
There is a smooth function
$$ U^+ = U_1 + U_{\frac12}  + U_0  + \ldots + U_{-J}: (-\infty, \td\tau] \lto  \sV_{ \rot,\geq -J}  = \sV_{\rot, 1} \oplus \sV_{\rot, \frac12} \oplus \ldots \oplus \sV_{\rot, -J}$$
such that if we set
\[ U^{++}:= U_1 + U_{\frac12} + U_0, \qquad
e^{-R^2(\tau)} :=  \Vert U^{++} (\tau) \Vert_{L^2_f}^J \]
(here we allow $R(\tau) = \infty$ if $U^{++} (\tau) = 0$), 
then the following is true for  all $\tau \leq \td\tau$:
\begin{enumerate}[label=(\alph*)]
\item \label{Prop_PO_ancient_a} There is a smooth function $u_\tau : \DD_\tau \to \IR^{1+n'}$ with $\IB^k_{R(\tau)-1} \times \IS^{n-k} \subset \DD_\tau \subset \IB^k_{R(\tau)} \times \IS^{n-k}$ such that
\[ \Gamma_{\cyl}(u_\tau) = (\spt \td\MM)_\tau \cap \IB^{n+1+n'}_{R(\tau)} \subset \td\MM_\tau^{\reg} \]
and 
\begin{equation} \label{eq_u_UP_prop}
   \Vert u_\tau - U^{+}(\tau) \Vert_{C^m (\DD_\tau)} \leq C( J, m) \Vert U^{++}(\tau) \Vert_{L^2_f}^{J+1}. 
\end{equation}
We also have the following bounds
\begin{equation} \label{eq_prop_u_remaining_bounds}
  \Vert u_\tau \Vert_{C^m(\DD_\tau)} + \Vert U^+ (\tau) \Vert_{L^2_f} + \Vert U^+ (\tau) \Vert_{C^m(\DD_\tau)} + e^{-R^2(\tau)} \leq \eta. 
\end{equation}
\item \label{Prop_PO_ancient_b}  The evolution of $U^+$ is controlled by the following ODI
\[ \big\| \partial_\tau U^+ - L U^+  - Q_J^+ (U^+) \big\|_{L^2_{f}}
\leq  C( J) \Vert U^{++} \Vert^{J+1}_{L^2_{f}}   \]
\item \label{Prop_PO_ancient_c} We have $\Vert U^+ \Vert \leq 10 \Vert U^{++} \Vert$ and for $i=1, \ldots, 2 J$
\[ \Vert U_{-\frac12 i} \Vert_{L^2_f} \leq C( J) \Vert U^{++} \Vert^{\lceil i/2 \rceil + 1}_{L^2_f}.  \]
\item  \label{Prop_PO_ancient_d}
If $U^{++} \equiv 0$, then $\MM = (\MM_{\cyl} + (\bq, t_0))|_{(-\infty, T)}$.
Otherwise there is a constant $c_0( J, \xi) > 0$  and there are unique times $-\infty \leq \tau_{\frac12} \leq \tau_1  \leq \td\tau$ such that the following is true:
\begin{enumerate}[label=(d\arabic*)]
\item On $(-\infty, \tau_{\frac12})$ we have $\Vert U_{\frac12} \Vert_{L^2_f}, \Vert U_{1} \Vert_{L^2_f} \leq c_0 \Vert U_{0} \Vert_{L^2_f}$.
View $U_0(\tau) = \sum_{i,j} c_{ij}(\tau) \fp^{(2)}_{ij}$ as a time-dependent symmetric matrix, using the Hermite polynomials from \eqref{eq_Hermite}, and let $U_{0,\min/\max}(\tau)$ be its minimal/maximal spectral value.
Then on $(-\infty, \tau_{\frac12})$ we have $U_{0,\min} \leq -|U_{0,\max}|$ and
\[ \partial_\tau \Vert U_0 \Vert_{L^2_f} \leq \xi \Vert U_0 \Vert_{L^2_f}, \qquad |\partial_\tau U_{0,\min} + \sqrt{2} U_{0,\min}^2 | \leq \xi U_{0,\min}^2.  \]
\item 
On $(\tau_{\frac12}, \tau_1)$ we have $ c_0 \Vert U_{0} \Vert_{L^2_f} , \Vert U_{1} \Vert_{L^2_f}  \leq \Vert U_{\frac12} \Vert_{L^2_f}$ and
\[ \Vert \partial_\tau U_{\frac12} - \tfrac12 U_{\frac12} \Vert_{L^2_f} \leq \xi \Vert U_{\frac12} \Vert_{L^2_f}. \]
\item 
On $(\tau_{1}, \td\tau)$ we have $c_0 \Vert U_{0} \Vert_{L^2_f}, \Vert U_{\frac12} \Vert_{L^2_f} ,  \leq \Vert U_{1} \Vert_{L^2_f}$ and
\[ \Vert \partial_\tau U_{1} -  U_{1} \Vert_{L^2_f} \leq \xi \Vert U_{1} \Vert_{L^2_f}. \]
\end{enumerate}
\end{enumerate}
\end{Proposition}

The following addendum to Proposition~\ref{Prop_PO_ancient} states that the property whether $\tau_{\frac12} = -\infty$ in Assertion~\ref{Prop_PO_ancient_d} is an inherent property of the flow $\MM$, so it is independent of the chosen constants.

\begin{Proposition} \label{Prop_Add_tau2}
Consider the setting of Proposition~\ref{Prop_PO_ancient}.
Then in Assertion~\ref{Prop_PO_ancient_d} the property whether $\tau_2 = -\infty$ depends only on the flow $\MM$, and not on the choice of $(\bq_0, t_0)$, $r_0$, or the auxiliary constants.  

In other words, apply Proposition~\ref{Prop_PO_ancient} with two different choices of $(\bq_0^{(i)}, t_0^{(i)})$, $r_0^{(i)}$, and constants $J^{(i)}, m^{(i)} , \eta^{(i)} , \xi^{(i)}$, for $i=0,1$.  
Let $\bq^{(i)}$ and $U^{+,(i)}$ be the point and the function satisfying the assertions of Proposition~\ref{Prop_PO_ancient} for these choices and let $-\infty \leq \tau_{\frac12}^{(i)} \leq \tau_1^{(i)}$ be as in Assertion~\ref{Prop_PO_ancient_d}.  
Then
\[
\tau_{\frac12}^{(0)} = -\infty \;\;\Longleftrightarrow\;\; \tau_{\frac12}^{(1)} = -\infty.
\]
\end{Proposition}

Proposition~\ref{Prop_Add_tau2} motivates the following definition.

\begin{Definition} \label{Def_dominant_modes}
Let $\MM$ be an asymptotically cylindrical mean curvature flow in $\IR^{n+1+n'} \times (-\infty,T)$.  
We say that $\MM$ has \textbf{dominant quadratic mode} if for some (equivalently, for every) choice of $(\bq_0, t_0)$, $r_0$, $\bq$, $J$, $m$, $\eta$, $\xi$ and associated function $U^+$, which satisfy the assertions of Proposition~\ref{Prop_PO_ancient}, we have $\tau_{\frac12} > -\infty$.  
We say that $\MM$ has \textbf{dominant linear mode} if it does not have dominant quadratic mode and is not homothetic to a round shrinking cylinder (so \emph{the} round shrinking cylinder $\MM_{\cyl}$ shifted in time and/or space).
\end{Definition}

We will study flows with dominant linear mode more closely in Section~\ref{sec_dom_lin} and flows with dominant quadratic mode in Section~\ref{sec_dom_quadratic}.

\medskip

The proof of Proposition~\ref{Prop_PO_ancient} relies on two lemmas.
The first is a special case of the Gagliardo-Nirenberg inequality.

\begin{Lemma} \label{Lem_GN}
If $u \in C^{m+1}(\IB_1^n)$ for $m \geq 0$, then
\[ \Vert u \Vert_{C^{m}(\IB_{\frac12}^n)} \leq C(m,n)  \Vert u \Vert_{C^{m+1}(\IB_1^n)}^{1-\frac1{n/2+m+1}} \Vert u \Vert_{L^2(\IB_1^n)}^{\frac1{n/2+m+1}}. \]
\end{Lemma}

\begin{proof}
After multiplying $u$ with an appropriate constant, we may assume that $\Vert u \Vert_{C^{m+1}(\IB_1^n)}=1$.
So after shifting $u$ and restricting to $\IB^n_{\frac12}$, it suffices to bound its first $m$ derivatives at the origin, so we need to show
\begin{equation} \label{eq_GN_simplified}
 \Vert u \Vert_{C^{m+1}(\IB_{\frac12}^n)}, \; \Vert u \Vert_{L^2(\IB^n_{\frac12})} \leq 1 \quad \Longrightarrow \quad \sum_{i=0}^m |\nabla^i u |(\bO) \leq C(m,n) \Vert u \Vert_{L^2(\IB^n_{\frac12})}^{\frac1{n/2+m+1}}. 
\end{equation}
Let $v$ be the $m$-th Taylor polynomial of $u$.
Then for $0 < r \leq \frac12$ we have
\[ \sup_{\IB^n_r} |u-v| \leq C(m,n) r^{m+1}, \]
so if $v_r (x) := v(r^{-1} x)$, then for $r \in (0,\frac12]$
\[ \Vert v_r \Vert_{L^2(\IB^n_{1})} 
= r^{-n/2} \Vert v \Vert_{L^2(\IB^n_r)} 
\leq r^{-n/2} \big( \Vert u \Vert_{L^2(\IB^n_r)} +  \Vert u - v \Vert_{L^2(\IB^n_r)} \big)
\leq r^{-n/2} \Vert u \Vert_{L^2(\IB^n_{\frac12})} + C(m,n) r^{m+1}. \]
Since the restriction map from space of polynomials of degree $\leq m$ into $L^2(\IB^n_{1})$ is an injection, we get
\[ \sum_{i=0}^m r^{i} |\nabla^i u |(\bO)  = \sum_{i=0}^m |\nabla^i v_r |(\bO) \leq C(m,n) \Vert v_r \Vert_{L^2(\IB^n_{1})} 
\leq C(m,n) r^{-n/2} \Vert u \Vert_{L^2(\IB^n_{\frac12})} + C(m,n) r^{m+1}. \]
Dividing both sides by $r^m$ and setting $r = \frac12 \Vert u \Vert_{L^2(\IB^n_{\frac12})}^{\frac1{n/2+m+1}}$ implies \eqref{eq_GN_simplified}.
\end{proof}

The second lemma characterizes the effect of Euclidean transformations on graphs over cylinders.

\begin{Lemma} \label{Lem_up_minus_u}
If $R \geq \underline R$, $\eps \leq \ov\eps$ and $m \geq 0$, then the following is true.

Consider a submanifold $M' = \Gamma_{\cyl} (u') \subset \IR^{n+1+n'}$ for some smooth function $u' : \IB^k_R \times \IS^{n-k} \to \IR^{1+n'}$ with $\Vert u' \Vert_{C^{m+1}} \leq 1$.
Consider $A \in O(n+1+n')$ and $\mathbf b \in \IR^{n+1+n'}$ such that $\Vert A-\id \Vert \leq \frac{\eps}R$ and $|\mathbf b| \leq \eps$.
Then $M := A M' + \mathbf b = \Gamma_{\cyl}(u)$ for some smooth function $u : \DD \to \IR^{1+n'}$ with $\IB^k_{\frac12 R} \times \IS^{n-k} \subset \DD$ and we have
\[ \Vert u' - u \Vert_{C^m (\IB^k_{\frac12 R} \times \IS^{n-k})} \leq C(m) \eps. \]
\end{Lemma}

\begin{proof}
Fix $m$.
In the following, we denote by $C$ a generic constant, which may depend on $m$.
It is not hard to see that for sufficiently small $\eps$ the projection $\proj_{\cyl} : M \to M_{\cyl}$ is a diffeomorphism onto its image and the image contains $\IB^k_{\frac12 R} \times \IS^{n-k}$.
Specifically, write $u'(\bx, \by) = (u'_1(\bx,\by), u'_2(\bx, \by)) \in \IR \times \IR^{n'}$ and consider the parameterization of $M'$ given by
\[ \psi : \IB^k_R \times \IS^{n-k} \to M', \qquad (\bx,\by) \mapsto \big(\bx,  (1+u'_1(\bx, \by)) \by, u'_2(\bx, \by) \big)  . \]
Then for sufficiently large $R$ the map
\[ \chi:= \proj_{M_{\cyl}} (A \psi + \mathbf b) :  \IB^k_R \times \IS^{n-k} \lto \DD \]
is a diffeomorphism with $\Vert \chi - \id \Vert_{C^{m+1}} \leq C \frac{\eps}{R} R +  C \eps \leq C \eps$, where the difference is taken with respect to the ambient vector space structure on $\IR^{n+1}$.
If $\eps$ is sufficiently small, then the implicit function theorem implies a comparable bound for the inverse map: 
\[ \Vert \chi^{-1} - \id \Vert_{C^{m+1}} \leq C \eps. \]

If we write $u(\bx, \by) = (u_1(\bx,\by), u_2(\bx, \by)) \in \IR \times \IR^{n'}$, then
\begin{equation*} 
  A \psi ( \chi^{-1}(\bx, \by)) + \mathbf b =  \big(\bx,  (1+u_1(\bx, \by)) \by, u_2(\bx, \by) \big)  
\end{equation*}
So over the domain $\IB^k_{\frac12 R} \times \IS^{n-k}$
\begin{align*}
   \Vert u' - u \Vert_{C^m} 
   &\leq C \Vert A \circ \psi \circ \chi^{-1} + \mathbf b - \psi \Vert_{C^m} \\
&\leq C \Vert \nabla (A \circ \psi) \Vert_{C^{m}} \Vert \chi^{-1} - \id \Vert_{C^{m}}  + C \Vert A \circ \psi  - \psi \Vert_{C^m} + C\eps  \\
&\leq C \eps   + \Vert A - \id \Vert \cdot \Vert \psi \Vert_{C^m} +C \eps
\leq  C\eps + C \tfrac{\eps}R R  + C\eps
\leq C\eps.
\end{align*}
This finishes the proof.
\end{proof}
\bigskip

\begin{proof}[Proof of Proposition~\ref{Prop_PO_ancient}.]
By Lemma~\ref{Lem_cyl_properties} we may assume without loss of generality that $\MM$ is $\delta$-close to $M_{\cyl}$  at time $t_0 - r^2$ and \emph{at all  scales $r \geq r_0$.}
Fix $J, m,   \eta, \eps$ as in the statement of the proposition and choose
\[ J' := 2(n+m+4) J, \qquad \la' := -  J' - \tfrac12   \]
Let $R^\#, \eta', \eps', R^*, > 0$ be constants whose values we will determine in the course of the proof.
Throughout the argument, we will impose  bounds on these constants and on the constant $\delta$ while respecting the following order
\[ \eta,\; \xi,\; m,\; J',\; \la',\; R^\#,\; \eta',\; \eps',\; R^*,\; \delta. \]
The first five constants from this list have already been chosen and will remain fixed.
For each of the other constants, we will only require bounds involving constants appearing earlier in this list.
So, for example, $R^\#$ may only be chosen based on $\eta,\xi, m, J', \la'$.
This will ensure that these constants can be chosen successively in the indicated order to fulfill all required bounds.
We will frequently omit the ``$L^2_f$'' subscript when there is no chance of confusion.

Apply Proposition~\ref{Prop_gauging}, assuming that $\delta \leq \ov\delta (\eta, R^\#)$ (where we substitute $\eps \leftarrow \frac12 \eta$ in this proposition) to the flow $\MM$ restricted to $(-\infty, t_0 - r_0^2]$ and consider the resulting family $(S_t \in E(n+1+n'))_{t \leq t_0 - r_0^2}$ of Euclidean motions and Killing fields $(\bY_t)_{t \leq t_0 - r_0^2}$ with $(t_0 - t) \partial_t S_t = \bY_t \circ S_t$, as well as the rescaled modified flow $\td\MM'_\tau = e^{\tau/2} S_{t_0 - e^{-\tau}} (\MM_{t_0 - e^{-\tau}} )$ with respect to $\td\bY_\tau = (e^{\tau/2})_* \bY_{t_0 - e^{-\tau}}$.
By Proposition~\ref{Prop_gauging}\ref{Prop_gauging_c}, \ref{Prop_gauging_f}, we know that $\td\MM^{\prime,\reg}_\tau \to M_{\cyl}$ locally smoothly as $\tau \to -\infty$.
Since the same is true for $\MM^{\reg}_t$, we find that, for some sequence $t_j \to -\infty$ with the following property:
Consider the rotational components $dS_{t_j} = S_{t_j} - S_{t_j} (\bO) \in O(n+1+n')$.
Then we have $dS_{t_j} \to S'$ for some limit $S' \in O(n+1+n')$ with $S' (M_{\cyl}) = M_{\cyl}$.
So $S' \in O(k) \times O(n-k+1) \times O(n')$.
Since the gauging condition is invariant under these rotations, we may replace $(S_t)$ with $( S^{\prime -1} \circ S_t )$ and assume without loss of generality that $S' = \id$.
So to summarize
\begin{equation} \label{eq_Stjtoid}
 dS_{t_j} = S_{t_j} - S_{t_j} (\bO) \lto \id. 
\end{equation}
We also record the following consequence from Proposition~\ref{Prop_gauging}\ref{Prop_gauging_c}:
\begin{equation} \label{eq_Sq0}
 |S_{t_0 - r_0^2}(\bq_0)| \leq \tfrac12 \eta r_0. 
\end{equation}

Let us now apply Proposition~\ref{Prop_PDE_ODI_MCF} with the parameters $m+2, J', \la', \eta' \leq \ov\eta(\la'), \eps' \leq \ov\eps(\eta',m+2)$  and some $R^* \geq \underline R^* (m+2, \lb J', \lb \la', \lb, \eta', \lb \eps')$.
Assumptions~\ref{Prop_PDE_ODI_MCF_i}--\ref{Prop_PDE_ODI_MCF_iii} can be ensured via Proposition~\ref{Prop_gauging}\ref{Prop_gauging_c}, \ref{Prop_gauging_f} if we assume that $\delta \leq \ov\delta (m+2, \eta', R^*)$.
Call the resulting functions $R', \UU^{-,\prime} :(-\infty,  \td\tau] \to [R^*, \infty)$ and $U^{+,\prime} : (-\infty, \td\tau] \to \sV_{> \la'}$ and $U^{-,\prime} : (-\infty, \td\tau] \to \sV_{\leq \la'}$.
We will also denote the function representing $\td\MM^{\prime, \reg}_\tau$ as a graph over the cylinder by $u'_\tau : \DD'_\tau \to \IR^{1+n'}$.
The purpose of the prime notation is to emphasize that these functions are not the same as the ones in the statement of the proposition.
Set
\[ U^+ (\tau) := \PP_{\sV_{\rot, \geq - J}} U^{+,\prime}(\tau) . \]
We record that for $\tau \leq \td\tau$
\[ \IB^k_{R'(\tau)} \times \IS^{n-k} \subset \DD'_\tau, \qquad
\Vert u'_\tau \Vert_{C^{m+2} (\IB^k_{R'(\tau)} \times \IS^{n-k})} \leq \eta'. \]

Next, we apply Propositions~\ref{Prop_PDE_ODI_MCF_gauged}, assuming in addition that $R^\# \geq \underline R^\#, \eta' \leq \ov\eta(\la', R^\#)$ and $R^* \geq R^\#+10$ and then Proposition~\ref{Prop_Tisosc}\ref{Prop_Tisosc_c} assuming $R^\# \geq \underline R^\#(J')$ and $\eta' \leq \ov\eta(J', \la', R^\#)$ and finally Proposition~\ref{Prop_ODE_subintervals} assuming $\eta' \leq \ov\eta(\la, J, \xi)$.
We obtain the following bounds, which hold on the time-interval $(-\infty, \td\tau]$,
\begin{align}
 \| \PP_{\sV_{\osc}} U^{+,\prime} \| + \UU^- + \| \td\bY \| 
 &\leq C( J) \| U^{++} \|^{\lceil J'/2\rceil + 1} \label{eq_Jp_pol_1} \\
 \| \PP_{\sV_{\rot, \leq - \frac12 i}} U^{+,\prime} \| 
 &\leq C( J) \| U^{++} \|^{\lceil i/2 \rceil + 1}, \qquad i = 1, \ldots, J' .  \label{eq_Jp_pol_2} 
\end{align}
Assertion~\ref{Prop_PO_ancient_c} of this proposition now follows directly from \eqref{eq_Jp_pol_2}; the bound 
\begin{equation} \label{eq_Up10Uppproof}
\Vert U^+ \Vert \leq 10 \Vert U^{++} \Vert
\end{equation}
follows for $\eta' \leq \ov\eta' (\la', J')$.
For Assertion~\ref{Prop_PO_ancient_b} is a consequence of Proposition~\ref{Prop_PDE_ODI_MCF}\ref{Prop_PDE_ODI_MCF_b}, via \eqref{eq_Jp_pol_1} and the following bound, which is implied by \eqref{eq_Jp_pol_1} and \eqref{eq_Jp_pol_2}:
\begin{equation} \label{eq_UppUpJp2} 
\Vert U^{+,\prime} - U^+ \Vert \leq \| \PP_{\sV_{\osc}} U^{+,\prime} \| +  \| \PP_{\sV_{\rot, \leq - J-\frac12}} U^{+,\prime} \| \leq C(J') \Vert U^{++} \Vert^{J+2}. 
\end{equation}
Assertion~\ref{Prop_PO_ancient_d} is a restatement of Proposition~\ref{Prop_ODE_subintervals}.

It remains to prove the bound \eqref{eq_bqmbq0} and Assertion~\ref{Prop_PO_ancient_a}.
To do so, we first establish similar bounds on $u'_\tau$.

\begin{Claim} \label{Cl_upUP}
If $R^* \geq \underline R^*(J', \eta')$, then for any $\tau \leq  \td\tau$ we have $\IB^k_{2R(\tau)} \times \IS^{n-k} \subset \DD'_\tau$ and
\begin{align}
 \Vert u'_\tau - U^+(\tau) \Vert_{C^{m+1}(\IB^k_{2R(\tau)} \times \IS^{n-k})} &\leq C(J,m) \Vert U^{++} (\tau) \Vert^{J+1} \label{eq_upmUp} \\
  \Vert U^+(\tau) \Vert_{C^{m+1}(\IB^k_{2R(\tau)} \times \IS^{n-k})} &\leq C(J,m) \sqrt{\eta'} \label{eq_Upsqrt}
\end{align}
\end{Claim}

\begin{proof}
Fix $\tau \leq \td\tau$ and write $U^{\pm,\prime} = U^{\pm,\prime}(\tau)$, $U^+ = U^+(\tau)$ and $U^{++} = U^{++} (\tau)$.
Recall that by the statement of the proposition we have $e^{-R^2(\tau)} := \Vert U^{++} (\tau) \Vert^J$.
So if $R(\tau) \geq R^* \geq \underline R^*(J', \eta')$, then the bound on $\UU^-$ from \eqref{eq_Jp_pol_1} implies that 
\[ R'(\tau) \geq 4 R(\tau). \]
Therefore if we write $B_a := \IB^k_{a R(\tau)} \times \IS^{n-k}$, then
\[ u'_{\tau} = u'_{\tau} \omega_{R'(\tau)} = U^{+,\prime} + U^{-,\prime}  \qquad \text{on} \quad B_{4}. \]
The bound \eqref{eq_Jp_pol_1} implies that
\[ \Vert U^{-,\prime} \Vert_{L^2_f (B_{4})}  \leq C( J) \Vert U^{++} \Vert^{(n+2m+6) J}. \]
So since $e^{-f} \geq e^{-(4R(\tau))^2/4}$ on $B_4$, we obtain
\begin{multline*}
 \Vert U^{-,\prime} \Vert_{L^2(B_4)} = \bigg( \int_{B_{4}} |U^{-,\prime}|^2 dg_{\cyl}  \bigg)^{1/2} 
\leq e^{(4R(\tau))^2/8} \bigg( \int_{B_{4}} |U^{-,\prime}|^2 e^{-f} dg_{\cyl}  \bigg)^{1/2}  \\
\leq C(J) \Vert U^{++} \Vert^{-2J} \Vert U^{++} \Vert^{(n+m+4) (J+1)}
\leq C(J)  \Vert U^{++} \Vert^{(n+m+2) (J+1)}.
\end{multline*}
On the other hand, using Lemma~\ref{Lem_polynomial_bounds} and \eqref{eq_Up10Uppproof}, we find that (recall $\la' = - J' - \tfrac12$)
\begin{multline} \label{eq_Upp_pol}
   \Vert U^{+,\prime} \Vert_{C^{m+2}(B_4)} 
\leq C(J,m) R^{C(J,m)}(\tau) \Vert U^{+,\prime} \Vert_{L^2_f} 
\leq C(J,m) R^{C(J,m)}(\tau) \Vert U^{++} \Vert_{L^2_f} \\
= C(J,m) R^{C(J,m)}(\tau)  e^{-\frac{R^2(\tau)}{2J}} \Vert U^{++} \Vert_{L^2_f}^{1/2} 
\leq C( J,m) \sqrt{\eta'}.
\end{multline}
In the same way we can bound $U^+ = \PP_{\sV_{\rot,\geq - J}} U^{+,\prime}$, proving \eqref{eq_Upsqrt}, and we obtain
\[  \Vert U^{-,\prime} \Vert_{C^{m+2}(B_{4})} 
\leq \Vert u'_\tau \Vert_{C^{m+2}(B_{4})} + \Vert U^{+,\prime} \Vert_{C^{m+2}(B_{4})} 
\leq \eta' + C(J,m) \sqrt{\eta'}
\leq C( J,m). \]
So using Lemma~\ref{Lem_GN}, we obtain that
\begin{equation} \label{eq_Ump_Cmp1}
    \Vert U^{-,\prime} \Vert_{C^{m+1}(B_{2})}
\leq C( J, m) \Vert U^{-,\prime} \Vert_{L^2(B_{4})}^{\frac1{n/2+m+2}}  
\leq C(J,m)  \Vert U^{++} \Vert^{J+1} . 
\end{equation}

Now as in \eqref{eq_Upp_pol} we can bound, using \eqref{eq_UppUpJp2},
\begin{multline*} 
    \Vert  U^{+,\prime} - U^+ \Vert_{C^{m+1}(B_2)}
\leq C(J,m) R^{C(J,m)}(\tau) \Vert  U^{+,\prime} - U^+ \Vert 
\leq C(J,m) R^{C(J,m)}(\tau) \Vert U^{++} \Vert^{J+2} \\
= C(J,m) R^{C(J,m)}(\tau) e^{-R^2(\tau)} \Vert U^{++} \Vert^{J+1} 
\leq C(J,m) \Vert U^{++} \Vert^{J+1}.
\end{multline*}
The same bound holds for $\PP_{\sV_{\rot,< -J}} U^+$.
Combining this with \eqref{eq_Ump_Cmp1} implies
\begin{equation*}
 \Vert u'_\tau - U^{+} \Vert_{C^{m+1}(B_2)}
\leq \Vert U^{+,\prime} - U^+ \Vert_{C^{m+1}(B_2)} + \Vert U^{-,\prime} \Vert_{C^{m+1}(B_2)}  
\leq  C( J,m) \Vert U^{++} \Vert^{J+1}.
\end{equation*}
This shows \eqref{eq_upmUp}.
\end{proof}

Next, we need to convert the bound \eqref{eq_upmUp} on $u'_\tau$ into a bound on $u_\tau$.
To do so, we will write
\[ S_{t_0-e^{-\tau}} (\bx) = A_\tau (\bx - \bq_0)  +  \mathbf b_\tau \]
for a time-dependent rotational part $A_\tau \in O(n+1+n')$ and vector $\mathbf b_\tau = S_{t_0-e^{-\tau}} (\bq_0) \in \IR^{n+1+n'}$.
We need the following claim.

\begin{Claim} \label{Cl_S_converges}
If $\eta' \leq \ov\eta' ( J)$, then there is a $\bq \in \IR^{n+1+n'}$ satisfying \eqref{eq_bqmbq0} such that for $\tau \leq \td\tau$
\begin{align} \label{eq_S_converges}
 \big\| A_\tau  - \id \big\|  & \leq C(J) \Vert U^{++}(\tau) \Vert^{2(J+1)} \\
e^{\tau/2}  \big| \mathbf b_\tau - (\bq_0 - \bq) \big| &\leq C( J) \Vert U^{++}(\tau) \Vert^{2(J+1)} \label{eq_b_converges}
\end{align}
\end{Claim}

\begin{proof}
We use the bounds from Assertion~\ref{Prop_PO_ancient_d} of this proposition, where we temporarily set $\xi := (10(J+1))^{-1} \leq \frac1{10}$.
On the time-interval $(\tau_1, \td\tau)$ (if non-empty) we have $\partial_\tau \Vert U_1 \Vert \geq 0.4 \Vert U_1 \Vert$, so for $\tau', \tau \in [\tau_1, \td\tau]$ with $\tau' \leq \tau$
\[ \Vert U_1 (\tau') \Vert
\leq e^{-0.4(\tau-\tau')} \Vert U_1(\tau) \Vert 
 . \]
The same bound follows for $U_{\frac12}$ over the time-interval $[\tau_{\frac12}, \tau_1]$.
Combining these two bounds and using the fact that $\Vert U_{1}(\tau_{1}) \Vert =\Vert U_{\frac12}(\tau_{1}) \Vert$ implies that for $\tau', \tau \in [\tau_{\frac12}, \td\tau]$ with $\tau' \leq \tau$ we have
\begin{equation} \label{eq_Uppexp}
   \Vert U^{++} (\tau') \Vert \leq C e^{-0.4(\tau-\tau')} \Vert U^{++}(\tau) \Vert. 
\end{equation}
So using \eqref{eq_Jp_pol_1} we obtain for any $\tau \in [ \tau_{\frac12}, \td\tau]$
\begin{multline} \label{eq_tau12totau}
 \int_{\tau_{\frac12}}^\tau e^{(\tau-\tau')/2} \Vert \td\bY(\tau') \Vert d\tau' 
\leq C( J) \int_{\tau_{\frac12}}^\tau e^{(\tau-\tau')/2}  \Vert U^{++}(\tau') \Vert^{2(J+1)} d\tau'  \\
\leq C(J) \int_{\tau_{\frac12}}^\tau e^{(\tau-\tau')/2} e^{-0.8(J+1)(\tau-\tau')} \Vert U^{++}(\tau) \Vert^{2(J+1)} d\tau' 
\leq C( J) \Vert U^{++}(\tau) \Vert^{2(J+1)}. \end{multline}
On the time-interval $(-\infty, \tau_{\frac12})$, Assertion~\ref{Prop_PO_ancient_d} implies that $\partial_\tau U_{0,\min}^{-1} \geq  1$, so for $\tau' \leq \tau \leq \tau_{\frac12}$
\[ | U_{0,\min} (\tau') | \leq \frac{1}{-U_{0,\min}^{-1} (\tau)+\tau - \tau'}. \]
So using \eqref{eq_Jp_pol_1} we obtain that for any $\tau \leq \tau_{\frac12}$
\begin{align}
 \int_{-\infty}^{\tau} \Vert \td\bY(\tau') \Vert d\tau'
&\leq C( J)  \int_{-\infty}^\tau \Vert U^{++}(\tau') \Vert^{2J+3} d\tau' \notag \\
&\leq C( J) \int_{-\infty}^{\tau}  \frac{d\tau'}{(-U_{0,\min}^{-1}(\tau)+\tau - \tau')^{2J+3}} \notag \\
&\leq C( J) |U_{0,\min}(\tau)|^{2J+2} \notag \\
&\leq C( J) \Vert U^{++}(\tau) \Vert^{2J+2}.\label{eq_integralinfinitytau}
\end{align} 
Combining \eqref{eq_tau12totau} (note that the term $e^{(\tau-\tau')/2}$ in the integrand is $\geq 1$), \eqref{eq_integralinfinitytau} and \eqref{eq_Uppexp} implies that for any $\tau \leq \td\tau$
\begin{equation} \label{eq_int_Y_all}
   \int_{-\infty}^{\tau} \Vert \td\bY(\tau') \Vert d\tau' 
\leq C(J) \Vert U^{++}(\tau) \Vert^{2(J+1)}. \end{equation}
On the other hand, Assertion~\ref{Prop_PO_ancient_d} also gives us the coarser bound $\partial_\tau \Vert U_0 \Vert \leq (10J)^{-1} \Vert U_0 \Vert$, which implies that whenever $\tau \leq \tau' \leq \tau_{\frac12}$, then
\[ \Vert U^{++} (\tau') \Vert \leq C\Vert U_0 (\tau') \Vert \leq C e^{(10(J+1))^{-1}(\tau'-\tau)} \Vert U_0(\tau) \Vert \leq C e^{(10(J+1))^{-1}(\tau'-\tau)} \Vert U^{++}(\tau) \Vert. \]
Integrating this bound implies as before that for $\tau \leq \tau_{\frac12}$
\begin{multline*}
 \int_\tau^{\tau_{\frac12}} e^{-(\tau' -\tau)} \Vert \td\bY(\tau') \Vert d\tau' 
 \leq C( J) \int_\tau^{\tau_{\frac12}} e^{-(\tau' -\tau)} \Vert U^{++} (\tau') \Vert^{2(J+1)} d\tau' \\
 \leq  C( J) \int_\tau^{\tau_{\frac12}} e^{-(\tau' -\tau)} e^{0.2(\tau'-\tau)}\Vert U^{++} (\tau) \Vert^{2(J+1)} d\tau'
 \leq C( J)\Vert U^{++} (\tau) \Vert^{2(J+1)}.
 \end{multline*}

The bound \eqref{eq_S_converges} now follows from the fact that $\|\partial_\tau A_\tau \circ A_{\tau}^{-1}\| \leq C \Vert \td\bY (\tau) \Vert$ together with \eqref{eq_int_Y_all} and \eqref{eq_Stjtoid}.
Next, we bound the evolution of $\mathbf b_\tau$ as follows, recall that $\td\bY_\tau = (e^{\tau/2})_* \bY_{t_0 - e^{-\tau}}$,
\[ |\partial_\tau \mathbf b_\tau| 
= |\partial_\tau S_{t_0 - e^{-\tau}} (\bq_0)|
=  |\bY_{t_0 - e^{-\tau}} (S_{t_0 -e^{-\tau}}(\bq_0))|
= e^{-\tau/2} |\td\bY_{\tau} (e^{\tau/2} S_{t_0 -e^{-\tau}}(\bq_0))|. \]
Due to Proposition~\ref{Prop_gauging}\ref{Prop_gauging_c} we have $|e^{\tau/2} S_{t_0 -e^{-\tau}}(\bq_0)| \leq \frac12\eta \leq 1$ so
\begin{equation} \label{eq_dtb}
   |\partial_\tau \mathbf b_\tau| \leq Ce^{-\tau/2} \Vert \td\bY_\tau \Vert. 
\end{equation}
If $\tau_{\frac12} = -\infty$, then \eqref{eq_tau12totau} implies that $\mathbf b_{\tau_{\frac12}} = \mathbf b_{-\infty} := \lim_{\tau \to -\infty} \mathbf b_\tau$ exists.
Let us now set $\bq := \bq_0- \mathbf b_{\tau_\frac12}$, which also makes sense if $\tau_{\frac12} > -\infty$.
If $\tau > \tau_{\frac12}$, then the bound \eqref{eq_b_converges} follows by combining \eqref{eq_dtb} with \eqref{eq_tau12totau}.
If $\tau \leq \tau_{\frac12}$, then it follows by combining \eqref{eq_dtb} with \eqref{eq_int_Y_all}.

The bound \eqref{eq_bqmbq0} follows from \eqref{eq_Sq0} and \eqref{eq_b_converges} for $\tau = \td\tau$ as long as $\eta' \leq \ov\eta' ( J)$:
\begin{multline*} 
   |\bq_0-\bq| 
\leq  |\mathbf b_{\td\tau} |  +  \big| \mathbf b_{\td\tau} - (\bq_0 - \bq) \big| 
=  |S_{t_0 - r_0^2} (\bq_0) | + \big| \mathbf b_{\td\tau} - (\bq_0 - \bq) \big| \\
\leq \tfrac12 \eta r_0 + C( J) r_0 \Vert U^{++} (\td\tau) \Vert^{2(J+1)} 
\leq \tfrac12 \eta r_0 + C( J) r_0 (\eta')^{2(J+1)} 
\leq \eta r_0. \qedhere
\end{multline*}
\end{proof}
\medskip

We will now prove Assertion~\ref{Prop_PO_ancient_a}.
Note that
\begin{multline*}
 \td\MM^{\prime, \reg}_{\tau} 
= e^{\tau/2} A_{\tau} (\MM^{\reg}_{t_0 - e^{-\tau}} - \bq_0) + e^{\tau/2} \mathbf b_\tau
=  A_{\tau} \big(\td\MM^{\reg}_{\tau} - e^{\tau/2} (\bq_0-\bq) \big) + e^{\tau/2} \mathbf b_\tau \\
=  A_{\tau} \td\MM^{\reg}_{\tau} + e^{\tau/2} \big( \mathbf b_\tau - A_\tau (\bq_0-\bq)  \big)
\end{multline*}
and, using \eqref{eq_S_converges}, \eqref{eq_b_converges} and \eqref{eq_bqmbq0},
\begin{multline} \label{eq_bAqq_small}
   e^{\tau/2} \big|\mathbf b_\tau - A_\tau (\bq_0-\bq) \big|
\leq e^{\tau/2} \big|\mathbf b_\tau - (\bq_0-\bq) \big| + C e^{\td\tau/2} \big| A_\tau - \id \big| \cdot |\bq_0 - \bq| \\
\leq C( J) \Vert U^{++}(\tau) \Vert^{2(J+1)}.
\end{multline}
Assuming $R^* \geq \underline R^*$, we have $R(\tau) \leq e^{R^2(\tau)} = \Vert U^{++}(\tau) \Vert^{-J-1}$.
So we can apply Lemma~\ref{Lem_up_minus_u}, Claim~\ref{Cl_upUP}, \eqref{eq_S_converges} and \eqref{eq_bAqq_small} to obtain that
\[ \Vert u'_\tau - u_\tau \Vert_{C^m(\IB_{R(\tau)}^k \times \IS^{n-k})} \leq C( J', m) \Vert U^{++}(\tau) \Vert^{J+1}. \]
Combining this with \eqref{eq_upmUp} implies \eqref{eq_u_UP_prop}.
The bounds \eqref{eq_prop_u_remaining_bounds} follow for $\eta' \leq \ov\eta ( J, m, \eta)$, using \eqref{eq_Upsqrt}.
\end{proof}
\bigskip

\begin{proof}[Proof of Proposition~\ref{Prop_Add_tau2}.]
In the following $c, C$ will be generic constants that may depend on the flow $\MM$, the parameters $(\bq_0^{(i)}, t_0^{(i)})$, $r_0^{(i)}$, $J^{(i)}, m^{(i)}$,  $\eta^{(i)}$ and $\xi^{(i)}$, but not on time. 
Suppose by contradiction and without loss of generality that $\tau_{\frac12}^{(0)} = -\infty$, but $\tau_{\frac12}^{(1)} > -\infty$.
Then by integrating the bounds in Assertion~\ref{Prop_PO_ancient_d} of Proposition~\ref{Prop_PO_ancient} as in the proof of Claim~\ref{Cl_S_converges}, we obtain asymptotic bounds of the following form (recall that we have conveniently assumed $\xi^{(i)} \leq \frac1{10}$):
\[ \big\Vert U^{++, (0)}(\tau) \big\Vert_{L^2_f} \leq C e^{ (\frac12 - \xi^{(0)}) \tau/2} \leq C e^{c\tau}, \qquad \frac{c}{|\tau|} \leq \big\Vert U^{+, (1)}(\tau) \big\Vert_{L^2_f} \leq \frac{C}{|\tau|} . \]
Consider the functions $u_\tau^{(i)} : \DD^{(i)}_\tau \to \IR^{1+n'}$, which express $\td\MM^{(i),\reg}_\tau = e^{\tau/2} ( \MM^{\reg}_{t_0^{(i)} - e^{-\tau}} - \bq^{(i)})$ intersected with $\IB^{n+1+n'}_{R^{(i)} (\tau)}$ as a graph over the standard cylinder.
Due to Assertion~\ref{Prop_PO_ancient_a} of Proposition~\ref{Prop_PO_ancient} we have similar asymptotic bounds of the form
\begin{equation} \label{eq_cmco}
 \big\Vert u^{(0)}_\tau \big\Vert_{C^m(\IB_2^k \times \IS^{n-k})} \leq C e^{c\tau}, \qquad \big\Vert u^{(1)}_\tau \big\Vert_{C^0(\IB_1^k \times \IS^{n-k})} \geq \frac{c}{|\tau|}. 
\end{equation}
However if $t = t_0^{(0)} - e^{-\tau^{(0)}} = t_0^{(1)} - e^{-\tau^{(1)}}$, then $\td\MM^{(1), \reg}_{\tau^{(1)}} =\td\MM^{(0), \reg}_{\tau^{(1)}} - e^{\tau/2} (\bq^{(0)} - \bq^{(1)})$, so Lemma~\ref{Lem_up_minus_u} implies
\[  \big\Vert u^{(1)}_{\tau^{(1)}} \big\Vert_{C^0(\IB_1^k \times \IS^{n-k})} 
\leq  \big\Vert  u^{(0)}_{\tau^{(0)}} \big\Vert_{C^0(\IB_1^k \times \IS^{n-k})} +  \big\Vert u^{(1)}_{\tau^{(1)}} - u^{(0)}_{\tau^{(0)}} \big\Vert_{C^0(\IB_1^k \times \IS^{n-k})}
\leq C e^{c\tau^{(0)}}, \]
so due to \eqref{eq_cmco} we have for $t \ll 0$
\[  \frac{c}{\log(t^{(1)}_0 - t)} = 
\frac{c}{|\tau^{(1)}|} \leq C e^{c\tau^{(0)}} = C \big(t^{(0)}_0-t \big)^{-c}, \]
which yields a contradiction for $t \to -\infty$.
\end{proof}
\bigskip

\section{Flows with dominant linear mode} \label{sec_dom_lin}
\subsection{Statement of the results}
In this section we will consider asymptotically cylindrical mean curvature flows with dominant linear mode.
Our main result is the following:

\begin{Theorem} \label{Thm_bowl_unique}
If $\MM$ is an asymptotically $(n,k)$-cylindrical mean curvature flow in $\IR^{n+1} \times (-\infty,T)$ with dominant linear mode, then $\MM$ is homothetic to $\IR^{k-1} \times \MM_{\bowl}^{n-k+1}$, where $ \MM_{\bowl}^{n-k+1}$ is the $(n-k+1)$-dimensional bowl soliton.
\end{Theorem}

This theorem is stated in the codimension-one setting.  
However, the arguments are not restricted to this case and can be adapted to higher codimension without difficulty.  
Since the higher codimension case has already been reduced to codimension one in \cite{Colding_Minicozzi_codim1}, we restrict to codimension one here for convenience.

The proof of Theorem~\ref{Thm_bowl_unique} proceeds in three steps.  
In Subsection~\ref{subsec_asymptotics_bowl}, we study the asymptotic characterization from Proposition~\ref{Prop_PO_ancient} and analyze its dependence on basepoints and its behavior under translations and time-shifts.  
In Subsection~\ref{subsec_approx_bowl}, we use this to show that, up to rescaling and translation, $\MM$ is close to $\IR^{k-1} \times \MM_{\bowl}^{n-k+1}$ at bounded distance from the cap, with the closeness improving rapidly farther away.  
Finally, in Subsection~\ref{subsec_bowl_unique} we apply a comparison principle to conclude that $\MM$ coincides with $\IR^{k-1} \times \MM_{\bowl}^{n-k+1}$.

In the following we will assume that the dimensions $n, k$ are fixed, we will write $M_{\cyl} = M^{n,k}_{\cyl}$, $\MM_{\cyl} = \MM^{n,k}_{\cyl}$ and $\MM_{\bowl}=  \MM_{\bowl}^{n-k+1}$ and omit dependencies on these constants.

\subsection{Asymptotics at $-\infty$} \label{subsec_asymptotics_bowl}
In this subsection, we consider asymptotically cylindrical flows $\MM$ without dominant quadratic mode; such flows are either round shrinking cylinders or have dominant linear mode.
We will analyze their characterization from Proposition~\ref{Prop_PO_ancient} in terms of the leading mode $U^{++} = U_1 + U_{\frac12} + U_0$.
Our goal is to characterize the precise asymptotic behavior of this leading mode.
We will first state all results of this subsection and then carry out their proofs.

The first result shows that these asymptotics are fully determined by a finite set of parameters, denoted by $\ov a$ and $\ov b_i$.

\begin{Proposition} \label{Prop_ab_exist}
Consider the setting of Proposition~\ref{Prop_PO_ancient} for $J \geq 2$.
Suppose that $\MM$ does not have dominant quadratic mode, so $\tau_2 = -\infty$ in Assertion~\ref{Prop_PO_ancient_d}.
Then we can find unique numbers $\ov a, \ov b_i$ such that we have the following asymptotic characterization (using the Hermite polynomials from \eqref{eq_Hermite})
\begin{equation} \label{eq_Up_coeff}
 U^{+}(\tau) = 
\sum_i \ov b_i e^{\tau/2} \mathfrak p^{(1)}_i 
+  \Big(\ov a - \tfrac12 \sum_i \ov b_i^2 \tau \Big) e^\tau \mathfrak p^{(0)}
- \tfrac1{\sqrt{2}} \sum_{i,j} \ov b_i \ov b_j e^{\tau} \mathfrak p^{(2)}_{ij} + O( e^{1.2\tau} ). 
\end{equation}
Here  
$O( e^{1.2\tau} )$ denotes a term whose norm is bounded by $C e^{1.2\tau}$ for a constant $C$, which may depend on the function $U^{+}$ itself, but not on time.
\end{Proposition}

The next result characterizes the dependence of the coefficients $\ov a$ and $\ov b_i$ on the basepoint $(\bq, t_0)$, which is supplied by Proposition~\ref{Prop_PO_ancient}.
It also characterizes the location of the point $\bq$ and shows that the coefficients $\ov a$ and $\ov b_i$ are independent of the choices of scale $r_0$ and the auxiliary parameters from Proposition~\ref{Prop_PO_ancient_d}.

\begin{Proposition} \label{Prop_dependence_ab}
Suppose that $\MM$ is an asymptotically cylindrical mean curvature flow in $\IR^{n+1} \times (-\infty,T)$ without dominant quadratic mode.
Apply Proposition~\ref{Prop_PO_ancient} with two different choices of $(\bq_0^{(i)}, t_0^{(i)})$, $r_0^{(i)}$, and constants $J^{(i)} \geq 2$, $m^{(i)}$,  $\eta^{(i)}$ and $\xi^{(i)}$, for $i=0,1$.  
Let $\bq^{(i)}$ and $U^{+,(i)}$ be the point and the function satisfying the assertions of Proposition~\ref{Prop_PO_ancient} for these choices and let $\ov a^{(i)}$ and $\ov b_j^{(i)}$ be as in Proposition~\ref{Prop_ab_exist}.
Then
\[
\ov b_j^{(1)} = \ov b_j^{(0)}, \qquad \ov a^{(1)} =  \ov a^{(0)} + \frac1{\sqrt{2}} \sum_j \ov b_j^{(0)} (\bq^{(1)} - \bq^{(0)})_j - \tfrac12 (t_0^{(1)} - t_0^{(0)}).
\]
and
\begin{equation} \label{eq_proj_qs}
 \proj_{\bO^k \times \IR^{n-k+1}} (\bq^{(1)}) = \proj_{\bO^k \times \IR^{n-k+1}} (\bq^{(0)}). 
\end{equation}
\end{Proposition}
\medskip

Proposition~\ref{Prop_dependence_ab} shows that all points $\bq$ supplied by Proposition~\ref{Prop_PO_ancient} must lie on a common \textbf{axis} of the form $\IR^k \times \mathbf v$ for some $\mathbf v \in \IR^{n-k+1}$, which is uniquely determined by $\MM$.
Similarly, the coefficients $\ov b_j$ are independent of the basepoint $(\bq_0, t_0)$, the point $\bq$, the scale $r_0$, and the auxiliary parameters in Proposition~\ref{Prop_PO_ancient}, so we may regard $\ov b_j = \ov b_j(\MM)$ as intrinsic to the flow.
The same is true for the coefficient $\ov a = \ov a(\MM, \bq, t_0)$, except that it also depends on $\bq$ and $t_0$.
For fixed $\MM$, the map $(\bq, t_0) \mapsto \ov a(\MM, \bq, t_0)$ is affine linear with spatial gradient $\frac{1}{\sqrt{2}} \ov b_j(\MM)$ and time derivative $-\frac{1}{2}$.

The third result describes how the parameters $\ov a$ and $\ov b_j$ transform under translations, time shifts, and parabolic rescalings.

\begin{Proposition} \label{Prop_ab_transform}
Suppose that $\MM$ is an asymptotically cylindrical mean curvature flow in $\IR^{n+1} \times (-\infty,T)$ without dominant quadratic mode.
Suppose its axis as defined above is $\IR^k \times \mathbf v$.
If $\bp = (\bp', \bp'') \in \IR^{n+1} = \IR^{k} \times \IR^{n-k+1}$ and $\Delta T \in \IR$ and $\alpha > 0$, then the flow $\MM' = \alpha ( \MM - (\bp, \Delta T))$ has axis $\IR^k \times (\alpha \mathbf v -  \bp'')$ and we have
\begin{equation} \label{eq_ab_transform}
 \ov b_j (\MM') = \alpha \, \ov b_j (\MM), \qquad \ov a (\MM', \bq, t_0) = \alpha^2 \ov a(\MM, \alpha^{-1} \bq + \bp, \alpha^{-2} t_0 + \Delta T) - \sum_j \big( \ov b_j (\MM) \big)^2 \alpha^2 \log \alpha .
\end{equation}
\end{Proposition}

\begin{Remark} \label{Rmk_bowl_constant}
The last two propositions allow us to calculate $\ov b_j$ of the bowl soliton $\MM := \IR^{k-1} \times \MM_{\bowl}$, which we normalize to move at speed $1$ in the direction $\mathbf e_{k}$.
First, by symmetry, we must have $\ov b_1(\MM) = \ldots = \ov b_{k-1} (\MM) = 0$.
To find the last coefficient, we use the fact $\MM - (\Delta T \mathbf e_k, \Delta T) = \MM$ for all $\Delta T \in \IR$ to conclude
\begin{multline*}
 \ov a (\MM, \bq, t_0) = \ov a \big((\MM - (\Delta T \mathbf e_k, \Delta T), \bq, t_0 \big) = \ov a (\MM, \bq + \Delta T \mathbf e_k, t_0 + \Delta T) \\
 =\ov a (\MM, \bq, t_0)
+ \tfrac1{\sqrt{2}} \ov b_k (\MM) \Delta T - \tfrac12 \Delta T. 
\end{multline*}
It follows that $\ov b_k( \IR^{k-1} \times \MM_{\bowl}) = \frac1{\sqrt{2}}$.
\end{Remark}

The last result is a consequence of Proposition~\ref{Prop_ab_transform}.
Note that it implies that if $\ov b_j (\MM) = 0$, then $\MM$ must be a round shrinking cylinder.

\begin{Corollary} \label{Cor_ab_normalization}
Let $\MM$ be an asymptotically cylindrical mean curvature flow in $\IR^{n+1} \times (-\infty,T)$ with dominant linear mode.
Then there is a rotation $S \in SO(n+1)$ that preserves $\IR^k \times \bO^{n-k+1}$, a vector $\bp \in \IR^{n+1}$ and a factor $\alpha > 0$ such that $\MM' := \alpha (S\MM - (\bp, 0))$ has axis $\IR^k \times \bO^{n-k}$ and satisfies
\begin{equation} \label{eq_same_ab}
 \ov b_j (\MM') = \ov b_j ( \IR^{k-1} \times \MM_{\bowl}), \qquad
\ov a (\MM') = \ov a ( \IR^{k-1} \times \MM_{\bowl}). 
\end{equation}
\end{Corollary}
\bigskip

\begin{proof}[Proof of Proposition~\ref{Prop_ab_exist}.]
The bounds in Proposition~\ref{Prop_PO_ancient}\ref{Prop_PO_ancient_d} imply that for $\alpha = 1$ or $\frac12$ and $\tau \ll 0$
\[ \partial_\tau \Vert U_\alpha (\tau) \Vert \geq \alpha \Vert U_\alpha (\tau) \Vert - \eta \Vert U_\alpha (\tau) \Vert \geq  0.4  \Vert U_\alpha (\tau) \Vert. \]
So $\partial_\tau \big(  e^{-0.4 \tau} \Vert  U_\alpha (\tau) \Vert \big) \geq 0$.
Integrating this bound implies
$U_\alpha (\tau) = O( e^{0.4\tau})$, so
\begin{equation} \label{eq_Upp04}
 U^{++} (\tau) = O(e^{0.4\tau}). 
\end{equation}
It follows from Proposition~\ref{Prop_PO_ancient}\ref{Prop_PO_ancient_c} that
\begin{equation} \label{eq_UpUpp08}
 U^+(\tau) = U^{++}(\tau) + O(e^{0.8\tau}), 
\end{equation}
so
\begin{equation} \label{eq_QJQ2}
 Q_J^+(U^+(\tau)) 
= Q_2^+(U^+(\tau)) + O(e^{1.2\tau})
= Q_2^+(U^{++}(\tau)) + O( e^{1.2\tau}). 
\end{equation}
Therefore, if we write $U^{++}(\tau) = a(\tau) \mathfrak p^{(0)} + \sum_i b_i(\tau) \mathfrak p^{(1)}_i + \sum_{i,j} c_{ij}(\tau) \mathfrak p^{(2)}_{ij}$, then by Lemma~\ref{Lem_Q2}
\begin{alignat}{3}
\partial_\tau a  &= a  &&- \tfrac12 a^2 - \tfrac12 \sum_i  b_i^2 - \tfrac12 \sum_{i,j} c_{ij}^2 &&+ O( e^{1.2\tau})  \label{eq_dtau_a} \\
\partial_\tau b_i  &= \tfrac12 b_i &&- ab_i- \sqrt{2} \sum_l c_{il} b_l &&+ O(e^{1.2\tau}) \\
\partial_\tau c_{ij} &= && - \sqrt{2} \sum_l c_{il} c_{lj} - \tfrac{1}{\sqrt{2}} b_i b_j - a c_{ij} &&+ O( e^{1.2\tau})  \label{eq_dtau_c} 
\end{alignat}
Due to \eqref{eq_Upp04}, these identities imply $\partial_\tau a = a + O( e^{0.8\tau})$, $\partial_\tau b_i = \frac12 b_i + O( e^{0.8\tau})$ and $\partial_\tau c_{ij} =O( e^{0.8\tau})$, so
\[ \partial_\tau \big( e^{-\tau} a \big) = O( e^{-0.2 \tau}), \qquad
\partial_\tau \big( e^{-\tau} b_i \big) = O( e^{0.3 \tau}), \qquad
\partial_\tau c_{ij} =O( e^{0.8\tau}). \]
Integrating these bounds implies that for $\ov b_i, \ov c_{ij} \in \IR$
\begin{equation*}
 a(\tau) = O( e^{\tau} ) + O( e^{0.8\tau}) = O( e^{0.8\tau}), \qquad
b_i(\tau) = \ov b_i e^{\tau/2} + O( e^{0.8\tau}), \qquad 
c_{ij} (\tau) = \ov c_{ij} + O(e^{0.8\tau}).
\end{equation*}
Since $c_{ij} (\tau) = O(e^{0.4\tau})$, we must have $\bar c_{ij} = 0$.
Plugging these bounds back into \eqref{eq_dtau_a}--\eqref{eq_dtau_c} yields
\begin{alignat*}{3}
\partial_\tau a  &= a  && - \tfrac12 \sum_i \bar b_i^2  e^\tau &&+ O(e^{1.2\tau}) \\
\partial_\tau b_i  &= \tfrac12 b_i && &&+ O( e^{1.2\tau}) \\
\partial_\tau c_{ij} &= && - \tfrac{1}{\sqrt{2}} \bar b_i \bar b_j  e^{\tau} &&+ O( e^{1.2\tau}) 
\end{alignat*}
Integrating these bounds once again implies the desired asymptotics for $U^{++}$.

To see the asymptotics for $U^{+}(\tau)$, we need to bound its projections $U_{\frac{i}2}(\tau)$ onto $\sV_{\frac{i}2} \subset \sV_{\geq - J}$ for $\frac{i}2 < 0$.
If $\frac{i}2 < -1$, then Proposition~\ref{Prop_PO_ancient}\ref{Prop_PO_ancient_c} implies that these are bounded by $C\Vert U^{++}(\tau) \Vert^3 \leq C e^{1.2\tau}$.
If $\frac{i}2 \in \{ -\frac12, -1 \}$, then
\[ \partial_\tau \Vert U_{\frac{i}2}(\tau) \Vert \leq -\tfrac{1}2 \Vert U_{\frac{i}2}(\tau) \Vert + \Vert \PP_{\sV_{\frac{i}2}} Q^+_J(U^{+}(\tau)) \Vert + C(J)\Vert U^{++} \Vert^{J+1}. \]
Due to \eqref{eq_UpUpp08}, \eqref{eq_QJQ2} and the fact that $Q_2^+$ is quadratic and homogeneous, the last two terms are bounded by  
\begin{equation} \label{eq_i2PP}
 \Vert \PP_{\sV_{\frac{i}2}} Q^+_2(U^{++}(\tau)) \Vert + O(e^{1.2\tau}) \leq 
\Big\Vert \PP_{\sV_{\frac{i}2}} Q^+_2\Big( \sum_j \ov b_j e^{\tau/2} \fp^{(1)}_j \Big) \Big\Vert+ O(e^{1.2\tau}) . 
\end{equation}
The identity \eqref{eq_Q2} in Lemma~\ref{Lem_Q2} implies that $Q^+_2( \sum_j \ov b_j e^{\tau/2} \fp^{(1)}_j )$ is equal to $-\frac12$ times the square of a linear function, so it is contained in $\sV_{\geq 0}$ and therefore the first term on the right-hand side of \eqref{eq_i2PP} vanishes.
So $\partial_\tau \Vert U_{\frac{i}2}(\tau) \Vert \leq -\frac12 \Vert U_{\frac{i}2}(\tau) \Vert  +  O(e^{1.2\tau})$, which implies the desired bound.
\end{proof}
\bigskip

\begin{proof}[Proof of Proposition~\ref{Prop_dependence_ab}.]
In the following $C$ denotes a generic constant, which may depend on the chosen objects, but is independent of time.
Recall the asymptotic bounds from Proposition~\ref{Prop_PO_ancient}\ref{Prop_PO_ancient_c} for $i = 0,1$
\begin{equation} \label{eq_u_close_U_fromProp}
 \big\| u^{(i)}_\tau - U^{+,(i)}(\tau) \big\|_{C^1(\DD^{(i)}_\tau)} \leq C \Vert U^{++,(i)}(\tau) \Vert^{J+1} \leq C e^{1.5\tau}. 
\end{equation}
Fix some $\bz \in \IR^k$.
Then for $\tau \ll 0$ the intersection
\begin{equation} \label{eq_intersection_MM}
 (\bz \times \IR^{n-k+1}) \cap e^{\tau/2} \big( \MM^{\reg, (i)}_{t^{(i)}_0 - e^{-\tau}} - \bq^{(i)} \big) 
\end{equation}
contains a surface, which is a radial graph of $u^{(i)}_\tau(\bz, \cdot)$.
This surface is uniquely characterized as the component of \eqref{eq_intersection_MM} that bounds a region that contains $\bz \times \bO^{n-k+1}$.
Due to \eqref{eq_u_close_U_fromProp} this surface is $C e^{1.5\tau}$-Hausdorff close to a sphere of radius $\sqrt{2(n-k)} (1 + ( U^{+,(i)} (\tau) ) (\bz))$ centered at the origin; here $( U^{+,(i)} (\tau) ) (\bz)$ denotes the evaluation of $U^{+, (i)}(\tau) \in \sV_{\rot} \subset L^2_f (\IR^k \times \IS^{n-k})$ viewed as a function on $\IR^k$.
So if we write $\bq^{ (i)} = (\bq^{\prime,(i)}, \bq^{\prime\prime, (i)}) \in \IR^k \times \IR^{n-k+1}$, then the corresponding component of
\begin{equation} \label{eq_intersection_rescaled}
 \big((e^{-\tau/2} \bz + \bq^{\prime, (i)} )  \times \IR^{n-k-1}\big) \cap  \MM^{\reg, (i)}_{t^{(i)}_0 - e^{-\tau}} 
\end{equation}
still bounds a region containing the origin if $\tau \ll 0$ and it is $C e^{\tau}$-Hausdorff close to a sphere of radius $\sqrt{2(n-k)} (1 + ( U^{+,(i)} (\tau) ) (\bz)) e^{-\tau/2}$ centered at $ \bq^{\prime\prime, (i)}$.
Since we can characterize the intersections \eqref{eq_intersection_rescaled} in two different ways, for $i=0,1$, this yields the following comparison between $U^{+,(0)}$ and $U^{+,(1)}$.

\begin{Claim}
For two points $\bz_0, \bz_1 \in \IR^k$ and times $\tau_0, \tau_1$, which are sufficiently small depending on an upper bound on $|\bz_0|$ and $|\bz_1|$ the following is true.
If 
\begin{equation} \label{eq_t0t1z0z1}
 t_0^{(0)} - e^{-\tau_0} = t_0^{(1)} - e^{-\tau_1} \qquad \text{and} \qquad e^{-\tau_0/2} \bz_0 + \bq^{\prime, (0)} = e^{-\tau_1/2} \bz_1 + \bq^{\prime, (1)},
\end{equation}
then 
\begin{align}
 \Big| e^{-\tau_0/2} \big( 1 + ( U^{+,(0)} (\tau_0) ) (\bz_0) \big) - e^{-\tau_1/2} \big( 1 + ( U^{+,(1)} (\tau_1) ) (\bz_1) \big) \big| &\leq C  e^{\tau_0}+ C  e^{\tau_1},  \label{eq_U0U1} \\
 | \bq^{\prime\prime,(0)} - \bq^{\prime\prime, (1)}| &\leq  C  e^{\tau_0}+ C  e^{\tau_1}. \label{eq_q0q1}
\end{align}
\end{Claim}

\begin{proof}
The identities in \eqref{eq_t0t1z0z1} ensure that the intersections in \eqref{eq_intersection_rescaled} with $\bz = \bz_i$ and $\tau = \tau_i$ coincide.
The bound \eqref{eq_U0U1} reflects that the radii of the approximating spheres must be close, while \eqref{eq_q0q1} makes the corresponding statement for their centers.
\end{proof}

Let us now convert this into a more convenient form.

\begin{Claim} \label{Cl_Up12_comp}
We have $\bq^{\prime\prime,(0)} = \bq^{\prime\prime, (1)}$, so \eqref{eq_proj_qs} holds.
Moreover, for any $\bz \in \IR^k$ the following bound holds for $\tau \ll 0$
\begin{equation} \label{eq_Up12_1p5}
 \big| - \tfrac12 (t^{(1)}_0 - t^{(0)}_0) e^{\tau} +  ( U^{+,(0)} (\tau) ) (\bz) -  ( U^{+,(1)} (\tau_{1,\tau}) ) \big(\bz - e^{\tau/2} ( \bq^{\prime, (1)} - \bq^{\prime,(0)} ) \big) \big| \leq C  e^{1.5\tau}, 
\end{equation}
for some function $\tau \mapsto \tau_{1,\tau}$ with
$\tau_{1,\tau} = \tau - (t^{(1)}_1 - t^{(0)}_0) e^\tau + O(e^{2\tau})$.
\end{Claim}

\begin{proof}
Set $\tau_0 = \tau$ and $\bz_0 = \bz$ and solve both equations in \eqref{eq_t0t1z0z1} for $\tau_1$ and $\bz_1$.
This leads to
\begin{align}
 \tau_{1,\tau} &= -\log \big( t_0^{(1)} - t_0^{(0)} + e^{-\tau} \big)
= \tau  -\log \big( (t_0^{(1)} - t_0^{(0)} )e^{\tau} + 1 \big)
=
\tau - (t_0^{(1)} - t_0^{(0)}) e^{\tau} + O(e^{2\tau}) , \notag \\ 
\bz_{1,\tau} &= e^{(\tau_{1,\tau} - \tau)/2} \bz  - e^{\tau_{1,\tau}/2} \big( \bq^{\prime, (1)} - \bq^{\prime,(0)} \big) 
= \bz - e^{\tau/2} \big( \bq^{\prime, (1)} - \bq^{\prime,(0)} \big) + O(e^{\tau}).  \label{eq_z1tau}
\end{align}
So
\begin{equation} \label{eq_etaus_diff}
 e^{-\tau/2} - e^{-\tau_{1,\tau}/2} 
 = (1-e^{(\tau-\tau_{1,\tau})/2}) e^{-\tau/2}
 = -\tfrac12 (t^{(1)}_0 - t^{(0)}_0) e^{\tau/2} + O(e^\tau). 
\end{equation}
Note that both $|\tau_{1,\tau} - \tau|$ and $\bz_{1,\tau}$ are uniformly bounded, so \eqref{eq_U0U1} and \eqref{eq_q0q1} hold for $\bz_0 = \bz$, $\bz_1 = \bz_{1,\tau}$ and $\tau_0 = \tau$, $\tau_{1,\tau}$, as long as $\tau \ll 0$ and the right-hand sides of both equations can be bounded by $Ce^{\tau}$.
Letting $\tau \to -\infty$ in \eqref{eq_q0q1} implies $\bq^{\prime\prime,(0)} = \bq^{\prime\prime, (1)}$.

Next, we plug \eqref{eq_z1tau} and \eqref{eq_etaus_diff} into \eqref{eq_U0U1}.
Observe that $U^{+,(i)}(\tau)$ are polynomials of uniformly bounded degree whose coefficients are bounded by $C \Vert U^{+,(i)} \Vert \leq C e^{\tau/2}$.
So the $O(e^\tau)$-term in \eqref{eq_z1tau} causes an error of order $O(e^{1.5\tau})$-term in \eqref{eq_Up12_1p5}.
\end{proof}

Since the maps
\[ \bz \mapsto -\tfrac12 (t^{(1)}_0 - t^{(0)}_0) e^{\tau} + ( U^{+,(0)} (\tau) ) (\bz), \qquad \bz \mapsto ( U^{+,(1)} (\tau_{1,\tau}) ) \big(\bz - e^{\tau/2} ( \bq^{\prime, (1)} - \bq^{\prime,(0)} ) \big) \] 
are both polynomials of uniformly bounded degree, their coefficients can be determined by evaluation at finitely many points.
So Claim~\ref{Cl_Up12_comp} implies that the coefficients of both polynomials differ by a term of the form $O(e^{1.5\tau})$.
It follows that the $L^2_f$-projection of both polynomials, viewed as elements in $\sV_{\rot,> \la}$, to the subspace $\sV_1 \oplus \sV_{\frac12}$ of linear functions differ by a  term of the form $O(e^{1.5\tau})$.
By Proposition~\ref{Prop_dependence_ab}, the projections of these  polynomials are
\begin{alignat*}{3}
 \bz &\mapsto -\tfrac12 (t^{(1)}_0 - t^{(0)}_0) e^{\tau} + \frac1{\sqrt{2}} \sum_j \bar b_j^{(0)} e^{\tau/2} \bz_j &&+  \Big(\bar a^{(0)} - \tfrac12\sum_j (\bar b_j^{(0)})^2 \tau \Big) e^\tau  &&+ O(e^{1.2\tau}), \\
 \bz &\mapsto \frac1{\sqrt 2} \sum_j \bar b_j^{(1)} e^{\tau_{1,\tau}/2} \big(\bz - e^{\tau_{1,\tau}/2} ( \bq^{\prime, (1)} - \bq^{\prime,(0)} ) \big)_j  &&+ \Big(\bar a^{(1)} - \tfrac12 \sum_j (\bar b_j^{(1)})^2 \tau_{1,\tau} \Big) e^{\tau_{1,\tau}}   &&+ O(e^{1.2\tau}). 
\end{alignat*}
Here we have used the fact that the projection of $\bz \mapsto \ov b_i \ov b_j e^\tau \mathfrak p^{(2)}_{ij} (\bz - e^{\tau_{1,\tau}/2} ( \bq^{\prime, (1)} - \bq^{\prime,(0)} ) )$ onto the space of linear polynomials decays like $O(e^{1.5\tau})$.
Comparing coefficients of the polynomials above yields $\ov b_j^{(1)} = \ov b_j^{(0)}$ and
\begin{multline*}
 -\tfrac12 (t^{(1)}_0 - t^{(0)}_0) e^{\tau} +  \Big(\bar a^{(0)} - \tfrac12\sum_j (\bar b_j^{(0)})^2 \tau \Big) e^\tau   \\
 = - \frac1{\sqrt 2} \sum_j \bar b_j^{(1)} e^{\tau_{1,\tau}} ( \bq^{\prime, (1)} - \bq^{\prime,(0)} )_j  +  \Big(\bar a^{(1)} - \tfrac12\sum_j (\bar b_j^{(1)})^2 \tau_{1,\tau} \Big) e^{\tau_{1,\tau}}   + O(e^{1.2\tau}). 
 \end{multline*}
 Dividing by $e^\tau$ and noting that $e^{\tau_{1,\tau}-\tau} = O(e^\tau)$ implies
\begin{equation*}
 -\tfrac12 (t^{(1)}_0 - t^{(0)}_0) +  \bar a^{(0)} - \tfrac12\sum_j (\bar b_j^{(0)})^2 \tau    
 = - \frac1{\sqrt 2} \sum_j \bar b_j^{(1)}  ( \bq^{\prime, (1)} - \bq^{\prime,(0)} )_j  +  a^{(1)} - \tfrac12 \sum_j (\bar b_j^{(1)})^2 \tau   + O( e^{0.2\tau}). 
 \end{equation*}
This proves the desired identity for $a^{(1)}$.
\end{proof}
\bigskip

\begin{proof}[Proof of Proposition~\ref{Prop_ab_transform}.]
Suppose that all assertions of Proposition~\ref{Prop_PO_ancient} hold for $\MM'$ and $(\bq, t_0)$ and for some function $U^+$.
If $\alpha = 1$, then these assertions also hold for $\MM = \MM' + ( \bp,  \Delta T)$ and $(\bq +  \bp, t_0 +  \Delta T)$ and for the \emph{same} function $U^+$.
Since the constants $\ov a$ and $\ov b_j$ characterize the asymptotics of the function $U^+$, this immediately implies to the identity \eqref{eq_ab_transform}.
If $\alpha \neq 1$, then for $\tau \ll 0$ we have
\[ e^{\tau/2}  (\MM^{\prime, \reg}_{ t_0 - e^{-\tau}} - \bq ) =  e^{\tau/2+\log \alpha} (\MM^{\reg}_{\alpha^{-2} t_0 +  \Delta T - e^{-(\tau + 2 \log \alpha)}} -  \alpha^{-1}\bq -  \bp), \]
So the assertions of Proposition~\ref{Prop_PO_ancient} hold for $\MM$ and $(\alpha^{-1} \bq + \bp, \alpha^{-2} t_0 + \Delta T)$, but for $U^{+}(\tau)$ replaced with $U^{+}(\tau - 2 \log \alpha)$.
Comparing the asymptotics in Proposition~\ref{Prop_ab_exist} implies \eqref{eq_ab_transform}.
\end{proof}
\bigskip

\begin{proof}[Proof of Corollary~\ref{Cor_ab_normalization}.]
If $\ov b_j (\MM) \neq 0$ for some $j$, then this follows immediately from Proposition~\ref{Prop_ab_transform}.
So assume that $\ov b_j (\MM) = 0$ for all $j$.
Our goal will be to show that $\MM$ is a round shrinking cylinder, which is a contradiction according to Definition~\ref{Def_dominant_modes}.
After application of a translation, we may assume without loss of generality that the axis of $\MM$ is $\IR^k \times \bO^{n-k}$.
By Proposition~\ref{Prop_ab_transform} we can apply a time-shift so that $\ov a (\MM, \bq, 0) = 0$ for all $\bq \in\IR^k \times \bO^{n-k}$.
Now apply Proposition~\ref{Prop_PO_ancient} for an arbitrary point $\bq_0 \in \IR^{n+1}$, at time $t_0 = 0$ and for arbitrary parameters.
The assumptions of this proposition hold for large enough $r_0$.
Proposition~\ref{Prop_PO_ancient} provides a point $\bq \in \IR^k \times \bO^{n-k}$ at which we can study the rescaled flow in terms of the leading mode $U^{++}(\tau)$.
Proposition~\ref{Prop_ab_exist} implies  the asymptotic bound $\Vert U^{++} (\tau) \Vert \leq C e^{1.2\tau}$.
If $U^{++} \not\equiv 0$, then Proposition~\ref{Prop_PO_ancient}\ref{Prop_PO_ancient_d}, implies a differential inequality of the form $\partial_\tau \Vert U_\alpha (\tau) \Vert \leq 1.1 \Vert U_\alpha (\tau)\Vert$ for $\alpha = 0, \frac12$ or $1$, depending on which component of $U^{++}(\tau)$ dominates.
Fix $\tau < \tau_0$ and let $\tau' < \tau$.
Integrating this differential bound implies $\Vert U^{++}(\tau) \Vert \leq Ce^{1.1 (\tau- \tau')} \Vert U^{++}(\tau') \Vert \leq Ce^{1.1(\tau-\tau')} e^{1.2\tau'} \to 0$ for $\tau' \to -\infty$.
It follows that $U^{++}(\tau) \equiv 0$ and hence $u_\tau \equiv 0$ and $R(\tau) \equiv \infty$.
Thus $\MM = \MM_{\cyl} |_{(-\infty,T)}$.
\end{proof}
\bigskip

\subsection{Approximation by the bowl soliton} \label{subsec_approx_bowl}
In the following we denote by $\MM := \IR^{k-1} \times \MM_{\bowl}$ the bowl soliton normaized to move at speed $1$ in the direction $\mathbf e_k$.
We moreover assume that $\bO \in \MM_{\bowl, 0}$, so the tip is located at the origin at time $0$.
Our main result states that if $\MM'$ is another asymptotically cylindrical mean curvature flow with the same asymptotics as~$\MM$---which, in the case of a dominant linear mode, can always be achieved by Corollary~\ref{Cor_ab_normalization}---then $\MM'$ is close to~$\MM$ away from its cap region and asymptotic to $\MM$ at a high polynomial rate.

\begin{Lemma}\label{l:H^10}
Let $\MM'$ be an asymptotically $(n,k)$-cylindrical flow in $\IR^{n+1} \times (-\infty,T)$ with dominant linear mode and axis $\IR^k \times \bO^{n-k+1}$.
Assume that \eqref{eq_same_ab} holds.
Then there is a $C > 0$ such that for all $t < T$ the set $(\spt \MM')_t \setminus (\IR^{k-1} \times B(t \mathbf e_k, C))$ is a normal graph of a function $v : \MM_{t} \supset \DD_t \to \IR$ over $\MM_t$.
Moreover, $|v| \leq C H^{10}$, where $H$ denotes the mean curvature on $\MM$.
\end{Lemma}

\begin{proof}
Fix $J, m \geq 100$, let $\eta \leq 0.1$ be a constant whose value we will determine later and let $\delta (J,m, \eta)$ be the constant from Proposition~\ref{Prop_PO_ancient}.
We will omit dependence on the $J,m$ in the following.

For any point $(\bq_0, t_0) \in \IR^{n+1} \times \IR$ define $r_0 (\bq_0, t_0) > 0$ to be the smallest number such that $\MM - (\bq_0, t_0)$  is $\delta$-close to $M_{\cyl}$ at time $-r^2$ and scale $r$ for all $r > r_0$.
Define $r'_0 (\bq_0, t_0) \geq 0$ similarly for the flow $\MM'$.
Note that these constants are finite since $\MM$ and $\MM'$ are asymptotically cylindrical.

\begin{Claim} \label{Cl_r0rp0}
There is a constant $C_0 > 0$ such that the following holds for $\eta \leq \ov\eta$.
For any point $(\bq_0, t_0) \in \IR^{n+1} \times \IR$ there is a point $\bq \in \IR^k \times \bO^{n-k+1}$ with $|\bq - \bq_0| \leq \eta r'_0(\bq_0, t_0)$ such that for all
\begin{equation} \label{eq_SgeqC0}
 r > \max \{ r_0 (\bq, t_0), r'_0 (\bq_0, t_0) \} \qquad \text{with} \qquad  S := \frac{r}{ r_0 (\bq, t_0)} \geq C_0 
\end{equation}
 the intersection
$(\spt \MM')_{t_0 - r^2} \cap B(\bq, r \log S) $
 consists of regular points and is  the normal graph of a function $v : \DD \to \IR$ with
\[ (\spt \MM)_{t_0 - r^2} \cap B(\bq, C_0^{-1} r \sqrt{\log S}) \subset \DD \subset (\spt \MM)_{t_0 - r^2} \cap B(\bq, 2r \sqrt{\log S}). \]
and 
\[ r^{-1} |v| +  |\nabla v| \leq C S^{-40}. \]
\end{Claim}

\begin{proof}
The point $\bq$ is supplied by Proposition~\ref{Prop_PO_ancient} applied to $\MM'$ at scale $r'_0 = r'_0 (\bq_0, t_0)$.
By our assumption on the axis of $\MM'$ we must have $\bq \in \IR^k \times \bO^{n-k+1}$.
So since $\MM$ is rotationally symmetric about this axis, we can repeat this construction and apply Proposition~\ref{Prop_PO_ancient} at $(\bq, t_0)$ with $r_0 = r_0(\bq, t_0)$, without adjusting the basepoint.
We denote the resulting functions for $\MM'$ by $u'_\tau$, $U^{\prime, +}(\tau)$, $U^{\prime, ++}(\tau),R'(\tau)$ and for $\MM$ by $u_\tau$, $U^{+}(\tau)$, $U^{ ++}(\tau),R(\tau)$.
These are defined for $\tau \leq \td\tau' := -2 \log r'_0$ and $\tau \leq \td\tau := -2 \log r_0$, respectively.

Due to \eqref{eq_same_ab} and Proposition~\ref{Prop_ab_exist} the difference $V(\tau) := U^{\prime,+}(\tau) - U^{+}(\tau)$ satisfies the bound
\begin{equation} \label{eq_Vdecay1p2}
 \Vert V(\tau) \Vert \leq C' e^{1.2\tau} ,  
\end{equation}
where $C'$ is allowed to depend on $\MM'$, $\bq_0, t_0, \bq$, etc., but not on $\tau$.
As in the proof of Proposition~\ref{Prop_ab_exist}, we find that for $\tau \leq \td\tau$
\[ \partial_\tau \Vert U^{++}(\tau) \Vert \leq 0.4 \Vert U^{++}(\tau) \Vert, \]
which implies that for some generic uniform constant $C$ (which does \emph{not} depend on $\MM', \bq_0, t_0, \bq$)
\begin{equation} \label{eq_Upp04bound}
 \Vert U^{++} (\tau) \Vert \leq C e^{0.4 ( \tau- \td\tau)}. 
\end{equation}
Therefore, by Proposition~\ref{Prop_PO_ancient}\ref{Prop_PO_ancient_b}, we obtain for $\tau \leq \min\{ \td\tau, \td\tau' \}$
\begin{align*}
 \big\Vert \partial_\tau U^+ - L U^+ - Q^+_J(U^+) \big\Vert &\leq C e^{40(\tau - \td\tau)}, \\
 \big\Vert \partial_\tau U^{\prime,+} - L U^{\prime, +} - Q^+_J(U^{\prime,+}) \big\Vert 
 &\leq C \Vert U^{\prime, ++} \Vert^{100} \leq  C \Vert V \Vert^{100} + C e^{40(\tau - \td\tau)}. 
\end{align*}
Combining these bounds yields
\[ \big\Vert \partial_\tau V - L V - Q^+_J(U^+ + V) + Q^+_J (U^+) \big\Vert \leq C \Vert V \Vert^{100} +  C e^{40(\tau - \td\tau)}. \]
Since the largest eigenvalue of $L$ is $1$ and since we have a uniform bound of the form $\Vert V \Vert \leq C\eta$, this implies that for $\eta \leq \ov\eta$
\begin{multline*}
 \partial_\tau \Vert V(\tau) \Vert \leq1.01 \Vert V(\tau) \Vert + C \Vert U^+(\tau) \Vert \cdot \Vert V(\tau) \Vert + C e^{40(\tau-\td\tau)} \\
\leq \big( 1.01 + C e^{0.4(\tau-\tau_0)} \big) \Vert V(\tau) \Vert + C e^{40(\tau-\td\tau)}. 
\end{multline*}
So if $\tau \leq \min \{ \td\tau - 2 \log C_0, \td\tau' \}$ for some uniform $C_0$, then
\[ \partial_\tau \Vert V(\tau) \Vert\leq 1.1\Vert V(\tau) \Vert + C e^{40(\tau- \td\tau)} \quad \Rightarrow \quad \partial_\tau \big( e^{-1.1(\tau- \td\tau)} \Vert V(\tau) \Vert \big) \leq C e^{38.9(\tau- \td\tau)}. \]
Integrating this bound, for $\tau' \leq \tau \leq \min \{ \td\tau- 2 \log C_0, \td\tau' \}$ yields, using \eqref{eq_Vdecay1p2}, 
\begin{multline*}
 e^{-1.1(\tau- \td\tau)} \Vert V(\tau) \Vert
\leq C\int_{\tau'}^\tau  e^{38.9(\ov\tau- \td\tau)} d\ov\tau + e^{-1.1(\tau'- \td\tau)} \Vert V(\tau') \Vert  \\\xrightarrow[\tau' \to -\infty]{} C\int_{-\infty}^\tau  e^{38.9(\ov\tau- \td\tau)} d\ov\tau
\leq C e^{38.9(\tau- \td\tau)},
\end{multline*}
and hence 
\begin{equation} \label{eq_V40decay}
\Vert V(\tau) \Vert \leq C e^{40(\tau- \td\tau)}.
\end{equation}
So for $\tau \leq \min \{ \td\tau - 2 \log C_0, \td\tau' \}$ we have by Proposition~\ref{Prop_PO_ancient}\ref{Prop_PO_ancient_a} over the common domain 
\begin{multline*}
    \Vert u'_\tau - u_\tau \Vert_{C^{10}} 
    \leq \Vert V(\tau) \Vert + C \Vert U^{++}(\tau) \Vert^{100} +C \Vert U^{\prime, ++}(\tau) \Vert^{100} \\
     \leq \Vert V(\tau) \Vert + C \Vert U^{++}(\tau) \Vert^{100} +C \Vert V(\tau) \Vert^{100}
    \leq C e^{40(\tau- \td\tau)}. 
\end{multline*} 
Note that if $t_0 - r^2 = t_0 - e^{-\tau}$, then $S = e^{-(\tau - \td\tau)/2}$, so the bound $\tau \leq \td\tau -2 \log C_0$ is equivalent to the second bound in \eqref{eq_SgeqC0}.
By the definition of $R(\tau)$ and $R'(\tau)$ from Proposition~\ref{Prop_PO_ancient} and the bounds \eqref{eq_Upp04bound} and \eqref{eq_V40decay} we have $R(\tau), R'(\tau) \geq C^{-1} \sqrt{|\tau-\td\tau|} = C^{-1} \sqrt{2\log S}$.
Hence the claim follows, possibly after adjusting $C_0$. 
\end{proof}

\begin{Claim} \label{Cl_rp0_bounded}
There is a constant $C_1 > 0$ such that the following is true for all points $(\bq_0, t_0) \in \IR^{n+1} \times \IR$.
Let $\bq$ be the point satisfying the assertions of Claim~\ref{Cl_r0rp0}.
Then $r'_0(\bq_0, t_0) \leq C_1 r_0(\bq, t_0)$.
\end{Claim}

\begin{proof}
Assume that such a constant does not exist and pick a sequence of counterexamples $\bq_{0,i}, \bq_i, t_{0,i}$ with $r'_{0,i}/ r_{0,i} \to \infty$, where $r_{0,i} := r_0(\bq^i, t^i_0)$ and $r'_{0,i} := (\bq^i_0, t^i_0)$.
Since $\MM$ does not have cylindrical regions of arbitrary small size, the scale $r_{0,i}$ must be uniformly bounded from below, so $r'_{0,i} \to \infty$.
So if we consider the parabolic rescalings $\MM_i := (r'_{0,i})^{-1} (\MM - (\bq_{i}, t_{0,i}))$ and $\MM'_i := (r'_{0,i})^{-1} (\MM' - (\bq_{i}, t_{0,i}))$, then we have smooth convergence $\MM_i \to \MM_{\cyl}$.
On the other hand, for any $t < 0$, Claim~\ref{Cl_r0rp0} establishes closeness of $\MM'$ to $\MM$ at time $t_{0,i} - t (r'_{0,i})^2$ for sufficiently large $i$ and where $S_i = \sqrt{t} r'_{0,i} / r_{0,i} \to \infty$.
For the rescaled flows this means that we have locally smooth convergence $\MM'_i \to \MM_{\cyl}$ over the time-interval $(-\infty,-1)$.
So any subsequential Brakke limit of $\MM'_i$ must also locally smoothly converge to $\MM_{\cyl}$ over the \emph{entire} time-interval $(-\infty,0)$.
However, this contradicts the minimal choice or $r'_0(\bq_{0,i}, t_{0,i})$ for large $i$. 
\end{proof}

\begin{Claim} \label{Cl_MMpgraph}
There is a constant $C_2 > 0$ and for any $A > 0$ there is a constant $D(A) > 0$ such that the following is true.
Let $t < T$ and let $\bq_0 \in \IR^k \times \bO^{n-k+1}$ be sufficiently far away from the tip of $\MM_t$ in the sense that $x_k (\bq_0) -t \geq  D(A)$ (the left-hand side is the signed distance from the tip of $\MM_t$).
Set $r^2 := x_k (\bq_0) - t$.
Then the intersection $(\spt \MM')_t \cap B(\bq_0, A r)$ consists of regular points and is the normal graph of a function $v : \MM_t \supset \DD \to \IR$ with $|v| \leq C_2 H^{10}$.
\end{Claim}

\begin{proof}
Choose $t_0 = t+r^2$.
So $(\bq_0, t_0)$ lies on the tip of $\MM$, which implies a bound of the form $r_0(\bq_0, t_0) \leq C$.
The claim therefore follows by applying Claim~\ref{Cl_r0rp0} to $(\bq_0, t_0)$.
Note that Claim~\ref{Cl_rp0_bounded} provides an upper bound on $r'_0(\bq, t_0)$.
\end{proof}

Claim~\ref{Cl_MMpgraph} shows that for any $t < T$ the intersection of $(\spt \MM')_t$ with
\begin{equation} \label{eq_WW}
   \mathcal W^A_t := \bigcup_{\substack{\bq_0 \in \IR^k \times \bO^{n-k+1} \\ x_k(\bq_0) \geq t + D(A)}} B \big(\bq_0, A \sqrt{x_k(\bq_0)- t+1} \big). 
\end{equation}
is the normal graph of a function $v : \MM_t \supset \DD_t \to \IR$ over $\MM_t$ with $|v| \leq C_2 H^{10}$.
So to finish the proof of the lemma it suffices to show that, for sufficiently large $A, C_3$, the difference of $(\spt \MM')_t$ and $\mathcal W^A_t$ is contained in a solid cylinder of the form $\IR^{n-1} \times B(t \mathbf e_k, C_3)$.
To see this, note that for any $r > 0$ the $r$-neighborhood around $(\spt \MM)_t$ is contained in the union of the subset \eqref{eq_WW} with $\IR^{n-1} \times B(t \mathbf e_k, C_3)$ for sufficiently large $A, C_3$.
So it is enough to establish a uniform upper bound of the distance of points $\bp \in (\spt \MM')_t$ to $(\spt \MM)_t$.

Suppose  by contradiction that no such bound exists, so there is a sequence of points $\bp_i \in (\spt \MM')_{t_i}$ whose distance to $(\spt \MM)_{t_i}$ diverges.
Then we must also have $r_i := r_0 (\bp_i, t_i) \to \infty$.
Consider the parabolic rescalings $\MM_i := r_i^{ -1} (\MM - (\bp_{i}, t_i))$ and $\MM'_i := r_i^{ -1} (\MM' - (\bp_{i}, t_i))$ and pass to a subsequence such that we have convergence in the Brakke sense $\MM_i \to \MM_\infty$ and $\MM'_i \to \MM'_\infty$.
Since $(\bO,0) \in \spt \MM'_i$, we must also have $(\bO,0) \in \MM'_\infty$.
Since $r_i \to \infty$, the limit $\MM_\infty$ must be a round shrinking cylinder, but due to the definition of $r_0 (\bp_i, 0)$ it cannot be equal to $\MM_{\cyl}$; it must differ from it by a non-trivial translation in time and/or space.
Let $t^*$ be the extinction time of $\MM_\infty$.

Fix some $t < t^*$.
Then the time-slices $(\spt \MM_{i})_t$ of the blow-down sequence become more and more cylindrical, so the corresponding blowdowns $r_i^{-1} (\mathcal W^A_{t_i + tr^2_i} - \bp_i)$ must converge to a cylindrical region $\mathcal W^A_{\infty,t}$ around the axis of the limiting cylinder $(\spt \MM_{\infty})_t$ with $\mathcal W^A_{\infty,t} \to \IR^{n+1}$ as $A \to \infty$.
Since characterization of from Claim~\ref{Cl_MMpgraph} applies within this region and implies that $(\spt \MM'_{i})_t$ becomes closer and closer to $(\spt \MM_{i})_t$.
Therefore $\MM'_{\infty}|_{(-\infty, t^*)} = \MM_{\infty}|_{(-\infty, t^*)}$.

Suppose now that $t^* > 0$.
Then we can apply the same discussion from the last paragraph at time $t = 0$.
This implies that $\bp_i \in \mathcal W^A_{t_i}$ for large enough $A$ and $i$.
So by the bounds from Claim~\ref{Cl_MMpgraph} the points $\bp_i \in (\spt \MM'_i)_{t_i}$ have uniformly bounded distance from $(\spt \MM)_{t_i}$, in contradiction to our assumption.

So we must have $t^* \leq 0$.
Recall that $(\bO, 0) \in \spt \MM'_\infty$.
So by monotonicity of the Gaussian area we have $1 \leq \Theta^{\MM'_\infty}_{(\bO,0)} (-t) = \Theta^{\MM_\infty}_{(\bO,0)} (-t)$ for all $t < t^*$.
However, since $\MM_\infty \neq \MM_{\cyl}$, the last quantity must go to zero as $t \nearrow t^*$, which is a contradiction.
\end{proof}
\bigskip

\subsection{Proof of Theorem~\ref{Thm_bowl_unique}} \label{subsec_bowl_unique}

\begin{proof}[Proof of Theorem \ref{Thm_bowl_unique}]
By Corollary~\ref{Cor_ab_normalization}, it suffices to consider a flow $\MM'$ with axis $\IR^k \times \bO^{n-k+1}$ that satisfies \eqref{eq_same_ab}.
Without loss of generality, we may also assume that $\MM'$ is defined on the time-interval $(-\infty,0]$.
For convenience, we will write $\MM := \IR^{k-1} \times \MM_{\bowl}$.
We will show $\MM' = \MM$ by comparing $\MM'$ with time-translations of $\MM$.
The following claim, which is a consequence of Lemma~\ref{l:H^10}, is the crucial ingredient for this comparison principle.

\begin{Claim} \label{Cl_intersection}
There is a continuous function $D : \IR_+ \to \IR_+$ with the following property.
If $\Delta T \in \IR$, $\Delta T \neq 0$, then for all $t \leq 0$
\[ (\spt \MM')_t \cap (\spt \MM)_{t + \Delta T} \subset \IR^{k-1} \times B(t \mathbf e_k, D(|\Delta T|) ). \] 
Moreover, this intersection is empty if $|\Delta T|$ is large enough.
\end{Claim}

\begin{proof}
Recall that the profile function of the $(n-k+1)$-dimensional bowl soliton satisfies the asymptotics $F(x) \sim \sqrt{x}$, $F'(x) \sim x^{-1/2} \sim H(x)$.
So for a fixed $\Delta T \neq 0$ the time-slice $(\spt \MM)_{t + \Delta T} = \Delta T \mathbf \, \mathbf e_k + (\spt \MM)_{t }$ is the normal graph of a function $v_{\Delta T,t}$ on $(\spt \MM)_{t}$ with $|v_{\Delta T,t}| \geq c_{\Delta T}  H$, where $c_{\Delta T}  > 0$ can be chosen continuously on $\Delta T$.
The claim now follows since Lemma~\ref{l:H^10} establishes closeness of $(\spt \MM')_{t}$ and $(\spt \MM)_{t}$, which decays at a strictly faster rate.
\end{proof}

Let $\Delta T_+ \geq 0$ be minimal with the property that for all $\Delta T \in (\Delta T_+, \infty)$ we have
\begin{equation} \label{eq_sptsdisjoint}
 (\spt \MM')_t \cap (\spt \MM)_{t + \Delta T} = \emptyset \qquad \text{for all} \quad t \leq 0. 
 \end{equation}
 Claim~\ref{Cl_intersection} shows that $\Delta T_+ < \infty$.

\begin{Claim}
$\Delta T_+ = 0$.
\end{Claim}

\begin{proof}
Suppose by contradiction that $\Delta T_+ > 0$.
By the minimal choice of $\Delta T_+$ we can find a sequence $\Delta T_i \nearrow \Delta T_+$ such that there are times $t_i \leq 0$ with
\[ \bp_i \in (\spt \MM')_{t_i} \cap (\spt \MM)_{t_i + \Delta T_i} \neq \emptyset. \]
Write $\bp_i = (\bp'_i, \bp''_i) \in \IR^{k-1} \times \IR^{n-k+1}$ and $\MM^{\prime, (i)} := \MM' - ((\bp'_i, \bO)+ t_i \mathbf e_k, t_i) $.
Recall that $\MM = \MM - ((\bp'_i, \bO) + t_i \mathbf e_k, t_i )$.
So $\MM^{\prime, (i)}$ still satisfies \eqref{eq_sptsdisjoint} for all $\Delta T > \Delta T_+$, but we have
\[ \ov\bp_i := (\bO, \bp''_i- t_i \mathbf e_k) \in (\spt \MM^{\prime, (i)})_0 \cap (\spt \MM)_{\Delta T_i}. \]
Claim~\ref{Cl_intersection} implies that $\ov\bp_i$ is uniformly bounded, so after passing to a subsequence, we may assume that $\ov\bp_i \to \ov\bp_\infty$.
After passing to another subsequence, we may also assume that $\MM^{\prime, (i)} \to \MM^{\prime, (\infty)}$ in the Brakke sense.
Again, $\MM^{\prime, (\infty)}$ still satisfies \eqref{eq_sptsdisjoint} for all $\Delta T > \Delta T_+$, but $\ov\bp_\infty \in (\spt \MM^{\prime, (\infty)})_0 \cap (\spt \MM)_{\Delta T_+}$.
So the strong avoidance principle \cite[Theorem~3.4]{Choi_Haslhofer_Hershkovits_White_22} (see also \cite[Section~14]{Chodosh_MCF_notes}) implies 
\begin{equation} \label{eq_MMinMMp}
 (\spt \MM)_{\Delta T_+} \subset (\spt \MM^{\prime, (\infty)})_0.
\end{equation}

It follows from Lemma~\ref{l:H^10} that there is a uniform $C > 0$ such that for each $i$ the set $(\spt \MM^{\prime, (i)})_0 \setminus (\IR^{k-1} \times \IB^{n-k+1}_{C})$ is contained in $\bigcup_{\bp \in (\spt \MM)_0} \ov{B}(\bp , C H^{10}(\bp))$.
So the same property must also hold for $(\spt \MM^{\prime, (\infty)})_0$.
As explained in the proof of Claim~\ref{Cl_intersection} the radius $H^{10}(\bp)$ of these balls decays faster than the separation between $(\spt \MM)_0$ and $(\spt \MM)_{\Delta T_+}$, which yields the desired contradiction.
\end{proof}

It follows that \eqref{eq_sptsdisjoint} holds for all $\Delta T > 0$.
Similarly, we can show that it holds for all $\Delta T < 0$.
As the time-slices of $\MM$ sweep out all of $\IR^{n+1}$, this implies $\spt \MM' \subset \spt \MM$, so since the latter is connected, we have $\spt \MM' = \spt \MM$.
Lastly, since $\Theta^{\MM'}(\infty) < 2$, all tangent flows are multiplicity one planes, which implies $\MM' = \MM$.
\end{proof}
\bigskip

\section{Flows with dominant quadratic mode} \label{sec_dom_quadratic}
\subsection{Overview and statement of the main results}
In this section we study flows with dominant quadratic mode.
Our discussion will serve as the foundation for a more detailed classification in \cite{Bamler_Lai_MCF2}. 
Among other things we will define an invariant $\Qu(\MM)$ called the \emph{quadratic mode at $-\infty$,} which characterizes the asymptotic behavior of the quadratic mode as $\tau \to -\infty$.
We establish several key properties of this invariant and show that for every given value of $\Qu$, there exists a rotatinally symmetric flow $\MM$ realizing it.

For the remainder of this section we fix $1 \leq k < n \leq n$ and $n' \geq 0$ and we will omit dependencies on these constants.

We begin by giving a general characterization of flows with a dominant quadratic mode, which follows directly from Proposition~\ref{Prop_PO_ancient}.

\begin{Proposition}[Asymptotics in the case of dominant quadratic mode] \label{Prop_dom_qu_asymp}
Let $\MM$ be an asymptotically $(n,k)$-cylindrical mean curvature flow in $\IR^{n+1+n'} \times (-\infty,T)$ with dominant quadratic mode.
For any integers $J,  m \geq 10$ and $0 < \eta \leq \frac1{10}$ there is a constant $\delta (J,  m,\eta)> 0$ such that if $\MM$ is $\delta$-close to $M_{\cyl}$ at time $-r_0^2$ and scale $r_0$, then the following is true.
Consider the rescaled (but unmodified!) flow $\td\MM$; so $\td\MM^{\reg}_\tau = e^{\tau/2} \MM^{\reg}_{ - e^{-\tau}}$ over the time-interval $(-\infty, \td\tau]$ for $\td\tau := -2\log r_0$.
There is a smooth function
$$ U^+ = U_1 + U_{\frac12}  + U_0   + \ldots + U_{-J} : (-\infty, \td\tau) \lto  \sV_{ \rot, \geq - J}  = \sV_{\rot, 1} \oplus \sV_{\rot, \frac12} \oplus \ldots \oplus \sV_{\rot, -J}$$
such that if we set
\[ 
R(\tau) := J \sqrt{\log (\td\tau - \tau + 10)}, \]
then the following is true for  all $\tau \leq \td\tau$:
\begin{enumerate}[label=(\alph*)]
\item \label{Prop_dom_qu_asymp_a}  There is a smooth function $u_\tau : \DD_\tau \to \IR^{1+n'}$ with $\IB^k_{R(\tau)-1} \times \IS^{n-k} \subset \DD_\tau \subset \IB^k_{R(\tau)} \times \IS^{n-k}$ such that
\[ \Gamma_{\cyl}(u_\tau) = (\spt \td\MM)_\tau \cap \IB^{n+1+n'}_{R(\tau)} \subset \td\MM_\tau^{\reg} \]
and 
\begin{equation} \label{eq_utauUptau_qu}
   \Vert u_\tau - U^{+}(\tau) \Vert_{C^m (\DD_\tau)} \leq \eta  (\td\tau - \tau + 10)^{-J-1}. 
\end{equation}
\item  \label{Prop_dom_qu_asymp_b}  The evolution of $U^+$ is controlled by the following ODI
\[ \big\| \partial_\tau U^+ - L U^+  - Q_J^+ (U^+) \big\|_{L^2_{f}}
\leq  C(J)  (\td\tau - \tau + 10)^{-J-1}   \]
\item  \label{Prop_dom_qu_asymp_c} For $i=2,1, 0, -1, \ldots, -2J$
\[ \Vert U_{\frac12 i} \Vert_{L^2_f} \leq C(J)  (\td\tau - \tau + 10)^{-\lceil |i|/2 \rceil - 1}.  \]
Note that these bounds imply quadratic decay for $U_{1}$ and $U_{\frac12}$. 
We also have
\[ |U_{0,\min}|(\tau), \; |U_{0,\max}|(\tau) \leq \tfrac{1+\eta}{\sqrt 2} (\td\tau - \tau + 10)^{-1}. \]
\item   \label{Prop_dom_qu_asymp_d}
View $U_0(\tau) = \sum_{i,j} c_{ij}(\tau) \fp^{(2)}_{ij}$ as a time-dependent symmetric matrix, using the Hermite polynomials from \eqref{eq_Hermite}.
Then there is a maximal, non-zero solution $\ov U : (-\infty, \ov T) \to \IR^{k \times k}_{\leq 0}$, taking only non-positive definite values, to the ODE
\begin{equation} \label{eq_barU_ODE}
 \partial_\tau \ov U = - \sqrt{2} \ov U^2 + 2 \tr (\ov U^2) \ov U + C^* \ov U^3, 
\end{equation}
where $C^*(n-k) \in \IR$ is a dimensional constant, such that for $\tau \leq \td\tau - C$
\begin{equation} \label{eq_U0_close_ovU}
 \big\| U_0(\tau) - \ov{U}(\tau) \big\| \leq C (\td\tau - \tau+10)^{-3} . 
\end{equation}
Moreover, there is a matrix $\mathsf A \in \IR^{k \times k}_{\geq 0}$ and a constant $C^{**}(n,k,\rank \mathsf A) \in \IR$ such that $\ov U(\tau)$ and $\mathsf A$ have the same nullspace for all $\tau$ and on the range of $\mathsf A$ we have the following asymptotics for $\ov U(\tau)$:
\begin{equation} \label{eq_Ubar_asymp}
 \ov U(\tau) = \bigg( \log(\mathsf A) + \big( \sqrt 2 \, \tau + C^{**} \log (-\tau) \big) \mathsf I_k \bigg)^{-1} + O(|\tau|^{-3} \log |\tau|). 
\end{equation}
Here $\mathsf I_k$ denotes the identity matrix and the error term may not uniform in $J, m, \eta$.
\end{enumerate}
\end{Proposition}
\medskip

The following lemma gives us a convenient way of characterizing the solution $\ov U$ from Assertion~\ref{Prop_dom_qu_asymp_d} via a non-negative definite matrix.

\begin{Lemma} \label{Lem_ODE_UQ}
Consider the set $\mathfrak U^k$ of maximal solutions $\ov U : (-\infty, \ov T) \to \IR^{k \times k}_{\leq 0}$ to the ODE \eqref{eq_barU_ODE} with $\lim_{\tau \to -\infty} \ov U (\tau) = 0$.
There is a well-defined topology on $\mathfrak U^k$ such that convergence within $\mathfrak U$ is equivalent to pointwise convergence.
There are homeomorphisms $\Qu_k : \mathfrak U^k \to \IR^{k \times k}_{\geq 0}$, for all $k \geq 0$, such that the following is true for all $\ov U \in \mathfrak U$:
\begin{enumerate}[label=(\alph*)]
\item \label{Lem_ODE_UQ_a} For any orthogonal matrix $\mathsf S \in \IR^{k \times k}$ we have $\Qu_k (\mathsf S^T \ov U \mathsf S) = \mathsf S^T \Qu(\ov U) \mathsf S$.
\item \label{Lem_ODE_UQ_b} The nullspaces of $\ov U(\tau)$ (for any $\tau$) and $\Qu_k(\ov U)$ are the same.
\item \label{Lem_ODE_UQ_c} If we regard $\ov U \in \mathfrak U^k$ as an element $\ov U' \in \mathfrak U^{k'}$ for $k' > k$ by extending $\ov U(\tau)$ with zero entries, then $\Qu_{k'}(\ov U')$ is obtained from $\Qu_k(\ov U)$ by extending it with zero entries in the same manner.
\item \label{Lem_ODE_UQ_d} If $\ov U' (\tau) = \ov U (\tau - 2\log \la)$, then $\Qu_k(\ov U') = \la \Qu_k(\ov U)$.
\end{enumerate}
\end{Lemma}

The precise construction of $\Qu$ in Lemma~\ref{Lem_ODE_UQ} is somehwat ad hoc and not important for our discussion.
A more canonical definition may be obtained via an asymptotic expansion of $\ov U$, however, this would be more technical and the present construction is adequate for our purposes.
Using this identification, we can now define the following invariant of $\MM$.

\begin{Definition} \label{Def_Qu}
Let $\MM$ be an asymptotically $(n,k)$-cylindrical mean curvature flow in $\IR^{n + 1} \times (-\infty, T)$.
If $\MM$ has dominant quadratic mode, then in the context of Proposition~\ref{Prop_dom_qu_asymp}\ref{Prop_dom_qu_asymp_d} we define $\Qu(\MM) := \Qu_k(\ov U) \in \IR^{k \times k}_{\geq 0}$.
If $\MM$ has dominant linear mode or is a round shrinking cylinder, then we set $\Qu(\MM) := 0$.
We call $\Qu(\MM)$ the \textbf{quadratic mode at $-\infty$.}
\end{Definition}
\medskip

The next proposition summarizes basic properties of $\Qu(\MM)$, for example, its behavior under translation and parabolic rescalings.

\begin{Proposition} \label{Prop_Q_basic_properties}
If $\MM$ is an asymptotically cylindrical flow in $\IR^{n+1+n'} \times (-\infty, T)$, then $\Qu(\MM)$ is well defined and does not depend on the choices of the parameters $J, m, \eta, \td\tau$.
For any $(\mathbf v, \Delta T) \in \IR^{n+1+n'} \times \IR$ and $\la > 0$ we have
\begin{equation} \label{eq_Q_identitites}
 \Qu(\MM + (\mathbf v, \Delta T)) = \Qu(\MM), \qquad \Qu(\la \MM) = \la \Qu(\MM) . 
\end{equation}
Moreover if $\mathsf S : \IR^{n+1+k'} \to \IR^{n+1+k'}$ is an orthogonal map, which is block-form with respect to the splitting $\IR^{n+1+k'} = \IR^k \times \IR^{n-k+1} \times \IR^{n'+1}$ and if $\mathsf S_0 \in \IR^{k \times k}$ is its first block, then
\begin{equation} \label{eq_Q_identity_2}
 \Qu( \mathsf S \MM ) = \mathsf S_0^T \Qu(\MM) \mathsf S_0. 
\end{equation}
Lastly, $\Qu(\MM) = 0$ if and only if $\MM$ has dominant linear mode (and hence is homothetic to $\MM^{n-k+1}_{\bowl} \times \IR^{k-1}$) or is a round shrinking cylinder.
\end{Proposition}
\medskip

The next proposition shows that the quadratic mode at $-\infty$ determines the asymptotic behavior of the flow as $\tau \to -\infty$ to arbitrarily high polynomial order.

\begin{Proposition} \label{Prop_same_Q_close}
For any integers $J,  m \geq 10$ and $0 < \eta \leq \frac1{10}$ there is a constant $\delta (J,  m,\eta)> 0$ with the following property.
Let $\MM^{(0)}$ and $\MM^{(1)}$ be two asymptotically $(n,k)$-cylindrical mean curvature flows in $\IR^{n+1+n'} \times (-\infty,T)$ with $\Qu(\MM^{(0)}) = \Qu(\MM^{(1)})$.
Suppose that for some $r_0 > 0$ the flows $\MM^{(0)}, \MM^{(1)}$ are $\delta$-close to $M_{\cyl}$ at scale $r_0$ and time $-r_0^2$.

Let $u^{(i)}_\tau : \DD_\tau \to \IR^{1+n'}$ be the corresponding functions from Proposition~\ref{Prop_dom_qu_asymp}, representing the rescaled flows $\td\MM^{(i),\reg}_\tau = e^{\tau/2} \MM^{(i),\reg}_{t_0  - e^{-\tau}}$, for $\tau \leq \td\tau := -2\log r_0$, as a graphs over the round cylinder.
Then
\[ \Vert u^{(1)}_\tau - u^{(0)}_\tau \Vert_{C^m(\IB^k_{R(\tau)-1} \times \IS^{n-k})} \leq C(J,m) (\td\tau - \tau + 10)^{-J}, \]
where $R(\tau) := J \sqrt{\log (\td\tau - \tau + 10)}$ as before.
\end{Proposition}
\medskip

In the next lemma we study the behavior of $\Qu(\MM)$ under Brakke convergence.
It states that $\Qu$ is continuous and almost proper.

\begin{Proposition} \label{Prop_Q_continuous}
Let $\MM^{(i)}$ be a sequence of asymptotically $(n,k)$-cylindrical mean curvature flows in $\IR^{n+1+n'} \times (-\infty, T_i)$.
Then the following is true:
\begin{enumerate}[label=(\alph*)]
\item \label{Prop_Q_continuous_a} If $\MM^{(i)} \to \MM^{(\infty)}$ in the Brakke sense and if the limit $\MM^{(\infty)}$ is also an asymptotically $(n,k)$-cylindrical mean curvature flow, then $\Qu(\MM^{(i)}) \to \Qu(\MM^{(\infty)})$.
\item \label{Prop_Q_continuous_b} Suppose that $n' = 0$ and $\Vert \Qu(\MM^{(i)}) \Vert$ is uniformly bounded from above.
Then, after passing to a subsequence, we have convergence $\MM^{(i)} \to \MM^{(\infty)}$ in the Brakke sense and the limit $\MM^{(\infty)}$ is one of the following:
\begin{itemize}
\item an asymptotically $(n,k)$-cylindrical mean curvature flow,
\item an affine plane, or
\item empty.
\end{itemize}
\item \label{Prop_Q_continuous_c} If $n'=0$ and $\Qu(\MM^{(i)}) \to 0$, then after passing to a subsequence, we have convergence $\MM^{(i)} \to \MM^{(\infty)}$ in the Brakke sense and the limit $\MM^{(\infty)}$ is one of the following:
\begin{itemize}
\item a flow homothetic to $\MM_{\bowl} \times \IR^{k-1}$,
\item a round shrinking cylinder,
\item an affine plane, or
\item empty.
\end{itemize}
\end{enumerate}
\end{Proposition}
\medskip

Lastly, we show that every quadratic mode at $-\infty$ can be realized by a convex, rotationally symmetric and non-collapsed flow with $\IZ^k_2$-symmetry.
Such flows were already constructed in \cite{DuHaslhofer2021} (see also \cite{HaslhoferHershkovits2016,White_03,HoffmanIlmanenMartinWhite2019} for related work), but their parameterization was expressed in terms of a different quantity.
The following theorem establishes the existence in the context of our quantity $\Qu$.

\begin{Theorem} \label{Thm_existence_oval}
Let $\Qu' \in \IR^{k \times k}_{\geq 0}$ be a symmetric, non-negative definite matrix.
Then there is an $(n,k)$-cylindrical mean curvature flow in $\IR^{n+1} \times (-\infty, T)$ with $\Qu(\MM) = \MM'$.
Moreover, $\MM$ is convex, non-collapsed, rotationally symmetric and invariant under reflections perpendicular to every spectral direction of $\Qu'$.
It goes extinct at time $0$ and has uniformly bounded second fundamental form on time-intervals of the form $(-\infty,T]$ for $T < 0$.
Moreover, $\MM$ splits as a product of a compact flow, which becomes extinct in a round singularity, with the nullspace of $\Qu'$.
\end{Theorem}
\medskip

The proof of Theorem~\ref{Thm_existence_oval} uses a general strategy of fixed point theorem in the context of a new topological framework.
\bigskip

\subsection{Proof of Proposition~\ref{Prop_dom_qu_asymp}}

\begin{proof}[Proof of Proposition~\ref{Prop_dom_qu_asymp}.]
Observe that the proposition is invariant under parabolic scaling.
Specifically, rescaling $\MM$ and $r_0$ by a factor $\la> 0$ changes $\td\tau$ by an additive constant $2\log \la$ and $U^+, \ov U$ are shifted by the same constant so that the bounds in Assertions~\ref{Prop_dom_qu_asymp_a}--\ref{Prop_dom_qu_asymp_d} remain preserved.
So we may assume in the following without loss of generality that $\td\tau = -10$, resulting in the simplification $\td\tau - \tau + 10 = -\tau$.

Let $\eta' \in (0, \eta)$ be a constant whose value we will determine later and apply Proposition~\ref{Prop_PO_ancient} for $J \leftarrow 100 J^2$, $m \leftarrow m+1$, $\eta \leftarrow \eta'$ and $(\bq_0 , t_0 ) \leftarrow (\bO,0)$.
Call the resulting functions $U^{\prime,+} : (-\infty, \td\tau] \to \sV_{\rot, \geq -J-1}$ and $R' : (-\infty, \td\tau] \to \IR$ and express $e^{\tau/2} (\MM^{\reg}_{-e^{-\tau}} - \bq)$ intersected with $\IB^{n+1+n'}_{R'(\tau)}$ as the graph of a function $u'_\tau$ whose domain contains $\IB^k_{R'(\tau)-1} \times \IS^{n-k}$.
Recall that $|\bq| \leq \eta' e^{\td\tau/2}$.
By Lemma~\ref{Lem_up_minus_u}, we can express the subset of $e^{\tau/2} \MM^{\reg}_{-e^{-\tau}}$ described by $u'_\tau$ as the graph of a function $u_\tau$ whose domain contains $\IB^k_{\frac12 (R'(\tau)-1)} \times \IS^{n-k}$ and over this domain we have
\begin{equation}  \label{eq_uptauutau_qu}
\Vert u'_\tau - u_\tau \Vert_{C^m} \leq C \eta' e^{\tau/2}. 
\end{equation}
Let $U^{+} := \PP_{\sV_{\rot, \geq -J}} U^{\prime,+}(\tau)$ and $U^{++}(\tau) := \PP_{\sV^{++}} U^{\prime,+}(\tau)$.

\begin{Claim} \label{Cl_Upp_tau_m1}
If $\eta' \leq \ov\eta' (J, m, \eta)$, then for some dimensional constant $C$,
\[ \Vert U^{++} (\tau) \Vert \leq C  \min\{|\tau|^{-1}, \eta' \}, \qquad 
|U_{0,\min}|(\tau), \; |U_{0,\max}|(\tau) \leq \frac{1+\eta}{\sqrt 2} |\tau|^{-1}  . \]
\end{Claim}

\begin{proof}
We use Proposition~\ref{Prop_PO_ancient}\ref{Prop_PO_ancient_d} for $\eps = 0.1 \eta \sqrt{2}$.
Integrating the differential bounds on $( \tau_{\frac12}, -10]$ implies that $\Vert U^{++} (\tau) \Vert \leq C e^{\frac12 \tau} \leq C |\tau|^{-1}$, which in combination with $\Vert U^{++} \Vert \leq C \eta'$ implies the first asserted bound on this time-interval.
The second asserted bound follows similarly as long as $\tau$ is sufficiently small.
If $\tau$ is bounded from below, then it follows from the first bound for small $\eta'$.
On the time-interval $(-\infty, \tau_{\frac12}]$ we have $\partial_\tau |U_{0,\min}|^{-1}(\tau) \leq - (1- 0.1 \eta) \sqrt 2$, so 
\[ |U_{0,\min}|^{-1}(\tau) \geq C^{-1} e^{-\frac12 \tau_{\frac12}} +  (1- 0.1 \eta) \sqrt 2 (\tau_{\frac12} - \tau) \geq -(1- 0.1 \eta) \sqrt 2 \, \tau -C , \]
 which implies the asserted bounds for sufficiently small $\tau \in (-\infty, \tau_{\frac12}]$.
 If $\tau$ is bounded from below, then these bounds hold trivially for a sufficiently small choice of $\eta'$.
\end{proof}

Recall that 
\[ e^{-(R'(\tau))^2} = \| U^{++}(\tau) \|^{100 J^2} \leq C |\tau|^{-100 J^2}. \]
So there is a uniform $T^* \leq -10$ such that for $\tau \leq T^*$ we have
\[ R'(\tau) \geq 10 J \log {\log|\tau| + C } \geq 2 J \sqrt{\log |\tau|} = 2 R(\tau). \]
On the other hand, the same bound holds for $\tau \in [T^*,-10]$ as long as we choose $\eta'\leq\ov\eta'(J,T^*)$.
Combining this with the claim and \eqref{eq_uptauutau_qu} implies that for a domain $\DD_\tau$ as specified in Assertion~\ref{Prop_dom_qu_asymp_a}
\begin{equation} \label{eq_utauUppCJM} 
\Vert u_\tau - U^{\prime,+}(\tau) \Vert_{C^m} \leq C(J,m) |\tau|^{-100 J^2 - 1}. 
\end{equation}
On the other hand, Lemma~\ref{Lem_polynomial_bounds} and Proposition~\ref{Prop_PO_ancient}\ref{Prop_PO_ancient_b} imply that for $\eta' \leq \ov\eta'(J,m,\eta)$
\begin{align*} 
\Vert U^{\prime,+}(\tau) - U^{+}(\tau) \Vert_{C^m(\DD)} 
&\leq C(J,m) R^{C(J,m)}(\tau) \Vert U^{\prime,+}(\tau) - U^{+}(\tau) \Vert_{L^2_f} \displaybreak[1] \\
&\leq C(J,m) R^{C(J,m)}(\tau) \Vert U^{++}(\tau) \Vert_{L^2_f}^{J+2} \displaybreak[1] \\
&\leq C(J,m) (\log|\tau|)^{C(J,m)} |\tau|^{-J-1.5} \| U^{++}(\tau) \|_{L^2_f}^{1/2}\displaybreak[1] \\
&\leq C(J,m) (\eta')^{1/2} |\tau|^{-J-1} 
\leq \eta |\tau|^{-J-1}.
\end{align*}
Combining this with \eqref{eq_utauUppCJM} implies 
Assertion~\ref{Prop_dom_qu_asymp_a}.

Assertion~\ref{Prop_dom_qu_asymp_b} now follows immediately from Proposition~\ref{Prop_PO_ancient}\ref{Prop_PO_ancient_b} combined with Claim~\ref{Cl_Upp_tau_m1}.
The bounds in Assertion~\ref{Prop_dom_qu_asymp_c}, for $i \geq 0$, follow similarly from Proposition~\ref{Prop_PO_ancient}\ref{Prop_PO_ancient_c}.
Consider now the case $i = 1, 2$ and observe that by Assertion~\ref{Prop_dom_qu_asymp_b} we have
\[ \big\Vert \partial_\tau U^+_{\frac{1}2 i} - \tfrac{i}2 U^+_{\frac{1}2 i} - \PP_{\sV_{\frac12 i}} Q^+_J(U^+) \big\Vert \leq C(J) |\tau|^{-J-1}, \] 
which implies that
\[ \partial_\tau \Vert U^+_{\frac{1}2 i} \Vert \geq \tfrac{i}2 \Vert U^+_{\frac{1}2 i} \Vert - C |\tau|^{-2} \qquad \Rightarrow \qquad
\partial_\tau \big( e^{-i/2 \cdot \tau} \Vert U^+_{\frac{1}2 i} \Vert \big) \geq -C |\tau|^{-2}e^{-i/2 \cdot \tau}. \]
Integrating this backwards in time implies that
\begin{multline*}
 \Vert U^+_{\frac{1}2 i} (\tau) \Vert 
\leq C e^{i/2 \cdot \tau} +  C e^{i/2 \cdot \tau} \int_{\tau}^{-10} |\tau'|^{-2}e^{-i/2 \cdot \tau'} d\tau' \\
\leq C|\tau|^{-2} +  C |\tau|^{-2} \int_{0}^{|\tau|-10} \bigg|\frac{\tau}{\tau+s}\bigg|^2 e^{-s/2} ds \leq C|\tau|^{-2}. 
\end{multline*}
The last bound follows by splitting the last integral into integrals over $[0,\frac12 |\tau|]$ and $[\frac12|\tau|, |\tau|-10]$.
This concludes the proof of Assertion~\ref{Prop_dom_qu_asymp_c}.

Assertion~\ref{Prop_dom_qu_asymp_d} is a consequence of the next lemma.
To see that the constant $C$ in \eqref{eq_U0_close_ovU} is independent of $J$ note that we can apply the proposition for different choices $J_1, J_2$ of $J$, resulting in different functions $U^+$, which must differ by at most $C(\min\{ J_1, J_2 \}) |\tau|^{-\min\{ J_1, J_2 \}}$.
\end{proof}
\medskip

\begin{Lemma} \label{Lem_ODE_neutral_mode}
Consider a solution $U^+ : (-\infty, -10] \to \sV_{ \rot,\geq - 1}$ to the ODI
\begin{equation} \label{eq_ODI_J3}
 \big\| \partial_\tau U^+ - L U^+  - Q_3^+ (U^+) \big\|
\leq   C_0 (- \tau )^{-4} 
\end{equation}
such that\footnote{The second bound in \eqref{eq_Uptaum1} could be omitted, but it is available in the application of Lemma~\ref{Lem_ODE_neutral_mode} and hence simplifies the proof.}
\begin{equation} \label{eq_Uptaum1}
 \Vert U^+ \Vert \leq C_0 |\tau|^{-1}, \quad \| U^+ - U_0 \| \leq C_0 |\tau|^{-2}, \quad |U_{0,\min}|(\tau), \; |U_{0,\max}|(\tau) \leq  \frac{1.1}{\sqrt{2}}  |\tau|^{-1} , 
\end{equation}
where $C_0 > 0$.
Then $U_0 = \PP_{\sV_0} U^+$ satisfies Assertion~\ref{Prop_dom_qu_asymp_d} in Proposition~\ref{Prop_dom_qu_asymp}, where the constant $C$ in \eqref{eq_U0_close_ovU} can be chosen depending only on $n, k$ and $C_0$.
\end{Lemma}

\begin{proof}
In the following, $C$ will denote a generic constant that may depend on $n, k, C_0$ and we will write $U^+ = U_{1} + U_{\frac12} + U_0 + U_{-\frac12} + U_{-1} \in \sV_{\rot, 1} \oplus \ldots \oplus \sV_{\rot,-1}$.
We first estimate $U_{\frac{i}2}$ for $i \neq 0$ in terms of $U_0$ via the following claim.

\begin{Claim}
Consider the degree 2 homogeneous polynomials of $Q'_1,  Q'_{\frac12},  Q'_{-\frac12}, Q'_{-1} :  \sV_{\rot,0} \to \sV_{\rot,0}$ defined by
\[ Q'_{\frac{i}2} (  U' )  := - \tfrac{2}i \PP_{\sV_{\frac{i}2}} \big( Q_2^+ (  U' ) \big). \]
An explicit formula for these functions can be obtained from \eqref{eq_Q2_V1}--\eqref{eq_Q2_Vm1}.
For $i \in \{ -2, -1, 1, 2 \}$ we have
\begin{equation} \label{eq_Ui2QpU0}
 \big\| U_{\frac{i}2}  - Q'_{\frac{i}2} \big(  U_0) \big) \big\| \leq C |\tau|^{-3}. 
\end{equation}
Moreover, if we set $Q'( U') := Q'_1 ( U') + Q'_{\frac12} (U') + U' + Q'_{-\frac12}(U') + Q'_{-1}(U')$, then
\begin{equation} \label{eq_UmQp}
 \Vert U^+ - Q'(U_0) \Vert \leq C  |\tau|^{-3}. 
\end{equation}
\end{Claim}

\begin{proof}
Set $X_i := U_{\frac{i}2}  - Q'_{\frac{i}2} (  U_0)$ and note that
\begin{align*}
 \partial_\tau X_i
&= \tfrac{i}2    U +\PP_{\sV_{\frac{i}2}} \big( Q_3^+ (U^+) \big) - \partial_\tau \big( Q'_{\frac{i}2} ( U_0)  \big) + \PP_{\sV_{\frac{i}2}} \big( \partial_\tau U^+ - LU^+ - Q_3^+(U^+) \big)  \\
&= \tfrac{i}2 X_i  + \PP_{\sV_{\frac{i}2}} \big( Q_3^+ (U^+) - Q_2^+ (U_0 ) \big)  - \partial_\tau \big( Q'_{\frac{i}2} (  U_0)  \big)  + \PP_{\sV_{\frac{i}2}} \big( \partial_\tau U^+ - LU^+ - Q_3^+(U^+) \big).
\end{align*}
We estimate that
\begin{align*}
 \big\| \PP_{\sV_{\frac{i}2}} \big( Q_3^+ (U) - Q_2^+ (U_0 ) \big) \big\|
&\leq \| Q_3^+ (U) - Q_2^+ (U) \| +  \| U  - U_0 \| \, \| U \| \leq C |\tau|^{-3} \\
 \big\| \partial_\tau \big( Q'_{\frac{i}2} (  U_0 ) \big) \big\| &\leq C \Vert \partial_\tau U_0 \Vert \, \Vert U_0  \Vert
\leq C |\tau|^{-3}, \\
\big\|  \PP_{\sV_{\frac{i}2}} \big( \partial_\tau U^+ - LU^+ - Q_3^+(U^+) \big) \big\| &\leq
\|  \partial_\tau U^+ - LU^+ - Q_3^+(U^+) \big\| \leq C|\tau|^{-3}.
\end{align*}
Hence
\[ \big| \partial_\tau \big( e^{-\frac{i}2 \tau} \Vert X_i(\tau) \Vert \big) \big|
\leq C e^{-\frac{i}2 \tau} |\tau|^{-3}. \]
If $i < 0$, then this implies that for $\tau^* < \tau \leq \tau_2$
\[ \Vert X_i (\tau) \Vert \leq e^{\frac{i}2 (\tau - \tau^*)} \Vert X_i(\tau^*) \Vert + C \int_{\tau^*}^\tau |\tau'|^{-3} e^{\frac{i}2 (\tau- \tau')}  d\tau' . \] 
Letting $\tau^* \to -\infty$ implies
\[ \Vert X_i(\tau) \Vert \leq  C \int_{-\infty}^\tau |\tau'|^{-3} e^{\frac{i}2 (\tau- \tau')}  d\tau'
\leq  C|\tau|^{-3} \int_{-\infty}^\tau  e^{\frac{i}2 (\tau- \tau')}  d\tau' \leq  C|\tau|^{-3}. \]
On the other hand, if $i > 0$, then
\begin{align*}
 \Vert X_i(\tau) \Vert 
 &\leq e^{\frac{i}2 (\tau + 10)} \Vert X_i(-10) \Vert + C \int^{-10}_\tau |\tau'|^{-3} e^{\frac{i}2 (\tau- \tau')}  d\tau'  \\
 &\leq Ce^{\tau/2}   + C \int^{\tau/2}_\tau |\tau'|^{-3} e^{\frac{i}2 (\tau- \tau')}  d\tau'  + C \int^{-10}_{\tau/2} |\tau'|^{-3} e^{\frac{i}2 (\tau- \tau')}  d\tau' \\
  &\leq C|\tau|^{-3}   + C |\tau|^{-3} \int^{\tau/2}_\tau  e^{\frac{i}2 (\tau- \tau')}  d\tau'  + C \int^{-10}_{\tau/2}  e^{\frac{i}2 (\tau- \tau')}  d\tau' \\
 &\leq C|\tau|^{-3}.
\end{align*}
This proves \eqref{eq_Ui2QpU0}.
The bound \eqref{eq_UmQp} follows immediately.
\end{proof}

The next claim characterizes the evolution of $U_0$ to higher order.

\begin{Claim} \label{Cl_ovQ}
There is a constant $C^*(n-k) \in \IR$ with the following property.
Identify $\sV_{\rot, 0} \cong \IR^{k \times k}_{\sym}$ by expressing elements $U' \in \sV_{\rot, 0}$ as $U'= \sum_{i,j} c_{ij} \mathfrak p_{ij}^{(2)}$ for symmetric matrices $(c_{ij})$ and define the map $\ov Q : \sV_{\rot, 0} \to \sV_{\rot, 0}$ by
\[ \ov Q(U') := -\sqrt{2}  U^{\prime 2} + 2 \tr (\ov U^{\prime 2})  U' + C^* \ov U^{\prime 3} . \] 
Then the following is true:
\begin{enumerate}[label=(\alph*)]
\item\label{Cl_ovQ_a} $\Vert  \ov Q(U') - \PP_{\sV_0}(Q_3^+ (Q'(U')))  \Vert \leq C \Vert U' \Vert^4$ for all $ U' \in \sV_{\rot,0}$ with $\Vert  U' \Vert \leq 1$.
\item \label{Cl_ovQ_b} $\Vert \partial_\tau U_0 - \ov Q(U_0) \Vert \leq C |\tau|^{-4}$.
\end{enumerate}
\end{Claim}

\begin{proof}
For Assertion~\ref{Cl_ovQ_a} we may assume without loss of generality that $U' = \sum_i c_{ii} \mathfrak p^{(2)}_{ii}$ is in diagonal form.
Then using \eqref{eq_Q2_V1}--\eqref{eq_Q2_Vm1}, we obtain that
\[ Q'(U') =  \tfrac12 \sum_i c_{ii}^2 \mathfrak p^{(0)} + \sum_i c_{ii} \mathfrak p^{(2)}_{ii} -\tfrac12 \sum_{i \neq j} c_{ii} c_{jj} \mathfrak p_{ii}^{(2)} \mathfrak p_{jj}^{(2)} - \tfrac{\sqrt{6}}2 \sum_i c_{ii}^2 \mathfrak p^{(4)}_{iiii}.  \]
So by \eqref{eq_Q3_V0} we have for some $C^*(n-k) \in \IR$
\begin{align*}
 \PP_{\sV_0} \big( Q_3^+(Q'(U')) \big)
&= \sum_i \Big( - \sqrt{2} c_{ii}^2
- \tfrac12 \sum_j c_{jj}^2 c_{ii}
 + \tfrac32 \sum_j c_{jj}^2 c_{ii} + C_1(n-k) c_{ii}^3 \\
&\qquad\qquad + \sum_{j\neq i} c_{jj} c_{jj} c_{ii} 
 - \tfrac{\sqrt{6}}2 C_2 c_{ii}^3 \Big) \mathfrak p^{(2)}_{ii} + O( \Vert U' \Vert^4 ) \\
&= \sum_i \Big( - \sqrt{2} c_{ii}^2
+2 \sum_j c_{jj}^2 c_{ii} + C^*(n-k) c_{ii}^3 \Big) \mathfrak p_{ii}^{(2)} + O( \Vert U' \Vert^4 ) \\
&= \ov Q (U') + O( \Vert U' \Vert^4 ). 
\end{align*}
For Assertion~\ref{Cl_ovQ_b}, we combine this with the bounds \eqref{eq_ODI_J3}, \eqref{eq_UmQp} and \eqref{eq_Uptaum1} to obtain
\begin{align*}
 \Vert \partial_\tau U_0 - \ov Q (U_0) \Vert
&\leq 
\big\Vert \PP_{\sV_0} \big( \partial_\tau U^+ - LU^+ - Q_3^+ (U^+) \big) \big\Vert
+ \big\Vert \PP_{\sV_0} \big( Q_3^+ (U^+) -  Q_3 (Q'(U_0))\big)  \big\Vert \\
&\qquad + \big\Vert \PP_{\sV_0} \big(   Q_3 (Q'(U_0)) - \ov Q(U_0) \big)  \big\Vert \\
&\leq C_0 |\tau|^{-4} + C|\tau|^{-4} + C \Vert U_0 \Vert^4  \leq C|\tau|^{-4}.
\end{align*}
This finishes the proof of the claim.
\end{proof}

Fix a sequence $\tau_i \to \infty$ and consider maximal solutions $\ov U_{i} : [\tau_i, \ov T_i) \to \IR^{k \times k}_{\sym}$ to the ODE $\partial_\tau \ov U_i = \ov Q(\ov U_i)$ with initial condition $\ov U_i (\tau_i) = U_0(\tau_i)$.
Choose $\tau^*_i \in [\tau_i, \ov T_i]$ maximal that we have the operator norm bound $\Vert \ov U_i (\tau) \Vert_{op} \leq \frac{1.2}{\sqrt{2}}|\tau|^{-1}$ for all $\tau \in (\tau_i, \tau^*_i]$.
Due to the second bound in \eqref{eq_Uptaum1} we must have $\tau_i^* > \tau_i$ for sufficiently large $i$.
We have the following bound on $[\tau_i, \tau_i^*]$
\begin{align}
 \partial_\tau \Vert U_0 - \ov U_{i} \Vert_{op}
&\leq  \Vert \ov Q(U_0) - \ov Q(\ov U_i) \Vert_{op} + \Vert \partial_\tau U_0 - \ov Q(U_0) \Vert_{op} \notag \\
&\leq \sqrt{2} \Vert  U_0^2 - \ov U_i^2 \Vert_{op} 
+ \Vert  U_0 * U_0 * U_0 - \ov U_i * \ov U_i * \ov U_i  \Vert_{op} 
 +C |\tau|^{-4} \notag \\
 &\leq \tfrac12 \sqrt{2} \big\Vert (U_0 + \ov U_i ) (U_0 - \ov U_i ) \big\Vert_{op}
 + \tfrac12 \sqrt{2} \big\Vert (U_0 - \ov U_i ) (U_0 + \ov U_i ) \big\Vert_{op} \notag \\
 &\qquad
 + C |\tau|^{-2} \Vert U_0 - \ov U_i \Vert_{op}  +C|\tau|^{-4} \notag \\
 &\leq \sqrt{2} \Vert U_0 -\ov U_i \Vert_{op} \Vert U_0 + \ov U_i \Vert_{op}+ C |\tau|^{-2} \big\Vert U_0 - \ov U_i \big\Vert_{op}  + C|\tau|^{-4} \notag \\
 &\leq \bigg( \sqrt{2} \cdot 2 \cdot  \frac{1.2 }{\sqrt{2}} + C|\tau|^{-1} \bigg) |\tau|^{-1} \Vert U_0 -\ov U_i \Vert_{op}   + C|\tau|^{-4}. \label{eq_dtau_U_0Ui} 
\end{align}
Choose $T = T(n,k, C_0) \geq 10$ such that the term in the parentheses is $\leq 2.3$ for $\tau \in [\tau_i, \min\{ \tau^*_i, -T \}]$, so for such $\tau$ we have
\[ \partial_\tau \big( |\tau|^{2.3} \Vert U_0(\tau) - \ov U_i (\tau) \Vert \big) \leq C |\tau|^{-1.7}. \]
Integrating this starting from $\tau_i$ implies that for all $\tau \in [\tau_i, \min\{ \tau^*_i, -T \}]$
\begin{equation} \label{eq_U0ovUi}
 \Vert U_0 (\tau) - \ov U_i (\tau) \Vert \leq C |\tau|^{-3}. 
\end{equation}

\begin{Claim}
There is a $T' = T'(n,k, C_0) \geq T$ such that $\tau^*_i \geq - T'$ for large $i$.
\end{Claim}

\begin{proof}
Choose $T' >T$ large enough such that whenever $\tau \leq -T'$, then the right-hand side in \eqref{eq_U0ovUi} is $< \frac{0.1}{\sqrt 2} |\tau|^{-1}$.
Suppose that $\tau^*_i \leq T'$.
Then $\Vert \ov U_i (\tau^*_i) \Vert_{op} = \frac{1.2}{\sqrt 2} |\tau^*_i|^{-1}$.
On the other hand, by second bound in \eqref{eq_Uptaum1} we have $\Vert  U_0 (\tau^*_i) \Vert_{op} \leq \frac{1.1}{\sqrt 2} |\tau^*_i|^{-1}$.
So $\Vert U_0 (\tau^*_i) - \ov U_i (\tau^*_i) \Vert_{op} \geq \frac{0.1}{\sqrt 2} |\tau^*_i|^{-1}$, in contradiction to \eqref{eq_U0ovUi}.
\end{proof}

Letting $i \to -\infty$, we obtain convergence $\ov U_i \to \ov U$ such that $\Vert U_0 (\tau) - \ov U (\tau) \Vert_{op} \leq C |\tau|^{-3}$ for $\tau \leq - T'$.

\begin{Claim} \label{Cl_nonneg}
$\ov U$  takes only non-positive definite values and its nullspace is invariant in time.
\end{Claim}

\begin{proof}
We have $\Vert \ov U \Vert \leq \Vert U_0 \Vert + C |\tau|^{-3} \leq C |\tau|^{-1}$ for small enough $\tau$.
Since the ODE \eqref{eq_barU_ODE} preserves a spectral basis, it reduces to an ODE system for the spectral values $\la_i$, which is of the form
\begin{equation} \label{eq_laievol1}
 \partial_\tau \la_i = - \sqrt 2 \la_i^2 + 2 \tr(\ov U^2) \la_i + C^* \la_i^3  = - \sqrt 2 \la_i^2 + O(|\tau|^{-2} ) \la_i. 
\end{equation}
Since non-negativity is preserved under this ODE, for each $i$ we either have $\la_i(\tau) > 0$, $\la_i (\tau) < 0$ or $\la_i(\tau)$ for all $\tau$ .
Suppose by contradiction that the former is true for some $i$.

We first claim that $\liminf_{\tau \to -\infty} |\tau| \la_i(\tau) > c > 0$.
Suppose not.
Then for small enough $\tau$, we have the following bound whenever $|\tau| \la_i(\tau) \leq a \leq 1$
\[ \partial_\tau (|\tau| \la_i) = - \la_i - \sqrt 2 |\tau| \la_i^2 + O(|\tau|^{-2}) |\tau| \la_i < 0. \]
So $|\tau| \la_i(\tau) \leq a$ remains preserved.
Letting $a \to 0$ implies that $\la_i \equiv 0$, which contradicts our assumption.
Therefore, $\la_i(\tau) \geq c |\tau|^{-1}$ for $\tau \ll 0$.
Plugging this back into \eqref{eq_laievol1} implies that for some $C' > 0$ and $\tau \ll 0$
\[ \partial_\tau \la_i^{-1} = \sqrt 2 + O(|\tau|^{-2}) \la_i^{-1} \geq \sqrt 2 - C'|\tau|^{-2} \cdot c^{-1} |\tau| \geq \sqrt 2 - C' c^{-1} |\tau|^{-1} \geq 1, \]
which contradicts the positivity of $\la_i^{-1}$.
\end{proof}

It remains to derive the asymptotic characterization \eqref{eq_Ubar_asymp}.
By Claim~\ref{Cl_nonneg} it suffices to consider the case in which $\ov U$ has trivial nullspace, which allows us to rewrite the ODE as
\begin{equation} \label{eq_dtaubarUm1}
\partial_\tau \ov U^{-1} = \sqrt{2} \, \mathsf I_k - 2 \tr(\ov U^2) \ov U^{-1} - C^* \ov U. 
\end{equation}
So the smallest spectral value $\la_{\min}(\tau)$ of $\ov U(\tau)$ satisfies
\[ \partial_\tau \la_{\min}^{-1} = \sqrt{2}  - 2 \tr(\ov U^2) \la_{\min}^{-1} - C^* \la_{\min} \geq \sqrt{2}  - 2 k  \la_{\min} - C^* \la_{\min}. \]
So for sufficiently small $\tau$ we have $\partial_\tau \la_{\min}^{-1} \geq 1$, which implies that $\ov U (\tau)  = O(|\tau|^{-1})$.
On the other hand, the maximal spectral value satisfies for sufficiently small $\tau$
\[ \partial_\tau \la_{\max}^{-1} = \sqrt{2}  - 2 \tr(\ov U^2) \la_{\max}^{-1} - C^* \la_{\max} \geq \sqrt{2}  - C^* \la_{\max}. \]
Thus all spectral values of $\ov U^{-1}(\tau)$ must go to $-\infty$ linearly as $\tau \to -\infty$.
It follows that, implying
\[ \partial_\tau \ov U^{-1}(\tau) = \sqrt{2}  \, \mathsf I_k + O(|\tau|^{-1}). \]
Integrating this bound implies
\[ \ov U(\tau) = \tfrac1{\sqrt{2}} \tau^{-1} \, \mathsf I_k + O(|\tau|^{-2} \log |\tau|). \]
Plugging this back into \eqref{eq_dtaubarUm1} gives us for some dimensional constant $C^{**}$
\[ \partial_\tau \ov U^{-1}(\tau) = \big(\sqrt{2} + C^{**}\tau^{-1} \big) \mathsf I_k + O( |\tau|^{-2} \log |\tau|) . \]
Integrating this again implies that for some matrix $\mathsf B \in \IR^{k \times k}_{\sym}$
\[ \ov U^{-1}(\tau) = \mathsf B + \big( \sqrt 2 \tau + C^{**} \log (-\tau) \big) \mathsf I_k + O(|\tau|^{-1} \log |\tau|), \]
which implies
\[ U(\tau) = \bigg( \mathsf B + \big( \sqrt 2 \, \tau + C^{**} \log |\tau| \big) \mathsf I_k \bigg)^{-1} + O(|\tau|^{-3} \log |\tau|), \]
and thus \eqref{eq_Ubar_asymp} for $\mathsf A := \exp (\mathsf B)$.
\end{proof}
\bigskip

\subsection{Proof of Lemma~\ref{Lem_ODE_UQ}}

\begin{proof}[Proof of Lemma~\ref{Lem_ODE_UQ}.]
For simplicity, we will express the ODE \eqref{eq_barU_ODE} as $\partial_\tau U = Q(U)$, where $Q(U) := - \sqrt 2 U^2 + 2 \tr (U^2) U + C^* U^3$ and $U : (-\infty, T_U) \to \IR^{k \times k}_{\leq 0}$ is always assumed to be defined on a maximal time-interval.
Note that $C^*$ only depends on $n-k$; define $c := (10 C^*)^{-1}$.

\begin{Claim} \label{Cl_unique_tau_U}
For every $0 \neq U \in \mathfrak U^k$ there is a unique $\tau_U \in (-\infty, T_U)$ such that the smallest spectral value of $U(\tau_U)$ equals $-c$.
Moreover, the smallest spectral value for $U(\tau)$ is $> -c$ for all $\tau < \tau_U$ and $< -c$ for all $\tau > \tau_U$.
\end{Claim}

\begin{proof}
Since the ODE for $U$ preserves the spectral decomposition, we may assume without loss of generality that $U(\tau)$ is diagonal with entries $\la_i(\tau)$ for all $i$.
These entries satisfy the ODE
\[ \partial_\tau \la_i = -\sqrt 2 \la_i^2 + 2 \tr\Big( \sum_i \la_i^2 \Big) \la_i + C^* \la_i^3. \]
It follows that $\la_{\min} (\tau) := \min_i \la_i(\tau)$ satisfies the following bound whenever $\la_{\min}(\tau) \geq -c$
\begin{equation} \label{eq_partial_tau_la_min}
 \partial_\tau \la_{\min} \leq - \sqrt 2 \la_{\min}^2  + C^* \la_{\min}^3 \leq - (\sqrt 2 - 0.1 ) \la_{\min}^2 < 0. 
\end{equation}
It follows that $\la_{\min}^{-1}([-c,0))$ is a disjoint union of closed intervals that doesn't have right endpoints, so it must be either empty, all of $(-\infty, T_U)$ or of the form $(-\infty,\la_{\min}^{-1}(\{-c\})]$.
The first case cannot occur since $\la_{\min} (\tau) \to 0$ as $\tau \to -\infty$.
In the second case, we would have $\la_{\min}(\tau) \geq -c$ for all $\tau$ and $T_U = \infty$.
However, \eqref{eq_partial_tau_la_min} implies $\partial_\tau \la_{\min}^{-1} \geq \sqrt 2 - 0.1$, which contradicts this bound.
So the third case must occur, which shows the unique existence.
\end{proof}

Claim~\ref{Cl_unique_tau_U} implies that the map $\mathfrak U^k \to (-\infty,\infty]$, $U \mapsto \tau_U$ is continuous if we set $\tau_{U} := \infty$ for $U \equiv 0$.
If $U \equiv 0$, we define $\Qu_k(U) := 0$ and if $U \not\equiv 0$, then we set
\[ \Qu_k (U) := - c^{-1} e^{-\frac12 \tau_U} U(\tau_U) \in \IR^{k \times k}_{\geq 0}. \]
It is clear that $\Qu_k$ is continuous and that it satisfies the Assertions~\ref{Lem_ODE_UQ_a}--\ref{Lem_ODE_UQ_d}.
Moreover, $\Qu_k$ is invertible; its inverse function can be described as follows: If $0 \neq \Qu' \in \IR^{k \times k}_{\geq 0}$, let $a_{\Qu'} > 0$ be its largest spectral value.
Then $\Qu_k^{-1} (\Qu')$ is the unique solution $\partial_\tau U = Q(U)$ with the initial condition $U (-2\log a_{\Qu'}) = - c a_{\Qu'}^{-1} \Qu'$.
It is clear that $\Qu_k^{-1} : \IR^{k \times k}_{\geq 0} \to \mathfrak U^k$ extends to a continuous function with $\Qu_k^{-1}(0) = 0$.
\end{proof}
\bigskip

\subsection{Proof of Proposition~\ref{Prop_Q_basic_properties}}

We will use the following lemma.

\begin{Lemma} \label{Lem_ODE_sol_same}
Let $\ov U_1, \ov U_2$ be two solutions to the ODE \eqref{eq_barU_ODE} in Assertion~\ref{Prop_dom_qu_asymp_d} of Proposition~\ref{Prop_dom_qu_asymp}.
Suppose that $\ov U_1(\tau), \ov U_2 (\tau) \to 0$ as $\tau \to -\infty$ and that $\Vert \ov U_1(\tau) - \ov U_2(\tau) \Vert \leq C |\tau|^{-3}$ for $\tau \ll 0$.
Then $\ov U_1 = \ov U_2$.
\end{Lemma}

\begin{proof}
By \eqref{eq_partial_tau_la_min} in the proof of Claim~\ref{Cl_unique_tau_U} above, the smallest spectral value of $\ov U_1(\tau)$ or $\ov U_2(\tau)$ satisfies the following differential inequality for $\tau \ll 0$
\[ \partial_\tau \la_{\min}^{-1} \geq \sqrt 2 - 0.1. \]
Hence we have the following bound on the operator norm
\[ \Vert \ov U_i (\tau) \Vert_{op} \leq \frac1{\sqrt{2} - 0.1} |\tau|^{-1} \leq \frac{1.2}{\sqrt 2} |\tau|^{-2}. \]
We can now compute similarly as in \eqref{eq_dtau_U_0Ui} that for $\tau \ll 0$
\begin{align*}
 \partial_\tau \Vert \ov U_1 - \ov U_2 \Vert_{op}
&\leq \sqrt 2 \Vert \ov U_1^2 - \ov U_2^2 \Vert_{op} + C |\tau|^{-2} \Vert \ov U_1 - \ov U_2 \Vert_{op}  \\
&\leq \sqrt 2 \Vert \ov U_1 + \ov U_2 \Vert_{op} \Vert \ov U_1 - \ov U_2 \Vert_{op}  + C |\tau|^{-2} \Vert \ov U_1 - \ov U_2 \Vert_{op} \\
&\leq 2.4 |\tau|^{-1} \Vert \ov U_1 - \ov U_2 \Vert_{op}  + C |\tau|^{-2} \Vert \ov U_1 - \ov U_2 \Vert_{op} \\
&\leq 2.5 |\tau|^{-1} \Vert \ov U_1 - \ov U_2 \Vert_{op} ,
\end{align*}
which implies that for $\tau \ll 0$
\[ \partial_\tau \big( |\tau|^{2.5} \Vert \ov U_1 (\tau) - \ov U_2 (\tau) \Vert_{op} \big) \leq 0. \]
On the other hand, the term within the parentheses goes to $0$ as $\tau \to -\infty$, which implies that $\ov U_1 (\tau) = \ov U_2 (\tau)$ for $\tau \ll 0$, so $\ov U_1  = \ov U_2 $.
\end{proof}
\bigskip

\begin{proof}[Proof of Proposition~\ref{Prop_Q_basic_properties}.]
We first establish the well-definedness of $\Qu(\MM)$ and the first identity in \eqref{eq_Q_identitites} for $\Delta T = 0$, assuming that $\MM$ has dominant quadratic mode.
Set $\MM' := \MM + (\mathbf v, 0)$, where for the purpose of proving well-definedness we may set $\mathbf v = \bO$.
Apply Proposition~\ref{Prop_dom_qu_asymp} to both flows $\MM, \MM'$, possibly with different choices of parameters $(J,m, \eta, \td\tau)$ and $(J',m',\eta', \td\tau')$.
We will decorate the resulting objects for the second flow with a primes and we will show that $\ov U = \ov U'$.

Since the time-slices of the rescaled flows $e^{\tau/2} \MM_{- e^{-\tau}}$ and $e^{\tau/2} \MM'_{-e^{-\tau}}$ differ only by a translation by the vector $e^{\tau/2} \mathbf v$, we can use Lemma~\ref{Lem_up_minus_u} to conclude that for some uniform constant $C' > 0$
\[ \Vert u'_\tau - u_\tau \Vert_{C^{10}(\IB^k_1 \times \IS^{n-k})} \leq C' e^{\tau/2}. \]
So by Proposition~\ref{Prop_dom_qu_asymp}\ref{Prop_dom_qu_asymp_a} we have
\[ \Vert U^+(\tau) - U^{+,\prime}(\tau) \Vert \leq \tfrac1{10} (\td\tau - \tau + 10)^{-10} + \tfrac1{10} (\td\tau' - \tau + 10)^{-10} + C' e^{\tau/2} = O((-\tau)^{-10}). \]
So by \eqref{eq_U0_close_ovU} we obtain that
\[ \Vert \ov U (\tau) - \ov U'(\tau) \Vert \leq O((-\tau)^{-3}), \]
which implies $\ov U = \ov U'$ via Lemma~\ref{Lem_ODE_sol_same}.
This shows the first identity of \eqref{eq_Q_identitites} for $\Delta T = 0$ and establishes well-definedness of $\Qu(\MM)$ if $\MM$ has dominant quadratic mode.
If $\MM$ does not have dominant quadratic mode, then $\Qu(\MM) = 0$ and by switching the roles of $\MM$ and $\MM'$ in the previous argument, we find that $\MM'$ can't have dominant quadratic mode, so $\Qu(\MM') =0 = \Qu(\MM)$.

To finish the proof of the first identity in \eqref{eq_Q_identitites}, it suffices to consider the case $\MM' := \MM + (\mathbf v, \Delta T)$.
It is not hard to see that whenever $-e^{-\tau} + \Delta T = - e^{-\tau'}$, then the functions $u_{\tau}$ and $u_{\tau'}$ only differ by a constant function in the first component whose norm is $\leq C' e^{-\tau}$.
So as in the previous discussion, we obtain that 
\[ \Vert U^+(\tau) - U^{+,\prime}(\tau') \Vert \leq \tfrac1{10} (\td\tau - \tau + 10)^{-10} + \tfrac1{10} (\td\tau' - \tau' + 10)^{-10} + C' e^{\tau} = O((-\tau)^{-10}), \]
which implies
\[  \big\Vert \ov U (\tau) - \ov U'\big({- \log(e^{-\tau}-t_0) }\big) \big\Vert \leq O((-\tau)^{-3}) \]
Since $\ov U'$ is Lipschitz, this implies that
\[ \big\Vert \ov U (\tau) - \ov U'(\tau) \big\Vert \leq O((-\tau)^{-3}) + O(e^{\tau}) = O((-\tau)^{-3}), \]
so again $\ov U = \ov U'$ by Lemma~\ref{Lem_ODE_sol_same}.

The second identity \eqref{eq_Q_identitites} and \eqref{eq_Q_identity_2} follow directly from Lemma~\ref{Lem_ODE_UQ}.
Lastly, if $\Qu(\MM) = 0$, then by Proposition~\ref{Prop_dom_qu_asymp}\ref{Prop_dom_qu_asymp_d} the flow $\MM$ cannot have dominant quadratic mode.
In this case we have exponential convergence by Propositions~\ref{Prop_ab_exist}, which is preserved under translation and parabolic scaling.
\end{proof}
\bigskip

\subsection{Proof of Proposition~\ref{Prop_same_Q_close}}

\begin{proof}[Proof of Proposition~\ref{Prop_same_Q_close}.]
As in the proof of Proposition~\ref{Prop_dom_qu_asymp} we may assume again without loss of generality that $\td\tau = -10$.
Apply Proposition~\ref{Prop_dom_qu_asymp} with $J \leftarrow J+2$.
In the following we will show the bound
\begin{equation} \label{eq_desired_diff_bound}
 \Vert U^{(1),+}(\tau) - U^{(0),+}(\tau) \Vert_{L^2_f} \leq C(J) |\tau|^{-J-1}. 
\end{equation}
Indeed, due to Assertions~\ref{Prop_dom_qu_asymp_a} of this proposition it suffices to prove the same bound for the corresponding functions $U^{(1),+}$ and $U^{(0),+}$ and if \eqref{eq_desired_diff_bound} is true, then by Lemma~\ref{Lem_polynomial_bounds} we have
\begin{multline*}
    \Vert U^{(1),+}(\tau) - U^{(0),+}(\tau) \Vert_{L^2_f} 
\leq C(J,m) R^{C(J,m)}(\tau) \Vert U^{(1),+}(\tau) - U^{(0),+}(\tau) \Vert_{L^2_f} \\
\leq C(J,m) (\log (-\tau))^{C(J,m)} |\tau|^{-J-1} 
\leq C(J,m) |\tau|^{-J}.
\end{multline*}

In the following $C$ will denote a generic uniform constants, which may depend on $J$.
Let $V(\tau) := U^{(1),+}(\tau) - U^{(0),+}(\tau)$ and write $V(\tau) = V_+(\tau) + V_0(\tau) + V_-(\tau) \in \sV_{\rot,> 0} \oplus \sV_{\rot,0} \oplus \sV_{\rot,< 0}$.
By Proposition~\ref{Prop_dom_qu_asymp}\ref{Prop_dom_qu_asymp_c} we have $\Vert U^{(i),+}(\tau)\Vert \leq C |\tau|^{-1}$, so by combining Proposition~\ref{Prop_dom_qu_asymp}\ref{Prop_dom_qu_asymp_b} for both flows we get for some $C_0 = C_0(J) > 0$, which we will henceforth fix,
\begin{align}
 \Vert \partial_\tau V_{\pm} - L V_{\pm} \Vert &\leq C_0|\tau|^{-1} \Vert V \Vert + C_0 |\tau|^{-J-2},  \label{eq_partialtauVpm}\\
  \Vert \partial_\tau V_{0}  \Vert &\leq C_0|\tau|^{-1} \Vert V \Vert + C_0 |\tau|^{-J-2}, \label{eq_partialatauV0} \\
 \Vert \partial_\tau V_0  - \PP_{\sV_{0}} ( Q_2^+ (U^{(1),+}) - Q_2^+ (U^{(0),+}) ) \Vert &\leq C_0|\tau|^{-2} \Vert V \Vert + C_0 |\tau|^{-J-2}.  \label{eq_evol_V0}
\end{align}

\begin{Claim}
There is a time $\tau_0 \in [-\infty, -10]$ such that a bound of the form \eqref{eq_desired_diff_bound} holds on $[\tau_0, -10]$ and for $\tau \leq \tau_0$ we have
\begin{equation} \label{eq_VpV0}
 \Vert V_{\pm} (\tau) \Vert \leq C |\tau|^{-1}   \Vert V_0 (\tau) \Vert + C |\tau|^{-J-1}. 
\end{equation}

\end{Claim}

\begin{proof}
Let $c > 0$ such that the only eigenvalue of $L$ contained in $(-10c,10c)$ is $0$ and consider the functions
\[ f_1 (\tau) := \Vert V_- (\tau) \Vert, \qquad f_2(\tau) :=  \tfrac{C_0}{c} |\tau|^{-1}   \Vert V_0 (\tau) \Vert +  |\tau|^{-J-1}, \qquad
f_3(\tau) := \Vert V_+(\tau) \Vert.   \]
Note that
\[ \Vert V (\tau) \Vert \leq f_1(\tau) + \tfrac{c}{C_0} |\tau| f_2(\tau) + f_3(\tau). \]
Define $I_i \subset (-\infty,-10]$ to be the subset of times at which $\max \{ f_1, f_2, f_3 \} = f_i$.
Then for $\tau \in I_1$ with $\tau \leq -C(C_0, J)$ we obtain from \eqref{eq_partialtauVpm}
\begin{multline} \label{eq_f1p}
 f'_1 (\tau) \leq -10 c \Vert V_-(\tau) \Vert + C_0 |\tau|^{-1} \Vert V(\tau) \Vert + C_0|\tau|^{-J-2} \\
 \leq -10 c f_1(\tau) + C_0 |\tau|^{-1} \big( 2 + \tfrac{c}{C_0} |\tau| \big) f_1(\tau) + cf_1(\tau) 
 \leq - 5c f_1(\tau).
 \end{multline}
Similarly, for $\tau \in I_3$ with  $\tau \leq -C(C_0, J)$
\[  f'_3(\tau) \geq 5c f_3(\tau). \]
We also obtain that if $\tau \in I_2$ with $\tau \leq -C(C_0, J)$, then by \eqref{eq_partialatauV0}
\begin{align*}
 |f'_2(\tau)| &\leq \tfrac{C_0}{c} |\tau|^{-2}   \Vert V_0 (\tau) \Vert + \tfrac{C_0}{c} |\tau|^{-1} \cdot C_0 |\tau|^{-1} \Vert V(\tau) \Vert + \tfrac{C_0}{c} |\tau|^{-1} \cdot |\tau|^{-J-2} + (J+1) |\tau|^{-J-2} \\
 &\leq |\tau|^{-1} f_2(\tau) + \tfrac{C_0^2}{c} |\tau|^{-2}  \big( 2 + \tfrac{c}{C_0}|\tau| \big) f_2(\tau) + \tfrac{C_0}{c} |\tau|^{-1} f_2(\tau) + (J+1)|\tau|^{-1} f_2(\tau) 
 \leq 2cf_2(\tau).
\end{align*}
So there is a time $T = T(C_0, J) \geq 10$ such that whenever $\tau \in I_i \cap I_j \cap (-\infty,-T]$ with $i < j$, then then $f'_i (\tau) < f'_j(\tau)$.
It follows that $I_1 \cap (-\infty, -T], I_2\cap (-\infty, -T], I_3\cap (-\infty, -T]$ decompose $(-\infty,-T]$ into (possibly empty) consecutive sub-intervals in this order.
So if $I_1 \cap (-\infty,-T]$ is non-empty, then it is of the form $(-\infty, -\tau_1]$.
However, by \eqref{eq_f1p} the function $f_1$ decreases exponentially on this interval, which contradicts the fact that it is uniformly bounded and non-zero.
So $I_1 \cap (-\infty,-T] = \emptyset$ and we must have $I_2 \cap (-\infty, -T] = (-\infty, \tau_0]$ and $I_3 \cap (-\infty, -T] = [\tau_0, -T]$ for some $\tau_0 \in [-\infty, -T]$.
The bound \eqref{eq_VpV0} holds on $(-\infty, -\tau_0]$ by definition.
The bound \eqref{eq_desired_diff_bound} holds on $[-T,-10]$ for a suitable $C(J)$, because the left-hand side is bounded from above by a uniform constant and $T$ only depends on $J$.
For $\tau \in [\tau_0, -T]$ the bound \eqref{eq_desired_diff_bound} implies $$f_3(\tau) \leq e^{-5c(-T-\tau)} f_3(-T) \leq C(J) e^{-5c |\tau|} \leq C(J) |\tau|^{-J},$$ which finishes the proof of the claim.
\end{proof}

So it remains to show \eqref{eq_desired_diff_bound} on $(-\infty, \tau_0]$.
Due to \eqref{eq_VpV0} it suffices to show that $\Vert V_0 \Vert \leq C(J) |\tau|^{-J}$.
To do this let us now consider the evolution inequality \eqref{eq_evol_V0} and write $\PP_{\sV_{\rot,0}} ( Q_2^+ (U')) = Q(U', U')$ as a symmetric bilinear form $Q : \sV_{\rot, \geq -J-2} \times \sV_{\rot, \geq -J-2} \to \sV_{\rot, 0}$, so
\begin{equation} \label{eq_Pv0Qp2}
  \PP_{\sV_{0}} ( Q_2^+ (U^{(1),+}) - Q_2^+ (U^{(0),+}) ) =   Q( U^{(0),+}, V) + Q( U^{(1),+}, V).  
\end{equation}
By Proposition~\ref{Prop_dom_qu_asymp}\ref{Prop_dom_qu_asymp_c} and \eqref{eq_VpV0} we have for $i = 0,1$
\begin{multline} \label{eq_UpIk}
 \big\Vert Q( U^{(i),+}, V) - Q( U^{(i)}_0, V_0)  \big\Vert
 \leq \big\Vert Q( U^{(i),+}, V) - Q( U^{(i)}_0, V)  \big\Vert + \big\Vert Q( U^{(i)}_0, V) - Q( U^{(i)}_0, V_0)  \big\Vert  \\
\leq 
C |\tau|^{-2} \Vert V \Vert + C|\tau|^{-1} \Vert V - V_0 \Vert \leq C |\tau|^{-2} \Vert V_0 \Vert + C |\tau|^{-J-2}. 
\end{multline}
Combining this bound with \eqref{eq_evol_V0} and \eqref{eq_Pv0Qp2} implies that on $(-\infty, \tau_0]$
\[ \Vert \partial_\tau V_0  - Q( U^{(0)}_0, V_0) - Q( U^{(1)}_0, V_0) \Vert \leq C|\tau|^{-2} \Vert V_0 \Vert +C |\tau|^{-J-2}.  \]
Using the fact that $Q(U', U'') = -\sqrt{2} U' U''$ from Lemma~\ref{Lem_Q2} and the bound from Proposition~\ref{Prop_dom_qu_asymp}\ref{Prop_dom_qu_asymp_c} for $\eta = \frac1{10}$, we obtain using the operator norm $\Vert \cdot \Vert_{op}$
\[ \partial_\tau \Vert V_0 \Vert_{op} \leq 2 \cdot \sqrt 2 \cdot \frac{1.1}{\sqrt{2}} |\tau|^{-1} \Vert V_0 \Vert_{op} + C |\tau|^{-2} \Vert V_0 \Vert_{op} + C|\tau|^{-J-2}. \]
It follows that there is a constant $T = T(J)$ such that on $(-\infty, \min \{ \tau_0, -T \} ]$
\[ \partial_\tau \Vert V_0 \Vert_{op} \leq 2.3 |\tau|^{-1} \Vert V_0 \Vert_{op} +  C |\tau|^{-J-2}. \]
So
\[ \partial_\tau \big( |\tau|^{2.3} \Vert V_0 \Vert_{op} \big) \leq C |\tau|^{-J-2+2.3}. \]
Since $\Qu(\MM^{(0)}) = \Qu(\MM^{(1)})$, the term in the parantheses must go to zero as $\tau \to -\infty$, which implies that on $(-\infty, \min \{ \tau_0, -T \} ]$
\[ \Vert V_0(\tau) \Vert_{op} \leq C |\tau|^{-J-1}. \]
This proves the desired bound \eqref{eq_desired_diff_bound} for $\tau \leq -T$.
On the other hand, since $T$ only depends on $J$, the bound for $\tau \geq -T$ holds trivially for a sufficiently large choice of $C(J)$.
\end{proof}
\bigskip

\subsection{Proof of Proposition~\ref{Prop_Q_continuous}}

\begin{proof}[Proof of Proposition~\ref{Prop_Q_continuous}.]
In all three assertions, we may pass to subsequences without loss of generality.
In fact, this is already part of the statements of Assertions~\ref{Prop_Q_continuous_b}, \ref{Prop_Q_continuous_c}.
For Assertion~\ref{Prop_Q_continuous_a} recall that it suffices to establish subsequential convergence for every given subsequence.

So in all assertions, may assume that we have convergence $\MM^{(i)} \to \MM^{(\infty)}$ in the Brakke sense, possibly after passing to a subsequence.
By the same argument, we may also assume that all flows $\MM^{(i)}$, $i < \infty$, have either dominant quadratic mode or none do so.
In the second case the assumptions follow using Theorem~\ref{Thm_bowl_unique} and Proposition~\ref{Prop_ab_exist}.
So assume that all flows $\MM^{(i)}$, $i < \infty$, have  dominant quadratic mode.
\medskip

Let us first prove Assertion~\ref{Prop_Q_continuous_a}, so assume that $\MM^{(\infty)}$ is also asymptotically $(n,k)$-cylindrical.
Then Proposition~\ref{Prop_dom_qu_asymp} can be applied for $J=m=10$ and $\eta = \frac1{10}$ and for a uniform choice of $\td\tau$, as long as $i$ is large enough.
Let $U^{+, (i)} : (-\infty, \td\tau) \to \sV_{\rot,\geq -10}$ be the corresponding functions and  and $\ov U^{(i)}$ the corresponding ODE solutions from Proposition~\ref{Prop_dom_qu_asymp}\ref{Prop_dom_qu_asymp_d}.
Combining the bounds of Assertions~\ref{Prop_dom_qu_asymp_c}, \ref{Prop_dom_qu_asymp_d} from Proposition~\ref{Prop_dom_qu_asymp} implies uniform bounds on $\ov U^{(i)} (\td\tau)$.
So after passing to another subsequence, we may assume that $\ov U^{(i)}$ converges to another solution $\ov U^{\prime, (\infty)}$ of the ODE \eqref{eq_barU_ODE}.

Suppose first that $\MM^{(\infty)}$ also has dominant quadratic mode and repeat the previous construction for $i = \infty$.
Taking \eqref{eq_utauUptau_qu} to the limit implies that for any fixed $\tau \leq \td\tau$ we have for large $i$
\[ \Vert U^{+,(i)}(\tau) - U^{+,(\infty)}(\tau) \Vert \leq 2\eta(\td\tau - \tau + 10)^{-J-1} . \]
So by combining this with \eqref{eq_U0_close_ovU} yields for large $i$
\[ \big\Vert \ov U^{(i)}(\tau) - \ov U^{(\infty)}(\tau) \big\Vert \leq 2C(\td\tau - \tau + 10)^{-3} + 2\eta(\td\tau - \tau + 10)^{-J-1} \leq 3C (\td\tau - \tau + 10)^{-3} . \]
Taking this to the limit implies that for all $\tau \leq \td\tau$
\[  \big\Vert \ov U^{(\infty),\prime}(\tau) - \ov U^{(\infty)}(\tau) \big\Vert \leq 3C (\td\tau - \tau + 10)^{-3} . \]
So by Lemma~\ref{Lem_ODE_sol_same} we have  $\ov U^{(\infty),\prime} = \ov U^{(\infty)}$, which implies that $\ov U^{(i)} \to \ov U^{(\infty)}$ and therefore $\Qu(\MM^{(i)}) \to \Qu(\MM^{(\infty)})$.

Next suppose that $\MM^{(\infty)}$ does \emph{not} have dominant quadratic mode and we need to show that $\Qu(\MM^{(i)}) \to 0$, which is equivalent to $\ov U^{(\infty),\prime} \equiv 0$.
In this case, we can apply Proposition~\ref{Prop_PO_ancient} to deduce exponential convergence of the rescaled flow to a cylinder as $\tau \to -\infty$.
(If $n'=0$, then we can also apply Theorem~\ref{Thm_bowl_unique} and use basic knowledge of the evolution of both model metrics.)
This implies that for every fixed $\tau \leq \td\tau$ the  functions $u^{(i)}_\tau$ from Proposition~\ref{Prop_dom_qu_asymp} satisfy the following bound for for large $i$:
\[ \big\Vert u_\tau^{(i)} \big\Vert_{C^m(\IB_1^k \times \IS^{n-k})} \leq C' e^{\tau/2}. \]
Here $C' > 0$ is a constant that is independent of $\tau$.
Arguing as before, we obtain that for any fixed $\tau \leq \td\tau$ we have the following bound for large $i$
\[ \Vert U^{+,(i)}(\tau)  \Vert \leq 2\eta(\td\tau - \tau + 10)^{-J-1} + C' e^{\tau/2} , \]
which implies
\[ \big\Vert \ov U^{(i)} (\tau) \big\Vert \leq 3C(\td\tau- \tau + 10)^{-3} + C' e^{\tau/2} . \]
Taking $i \to \infty$ implies 
\[  \big\Vert \ov U^{(\infty),\prime}(\tau)  \big\Vert \leq 3C (\td\tau - \tau + 10)^{-3} + C' e^{\tau/2} , \]
so $\ov U^{(\infty),\prime} \equiv 0$ by Lemma~\ref{Lem_ODE_sol_same}.
This concludes the proof of Assertion~\ref{Prop_Q_continuous_a}.
\medskip

To prove Assertion~\ref{Prop_Q_continuous_b}, fix a constant $\delta_0 > 0$, whose value we will choose later and recall that $\MM^{(i)} \to \MM^{(\infty)}$ in the Brakke sense.
For each $i$ choose $\la_i > 0$ maximal with the property that $\la \MM^{(i)}$ is $\delta_0$-close to $M_{\cyl}$ at time $-1$ for all $\la \in (0, \la_i)$.
After passing to a subsequence, we may assume that $\la_i \to \la_\infty \in [0, \infty]$.
If $\la_\infty > 0$, then for all $\la \leq \la_\infty$ the flow $\la \MM^{(\infty)}$ is $2\delta_0$-close to $M_{\cyl}$ at time $-1$.
So if we choose $\delta_0$ smaller than some dimensional constant, then by Corollary~\ref{Cor_unique_tangent_infinity} there is an $S \in O(n+1)$ such that $S\MM^{(\infty)}$ is asymptotically cylindrical.
So if $\delta' > 0$ is a constant, whose value we will determine later, then $S\MM^{(\infty)}$ is $\delta'$-close to $M_{\cyl}$ at scale $\sqrt{-t'_{\delta'}}$ at some time $t'_{\delta'} < 0$ and hence $S\MM^{(i)}$ is $2\delta'$close to $M_{\cyl}$ at scale $\sqrt{-t'_{\delta'}}$ at time $t'_{\delta'}$ for large $i$.
But $\MM^{(i)}$ is $\delta'$-close to $M_{\cyl}$ at scale $\sqrt{-t''_{\delta'}}$ at some time $t''_{\delta'} < t'_{\delta'}$.
So if $\delta'$ is small enough, then we obtain a contradiction using Theorem~\ref{Thm_stability_necks} unless $S M_{\cyl}=M_{\cyl}$, which implies that $\MM^{(\infty)}$ is asymptotically cylindrical.

Assume now that $\la_\infty = 0$.
After passing to a subsequence, we may assume that $\la_i \MM^{(i)} \to \MM^{(\infty), \prime}$ in the Brakke sense.
By the same argument as before, we obtain that for all $\la \leq 1$ the flow $\la \MM^{(\infty), \prime}$ is $2\delta_0$-close to $M_{\cyl}$ at time $-1$ and hence $\MM^{(\infty), \prime}$ is asymptotically cylindrical.
Assertion~\ref{Prop_Q_continuous_a} and Proposition~\ref{Prop_Q_basic_properties} imply that
\[ \Qu(\MM^{(\infty), \prime}) = \lim_{i \to \infty} \Qu (\la_i \MM^{(i)}) = \lim_{i \to \infty} \la_i \Qu(\MM^{(i)}) = 0. \] 
So by Theorem~\ref{Thm_bowl_unique}, the flow $\MM^{(\infty), \prime}$ is homothetic to $\MM_{\cyl}$ or $\MM_{\bowl} \times \IR^{k-1}$.
If $\MM^{(\infty), \prime}$ is regular near $(\bO, 0)$, then $\MM^{(\infty)}$ is either empty or equal to the tangent space of $\MM^{(\infty), \prime}$ at this point.
Suppose now that $\MM^{(\infty), \prime}$ is not regular near $(\bO, 0)$, so it must be equal to $\MM_{\cyl}$.
By construction we can find $\la'_i \to 1$ such that $\la'_i \la_i \MM^{(i)}$ is \emph{not} $\delta_0$-close to $M_{\cyl}$ at time $-1$.
However, this contradicts the local smooth convergence $\la_i \MM^{(i)} \to \MM_{\cyl} = \MM^{(\infty), \prime}$.
This concludes the proof of Assertion~\ref{Prop_Q_continuous_b}.
\medskip

Assertion~\ref{Prop_Q_continuous_c} is a direct consequence of Assertions~\ref{Prop_Q_continuous_a} and \ref{Prop_Q_continuous_b} and Theorem~\ref{Thm_bowl_unique}.
\end{proof}
\bigskip

\subsection{Proof of Theorem~\ref{Thm_existence_oval}}

\begin{proof}[Proof of Theorem~\ref{Thm_existence_oval}.]
We fix a Brakke distance $d_{Brakke}$ between convex, rotationally symmetric and $\IZ_2$-symmetric flows defined on time-intervals of the form $[-T,0)$.
The precise choice of this distance will be irrelevant; it will only be important that for any sequence of such mean curvature flows $\MM_i$, defined over time-intervals $[-T_i, 0)$ or $(-\infty,0)$, and a flow $\MM_\infty$, defined over $(-\infty, 0)$ the statement $d_{Brakke}(\MM_i,\MM_\infty) \to 0$ is equivalent to $T_i \to \infty$ and $\MM_i \to \MM_\infty$ in the Brakke sense.

Let $\mathfrak M$ be the space of all rotationally symmetric, $\IZ_2$-symmetric, convex, asymptotically cylindrical flows that go extinct at time $0$ and consider the map $\Qu : \mathfrak M \to \IR^{k \times k}$, whose image only consists of diagonal maps due to the $\IZ_2$-symmetry (see Proposition~\ref{Prop_Q_basic_properties}).
So we may view this map as a map of the form
\[ \Qu : \mathfrak M \to \IR^{k}_{\geq 0} \]
and we need to show that this map is surjective.

\begin{Claim} \label{Cl_Q_cont_proper}
$\Qu$ is a continuous proper map.
\end{Claim}

\begin{proof}
The continuity follows from Proposition~\ref{Prop_Q_continuous}.
For the properness, we use Assertion~\ref{Prop_Q_continuous_b} of that proposition: if $\Qu(\MM_i)$ is uniformly bounded, then we have subsequential convergence $\MM_i \to \MM_\infty$, which may be asymptotically cylindrical, empty or an affine plane.
Since all $\MM_i$ go extinct exactly at time $0$, the same is true for the limit, which rules out the last two cases; so $\MM_\infty \in \mathfrak M$.
\end{proof}

Fix $\Qu' = (Q'_1, \ldots, Q'_k) \in \IR^k_{\geq 0}$ for the remainder of the proof.
We will show that $\Qu' \in \Qu(\MM)$.
By induction on $k$, we may assume that $Q'_1, \ldots, Q'_k > 0$.

Let $\delta \in (0,1)$ be a constant whose value we will determine later and let $\Delta^k_\delta := \{ x_i \geq 0, x_1+\ldots + x_k \leq \delta\} \subset \IR^k$ be a simplex of size $\delta$.
For any $\mathbf b = (\mathbf b_1, \ldots, \mathbf b_k) \in \Delta^k_\delta$ consider the maximal mean curvature flow $\MM^{\mathbf b}$ on $[-T_{\mathbf b}, 0)$ starting from the ellipsoid
\[ \MM^{\mathbf b}_{-T_{\mathbf b}} =  \big\{ (\mathbf b_1 x_1)^2 + \ldots + (\mathbf b_k x_k)^2  +  x_{k+1}^2 + \ldots +  x_{n+1}^2 = 1 \big\}, \] 
where $T_{\mathbf b}$ is chosen such that the extinction time is $0$.
Note that $T_{\mathbf b}$ and $\MM^{ \mathbf b}$ depend continuously on the parameter $\mathbf b$.

\begin{Claim} \label{Cl_may_take_limit}
There are constants $\ov\eps, \delta  > 0$ with the following property.
Suppose that $a_i \to \infty$, $\mathbf b_i \in \Delta^k_\delta$ and $d_{Brakke}(a_i \MM^{ \mathbf b_i},\MM'_i) \leq \ov\eps$ for some sequence $\MM'_i \in \mathfrak M$ with $|\Qu(\MM'_i)| \leq |\Qu'|$.
Then after passing to a subsequence, we have $a_i \MM^{\mathbf b_i} \to \MM_\infty \in \mathfrak M$ in the Brakke sense and $\mathbf b_i \to \bO$.
\end{Claim}

\begin{proof}
After passing to a subsequence, we may always assume that we have convergence $a_i \MM^{\mathbf b_i} \to \MM_\infty$ to an ancient flow, convex, rotationally symmetric and $\IZ_2$-symmetric flow.
As in the proof of Claim~\ref{Cl_Q_cont_proper}, it is clear that $\MM_\infty$ must go extinct at time $0$.
So it remains to prove that $\MM_\infty$ is asymptotically cylindrical.
Moreover, by Corollary~\ref{Cor_unique_tangent_infinity} and the symmetries it is enough to show that $\MM_{\infty, t}$ is $\delta_0$-close to $M_{\cyl}$ at scale $\sqrt{-t}$ for $t \ll 0$, where $\delta_0 > 0$ is a dimensional constant.

Let $\delta_1 \in (0,\frac12 \delta_0)$ be a constant whose value we will determine later and choose $\delta \leq \ov\delta(\delta_1)$ such that for all $\mathbf b \in \Delta^k_\delta$ the initial time-slice $\MM^{\mathbf b}_{-T_{\mathbf b}}$ is $ \delta_1$-close to $M_{\cyl}$ at scale $\sqrt{T_{\mathbf b}}$.
For each $i$ choose $\theta_i \in [0,T_{\mathbf b_i}]$ minimal with the property that $ \MM^{\mathbf b_i}_{t}$ is $\frac12\delta_0$-close to $M_{\cyl}$ at scale $\sqrt{-t}$ for all $t\in [-T_{\mathbf b_i}, -\theta_i)$.
We now claim that if $\delta_1 \leq \ov\delta_1(\delta_0)$, then $\MM^{\mathbf b_i}_{t}$ is \emph{not} $\delta_1$-close to $M_{\cyl}$ at scale $\sqrt{-t}$ for all $t \in [-\theta_i, 0)$.
Indeed if $\MM^{\mathbf b_i}_{t}$ was $\delta_1$-close to $M_{\cyl}$ at scale $\sqrt{-t}$ for some $t \in [-\theta_i, 0)$, then we could apply Theorem~\ref{Thm_stability_necks} to the time-interval $[-T_{\mathbf b_i}, t]$ and conclude that $\MM^{\mathbf b_i}_{t}$ was $\frac12 \delta_0$-close  $M_{\cyl}$ at scale $\sqrt{-t'}$ for $t' \approx -\theta_i$, in contradiction to the minimal choice of $\theta_i$.

Suppose first that $a_i^2 \theta_i$ remains bounded.
Then by our discussion in the last paragraph we have $\MM_\infty \in \mathfrak M$.
{Moreover, since $\theta_i \to 0$, we must have $\mathbf b_i \to \bO$.
Specifically, if for a subsequence we would have $\mathbf b_i \to \mathbf b_\infty \neq \bO$, then $\MM^{\mathbf b_i} \to \MM^{\mathbf b_\infty}$, where the limit must develop a singularity modeled on a sphere or on a cylinder with strictly less $\IR$ factors than $M_{\cyl}$, which contradicts the choice of $\theta_i$ and the fact that $\theta_i \to 0$.}

Suppose now by contradiction that $a_i^2 \theta_i \to \infty$.
Then $\MM_{\infty, t}$ is \emph{not} $\frac12 \delta_1$-close to $M_{\cyl}$ at scale $\sqrt{-t}$ for all $t < 0$.
By the properness from Claim~\ref{Cl_Q_cont_proper} we may pass to a subsequence and assume that $\MM'_i \to \MM'_\infty \in \mathfrak M$ with $d_{Brakke}(\MM_\infty, \MM'_\infty) \leq \ov\eps$ and $|\Qu(\MM'_\infty)| \leq |\Qu'|$.

So, in summary, if the claim was false, then we would find a sequence $\ov\eps_j \to 0$, ancient flows $\MM^*_j$ with the property that $ \MM^*_{j,t}$ is \emph{not} $\frac12 \delta_1$-close to $M_{\cyl}$ at scale $\sqrt{-t}$ for all $j$ and $t < 0$, as well as flows $\MM^{*, \prime}_j \in \mathfrak M$ such that $d_{Brakke}(\MM^*_j, \MM^{*, \prime}_j) \leq \ov\eps_j$ and $|\Qu(\MM^{*, \prime}_j)| \leq |\Qu'|$.
By the properness from Claim~\ref{Cl_Q_cont_proper} we may again pass to a subsequence such that $\MM^{*, \prime}_j \to \MM^{*, \prime}_\infty$.
It follows that $\MM^*_j \to \MM^{*, \prime}_\infty$.
However, since $\MM^{*, \prime}_\infty$ is asymptotically cylindrical, this contradicts the property of $\MM^*_j$ for large $j$.
\end{proof}

For $j = 1, \ldots, k$ let
\[ X_j := \IR^k_{\geq 0} \cap \Big\{  \frac{x_j}{Q'_j} \leq  \max_{i\neq j}\frac{x_i}{Q'_i} \Big\} \cap \ov\IB^k_{|\Qu'|}. \]
These sets are chosen such that
\[ \{ \Qu' \} = X_1 \cap \ldots \cap X_k \cap \partial(X_1 \cup \ldots \cup X_k), \]
where the boundary is taken within $\IR^k_{\geq 0}$; so $\partial(X_1 \cup \ldots \cup X_k)$ is equal to the intersection of a sphere with $\IR^k_{\geq 0}$.
For any $a$ and $\eps \in (0,\ov\eps]$ and $j = 1, \ldots, k$ we define the following covering $U_0^{a,\eps} \cup \ldots \cup U_k^{a,\eps} = \Delta^k_\delta$:
\begin{align*}
 U_j^{a,\eps} &:= \big\{ \mathbf b \in \Delta^k_\delta \;\; : \;\; \text{there is $\MM' \in \mathfrak M$ such that $\Qu(\MM') \in X_j$ and $d_{Brakke}( a\MM^{\mathbf b}, \MM') < \eps$} \big\} \\
 U_0^{a, \eps} &:= \big\{ \mathbf b \in \Delta^k_\delta \;\; : \;\; \text{$d_{Brakke}( a\MM^{\mathbf b}, \MM) > 
 \tfrac12\eps$ for all $\MM' \in \mathfrak M$ with $\Qu(\MM') \in X_1 \cup \ldots \cup X_k$} \big\} 
\end{align*}
Let $F_j := \Delta_\delta^k \cap \{ x_j = 0 \}$ be the face of $\Delta_\delta^k$ corresponding to the $j$-th coordinate plane for $j = 1, \ldots, k$, and let $F_0 := \Delta_\delta^k \cap \{ x_1 + \ldots + x_k = \delta \}$ be the face opposite the origin.

\begin{Claim}
If $a \geq \underline{a}(\eps)$, then
\[ F_0 \subset U_0^{a,\eps} \qquad \text{and} \qquad F_j \subset   U_0^{a,\eps} \cup  U_j^{a,\eps} \quad \text{for all} \quad j = 1, \ldots, k. \]
\end{Claim}

\begin{proof}
Fix $\eps > 0$ and suppose that one of the inclusions fails for some face $F_j$ and a sequence $a_i \to \infty$.
So there is a sequence $\mathbf b_i \in F_j$, but $\mathbf b_i \not\in U_0^{a_i,\eps} \cup  U_j^{a_i,\eps}$.
The fact $\mathbf b_i \not\in U_0^{a_i,\eps}$ implies that there is a $\MM'_i \in \mathfrak M$ with $|\Qu(\MM'_i)| \leq |\Qu'|$ such that $d_{Brakke}(a_i \MM^{\mathbf b_i} , \MM'_i) \leq \eps$.
So by Claim~\ref{Cl_may_take_limit} we get that $a_i \MM^{\mathbf b_i} \to \MM_\infty \in \mathfrak M$ and $\mathbf b_i \to 0$.
This rules out the case $j =0$.

Consider now the case $j > 0$, so the $j$-th component of $\mathbf b_i$ vanishes for all $i$.
So $a_i \MM^{\mathbf b_i}$ and thus also its limit split off a line.
It follows that the $j$-th component of $\Qu(\MM_\infty)$ vanishes as well.
But this implies that $\Qu(\MM_\infty) \in X_j$ and hence $\mathbf b_i \in U_j^{a_i,\eps}$ for large $i$.
\end{proof}

\begin{Claim} \label{Cl_0_in_U}
If $a \geq \underline a(\eps)$, then $\bO$ is not contained in the closure of $U^{a,\eps}_0$.
\end{Claim}

\begin{proof}
This follows from the fact that $\MM^{\bO}$ is the round shrinking cylinder restricted to the time-interval $[-(2(n-k))^{-1},0)$. 
So for large $a$ we have $d_{Brakke}( a\MM^{\bO}, \MM_{\cyl}) < 
 \tfrac12\eps$.
\end{proof}

\begin{Claim}
If $\eps \leq\ov\eps$ and $a \geq \underline a (\eps)$ there is a $\mathbf b_{a,\eps} \in U_0^{a,\eps} \cap \ldots \cap U_k^{a, \eps}$.
\end{Claim}

\begin{proof}
We may choose continuous functions $\varphi_0, \ldots, \varphi_k : \Delta^k_\delta \to [0,1]$ such that $U_j^{a,\eps} = \{ \varphi_j > 0 \}$.
After dividing these functions by $\sum_{j=0}^k \varphi_j > 0$, we may assume that $\sum_{j=0}^k \varphi_j = 1$.
For $A \geq 1$ consider the map $f_A : \Delta^k_\delta \to \Delta^k_\delta$ defined by $f_A(\mathbf b) = \delta (1-\varphi_0(\mathbf b))^A \cdot (\varphi_1(\mathbf b), \ldots, \varphi_k(\mathbf b))$.
By Brouwer's fixed point theorem there is a $\mathbf b_A \in \Delta^k_\delta$ such that $f_A(\mathbf b_A) = \mathbf b_A$.
It suffices to show that all entries of $\varphi_j (\mathbf b_A ) > 0$ for all $j = 0, \ldots, k$ and sufficiently large $A$.
Suppose first that $\varphi_0(\mathbf b_A) = 0$.
Then $\mathbf b_A = f_A(\mathbf b_A) \in F_0 \subset \{ \varphi_0 > 0 \}$, which is a contradiction.
Next suppose by contradiction that for any $A \geq 1$ we have $\varphi_{j_A}(\mathbf b_A) = 0$ for some $j_A \in \{ 1, \ldots, k \}$.
Then that $j$-th component of $\mathbf b_A = f_A(\mathbf b_A)$ vanishes, so $\mathbf b_A \in F_{j_A} \subset  \{ \varphi_0 > 0 \} \cup \{ \varphi_{j_A} > 0 \}$, which implies $\mathbf b_A \in \{ \varphi_0 > 0 \} = U^{a,\eps}_0$ and $\mathbf b_A \in F_{j_A}$.
So we can find a  sequence $A_i \to \infty$ such that $\mathbf b_{A_i} \to \mathbf b_\infty$ and $\mathbf b_{A_i}  \in F_j \cap (U^{a,\eps}_0 \setminus U^{a,\eps}_j)$ for a fixed $j \in \{ 1, \ldots, k \}$.
Since $\mathbf b_{A_i} \in  U^{a,\eps}_0$, we must have $\mathbf b_\infty \neq \bO$ due to Claim~\ref{Cl_0_in_U}.
But since $f_{A_i} (\mathbf b_{A_i}) = \mathbf b_{A_i}  \to \mathbf b_\infty$, this implies that $\varphi_0(\mathbf b_\infty) = 0$.
Since $F_j \subset   U_0^{a,\eps} \cup  U_j^{a,\eps}$, we hence must have $\mathbf b_\infty \in U_j^{a,\eps}$ and therefore $\mathbf b_{A_i} \in U_j^{a,\eps}$, which is a contradiction.%
\end{proof}

Now fix a small $\eps \in (0,\ov\eps]$  and consider a sequence $a_i \to \infty$.
Since $\mathbf b_{a_i,\eps}\in U_j^{a, \eps}$ for each $j=1,...,k$, we can find flows $\MM'_{j,\eps} \in \mathfrak M$ with $\Qu(\MM'_{j,\eps}) \in X_j$ such that
\[ d_{Brakke} (a_i \MM^{\mathbf b_{a_i,\eps}}, \MM'_{j,\eps}) < \eps. \]
So by Claim~\ref{Cl_may_take_limit}, we may pass to a subsequence such that  we have convergence in the Brakke sense $a_i \MM^{\mathbf b_{a_i,\eps}} \to \MM_{\infty,\eps} \in \mathfrak M$ with
\begin{equation} \label{eq_brakke_leqeps}
  d_{Brakke} (\MM_{\infty,\eps}, \MM'_{j,\eps}) \leq \eps. 
\end{equation}
In addition, the fact that $\mathbf b_{a_i,\eps} \in U^{a_i, \eps}_0$ implies that $\Qu(\MM_{\infty,\eps}) \not\in X_1 \cup \ldots \cup X_k$ (here we have taken $\MM' = \MM_{\infty,\eps}$ in the definition of $U^{a_i,\eps}_0$).

By the properness from Claim~\ref{Cl_Q_cont_proper} we can find a sequence $\eps_i \to 0$ such that $\MM'_{j,\eps_i} \to \MM'_{j}\in \mathfrak M$.
Taking \eqref{eq_brakke_leqeps} to the limit implies that $\MM'_{1} = \ldots = \MM'_{k} =: \MM'$ and $\MM_{\infty,\eps_i} \to \MM'$.
So by the continuity of $\Qu$ we have $\Qu(\MM') \in X_1 \cap \ldots \cap X_k \setminus \Int (X_1 \cup \ldots \cup X_k) = \{ \Qu' \}$, as desired.

To see that $\MM'$ has uniformly bounded second fundamental form on time-intervals of the form $(-\infty,T]$ for $T < 0$, suppose by contradiction that there is a sequence $(\bp_i,t_i) \in \spt \MM'$ with $|\mathbf H|(\bp_i, t_i) \to \infty$. {Since $\MM'$ has compact time-slices by  Lemma~\ref{Lem_time_slice_compact} below below, this implies that $t_i \to -\infty$.}
By Proposition~\ref{Prop_Q_continuous} we can pass to a subsequence such that we have convergence $\MM' - (\bp_i, t_i) \to \MM''$ in the Brakke sense, where the limit must be convex asymptotically $(n,k)$-cylindrical with $\Qu(\MM'') = \Qu'$.
Again by Lemma~\ref{Lem_time_slice_compact} below, $\MM''$ must have compact time-slices, which would imply that $\diam (\spt \MM'_{t_i - 1})$ is uniformly bounded.
{However, this contradicts the fact that $\MM'$ is asymptotically cylindrical.}
\end{proof}
\medskip

\begin{Lemma} \label{Lem_time_slice_compact}
Let $\MM$ be a smooth, convex asymptotically $(n,k)$-cylindrical flow in $\IR^{n+1} \times (-\infty,T)$ such that $\Qu(\MM)$ is invertible.
Then $\MM$ must have compact time-slices.
\end{Lemma}

\begin{proof}
Apply Proposition~\ref{Prop_dom_qu_asymp} to the rescaled flow $\td\MM$ for $J, m \geq 10$.
Assume by contradiction that $\MM$ does not have compact time-slices.
For every $\tau \ll 0$ there is a unit vector $\mathbf v_\tau \in \IR^k \times \bO^{n-k+1}$ such that the ray $\{ s \mathbf v_\tau \; : \; s \geq 0 \}$ that is contained in the region bounded by $(\spt \td\MM)_\tau$.
By convexity this implies that the same is true for the rays $\{ \bp + s \mathbf v_\tau \; : \; s > 0 \}$ whenever $\bp \in (\spt\td\MM)_\tau$.
Therefore, if we write $\mathbf v_\tau = (v_1(\tau), \ldots, v_k(\tau))$, then the function $u_\tau$ from Proposition~\ref{Prop_dom_qu_asymp} satisfies
\[ \sum_{j=1}^k v_j(\tau) \frac{\partial u_\tau}{\partial x^j} \geq 0. \]
By Assertion~\ref{Prop_dom_qu_asymp_a} of Proposition~\ref{Prop_dom_qu_asymp} this implies the asymptotic bound for every fixed $\bx \in \IR^k$
\[ \sum_{j=1}^k v_j(\tau) \frac{\partial U^+(\tau)}{\partial x^j} (\bx) \geq - O(|\tau|^{-10}) \]
and by Assertion~\ref{Prop_dom_qu_asymp_c} we obtain
\[ \sum_{j=1}^k v_j(\tau) \frac{\partial U_0(\tau)}{\partial x^j} (\bx) \geq - O(|\tau|^{-2}). \]
But by Assertion~\ref{Prop_dom_qu_asymp_d} we have local smooth convergence $(-\tau) U_0(\tau) \to -\sum_{j=1}^k \frac1{\sqrt{2}} \mathfrak p^{(2)}_{jj}$ as $\tau \to -\infty$, which gives a contradiction.
\end{proof}
\bigskip

\bibliography{references}	
\bibliographystyle{amsalpha}

\end{document}